\documentclass[a4paper,11pt]{memoir}
\captionnamefont{\usefont{T1}{ptm}{b}{n}\selectfont}
\captiontitlefont{\usefont{T1}{ptm}{m}{it}\selectfont}

\usepackage[utf8]{inputenc}
\usepackage[T1]{fontenc}
\usepackage[british]{babel}

\chapterstyle{bianchi}
\setafterparaskip{1ex}

\usepackage{makeidx}
\makeindex
\usepackage[square]{natbib}
\usepackage[svgnames]{xcolor}
\usepackage[colorlinks,urlcolor=DarkBlue,citecolor=DarkBlue,linkcolor=DarkBlue,anchorcolor=DarkBlue]{hyperref}
\usepackage{amsmath}
\usepackage{nicefrac}
\usepackage{wasysym}
\usepackage{booktabs}
\usepackage{ntheorem}
\usepackage{stmaryrd}
\usepackage{todonotes}
\usepackage{amssymb}
\usepackage{tgpagella}
\usepackage{todonotes}
\usepackage{enumitem}
\usepackage{longtable}
\usepackage{ifthen}
\setlist[enumerate]{label=\textit{(\roman*)}}
\setlist[enumerate]{itemsep=1pt,topsep=5pt}

\def\bold#1{\textbf{#1}}
\usepackage{tikz}
\usetikzlibrary{matrix,arrows,calc,positioning}

\tikzset{every scope/.style={>=angle 60,thick}}
\newcommand\Figref[1]{Figure \ref{#1}\ifthenelse{\value{page}=\pageref{#1}}{}{ on page \pageref{#1}}}
\newcommand\figref[1]{figure \ref{#1}\ifthenelse{\value{page}=\pageref{#1}}{}{ on page \pageref{#1}}}

\def\bt{\begin{center}\begin{tikzpicture}\matrix[matrix of math nodes,column sep=1cm, row sep=1cm]}
\def\bth{\begin{center}\begin{tikzpicture}\matrix[matrix of math nodes,column sep=.5cm, row sep=1cm]}
\def\et{\end{tikzpicture}\end{center}}
\def\btm#1#2{\begin{center}\begin{tikzpicture}\matrix[matrix of math nodes,column sep=#1, row sep=#2}
\def\bsc{\begin{scope}}
\def\esc{\end{scope}}

\def\|#1{\,\begin{tikzpicture}[baseline]\draw[-,thin](0,-3pt)--(0,6pt);\node[anchor=west] at (-1pt,-2pt) {$_{#1}\!\!\!$}; \end{tikzpicture}\,}

\def\hyphen{\,\text{-}\,}

\def\op{\mathit{op}}
\def\co{\mathit{co}}
\def\leq{\leqslant}
\def\geq{\geqslant}
\def\preceq{\preccurlyeq}
\def\succeq{\succcurlyeq}
\def\coloneq{\mathrel{\mathop:}=}
\def\eqcolon{=\mathrel{\mathop:}}
\def\cotensor{\oblong}
\def\otimes{\varotimes}

\def\Aut{\textit{Aut\,}}

\def\im{\mathbf{im}}

\def\Hom{\mathit{Hom}}

\def\ker{\mathsf{ker}}
\def\coker{\mathsf{coker}}
\def\Cent{\mathsf{Cent}}

\def\int{\mathsf{int}}
\def\Span{\mathsf{Span}}

\def\LatT1{\mathsf{Lat_{T1}}}
\newcommand{\id}{\operatorname{id}}
\newcommand{\can}{\operatorname{\mathit{can}}}

\def\k{\mathbf{k}}
\def\R{\mathbf{R}}

\def\bN{\mathbb{N}}

\def\bQ{\mathbb{Q}}

\def\bZ{\mathbb{Z}}
\def\bK{\mathbb{K}}
\def\bL{\mathbb{L}}
\def\bE{\mathbb{E}}
\def\bF{\mathbb{F}}

\def\Gal{\mathsf{Gal}}

\def\ad{\mathsf{ad}}
\def\Fix{\mathsf{Fix}}

\def\Aut{\mathsf{Aut}}

\def\End{\mathsf{End}}
\def\Der{\mathsf{Der}}
\def\Coalg{\mathit{Coalg}}
\def\Alg{\mathit{Alg}}

\def\Sub{\mathsf{Sub}}

\def\Id{\mathsf{Id}}

\def\coId{\mathsf{coId}}
\def\coid#1{\coId_{\mathit{#1}}}
\def\sup{\mathit{sup}}
\def\inf{\mathit{inf}}
\def\Quot{\mathsf{Quot}}
\def\qquot{\mathsf{Quot_{\mathit{gen}}}}
\def\qsub{\mathsf{Sub_{\mathit{gen}}}}
\def\qid{\mathsf{Id_{\mathit{gen}}}}
\def\id{\mathit{id}}
\def\dim{\mathit{dim}}
\def\Con{\mathsf{Con}}
\def\Map{\mathsf{Map}}

\def\Hom{\mathsf{Hom}}

\def\Vect{\mathit{Vect}}
\def\Mod{\mathit{Mod}}
\def\mod#1{\Mod_{#1}}

\def\modulo{\mathit{mod}}

\theoremstyle{break}
\theorembodyfont{\itshape}
\theoremheaderfont{\usefont{T1}{ptm}{b}{sl}\normalsize}
\newtheorem{theorem}{Theorem}[section] 

\newtheorem{proposition}[theorem]{Proposition}

\newenvironment{proof}{\par\noindent\textbf{Proof:}}{\nopagebreak[4]\hfill\(\square\)\medskip\\}
\newenvironment{proofof}[1]{\par\noindent\textbf{Proof of {#1}:}}{\nopagebreak[4]\hfill\(\square\)\medskip\\}

\newtheorem{lemma}[theorem]{Lemma}
\newtheorem{corollary}[theorem]{Corollary}
\newtheorem{definition}[theorem]{Definition}

\theoremstyle{nonumberbreak}
\theorembodyfont{\itshape}
\theoremheaderfont{\usefont{T1}{ptm}{b}{sl}\normalsize}
\renewtheorem{theorem*}{Theorem}
\renewtheorem{mtheorem*}{Main Theorem}
\renewtheorem{proposition*}{Proposition}
\renewtheorem{construction*}{Construction}
\renewtheorem{definition*}{Definition}
\renewtheorem{observation*}{Observation}
\renewtheorem{lemma*}{Lemma}
\theoremstyle{nonumberplain}
\renewtheorem{corollary*}{Corollary}

\theorembodyfont{\normalfont}
\theoremheaderfont{\usefont{T1}{ptm}{b}{sl}\normalsize}
\theoremstyle{plain}
\newtheorem{example}[theorem]{Example}

\theoremstyle{break}
\newtheorem{examplebr}[theorem]{Example}
\newtheorem{examplesbr}[theorem]{Examples}

\theoremstyle{nonumberplain}
\renewtheorem{example*}[theorem]{Example}
\renewtheorem{examples*}[theorem]{Examples}
\renewtheorem{exercise*}{Exercise}

\theorembodyfont{\slshape}
\theoremheaderfont{\usefont{T1}{ptm}{b}{sl}\normalsize}
\theoremstyle{plain}
\newtheorem{remark}[theorem]{Remark}
\newtheorem{note}[theorem]{Note}

\theoremstyle{break}
\newtheorem{remarkbr}[theorem]{Remark}

\theoremstyle{nonumberbreak}
\renewtheorem{notation*}{Notation}

\theoremstyle{nonumberplain}
\renewtheorem{remark*}{Remark}
\renewtheorem{note*}{Note}
\renewtheorem{question*}{Question}

\theorembodyfont{\normalfont}
\theoremheaderfont{\usefont{T1}{fxlf}{b}{ic}\normalsize}
\def\ov{\overline}
\def\fill#1{\hbox to #1pt{\hfill}}

\def\galoisskip{.8mm}
\def\galoislength{7mm}		
\def\mprrskip{1mm}		
\def\shortmprrskip{0.8mm}     	
\def\shortgaloislength{5mm} 
\def\longgaloislength{12mm} 
\def\arrowlength{6mm}
\def\arrowhight{2pt}

\def\longelementarrowlength{9mm} 	
\def\shortarrowlength{4.75mm}
\def\longarrowlength{10mm}

\def\horizontalskip{3pt}
\def\nodeposition{.5}		
\def\shortnodeposition{.3}	
\def\elmaptipheight{.9mm}	

\def\wmpr#1{\smash{\mathop{\hbox to 12pt{\rightarrowfill}}\limits^{#1}}}

\def\mpr#1{
\hspace{\horizontalskip}\begin{tikzpicture}[baseline=-1pt]
	\draw[->] (0,\arrowhight) -- node[above,pos=\nodeposition]{${\scriptstyle #1}$} (\arrowlength,\arrowhight);
\end{tikzpicture}
\hspace{\horizontalskip}}

\def\ir{
\hspace{\horizontalskip}\begin{tikzpicture}[baseline=-1pt]
	\draw[->] (0,\arrowhight) -- (\arrowlength,\arrowhight);
\end{tikzpicture}
\hspace{\horizontalskip}}

\def\lmonoir{
\hspace{\horizontalskip}\begin{tikzpicture}[baseline=-1pt]
	\draw[>->] (0,\arrowhight) -- (\longarrowlength,\arrowhight);
\end{tikzpicture}
\hspace{\horizontalskip}}

\def\lmonoir{
\hspace{\horizontalskip}\begin{tikzpicture}[baseline=-1pt]
    \draw[>->] (0,\arrowhight) -- (\longarrowlength,\arrowhight);
\end{tikzpicture} \hspace{\horizontalskip}}

\def\smpr#1{
\hspace{\horizontalskip}\begin{tikzpicture}[baseline]
    \draw[->] (0,\arrowhight) -- node[above,pos=\shortnodeposition]{${\scriptstyle #1}$} (\shortarrowlength,\arrowhight);
\end{tikzpicture} \hspace{\horizontalskip}}

\def\sir{
\hspace{\horizontalskip}\begin{tikzpicture}[baseline=-1pt]
    \draw[->] (0,\arrowhight) -- (\shortarrowlength,\arrowhight);
\end{tikzpicture} \hspace{\horizontalskip}}

\def\epr#1{
\hspace{\horizontalskip}\begin{tikzpicture}[baseline=-1pt]
    \draw[->>] (0,\arrowhight) -- node[above,pos=\nodeposition]{${\scriptstyle #1}$} (\arrowlength,\arrowhight);
\end{tikzpicture} \hspace{\horizontalskip}}

\def\eir{
\hspace{\horizontalskip}\begin{tikzpicture}[baseline=-1pt]
    \draw[->>] (0,\arrowhight) -- (\arrowlength,\arrowhight);
\end{tikzpicture} \hspace{\horizontalskip}}

\def\lmpr#1{
\hspace{\horizontalskip}\begin{tikzpicture}[baseline=-1pt]
    \draw[->] (0,\arrowhight) -- node[above,pos=\nodeposition]{${\scriptstyle #1}$} (\longarrowlength,\arrowhight);
\end{tikzpicture} \hspace{\horizontalskip}}

\def\smrp#1{\hspace{\horizontalskip}\smash{\mathop{\hbox to 12pt{\rightarrowfill}}\limits_{#1}}\hspace{\horizontalskip}} 

\def\smrp#1{\hspace{\horizontalskip}\smash{\mathop{\hbox to 12pt{\rightarrowfill}}\limits_{#1}}\hspace{\horizontalskip}} 

\def\mprr#1#2{
\hspace{\horizontalskip}\begin{tikzpicture}[baseline=-4pt]
\begin{scope}
\draw[->] (0,\mprrskip) -- node[above,pos=\nodeposition]{${\scriptstyle #1}$} (\galoislength,\mprrskip);
\draw[->] (0,-\mprrskip) -- node[below,pos=\nodeposition]{${\scriptstyle #2}$} (\galoislength,-\mprrskip);
\end{scope}
\end{tikzpicture}}

\def\smprr#1#2{
\hspace{\horizontalskip}\begin{tikzpicture}[baseline=-4pt]
\begin{scope}
\draw[->] (0,\shortmprrskip) -- node[above,pos=\nodeposition]{${\scriptstyle #1}$} (\shortgaloislength,\shortmprrskip);
\draw[->] (0,-\shortmprrskip) -- node[below,pos=\nodeposition]{${\scriptstyle #2}$} (\shortgaloislength,-\shortmprrskip);
\end{scope}
\end{tikzpicture}}

\def\lmprr#1#2{\begin{tikzpicture}[baseline=-4pt]
\hspace{\horizontalskip}
\begin{scope}
\draw[->] (0,\mprrskip) -- node[above,pos=\nodeposition]{${\scriptstyle #1}$} (\longgaloislength,\mprrskip);
\draw[->] (0,-\mprrskip) -- node[below,pos=\nodeposition]{${\scriptstyle #2}$} (\longgaloislength,-\mprrskip);
\end{scope}
\end{tikzpicture}
\hspace{\horizontalskip}}

\def\elr#1{
\hspace{\horizontalskip}\begin{tikzpicture}[baseline=-1pt]
\draw[->] (0,\arrowhight) -- node[above,pos=\nodeposition]{${\scriptstyle #1}$} (\arrowlength,\arrowhight);
\draw[-] ($(0,\arrowhight)+(0,\elmaptipheight)$) -- ($(0,\arrowhight)+(0,-\elmaptipheight)$);
\end{tikzpicture}
\hspace{\horizontalskip}}
\def\elmap#1{\elr{#1}} 

\def\ell#1{
\hspace{\horizontalskip}\begin{tikzpicture}[baseline=-1pt]
\draw[->] (0,\arrowhight) -- node[above,pos=\nodeposition]{${\scriptstyle #1}$} (-\arrowlength,\arrowhight);
\draw[-] ($(0,\arrowhight)+(0,\elmaptipheight)$) -- ($(0,\arrowhight)+(0,-\elmaptipheight)$);
\end{tikzpicture}
\hspace{\horizontalskip}}
\def\ellmap#1{\ell{#1}} 

\def\eli{
\hspace{\horizontalskip}\begin{tikzpicture}[baseline=-1pt]
\draw[->] (0,\arrowhight) --  (\arrowlength,\arrowhight);
\draw[-] ($(0,\arrowhight)+(0,\elmaptipheight)$) -- ($(0,\arrowhight)+(0,-\elmaptipheight)$);
\end{tikzpicture}
\hspace{\horizontalskip}}

\def\selr#1{
\hspace{\horizontalskip}\begin{tikzpicture}[baseline=-1pt]
\draw[->] (0,\arrowhight) -- node[above,pos=\nodeposition]{${\scriptstyle #1}$} (\shortarrowlength,\arrowhight);
\draw[-] ($(0,\arrowhight)+(0,\elmaptipheight)$) -- ($(0,\arrowhight)+(0,-\elmaptipheight)$);
\end{tikzpicture}
\hspace{\horizontalskip}}
\def\selmap#1{\selr{#1}} 

\def\selir{
\hspace{\horizontalskip}\begin{tikzpicture}[baseline=-1pt]
\draw[->] (0,\arrowhight) --  (\shortarrowlength,\arrowhight);
\draw[-] ($(0,\arrowhight)+(0,\elmaptipheight)$) -- ($(0,\arrowhight)+(0,-\elmaptipheight)$);
\end{tikzpicture}
\hspace{\horizontalskip}}
\def\seli{\selir} 

\def\lelr#1{
\hspace{\horizontalskip}
\begin{tikzpicture}[baseline=-1pt]
\draw[->] (0,\arrowhight) -- node[above,pos=\nodeposition]{${\scriptstyle #1}$} (\longelementarrowlength,\arrowhight);
\draw[-] ($(0,\arrowhight)+(0,\elmaptipheight)$) -- ($(0,\arrowhight)+(0,-\elmaptipheight)$);
\end{tikzpicture}
\hspace{\horizontalskip}}
\def\lelmap#1{\lelr{#1}}

\def\galois#1#2{
\hspace{\horizontalskip}
\begin{tikzpicture}[baseline=-4pt]
\begin{scope}
\draw[->] (0,\galoisskip) -- node[above,pos=.5]{${\scriptstyle #1}$} (\galoislength,\galoisskip);
\draw[<-] (0,-\galoisskip) -- node[below,pos=.6]{${\scriptstyle #2}$} (\galoislength,-\galoisskip);
\end{scope}
\end{tikzpicture}
\hspace{\horizontalskip}}

\def\lgalois#1#2{
\hspace{\horizontalskip}
\begin{tikzpicture}[baseline=-4pt]
\begin{scope}
\draw[->] (0,\galoisskip) -- node[above,pos=.5]{${\scriptstyle #1}$} (\longgaloislength,\galoisskip);
\draw[<-] (0,-\galoisskip) -- node[below,pos=.5]{${\scriptstyle #2}$} (\longgaloislength,-\galoisskip);
\end{scope}
\end{tikzpicture}
\hspace{\horizontalskip}}

\definecolor{snow} 		{rgb} 	{1, 0.98, 0.980}
\definecolor{GhostWhite} 	{rgb} 	{0.972, 0.972, 1}
\definecolor{WhiteSmoke} 	{rgb} 	{0.96, 0.96, 0.960}
\definecolor{gainsboro} 	{rgb} 	{0.862, 0.862, 0.862}
\definecolor{FloralWhite} 	{rgb} 	{1, 0.98, 0.941}
\definecolor{OldLace} 		{rgb} 	{0.992, 0.96, 0}
\definecolor{linen} 		{rgb} 	{0.98, 0.941, 0.90}
\definecolor{AntiqueWhite} 	{rgb} 	{0.98, 0.921, 0.843}
\definecolor{PapayaWhip} 	{rgb} 	{1, 0.937, 0.835}
\definecolor{BlanchedAlmond} 	{rgb} 	{1, 0.921, 0.8}
\definecolor{bisque} 		{rgb} 	{1, 0.894, 0.768}
\definecolor{PeachPuff} 	{rgb} 	{1, 0.854, 0.725}
\definecolor{NavajoWhite} 	{rgb} 	{1, 0.87, 0.678}
\definecolor{moccasin} 		{rgb} 	{1, 0.894, 0.70}
\definecolor{cornsilk} 		{rgb} 	{1, 0.972, 0.862}
\definecolor{ivory} 		{rgb} 	{1, 1, 0.941}
\definecolor{LemonChiffon} 	{rgb} 	{1, 0.98, 0.80}
\definecolor{seashell} 		{rgb} 	{1, 0.96, 0.933}
\definecolor{honeydew} 		{rgb} 	{0.941, 1, 0.941}
\definecolor{MintCream} 	{rgb} 	{0.96, 1, 0.980}
\definecolor{azure} 		{rgb} 	{0.941, 1, 1}
\definecolor{AliceBlue} 	{rgb} 	{0.941, 0.972, 1}
\definecolor{lavender} 		{rgb} 	{0.9, 0.9, 0.980}
\definecolor{LavenderBlush} 	{rgb} 	{1, 0.941, 0.960}
\definecolor{MistyRose} 	{rgb} 	{1, 0.894, 0.882}
\definecolor{white} 		{rgb} 	{1, 1, 1}
\definecolor{black} 		{rgb} 	{0, 0, 0}
\definecolor{DarkSlateGray} 	{rgb} 	{0.184, 0.3, 0.30}
\definecolor{DimGrey} 		{rgb} 	{0.411, 0.411, 0.411}
\definecolor{SlateGrey} 	{rgb} 	{0.439, 0.5, 0.564}
\definecolor{LightSlateGrey} 	{rgb} 	{0.466, 0.533, 0.6}
\definecolor{grey} 		{rgb} 	{0.745, 0.745, 0.745}
\definecolor{LightGrey} 	{rgb} 	{0.827, 0.827, 0.827}
\definecolor{MidnightBlue} 	{rgb} 	{0.098, 0.098, 0.439}
\definecolor{NavyBlue} 		{rgb} 	{0, 0, 0.5}
\definecolor{CornflowerBlue} 	{rgb} 	{0.392, 0.584, 0.929}
\definecolor{DarkSlateBlue} 	{rgb} 	{0.282, 0.239, 0.545}
\definecolor{SlateBlue} 	{rgb} 	{0.415, 0.352, 0.8}
\definecolor{MediumSlateBlue} 	{rgb} 	{0.482, 0.4, 0.933}
\definecolor{LightSlateBlue} 	{rgb} 	{0.517, 0.439, 1}
\definecolor{MediumBlue} 	{rgb} 	{0, 0, 0.8}
\definecolor{RoyalBlue} 	{rgb} 	{0.254, 0.411, 0.882}
\definecolor{DodgerBlue} 	{rgb} 	{0.117, 0.564, 1}
\definecolor{DeepSkyBlue} 	{rgb} 	{0, 0.749, 1}
\definecolor{SkyBlue} 		{rgb} 	{0.529, 0.8, 0.921}
\definecolor{LightSkyBlue} 	{rgb} 	{0.529, 0.8, 0.980}
\definecolor{SteelBlue} 	{rgb} 	{0.274, 0.5, 0.7}
\definecolor{LightSteelBlue} 	{rgb} 	{0.69, 0.768, 0.870}
\definecolor{LightBlue} 	{rgb} 	{0.678, 0.847, 0.9}
\definecolor{PowderBlue} 	{rgb} 	{0.69, 0.878, 0.9}
\definecolor{PaleTurquoise} 	{rgb} 	{0.686, 0.933, 0.933}
\definecolor{DarkTurquoise} 	{rgb} 	{0, 0.8, 0.819}
\definecolor{MediumTurquoise} 	{rgb} 	{0.282, 0.819, 0.8}
\definecolor{turquoise} 	{rgb} 	{0.25, 0.878, 0.815}
\definecolor{cyan} 		{rgb} 	{0, 1, 1}
\definecolor{LightCyan} 	{rgb} 	{0.878, 1, 1}
\definecolor{CadetBlue} 	{rgb} 	{0.372, 0.619, 0.627}
\definecolor{MediumAquamarine} 	{rgb} 	{0.4, 0.8, 0.666}
\definecolor{aquamarine} 	{rgb} 	{0.498, 1, 0.831}
\definecolor{DarkGreen} 	{rgb} 	{0, 0.392, 0}
\definecolor{DarkOliveGreen} 	{rgb} 	{0.333, 0.419, 0.184}
\definecolor{DarkSeaGreen} 	{rgb} 	{0.56, 0.737, 0.560}
\definecolor{SeaGreen} 		{rgb} 	{0.18, 0.545, 0.341}
\definecolor{MediumSeaGreen} 	{rgb} 	{0.235, 0.7, 0.443}
\definecolor{LightSeaGreen} 	{rgb} 	{0.125, 0.698, 0.666}
\definecolor{PaleGreen} 	{rgb} 	{0.596, 0.984, 0.596}
\definecolor{SpringGreen} 	{rgb} 	{0, 1, 0.498}
\definecolor{LawnGreen} 	{rgb} 	{0.486, 0.988, 0}
\definecolor{green} 		{rgb} 	{0, 1, 0}
\definecolor{chartreuse} 	{rgb} 	{0.498, 1, 0}
\definecolor{MediumSpringGreen} {rgb} 	{0, 0.98, 0.6}
\definecolor{GreenYellow} 	{rgb} 	{0.678, 1, 0.184}
\definecolor{LimeGreen} 	{rgb} 	{0.196, 0.8, 0.196}
\definecolor{YellowGreen} 	{rgb} 	{0.6, 0.8, 0.196}
\definecolor{ForestGreen} 	{rgb} 	{0.133, 0.545, 0.133}
\definecolor{OliveDrab} 	{rgb} 	{0.419, 0.556, 0.137}
\definecolor{DarkKhaki} 	{rgb} 	{0.741, 0.717, 0.419}
\definecolor{khaki} 		{rgb} 	{0.941, 0.9, 0.549}
\definecolor{pale} 		{rgb} 	{0.933, 0.9, 0.666}
\definecolor{PaleGoldenrod} 	{rgb} 	{0.933, 0.9, 0.666}
\definecolor{LightGoldenrodYellow} 	{rgb} 	{0.98, 0.98, 0.823}
\definecolor{LightYellow} 	{rgb} 	{1, 1, 0.878}
\definecolor{yellow} 		{rgb} 	{1, 1, 0}
\definecolor{gold} 		{rgb} 	{1, 0.843, 0}
\definecolor{LightGoldenrod} 	{rgb} 	{0.933, 0.866, 0.5}
\definecolor{goldenrod} 	{rgb} 	{0.854, 0.647, 0.125}
\definecolor{DarkGoldenrod} 	{rgb} 	{0.721, 0.525, 0.43}
\definecolor{RosyBrown} 	{rgb} 	{0.737, 0.56, 0.560}
\definecolor{IndianRed} 	{rgb} 	{0.8, 0.36, 0.360}
\definecolor{SaddleBrown} 	{rgb} 	{0.545, 0.27, 0.74}
\definecolor{sienna} 		{rgb} 	{0.627, 0.321, 0.176}
\definecolor{peru} 		{rgb} 	{0.8, 0.521, 0.247}
\definecolor{burlywood} 	{rgb} 	{0.87, 0.721, 0.529}
\definecolor{beige} 		{rgb} 	{0.96, 0.96, 0.862}
\definecolor{wheat} 		{rgb} 	{0.96, 0.87, 0.7}
\definecolor{SandyBrown} 	{rgb} 	{0.956, 0.643, 0.376}
\definecolor{tan} 		{rgb} 	{0.823, 0.7, 0.549}
\definecolor{chocolate} 	{rgb} 	{0.823, 0.411, 0.117}
\definecolor{FireBrick} 	{rgb} 	{0.698, 0.133, 0.133}
\definecolor{firebrick} 	{rgb} 	{0.698, 0.133, 0.133}
\definecolor{brown} 		{rgb} 	{0.647, 0.164, 0.164}
\definecolor{DarkSalmon} 	{rgb} 	{0.913, 0.588, 0.478}
\definecolor{salmon} 		{rgb} 	{0.98, 0.5, 0.447}
\definecolor{LightSalmon} 	{rgb} 	{1, 0.627, 0.478}
\definecolor{orange} 		{rgb} 	{1, 0.647, 0}
\definecolor{DarkOrange} 	{rgb} 	{1, 0.549, 0}
\definecolor{coral} 		{rgb} 	{1, 0.498, 0.313}
\definecolor{LightCoral} 	{rgb} 	{0.941, 0.5, 0.5}
\definecolor{tomato} 		{rgb} 	{1, 0.388, 0.278}
\definecolor{OrangeRed} 	{rgb} 	{1, 0.27, 0}
\definecolor{red} 		{rgb} 	{1, 0, 0}
\definecolor{HotPink} 		{rgb} 	{1, 0.411, 0.7}
\definecolor{DeepPink} 		{rgb} 	{1, 0.78, 0.576}
\definecolor{pink} 		{rgb} 	{1, 0.752, 0.796}
\definecolor{LightPink} 	{rgb} 	{1, 0.713, 0.756}
\definecolor{PaleVioletRed} 	{rgb} 	{0.858, 0.439, 0.576}
\definecolor{maroon} 		{rgb} 	{0.69, 0.188, 0.376}
\definecolor{MediumVioletRed} 	{rgb} 	{0.78, 0.82, 0.521}
\definecolor{VioletRed} 	{rgb} 	{0.815, 0.125, 0.564}
\definecolor{magenta} 		{rgb} 	{1, 0, 1}
\definecolor{violet} 		{rgb} 	{0.933, 0.5, 0.933}
\definecolor{plum} 		{rgb} 	{0.866, 0.627, 0.866}
\definecolor{orchid} 		{rgb} 	{0.854, 0.439, 0.839}
\definecolor{MediumOrchid} 	{rgb} 	{0.729, 0.333, 0.827}
\definecolor{DarkOrchid} 	{rgb} 	{0.6, 0.196, 0.8}
\definecolor{DarkViolet} 	{rgb} 	{0.58, 0, 0.827}
\definecolor{BlueViolet} 	{rgb} 	{0.541, 0.168, 0.886}
\definecolor{purple} 		{rgb} 	{0.627, 0.125, 0.941}
\definecolor{medium} 		{rgb} 	{0.576, 0.439, 0.858}
\definecolor{MediumPurple} 	{rgb} 	{0.576, 0.439, 0.858}
\definecolor{thistle} 		{rgb} 	{0.847, 0.749, 0.847}
\definecolor{snow1} 		{rgb} 	{1, 0.98, 0.980}
\definecolor{snow2} 		{rgb} 	{0.933, 0.913, 0.913}
\definecolor{snow3} 		{rgb} 	{0.8, 0.788, 0.788}
\definecolor{snow4} 		{rgb} 	{0.545, 0.537, 0.537}
\definecolor{seashell1} 	{rgb} 	{1, 0.96, 0.933}
\definecolor{seashell2} 	{rgb} 	{0.933, 0.898, 0.870}
\definecolor{seashell3} 	{rgb} 	{0.8, 0.772, 0.749}
\definecolor{seashell4} 	{rgb} 	{0.545, 0.525, 0.5}
\definecolor{AntiqueWhite1} 	{rgb} 	{1, 0.937, 0.858}
\definecolor{AntiqueWhite2} 	{rgb} 	{0.933, 0.874, 0.8}
\definecolor{AntiqueWhite3} 	{rgb} 	{0.8, 0.752, 0.690}
\definecolor{AntiqueWhite4} 	{rgb} 	{0.545, 0.513, 0.470}
\definecolor{bisque1} 		{rgb} 	{1, 0.894, 0.768}
\definecolor{bisque2} 		{rgb} 	{0.933, 0.835, 0.717}
\definecolor{bisque3} 		{rgb} 	{0.8, 0.717, 0.619}
\definecolor{bisque4} 		{rgb} 	{0.545, 0.49, 0.419}
\definecolor{PeachPuff1} 	{rgb} 	{1, 0.854, 0.725}
\definecolor{PeachPuff2} 	{rgb} 	{0.933, 0.796, 0.678}
\definecolor{PeachPuff3} 	{rgb} 	{0.8, 0.686, 0.584}
\definecolor{PeachPuff4} 	{rgb} 	{0.545, 0.466, 0.396}
\definecolor{NavajoWhite1} 	{rgb} 	{1, 0.87, 0.678}
\definecolor{NavajoWhite2} 	{rgb} 	{0.933, 0.811, 0.631}
\definecolor{NavajoWhite3} 	{rgb} 	{0.8, 0.7, 0.545}
\definecolor{NavajoWhite4} 	{rgb} 	{0.545, 0.474, 0.368}
\definecolor{LemonChiffon1} 	{rgb} 	{1, 0.98, 0.8}
\definecolor{LemonChiffon2} 	{rgb} 	{0.933, 0.913, 0.749}
\definecolor{LemonChiffon3} 	{rgb} 	{0.8, 0.788, 0.647}
\definecolor{LemonChiffon4} 	{rgb} 	{0.545, 0.537, 0.439}
\definecolor{cornsilk1} 	{rgb} 	{1, 0.972, 0.862}
\definecolor{cornsilk2} 	{rgb} 	{0.933, 0.9, 0.8}
\definecolor{cornsilk3} 	{rgb} 	{0.8, 0.784, 0.694}
\definecolor{cornsilk4} 	{rgb} 	{0.545, 0.533, 0.470}
\definecolor{ivory1} 		{rgb} 	{1, 1, 0.941}
\definecolor{ivory2} 		{rgb} 	{0.933, 0.933, 0.878}
\definecolor{ivory3} 		{rgb} 	{0.8, 0.8, 0.756}
\definecolor{ivory4} 		{rgb} 	{0.545, 0.545, 0.513}
\definecolor{honeydew1} 	{rgb} 	{0.941, 1, 0.941}
\definecolor{honeydew2} 	{rgb} 	{0.878, 0.933, 0.878}
\definecolor{honeydew3} 	{rgb} 	{0.756, 0.8, 0.756}
\definecolor{honeydew4} 	{rgb} 	{0.513, 0.545, 0.513}
\definecolor{LavenderBlush1} 	{rgb} 	{1, 0.941, 0.960}
\definecolor{LavenderBlush2} 	{rgb} 	{0.933, 0.878, 0.898}
\definecolor{LavenderBlush3} 	{rgb} 	{0.8, 0.756, 0.772}
\definecolor{LavenderBlush4} 	{rgb} 	{0.545, 0.513, 0.525}
\definecolor{MistyRose1} 	{rgb} 	{1, 0.894, 0.882}
\definecolor{MistyRose2} 	{rgb} 	{0.933, 0.835, 0.823}
\definecolor{MistyRose3} 	{rgb} 	{0.8, 0.717, 0.7}
\definecolor{MistyRose4} 	{rgb} 	{0.545, 0.49, 0.482}
\definecolor{azure1} 		{rgb} 	{0.941, 1, 1}
\definecolor{azure2} 		{rgb} 	{0.878, 0.933, 0.933}
\definecolor{azure3} 		{rgb} 	{0.756, 0.8, 0.8}
\definecolor{azure4} 		{rgb} 	{0.513, 0.545, 0.545}
\definecolor{SlateBlue1} 	{rgb} 	{0.513, 0.435, 1}
\definecolor{SlateBlue2} 	{rgb} 	{0.478, 0.4, 0.933}
\definecolor{SlateBlue3} 	{rgb} 	{0.411, 0.349, 0.8}
\definecolor{SlateBlue4} 	{rgb} 	{0.278, 0.235, 0.545}
\definecolor{RoyalBlue1} 	{rgb} 	{0.282, 0.462, 1}
\definecolor{RoyalBlue2} 	{rgb} 	{0.262, 0.431, 0.933}
\definecolor{RoyalBlue3} 	{rgb} 	{0.227, 0.372, 0.8}
\definecolor{RoyalBlue4} 	{rgb} 	{0.152, 0.25, 0.545}
\definecolor{blue1} 		{rgb} 	{0, 0, 1}
\definecolor{blue2} 		{rgb} 	{0, 0, 0.933}
\definecolor{blue3} 		{rgb} 	{0, 0, 0.8}
\definecolor{blue4} 		{rgb} 	{0, 0, 0.545}
\definecolor{DodgerBlue1} 	{rgb} 	{0.117, 0.564, 1}
\definecolor{DodgerBlue2} 	{rgb} 	{0.1, 0.525, 0.933}
\definecolor{DodgerBlue3} 	{rgb} 	{0.94, 0.454, 0.8}
\definecolor{DodgerBlue4} 	{rgb} 	{0.62, 0.3, 0.545}
\definecolor{SteelBlue1} 	{rgb} 	{0.388, 0.721, 1}
\definecolor{SteelBlue2} 	{rgb} 	{0.36, 0.674, 0.933}
\definecolor{SteelBlue3} 	{rgb} 	{0.3, 0.58, 0.8}
\definecolor{SteelBlue4} 	{rgb} 	{0.211, 0.392, 0.545}
\definecolor{DeepSkyBlue1} 	{rgb} 	{0, 0.749, 1}
\definecolor{DeepSkyBlue2} 	{rgb} 	{0, 0.698, 0.933}
\definecolor{DeepSkyBlue3} 	{rgb} 	{0, 0.6, 0.8}
\definecolor{DeepSkyBlue4} 	{rgb} 	{0, 0.4, 0.545}
\definecolor{SkyBlue1} 		{rgb} 	{0.529, 0.8, 1}
\definecolor{SkyBlue2} 		{rgb} 	{0.494, 0.752, 0.933}
\definecolor{SkyBlue3} 		{rgb} 	{0.423, 0.65, 0.8}
\definecolor{SkyBlue4} 		{rgb} 	{0.29, 0.439, 0.545}
\definecolor{LightSkyBlue1} 	{rgb} 	{0.69, 0.886, 1}
\definecolor{LightSkyBlue2} 	{rgb} 	{0.643, 0.827, 0.933}
\definecolor{LightSkyBlue3} 	{rgb} 	{0.552, 0.713, 0.8}
\definecolor{LightSkyBlue4} 	{rgb} 	{0.376, 0.482, 0.545}
\definecolor{SlateGray1} 	{rgb} 	{0.776, 0.886, 1}
\definecolor{SlateGray2} 	{rgb} 	{0.725, 0.827, 0.933}
\definecolor{SlateGray3} 	{rgb} 	{0.623, 0.713, 0.8}
\definecolor{SlateGray4} 	{rgb} 	{0.423, 0.482, 0.545}
\definecolor{LightSteelBlue1} 	{rgb} 	{0.792, 0.882, 1}
\definecolor{LightSteelBlue2} 	{rgb} 	{0.737, 0.823, 0.933}
\definecolor{LightSteelBlue3} 	{rgb} 	{0.635, 0.7, 0.8}
\definecolor{LightSteelBlue4} 	{rgb} 	{0.431, 0.482, 0.545}
\definecolor{LightBlue1} 	{rgb} 	{0.749, 0.937, 1}
\definecolor{LightBlue2} 	{rgb} 	{0.698, 0.874, 0.933}
\definecolor{LightBlue3} 	{rgb} 	{0.6, 0.752, 0.8}
\definecolor{LightBlue4} 	{rgb} 	{0.4, 0.513, 0.545}
\definecolor{LightCyan1} 	{rgb} 	{0.878, 1, 1}
\definecolor{LightCyan2} 	{rgb} 	{0.819, 0.933, 0.933}
\definecolor{LightCyan3} 	{rgb} 	{0.7, 0.8, 0.8}
\definecolor{LightCyan4} 	{rgb} 	{0.478, 0.545, 0.545}
\definecolor{PaleTurquoise1} 	{rgb} 	{0.733, 1, 1}
\definecolor{PaleTurquoise2} 	{rgb} 	{0.682, 0.933, 0.933}
\definecolor{PaleTurquoise3} 	{rgb} 	{0.588, 0.8, 0.8}
\definecolor{PaleTurquoise4} 	{rgb} 	{0.4, 0.545, 0.545}
\definecolor{CadetBlue1} 	{rgb} 	{0.596, 0.96, 1}
\definecolor{CadetBlue2} 	{rgb} 	{0.556, 0.898, 0.933}
\definecolor{CadetBlue3} 	{rgb} 	{0.478, 0.772, 0.8}
\definecolor{CadetBlue4} 	{rgb} 	{0.325, 0.525, 0.545}
\definecolor{turquoise1} 	{rgb} 	{0, 0.96, 1}
\definecolor{turquoise2} 	{rgb} 	{0, 0.898, 0.933}
\definecolor{turquoise3} 	{rgb} 	{0, 0.772, 0.8}
\definecolor{turquoise4} 	{rgb} 	{0, 0.525, 0.545}
\definecolor{cyan1} 		{rgb} 	{0, 1, 1}
\definecolor{cyan2} 		{rgb} 	{0, 0.933, 0.933}
\definecolor{cyan3} 		{rgb} 	{0, 0.8, 0.8}
\definecolor{cyan4} 		{rgb} 	{0, 0.545, 0.545}
\definecolor{DarkSlateGray1} 	{rgb} 	{0.592, 1, 1}
\definecolor{DarkSlateGray2} 	{rgb} 	{0.552, 0.933, 0.933}
\definecolor{DarkSlateGray3} 	{rgb} 	{0.474, 0.8, 0.8}
\definecolor{DarkSlateGray4} 	{rgb} 	{0.321, 0.545, 0.545}
\definecolor{aquamarine1} 	{rgb} 	{0.498, 1, 0.831}
\definecolor{aquamarine2} 	{rgb} 	{0.462, 0.933, 0.776}
\definecolor{aquamarine3} 	{rgb} 	{0.4, 0.8, 0.666}
\definecolor{aquamarine4} 	{rgb} 	{0.27, 0.545, 0.454}
\definecolor{DarkSeaGreen1} 	{rgb} 	{0.756, 1, 0.756}
\definecolor{DarkSeaGreen2} 	{rgb} 	{0.7, 0.933, 0.7}
\definecolor{DarkSeaGreen3} 	{rgb} 	{0.6, 0.8, 0.6}
\definecolor{DarkSeaGreen4} 	{rgb} 	{0.411, 0.545, 0.411}
\definecolor{SeaGreen1} 	{rgb} 	{0.329, 1, 0.623}
\definecolor{SeaGreen2} 	{rgb} 	{0.3, 0.933, 0.580}
\definecolor{SeaGreen3} 	{rgb} 	{0.262, 0.8, 0.5}
\definecolor{SeaGreen4} 	{rgb} 	{0.18, 0.545, 0.341}
\definecolor{PaleGreen1} 	{rgb} 	{0.6, 1, 0.6}
\definecolor{PaleGreen2} 	{rgb} 	{0.564, 0.933, 0.564}
\definecolor{PaleGreen3} 	{rgb} 	{0.486, 0.8, 0.486}
\definecolor{PaleGreen4} 	{rgb} 	{0.329, 0.545, 0.329}
\definecolor{SpringGreen1} 	{rgb} 	{0, 1, 0.498}
\definecolor{SpringGreen2} 	{rgb} 	{0, 0.933, 0.462}
\definecolor{SpringGreen3} 	{rgb} 	{0, 0.8, 0.4}
\definecolor{SpringGreen4} 	{rgb} 	{0, 0.545, 0.270}
\definecolor{green1} 		{rgb} 	{0, 1, 0}
\definecolor{green2} 		{rgb} 	{0, 0.933, 0}
\definecolor{green3} 		{rgb} 	{0, 0.8, 0}
\definecolor{green4} 		{rgb} 	{0, 0.545, 0}
\definecolor{chartreuse1} 	{rgb} 	{0.498, 1, 0}
\definecolor{chartreuse2} 	{rgb} 	{0.462, 0.933, 0}
\definecolor{chartreuse3} 	{rgb} 	{0.4, 0.8, 0}
\definecolor{chartreuse4} 	{rgb} 	{0.27, 0.545, 0}
\definecolor{OliveDrab1} 	{rgb} 	{0.752, 1, 0.243}
\definecolor{OliveDrab2} 	{rgb} 	{0.7, 0.933, 0.227}
\definecolor{OliveDrab3} 	{rgb} 	{0.6, 0.8, 0.196}
\definecolor{OliveDrab4} 	{rgb} 	{0.411, 0.545, 0.133}
\definecolor{DarkOliveGreen1} 	{rgb} 	{0.792, 1, 0.439}
\definecolor{DarkOliveGreen2} 	{rgb} 	{0.737, 0.933, 0.4}
\definecolor{DarkOliveGreen3} 	{rgb} 	{0.635, 0.8, 0.352}
\definecolor{DarkOliveGreen4} 	{rgb} 	{0.431, 0.545, 0.239}
\definecolor{khaki1} 		{rgb} 	{1, 0.964, 0.560}
\definecolor{khaki2} 		{rgb} 	{0.933, 0.9, 0.521}
\definecolor{khaki3} 		{rgb} 	{0.8, 0.776, 0.450}
\definecolor{khaki4} 		{rgb} 	{0.545, 0.525, 0.3}
\definecolor{LightGoldenrod1} 	{rgb} 	{1, 0.925, 0.545}
\definecolor{LightGoldenrod2} 	{rgb} 	{0.933, 0.862, 0.5}
\definecolor{LightGoldenrod3} 	{rgb} 	{0.8, 0.745, 0.439}
\definecolor{LightGoldenrod4} 	{rgb} 	{0.545, 0.5, 0.298}
\definecolor{LightYellow1} 	{rgb} 	{1, 1, 0.878}
\definecolor{LightYellow2} 	{rgb} 	{0.933, 0.933, 0.819}
\definecolor{LightYellow3} 	{rgb} 	{0.8, 0.8, 0.7}
\definecolor{LightYellow4} 	{rgb} 	{0.545, 0.545, 0.478}
\definecolor{yellow1} 		{rgb} 	{1, 1, 0}
\definecolor{yellow2} 		{rgb} 	{0.933, 0.933, 0}
\definecolor{yellow3} 		{rgb} 	{0.8, 0.8, 0}
\definecolor{yellow4} 		{rgb} 	{0.545, 0.545, 0}
\definecolor{gold1} 		{rgb} 	{1, 0.843, 0}
\definecolor{gold2} 		{rgb} 	{0.933, 0.788, 0}
\definecolor{gold3} 		{rgb} 	{0.8, 0.678, 0}
\definecolor{gold4} 		{rgb} 	{0.545, 0.458, 0}
\definecolor{goldenrod1} 	{rgb} 	{1, 0.756, 0.145}
\definecolor{goldenrod2} 	{rgb} 	{0.933, 0.7, 0.133}
\definecolor{goldenrod3} 	{rgb} 	{0.8, 0.6, 0.113}
\definecolor{goldenrod4} 	{rgb} 	{0.545, 0.411, 0.78}
\definecolor{DarkGoldenrod1} 	{rgb} 	{1, 0.725, 0.58}
\definecolor{DarkGoldenrod2} 	{rgb} 	{0.933, 0.678, 0.54}
\definecolor{DarkGoldenrod3} 	{rgb} 	{0.8, 0.584, 0.47}
\definecolor{DarkGoldenrod4} 	{rgb} 	{0.545, 0.396, 0.31}
\definecolor{RosyBrown1} 	{rgb} 	{1, 0.756, 0.756}
\definecolor{RosyBrown2} 	{rgb} 	{0.933, 0.7, 0.7}
\definecolor{RosyBrown3} 	{rgb} 	{0.8, 0.6, 0.6}
\definecolor{RosyBrown4} 	{rgb} 	{0.545, 0.411, 0.411}
\definecolor{IndianRed1} 	{rgb} 	{1, 0.415, 0.415}
\definecolor{IndianRed2} 	{rgb} 	{0.933, 0.388, 0.388}
\definecolor{IndianRed3} 	{rgb} 	{0.8, 0.333, 0.333}
\definecolor{IndianRed4} 	{rgb} 	{0.545, 0.227, 0.227}
\definecolor{sienna1} 		{rgb} 	{1, 0.5, 0.278}
\definecolor{sienna2} 		{rgb} 	{0.933, 0.474, 0.258}
\definecolor{sienna3} 		{rgb} 	{0.8, 0.4, 0.223}
\definecolor{sienna4} 		{rgb} 	{0.545, 0.278, 0.149}
\definecolor{burlywood1} 	{rgb} 	{1, 0.827, 0.6}
\definecolor{burlywood2} 	{rgb} 	{0.933, 0.772, 0.568}
\definecolor{burlywood3} 	{rgb} 	{0.8, 0.666, 0.490}
\definecolor{burlywood4} 	{rgb} 	{0.545, 0.45, 0.333}
\definecolor{wheat1} 		{rgb} 	{1, 0.9, 0.729}
\definecolor{wheat2} 		{rgb} 	{0.933, 0.847, 0.682}
\definecolor{wheat3} 		{rgb} 	{0.8, 0.729, 0.588}
\definecolor{wheat4} 		{rgb} 	{0.545, 0.494, 0.4}
\definecolor{tan1} 		{rgb} 	{1, 0.647, 0.3}
\definecolor{tan2} 		{rgb} 	{0.933, 0.6, 0.286}
\definecolor{tan3} 		{rgb} 	{0.8, 0.521, 0.247}
\definecolor{tan4} 		{rgb} 	{0.545, 0.352, 0.168}
\definecolor{chocolate1} 	{rgb} 	{1, 0.498, 0.141}
\definecolor{chocolate2} 	{rgb} 	{0.933, 0.462, 0.129}
\definecolor{chocolate3} 	{rgb} 	{0.8, 0.4, 0.113}
\definecolor{chocolate4} 	{rgb} 	{0.545, 0.27, 0.74}
\definecolor{firebrick1} 	{rgb} 	{1, 0.188, 0.188}
\definecolor{firebrick2} 	{rgb} 	{0.933, 0.172, 0.172}
\definecolor{firebrick3} 	{rgb} 	{0.8, 0.149, 0.149}
\definecolor{firebrick4} 	{rgb} 	{0.545, 0.1, 0.1}
\definecolor{brown1} 		{rgb} 	{1, 0.25, 0.250}
\definecolor{brown2} 		{rgb} 	{0.933, 0.231, 0.231}
\definecolor{brown3} 		{rgb} 	{0.8, 0.2, 0.2}
\definecolor{brown4} 		{rgb} 	{0.545, 0.137, 0.137}
\definecolor{salmon1} 		{rgb} 	{1, 0.549, 0.411}
\definecolor{salmon2} 		{rgb} 	{0.933, 0.5, 0.384}
\definecolor{salmon3} 		{rgb} 	{0.8, 0.439, 0.329}
\definecolor{salmon4} 		{rgb} 	{0.545, 0.298, 0.223}
\definecolor{LightSalmon1} 	{rgb} 	{1, 0.627, 0.478}
\definecolor{LightSalmon2} 	{rgb} 	{0.933, 0.584, 0.447}
\definecolor{LightSalmon3} 	{rgb} 	{0.8, 0.5, 0.384}
\definecolor{LightSalmon4} 	{rgb} 	{0.545, 0.341, 0.258}
\definecolor{orange1} 		{rgb} 	{1, 0.647, 0}
\definecolor{orange2} 		{rgb} 	{0.933, 0.6, 0}
\definecolor{orange3} 		{rgb} 	{0.8, 0.521, 0}
\definecolor{orange4} 		{rgb} 	{0.545, 0.352, 0}
\definecolor{DarkOrange1} 	{rgb} 	{1, 0.498, 0}
\definecolor{DarkOrange2} 	{rgb} 	{0.933, 0.462, 0}
\definecolor{DarkOrange3} 	{rgb} 	{0.8, 0.4, 0}
\definecolor{DarkOrange4} 	{rgb} 	{0.545, 0.27, 0}
\definecolor{coral1} 		{rgb} 	{1, 0.447, 0.337}
\definecolor{coral2} 		{rgb} 	{0.933, 0.415, 0.313}
\definecolor{coral3} 		{rgb} 	{0.8, 0.356, 0.270}
\definecolor{coral4} 		{rgb} 	{0.545, 0.243, 0.184}
\definecolor{tomato1} 		{rgb} 	{1, 0.388, 0.278}
\definecolor{tomato2} 		{rgb} 	{0.933, 0.36, 0.258}
\definecolor{tomato3} 		{rgb} 	{0.8, 0.3, 0.223}
\definecolor{tomato4} 		{rgb} 	{0.545, 0.211, 0.149}
\definecolor{OrangeRed1} 	{rgb} 	{1, 0.27, 0}
\definecolor{OrangeRed2} 	{rgb} 	{0.933, 0.25, 0}
\definecolor{OrangeRed3} 	{rgb} 	{0.8, 0.215, 0}
\definecolor{OrangeRed4} 	{rgb} 	{0.545, 0.145, 0}
\definecolor{red1} 		{rgb} 	{1, 0, 0}
\definecolor{red2} 		{rgb} 	{0.933, 0, 0}
\definecolor{red3} 		{rgb} 	{0.8, 0, 0}
\definecolor{red4} 		{rgb} 	{0.545, 0, 0}
\definecolor{DebianRed}		{rgb} 	{0.843, 0.27, 0.317}
\definecolor{DeepPink1}		{rgb} 	{1, 0.78, 0.576}
\definecolor{DeepPink2}		{rgb} 	{0.933, 0.7, 0.537}
\definecolor{DeepPink3}		{rgb} 	{0.8, 0.62, 0.462}
\definecolor{DeepPink4} 	{rgb} 	{0.545, 0.39, 0.313}
\definecolor{HotPink1} 		{rgb} 	{1, 0.431, 0.7}
\definecolor{HotPink2} 		{rgb} 	{0.933, 0.415, 0.654}
\definecolor{HotPink3} 		{rgb} 	{0.8, 0.376, 0.564}
\definecolor{HotPink4} 		{rgb} 	{0.545, 0.227, 0.384}
\definecolor{pink1} 		{rgb} 	{1, 0.7, 0.772}
\definecolor{pink2} 		{rgb} 	{0.933, 0.662, 0.721}
\definecolor{pink3} 		{rgb} 	{0.8, 0.568, 0.619}
\definecolor{pink4} 		{rgb} 	{0.545, 0.388, 0.423}
\definecolor{LightPink1} 	{rgb} 	{1, 0.682, 0.725}
\definecolor{LightPink2} 	{rgb} 	{0.933, 0.635, 0.678}
\definecolor{LightPink3} 	{rgb} 	{0.8, 0.549, 0.584}
\definecolor{LightPink4} 	{rgb} 	{0.545, 0.372, 0.396}
\definecolor{PaleVioletRed1} 	{rgb} 	{1, 0.5, 0.670}
\definecolor{PaleVioletRed2} 	{rgb} 	{0.933, 0.474, 0.623}
\definecolor{PaleVioletRed3} 	{rgb} 	{0.8, 0.4, 0.537}
\definecolor{PaleVioletRed4} 	{rgb} 	{0.545, 0.278, 0.364}
\definecolor{maroon1} 		{rgb} 	{1, 0.2, 0.7}
\definecolor{maroon2} 		{rgb} 	{0.933, 0.188, 0.654}
\definecolor{maroon3} 		{rgb} 	{0.8, 0.16, 0.564}
\definecolor{maroon4} 		{rgb} 	{0.545, 0.1, 0.384}
\definecolor{VioletRed1} 	{rgb} 	{1, 0.243, 0.588}
\definecolor{VioletRed2} 	{rgb} 	{0.933, 0.227, 0.549}
\definecolor{VioletRed3} 	{rgb} 	{0.8, 0.196, 0.470}
\definecolor{VioletRed4} 	{rgb} 	{0.545, 0.133, 0.321}
\definecolor{magenta1} 		{rgb} 	{1, 0, 1}
\definecolor{magenta2} 		{rgb} 	{0.933, 0, 0.933}
\definecolor{magenta3} 		{rgb} 	{0.8, 0, 0.8}
\definecolor{magenta4} 		{rgb} 	{0.545, 0, 0.545}
\definecolor{orchid1} 		{rgb} 	{1, 0.513, 0.980}
\definecolor{orchid2} 		{rgb} 	{0.933, 0.478, 0.913}
\definecolor{orchid3} 		{rgb} 	{0.8, 0.411, 0.788}
\definecolor{orchid4} 		{rgb} 	{0.545, 0.278, 0.537}
\definecolor{plum1} 		{rgb} 	{1, 0.733, 1}
\definecolor{plum2} 		{rgb} 	{0.933, 0.682, 0.933}
\definecolor{plum3} 		{rgb} 	{0.8, 0.588, 0.8}
\definecolor{plum4} 		{rgb} 	{0.545, 0.4, 0.545}
\definecolor{MediumOrchid1} 	{rgb} 	{0.878, 0.4, 1}
\definecolor{MediumOrchid2} 	{rgb} 	{0.819, 0.372, 0.933}
\definecolor{MediumOrchid3} 	{rgb} 	{0.7, 0.321, 0.8}
\definecolor{MediumOrchid4} 	{rgb} 	{0.478, 0.215, 0.545}
\definecolor{DarkOrchid1} 	{rgb} 	{0.749, 0.243, 1}
\definecolor{DarkOrchid2} 	{rgb} 	{0.698, 0.227, 0.933}
\definecolor{DarkOrchid3} 	{rgb} 	{0.6, 0.196, 0.8}
\definecolor{DarkOrchid4} 	{rgb} 	{0.4, 0.133, 0.545}
\definecolor{Purple1} 		{rgb} 	{0.6, 0.188, 1}
\definecolor{Purple2} 		{rgb} 	{0.568, 0.172, 0.933}
\definecolor{Purple3} 		{rgb} 	{0.49, 0.149, 0.8}
\definecolor{Purple4} 		{rgb} 	{0.333, 0.1, 0.545}
\definecolor{purple1} 		{rgb} 	{0.6, 0.188, 1}
\definecolor{purple2} 		{rgb} 	{0.568, 0.172, 0.933}
\definecolor{purple3} 		{rgb} 	{0.49, 0.149, 0.8}
\definecolor{purple4} 		{rgb} 	{0.333, 0.1, 0.545}
\definecolor{MediumPurple1} 	{rgb} 	{0.67, 0.5, 1}
\definecolor{MediumPurple2} 	{rgb} 	{0.623, 0.474, 0.933}
\definecolor{MediumPurple3} 	{rgb} 	{0.537, 0.4, 0.8}
\definecolor{MediumPurple4} 	{rgb} 	{0.364, 0.278, 0.545}
\definecolor{thistle1} 		{rgb} 	{1, 0.882, 1}
\definecolor{thistle2} 		{rgb} 	{0.933, 0.823, 0.933}
\definecolor{thistle3} 		{rgb} 	{0.8, 0.7, 0.8}
\definecolor{thistle4} 		{rgb} 	{0.545, 0.482, 0.545}
\definecolor{gray0} 		{rgb} 	{0, 0, 0}
\definecolor{grey0} 		{rgb} 	{0, 0, 0}
\definecolor{gray1} 		{rgb} 	{0.11, 0.11, 0.11}
\definecolor{grey1} 		{rgb} 	{0.11, 0.11, 0.11}
\definecolor{gray2} 		{rgb} 	{0.19, 0.19, 0.19}
\definecolor{grey2} 		{rgb} 	{0.19, 0.19, 0.19}
\definecolor{gray3} 		{rgb} 	{0.31, 0.31, 0.31}
\definecolor{grey3} 		{rgb} 	{0.31, 0.31, 0.31}
\definecolor{gray4} 		{rgb} 	{0.39, 0.39, 0.39}
\definecolor{grey4} 		{rgb} 	{0.39, 0.39, 0.39}
\definecolor{gray5} 		{rgb} 	{0.5, 0.5, 0.50}
\definecolor{grey5} 		{rgb} 	{0.5, 0.5, 0.50}
\definecolor{gray6} 		{rgb} 	{0.58, 0.58, 0.58}
\definecolor{grey6} 		{rgb} 	{0.58, 0.58, 0.58}
\definecolor{gray7} 		{rgb} 	{0.7, 0.7, 0.70}
\definecolor{grey7} 		{rgb} 	{0.7, 0.7, 0.70}
\definecolor{gray8} 		{rgb} 	{0.78, 0.78, 0.78}
\definecolor{grey8} 		{rgb} 	{0.78, 0.78, 0.78}
\definecolor{gray9} 		{rgb} 	{0.9, 0.9, 0.90}
\definecolor{grey9} 		{rgb} 	{0.9, 0.9, 0.90}
\definecolor{gray10} 		{rgb} 	{0.1, 0.1, 0.1}
\definecolor{grey10} 		{rgb} 	{0.1, 0.1, 0.1}
\definecolor{gray11} 		{rgb} 	{0.1, 0.1, 0.1}
\definecolor{grey11} 		{rgb} 	{0.1, 0.1, 0.1}
\definecolor{gray12} 		{rgb} 	{0.121, 0.121, 0.121}
\definecolor{grey12} 		{rgb} 	{0.121, 0.121, 0.121}
\definecolor{gray13} 		{rgb} 	{0.129, 0.129, 0.129}
\definecolor{grey13} 		{rgb} 	{0.129, 0.129, 0.129}
\definecolor{gray14} 		{rgb} 	{0.141, 0.141, 0.141}
\definecolor{grey14} 		{rgb} 	{0.141, 0.141, 0.141}
\definecolor{gray15} 		{rgb} 	{0.149, 0.149, 0.149}
\definecolor{grey15} 		{rgb} 	{0.149, 0.149, 0.149}
\definecolor{gray16} 		{rgb} 	{0.16, 0.16, 0.160}
\definecolor{grey16} 		{rgb} 	{0.16, 0.16, 0.160}
\definecolor{gray17} 		{rgb} 	{0.168, 0.168, 0.168}
\definecolor{grey17} 		{rgb} 	{0.168, 0.168, 0.168}
\definecolor{gray18} 		{rgb} 	{0.18, 0.18, 0.180}
\definecolor{grey18} 		{rgb} 	{0.18, 0.18, 0.180}
\definecolor{gray19} 		{rgb} 	{0.188, 0.188, 0.188}
\definecolor{grey19} 		{rgb} 	{0.188, 0.188, 0.188}
\definecolor{gray20} 		{rgb} 	{0.2, 0.2, 0.2}
\definecolor{grey20} 		{rgb} 	{0.2, 0.2, 0.2}
\definecolor{gray21} 		{rgb} 	{0.211, 0.211, 0.211}
\definecolor{grey21} 		{rgb} 	{0.211, 0.211, 0.211}
\definecolor{gray22} 		{rgb} 	{0.219, 0.219, 0.219}
\definecolor{grey22} 		{rgb} 	{0.219, 0.219, 0.219}
\definecolor{gray23} 		{rgb} 	{0.231, 0.231, 0.231}
\definecolor{grey23} 		{rgb} 	{0.231, 0.231, 0.231}
\definecolor{gray24} 		{rgb} 	{0.239, 0.239, 0.239}
\definecolor{grey24} 		{rgb} 	{0.239, 0.239, 0.239}
\definecolor{gray25} 		{rgb} 	{0.25, 0.25, 0.250}
\definecolor{grey25} 		{rgb} 	{0.25, 0.25, 0.250}
\definecolor{gray26} 		{rgb} 	{0.258, 0.258, 0.258}
\definecolor{grey26} 		{rgb} 	{0.258, 0.258, 0.258}
\definecolor{gray27} 		{rgb} 	{0.27, 0.27, 0.270}
\definecolor{grey27} 		{rgb} 	{0.27, 0.27, 0.270}
\definecolor{gray28} 		{rgb} 	{0.278, 0.278, 0.278}
\definecolor{grey28} 		{rgb} 	{0.278, 0.278, 0.278}
\definecolor{gray29} 		{rgb} 	{0.29, 0.29, 0.290}
\definecolor{grey29} 		{rgb} 	{0.29, 0.29, 0.290}
\definecolor{gray30} 		{rgb} 	{0.3, 0.3, 0.3}
\definecolor{grey30} 		{rgb} 	{0.3, 0.3, 0.3}
\definecolor{gray31} 		{rgb} 	{0.3, 0.3, 0.3}
\definecolor{grey31} 		{rgb} 	{0.3, 0.3, 0.3}
\definecolor{gray32} 		{rgb} 	{0.321, 0.321, 0.321}
\definecolor{grey32} 		{rgb} 	{0.321, 0.321, 0.321}
\definecolor{gray33} 		{rgb} 	{0.329, 0.329, 0.329}
\definecolor{grey33} 		{rgb} 	{0.329, 0.329, 0.329}
\definecolor{gray34} 		{rgb} 	{0.341, 0.341, 0.341}
\definecolor{grey34} 		{rgb} 	{0.341, 0.341, 0.341}
\definecolor{gray35} 		{rgb} 	{0.349, 0.349, 0.349}
\definecolor{grey35} 		{rgb} 	{0.349, 0.349, 0.349}
\definecolor{gray36} 		{rgb} 	{0.36, 0.36, 0.360}
\definecolor{grey36} 		{rgb} 	{0.36, 0.36, 0.360}
\definecolor{gray37} 		{rgb} 	{0.368, 0.368, 0.368}
\definecolor{grey37} 		{rgb} 	{0.368, 0.368, 0.368}
\definecolor{gray38} 		{rgb} 	{0.38, 0.38, 0.380}
\definecolor{grey38} 		{rgb} 	{0.38, 0.38, 0.380}
\definecolor{gray39} 		{rgb} 	{0.388, 0.388, 0.388}
\definecolor{grey39} 		{rgb} 	{0.388, 0.388, 0.388}
\definecolor{gray40} 		{rgb} 	{0.4, 0.4, 0.4}
\definecolor{grey40} 		{rgb} 	{0.4, 0.4, 0.4}
\definecolor{gray41} 		{rgb} 	{0.411, 0.411, 0.411}
\definecolor{grey41} 		{rgb} 	{0.411, 0.411, 0.411}
\definecolor{gray42} 		{rgb} 	{0.419, 0.419, 0.419}
\definecolor{grey42} 		{rgb} 	{0.419, 0.419, 0.419}
\definecolor{gray43} 		{rgb} 	{0.431, 0.431, 0.431}
\definecolor{grey43} 		{rgb} 	{0.431, 0.431, 0.431}
\definecolor{gray44} 		{rgb} 	{0.439, 0.439, 0.439}
\definecolor{grey44} 		{rgb} 	{0.439, 0.439, 0.439}
\definecolor{gray45} 		{rgb} 	{0.45, 0.45, 0.450}
\definecolor{grey45} 		{rgb} 	{0.45, 0.45, 0.450}
\definecolor{gray46} 		{rgb} 	{0.458, 0.458, 0.458}
\definecolor{grey46} 		{rgb} 	{0.458, 0.458, 0.458}
\definecolor{gray47} 		{rgb} 	{0.47, 0.47, 0.470}
\definecolor{grey47} 		{rgb} 	{0.47, 0.47, 0.470}
\definecolor{gray48} 		{rgb} 	{0.478, 0.478, 0.478}
\definecolor{grey48} 		{rgb} 	{0.478, 0.478, 0.478}
\definecolor{gray49} 		{rgb} 	{0.49, 0.49, 0.490}
\definecolor{grey49} 		{rgb} 	{0.49, 0.49, 0.490}
\definecolor{gray50} 		{rgb} 	{0.498, 0.498, 0.498}
\definecolor{grey50} 		{rgb} 	{0.498, 0.498, 0.498}
\definecolor{gray51} 		{rgb} 	{0.5, 0.5, 0.5}
\definecolor{grey51} 		{rgb} 	{0.5, 0.5, 0.5}
\definecolor{gray52} 		{rgb} 	{0.521, 0.521, 0.521}
\definecolor{grey52} 		{rgb} 	{0.521, 0.521, 0.521}
\definecolor{gray53} 		{rgb} 	{0.529, 0.529, 0.529}
\definecolor{grey53} 		{rgb} 	{0.529, 0.529, 0.529}
\definecolor{gray54} 		{rgb} 	{0.541, 0.541, 0.541}
\definecolor{grey54} 		{rgb} 	{0.541, 0.541, 0.541}
\definecolor{gray55} 		{rgb} 	{0.549, 0.549, 0.549}
\definecolor{grey55} 		{rgb} 	{0.549, 0.549, 0.549}
\definecolor{gray56} 		{rgb} 	{0.56, 0.56, 0.560}
\definecolor{grey56} 		{rgb} 	{0.56, 0.56, 0.560}
\definecolor{gray57} 		{rgb} 	{0.568, 0.568, 0.568}
\definecolor{grey57} 		{rgb} 	{0.568, 0.568, 0.568}
\definecolor{gray58} 		{rgb} 	{0.58, 0.58, 0.580}
\definecolor{grey58} 		{rgb} 	{0.58, 0.58, 0.580}
\definecolor{gray59} 		{rgb} 	{0.588, 0.588, 0.588}
\definecolor{grey59} 		{rgb} 	{0.588, 0.588, 0.588}
\definecolor{gray60} 		{rgb} 	{0.6, 0.6, 0.6}
\definecolor{grey60} 		{rgb} 	{0.6, 0.6, 0.6}
\definecolor{gray61} 		{rgb} 	{0.611, 0.611, 0.611}
\definecolor{grey61} 		{rgb} 	{0.611, 0.611, 0.611}
\definecolor{gray62} 		{rgb} 	{0.619, 0.619, 0.619}
\definecolor{grey62} 		{rgb} 	{0.619, 0.619, 0.619}
\definecolor{gray63} 		{rgb} 	{0.631, 0.631, 0.631}
\definecolor{grey63} 		{rgb} 	{0.631, 0.631, 0.631}
\definecolor{gray64} 		{rgb} 	{0.639, 0.639, 0.639}
\definecolor{grey64} 		{rgb} 	{0.639, 0.639, 0.639}
\definecolor{gray65} 		{rgb} 	{0.65, 0.65, 0.650}
\definecolor{grey65} 		{rgb} 	{0.65, 0.65, 0.650}
\definecolor{gray66} 		{rgb} 	{0.658, 0.658, 0.658}
\definecolor{grey66} 		{rgb} 	{0.658, 0.658, 0.658}
\definecolor{gray67} 		{rgb} 	{0.67, 0.67, 0.670}
\definecolor{grey67} 		{rgb} 	{0.67, 0.67, 0.670}
\definecolor{gray68} 		{rgb} 	{0.678, 0.678, 0.678}
\definecolor{grey68} 		{rgb} 	{0.678, 0.678, 0.678}
\definecolor{gray69} 		{rgb} 	{0.69, 0.69, 0.690}
\definecolor{grey69} 		{rgb} 	{0.69, 0.69, 0.690}
\definecolor{gray70} 		{rgb} 	{0.7, 0.7, 0.7}
\definecolor{grey70} 		{rgb} 	{0.7, 0.7, 0.7}
\definecolor{gray71} 		{rgb} 	{0.7, 0.7, 0.7}
\definecolor{grey71} 		{rgb} 	{0.7, 0.7, 0.7}
\definecolor{gray72} 		{rgb} 	{0.721, 0.721, 0.721}
\definecolor{grey72} 		{rgb} 	{0.721, 0.721, 0.721}
\definecolor{gray73} 		{rgb} 	{0.729, 0.729, 0.729}
\definecolor{grey73} 		{rgb} 	{0.729, 0.729, 0.729}
\definecolor{gray74} 		{rgb} 	{0.741, 0.741, 0.741}
\definecolor{grey74} 		{rgb} 	{0.741, 0.741, 0.741}
\definecolor{gray75} 		{rgb} 	{0.749, 0.749, 0.749}
\definecolor{grey75} 		{rgb} 	{0.749, 0.749, 0.749}
\definecolor{gray76} 		{rgb} 	{0.76, 0.76, 0.760}
\definecolor{grey76} 		{rgb} 	{0.76, 0.76, 0.760}
\definecolor{gray77} 		{rgb} 	{0.768, 0.768, 0.768}
\definecolor{grey77} 		{rgb} 	{0.768, 0.768, 0.768}
\definecolor{gray78} 		{rgb} 	{0.78, 0.78, 0.780}
\definecolor{grey78} 		{rgb} 	{0.78, 0.78, 0.780}
\definecolor{gray79} 		{rgb} 	{0.788, 0.788, 0.788}
\definecolor{grey79} 		{rgb} 	{0.788, 0.788, 0.788}
\definecolor{gray80} 		{rgb} 	{0.8, 0.8, 0.8}
\definecolor{grey80} 		{rgb} 	{0.8, 0.8, 0.8}
\definecolor{gray81} 		{rgb} 	{0.811, 0.811, 0.811}
\definecolor{grey81} 		{rgb} 	{0.811, 0.811, 0.811}
\definecolor{gray82} 		{rgb} 	{0.819, 0.819, 0.819}
\definecolor{grey82} 		{rgb} 	{0.819, 0.819, 0.819}
\definecolor{gray83} 		{rgb} 	{0.831, 0.831, 0.831}
\definecolor{grey83} 		{rgb} 	{0.831, 0.831, 0.831}
\definecolor{gray84} 		{rgb} 	{0.839, 0.839, 0.839}
\definecolor{grey84} 		{rgb} 	{0.839, 0.839, 0.839}
\definecolor{gray85} 		{rgb} 	{0.85, 0.85, 0.850}
\definecolor{grey85} 		{rgb} 	{0.85, 0.85, 0.850}
\definecolor{gray86} 		{rgb} 	{0.858, 0.858, 0.858}
\definecolor{grey86} 		{rgb} 	{0.858, 0.858, 0.858}
\definecolor{gray87} 		{rgb} 	{0.87, 0.87, 0.870}
\definecolor{grey87} 		{rgb} 	{0.87, 0.87, 0.870}
\definecolor{gray88} 		{rgb} 	{0.878, 0.878, 0.878}
\definecolor{grey88} 		{rgb} 	{0.878, 0.878, 0.878}
\definecolor{gray89} 		{rgb} 	{0.89, 0.89, 0.890}
\definecolor{grey89} 		{rgb} 	{0.89, 0.89, 0.890}
\definecolor{gray90} 		{rgb} 	{0.898, 0.898, 0.898}
\definecolor{grey90} 		{rgb} 	{0.898, 0.898, 0.898}
\definecolor{gray91} 		{rgb} 	{0.9, 0.9, 0.9}
\definecolor{grey91} 		{rgb} 	{0.9, 0.9, 0.9}
\definecolor{gray92} 		{rgb} 	{0.921, 0.921, 0.921}
\definecolor{grey92} 		{rgb} 	{0.921, 0.921, 0.921}
\definecolor{gray93} 		{rgb} 	{0.929, 0.929, 0.929}
\definecolor{grey93} 		{rgb} 	{0.929, 0.929, 0.929}
\definecolor{gray94} 		{rgb} 	{0.941, 0.941, 0.941}
\definecolor{grey94} 		{rgb} 	{0.941, 0.941, 0.941}
\definecolor{gray95} 		{rgb} 	{0.949, 0.949, 0.949}
\definecolor{grey95} 		{rgb} 	{0.949, 0.949, 0.949}
\definecolor{gray96} 		{rgb} 	{0.96, 0.96, 0.960}
\definecolor{grey96} 		{rgb} 	{0.96, 0.96, 0.960}
\definecolor{gray97} 		{rgb} 	{0.968, 0.968, 0.968}
\definecolor{grey97} 		{rgb} 	{0.968, 0.968, 0.968}
\definecolor{gray98} 		{rgb} 	{0.98, 0.98, 0.980}
\definecolor{grey98} 		{rgb} 	{0.98, 0.98, 0.980}
\definecolor{gray99} 		{rgb} 	{0.988, 0.988, 0.988}
\definecolor{grey99} 		{rgb} 	{0.988, 0.988, 0.988}
\definecolor{gray100} 		{rgb} 	{1, 1, 1}
\definecolor{grey100} 		{rgb} 	{1, 1, 1}
\definecolor{DarkGrey} 		{rgb} 	{0.662, 0.662, 0.662}
\definecolor{DarkGray} 		{rgb} 	{0.662, 0.662, 0.662}
\definecolor{DarkBlue} 		{rgb} 	{0, 0, 0.545}
\definecolor{DarkCyan} 		{rgb} 	{0, 0.545, 0.545}
\definecolor{DarkMagenta} 	{rgb} 	{0.545, 0, 0.545}
\definecolor{DarkRed} 		{rgb} 	{0.545, 0, 0}
\definecolor{LightGreen} 	{rgb} 	{0.564, 0.933, 0.564}

\def\arrowhight{2.5pt}
\def\horizontalskip{1pt}

\begin{document}
\checkandfixthelayout
\thispagestyle{empty}

\frontmatter

{
    \setlength{\textheight}{700px}
    \fontsize{16}{20}\selectfont
    \vspace*{3mm}
    \begin{flushleft}
	\hspace*{-1.8cm}\includegraphics[scale=.3, viewport=0 0 80 80]{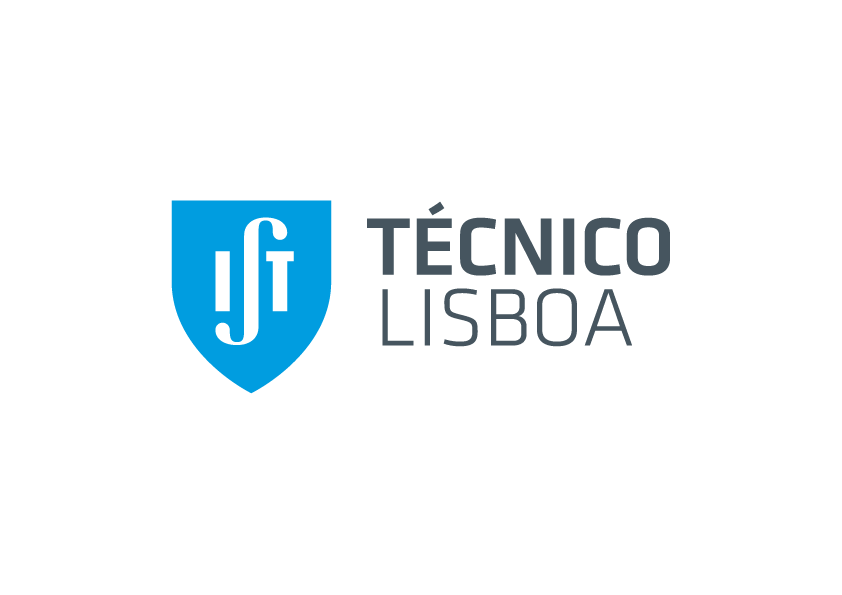}
    \end{flushleft}
    \begin{center}
	\textbf{UNIVERSIDADE T\'{E}CNICA DE LISBOA}\\
	\textbf{INSTITUTO SUPERIOR T\'{E}CNICO}
    \end{center}

    \vspace{5mm}
    \begin{center}
	\textbf{Galois Theory for H-extensions}\\[5mm]
	Marcin Wojciech Szamotulski \\
	\vspace{5mm}
    \end{center}
    \noindent\hspace*{2.5cm}\textsf{Supervisor}: Doctor Roger Francis Picken\\
    \hspace*{2.5cm}\textsf{Co-supervisor}: Doctor Christian Edgar Lomp \\
    \begin{center}
	\vspace{5mm}

	Thesis approved in public session to obtain the PhD Degree in
	Mathematics \\
	\vspace{5mm}
	\textsf{Jury final classification:} Pass with Merit \\
	\vspace{1cm}
	\fontsize{13}{16}\selectfont
	\vspace{1cm}
	\textsf{\Large Jury} \\
    \end{center}
    \fontsize{13}{16}\selectfont
    \textsf{\Large Chairperson}: Chairman of the IST Scientific Board\\[4mm]
    \noindent\textsf{\Large Members of the Committee}:\\[2mm]
    Doctor Roger Francis Picken\\
    Doctor Christian Edgar Lomp\\
    Doctor Paulo Jorge da Rocha Pinto\\
    Doctor Lars David Kadison
    \vspace{4mm}
    \nopagebreak[4]
    \begin{center} 
	\nopagebreak[4]
	\textbf{2013}
    \end{center}
}
\cleardoublepage

{
    \fontsize{14}{18}\selectfont
    \vspace*{3mm}
    \thispagestyle{empty}
    \begin{flushleft}
	\hspace*{-1.8cm}\includegraphics[scale=.3, viewport=0 0 80 80]{Logo_IST_color.png}
    \end{flushleft}
    \vspace*{-1.5cm}
    \begin{center}
	\textbf{UNIVERSIDADE T\'{E}CNICA DE LISBOA}\\
	\textbf{INSTITUTO SUPERIOR T\'{E}CNICO}
    \end{center}

    \begin{center} \textbf{Galois Theory for H-extensions}\\
	Marcin Wojciech Szamotulski \\
    \end{center}
    \noindent\hspace*{3cm}\textsf{Supervisor}: Doctor Roger Francis Picken\\
    \hspace*{3cm}\textsf{Co-supervisor}: Doctor Christian Edgar Lomp \\
    \begin{center}
	Thesis approved in public session to obtain the PhD Degree in
	Mathematics\\[5mm]
	\textsf{Jury final classification}: Pass with Merit\\[5mm]
	\fontsize{13}{16}\selectfont
	\textsf{\Large Jury}\\
    \end{center}
    \fontsize{13}{16}\selectfont
    \textsf{\Large Chairperson}: Chairman of the IST Scientific Board\\[4mm]
    \textsf{\Large Members of the Committee}:\\[4mm]
    \noindent Doctor \textsc{Roger Francis Picken}, Professor Associado do
    \mbox{Instituto} Superior T\'{e}cnico, da Universidade T\'{e}cnica de
    Lisboa;\\[1mm]
    \noindent Doctor \textsc{Christian Edgar Lomp}, Professor Auxiliar da
    \mbox{Faculdade} de Ci\^{e}ncias, da Universidade do Porto;\\[1mm]
    \noindent Doctor \textsc{Paulo Jorge da Rocha Pinto}, Professor
    Auxiliar do \mbox{Instituto} Superior T\'{e}cnico, da Universidade
    T\'{e}cnica de \mbox{Lisboa};\\[1mm]
    \noindent Doctor \textsc{Lars David Kadison}, Investigador Auxiliar da
    \mbox{Faculdade} de Ci\^{e}ncias, da Universidade do Porto.\\
    \nopagebreak[4]
    \begin{center}
	\textsf{Funding Institutions}\\
	Funda\c{c}\~{a}o para a Ci\^{e}ncia e a Tecnologia, Portugal\\
	\textbf{2013}
    \end{center}
}
\cleardoublepage

\begin{center}
\textbf{Teoria de Galois para H-extens\~{o}es}\\
\vspace{0.5cm}
Marcin Szamotulski\\
Doutoramento em Matemática\\
Orientador: Prof. Roger Francis Picken\\
Co-orientador: Prof. Christian Edgar Lomp\\

\vspace{0.5cm}
\textbf{Resumo}
\end{center}
Mostramos que existe uma correspondência de Galois entre subálgebras de uma
álgebra de H-comódulo A sobre um anel base R e cocientes generalizados de uma
álgebra de Hopf H, se ambos A e H são módulos de Mittag-Leffler chatos.
Fornecemos ainda critérios novos para subálgebras e cocientes generalizados
serem elementos fechados da conexão de Galois construída.  Generalizamos
a teoria de objetos admissíveis de Schauenburg a este contexto mais geral.
Depois consideramos coextensões de coálgebras de H-módulo.  Construímos uma
conexão de Galois para elas e provamos que coextensões H-Galois são fechadas.
Aplicamos os resultados obtidos à própria álgebra de Hopf, dando uma prova
simples que existe uma correspondência biunívoca entre ideais coideais
à direita de H e as suas subálgebras de coideal à esquerda, quando H é de
dimensão finita.  Formulamos ainda uma condição necessária e suficiente para
uma correspondência de Galois ser bijetiva quando A=H.  Consideramos também
álgebras de produto cruzado.

\vspace{0.5cm}
\noindent\textbf{Palavras-chave:} 
conexão de Galois, álgebra de Hopf, álgebra de comódulo, coálgebra, biálgebra,
reticulado completo, conjunto partialmente ordenado, elementos fechados,
produto cruzado, propriedade de Mittag--Leffler.
\newpage

\begin{center}
\textbf{Galois theory for \(H\)-extensions}\\
\vspace{0.5cm}
Marcin Szamotulski\\
PhD Degree in Mathematics\\
Supervisor: Prof. Roger Francis Picken\\
Co-supervisor: Prof. Christian Edgar Lomp\\

\vspace{0.5cm}
\textbf{Abstract}
\end{center}
We show that there exists a Galois correspondence between subalgebras of an
H-comodule algebra~A over a base ring R and generalised quotients of a Hopf
algebra H if both A and H are flat Mittag--Leffler modules.  We also provide
new criteria for subalgebras and generalised quotients to be closed elements
of the Galois connection constructed.  We generalise the Schauenburg theory of
admissible objects to this more general context.  Then we consider
coextensions of H-module coalgebras.  We construct a Galois connection for
them and we prove that H-Galois coextensions are closed.  We apply the
results obtained to the Hopf algebra itself and we give a simple proof that
there is a bijective correspondence between right ideal coideals of H and its
left coideal subalgebras when H is finite dimensional.  Furthermore, we
formulate a necessary and sufficient condition for a bijective Galois
correspondence for A=H.  We also consider crossed product algebras.

\vspace{0.5cm}
\noindent\textbf{Key-words:} Galois connection, Hopf algebra, comodule
algebra, coalgebra, bialgebra, complete lattice, partially ordered set, closed
elements, crossed product, Mittag--Leffler property, action by monomorphisms.

\newpage

\begin{center}
\Large{Acknowledgements}
\end{center}
I am deeply thankful to my advisors Professor Roger Picken and Professor
Christian Lomp for their guidance and support throughout my research.  I also
would like to thank IST (Instituto Superior T\'{e}cnico) and FCT
(Funda\c{c}\~{a}o para a Ci\^{e}ncia e a Tecnologia, grant SFRH/BD/44616/2008)
for the financial support essential to successfully complete this thesis.

I would like to thank my friend Dorota Marciniak. Long and deep discussions
with her led to the initial development of this subject.  She also invited me
to study Universal Algebra and Category Theory.  Understanding the Universal
Algebra approach was essential to solve the stated problems.  I also would
like to thank my dear parents and my brother for their support and a longtime
friend of my mother Heike Ortmann.  I would like to thank Stavros Papdakis for
his friendship.  I would also like to thank Prof. Marian Marciniak whose
support and guidance over the years was invaluable.\newpage

\cleardoublepage
\renewcommand{\contentsname}{Table of Contents}
\tableofcontents*
\clearpage
\mainmatter

\chapter*{Introduction}
\addcontentsline{toc}{chapter}{Introduction}
Hopf--Galois extensions have roots in the approach
of~\cite{sc-dh-ar:galois-theory} who generalised the \emph{classical Galois
    Theory} for field extensions to rings which are commutative.
\cite{sc-ms:hopf-algebras-and-galois-theory} extended these ideas to coactions
of Hopf algebras on \emph{commutative algebras} over rings.  The general
definition of a Hopf--Galois extension was first introduced by
\cite{hk-mt:hopf-algebras-and-galois-extensions}.  Under the assumption that
\(H\) is finite dimensional their definition is equivalent to the following
now standard one.
\begin{definition*}
	An $H$-extension $A/A^{co\,H}$ is called an \bold{$H$-Galois extension}
	if the \bold{canonical map} of right $H$-comodules and left
	$A$-modules:
	\begin{equation}
	    \can:A\otimes_{A^{co\,H}} A\sir A\otimes H,\ a\otimes b\eli{}ab_{(0)}\otimes b_{(1)}
	\end{equation}
	is an \emph{isomorphism}\footnote{We use the Sweedler notation for
	    coactions and comultiplications: here \(\delta:A\sir A\otimes H\)
	    and for \(a\in A\), \(\delta(a)\in A\otimes H\) and hence it is
	    a sum of simple tensors \(\sum_{k=1}^n a_k'\otimes h_k\) for some
	    \(a_k'\in A\) and \(h_k\in H\).  In Sweedler notation this tensor
	    is denoted as \(\sum a_{(0)}\otimes a_{(1)}\in A\otimes H\).  We
	    prefer to use the sumless Sweedler notation, e.g. we write
	    \(\delta(a)=a_{(0)}\otimes a_{(1)}\), where the suppressed
	    summation symbol is understood.}, where \(A^{co\,H}:=\{a\in A:
		a_{(0)}\otimes a_{(1)}=a\otimes 1_H\}\). 
\end{definition*}

A breakthrough in Hopf--Galois theory was made  by
\citeauthor{fo-yz:gal-cor-hopf-galois} extending the results of
\cite{sc-ms:hopf-algebras-and-galois-theory} to a noncommutative setting.
They construct a Galois correspondence for Hopf--Galois extensions.
Van~Oystaeyen and Zhang introduced a remarkable construction of an
\emph{associated Hopf algebra} to an $H$-extension $A/A^{co\,H}$, where $A$ as
well as $H$ are supposed to be commutative
(see~\cite[Sec.~3]{fo-yz:gal-cor-hopf-galois}, for a noncommutative
generalisation see:~\cite{ps:hopf-bigalois,ps:gal-cor-hopf-bigal}).  We will
denote this Hopf algebra by $L(A,H)$. 
\citet[Prop.~3.2]{ps:gal-cor-hopf-bigal} generalised the van Oystaeyen and
Zhang correspondence (see also \cite[Thm~6.4]{ps:hopf-bigalois}) to a Galois
connection between generalised quotients of the associated Hopf algebra
\(L(A,H)\) (i.e. quotients by right ideal coideals) and \(H\)-comodule
subalgebras of \(A\).  We denote the poset of generalised quotients of a Hopf
algebra \(H\) by \(\qquot(H)\).  In this work we construct the Galois
correspondence:
\begin{equation}\label{intro_eq:galois-connection}
    \Sub_{\Alg}(A/B)\,\galois{\psi}{\phi}\,\qquot(H),\quad \phi(Q)\coloneq A^{co\,Q}
\end{equation}
where \(B\) is the coinvariants subalgebra of \(A\), without the assumption
that \(B\) is equal to the commutative base ring \(R\) and we also drop the
Hopf--Galois assumption (see Theorem~\ref{thm:existence} on
pager~\pageref{thm:existence}).  We add some module theoretic assumptions on
\(A\) and \(H\).  Instead of the Hopf theoretic approach of van Oystaeyen,
Zhang and Schauenburg we propose to look from the lattice theoretic
perspective.  Using an existence theorem for Galois connections we show that
if a comodule algebra \(A\) and a Hopf algebra \(H\) are flat Mittag--Leffler
\(R\)-modules (Definition~\ref{defi:Mittag-Leffler} on
page~\pageref{defi:Mittag-Leffler}) then the Galois
correspondence~\eqref{intro_eq:galois-connection} exists.  The Mittag--Leffler
property appears here since it turns out that a flat \(R\)-module \(M\) has
this property if and only if the endofunctor \(M\otimes-\) of the category of
left \(R\)-modules preserves arbitrary intersections of submodules (see
Definition~\ref{defi:intersection_property} and
Corollary~\ref{cor:mittag-leffler} on page~\pageref{cor:mittag-leffler}).  We
consider modules with the intersection property in
Chapter~\ref{chap:modules_with_int_property}, where we also give examples of
flat and faithfully flat modules which fail to have it.  For an
\(H\)-extension \(A/A^{co\,H}\) over a field we show that \(\psi(S)=H/K_S^+H\)
where \(K_S\) is the smallest left coideal subalgebra of \(H\) with the
property: \(\delta_A(S)\subseteq A\otimes K_S\), where \(\delta_A\) is the
comodule structure map of \(A\) (see Theorem~\ref{thm:connection_over_field}
on page~\pageref{thm:connection_over_field}).  Then we discuss Galois
closedness of the generalised quotients and subalgebras
in~\eqref{intro_eq:galois-connection}.  We show that a subalgebra
\(S\in\Sub(A/B)\) is closed if and only if the following canonical map:
\(\can_S:S\otimes A\sir A\cotensor_{\psi(S)}H\) \(\can(s\otimes
    a)=sa_{(0)}\otimes a_{(1)}\), is an isomorphism (see
Theorem~\ref{thm:closed_subalgebras} on \pageref{thm:closed_subalgebras}).  We
show that if a generalised quotient \(Q\) is such that \(A/A^{co\,Q}\) is
\(Q\)-Galois then it is necessarily closed under the assumptions that the
canonical map of \(A/A^{co\,H}\) is onto and the unit map \(1_A:R\sir A\) is
a pure monomorphism of left \(R\)-modules (Corollary~\ref{cor:Q-Galois_closed}
on page~\pageref{cor:Q-Galois_closed}).  Later we prove that this is also
a necessary condition for Galois closedness if \(A=H\) or, more generally, if
\(A/A^{co\,H}\) is a crossed product, \(H\) is flat and \(A^{co\,H}\) is
a flat Mittag--Leffler \(R\)-module (Theorem~\ref{thm:cleft-case} on
page~\pageref{thm:cleft-case}).  For \(H\)-Galois extensions over a field we
prove that the canonical map \(\can_Q:A\otimes_{A^{\co\,Q}}A\sir A\otimes Q\)
is isomorphic if and only if \(Q\) is a closed element and the map
\(\delta_A\otimes\delta_A:A\otimes_{A^{\co\,Q}}A\sir(A\otimes
    H)\otimes_{A\otimes H^{\co\,Q}}(A\otimes H)\) is injective
(Theorem~\ref{thm:Q-Galois_closed_over_field} on
page~\pageref{thm:Q-Galois_closed_over_field}).  We also consider the dual
case of \(H\)-module coalgebras, which later gives us a simple proof of
the bijective correspondence between generalised quotients and left coideal
subalgebras of~\(H\) if it is finite dimensional
(Theorem~\ref{thm:newTakeuchi} on page~\pageref{thm:newTakeuchi}):
\begin{equation}
    \qsub(H)\galois{\psi}{\phi}\qquot(H),\quad\phi(Q)\coloneq H^{co\,Q},\;\psi(K)=H/K^+H
\end{equation}
\citet[Thm.~3.10]{ps:gal-cor-hopf-bigal} showed that this Galois
correspondence restricts to a bijection between admissible quotients and
subalgebras, where admissibility assumes, among other things, faithful
flatness (for subobjects) and faithful coflatness (for quotient objects) and
\(H\) is required to be flat over the base ring.  The bijectivity of this
correspondence, without the assumptions of any faithfully (co)flatness was
proved by~\cite{ss:projectivity-over-comodule-algebras}.  He showed this by
proving that a finite dimensional Hopf algebra is free over any of its left
coideal subalgebra.  Our independent proof makes no use of this result.  We
also characterise closed elements of this Galois correspondence for Hopf
algebras which are flat over the base ring (Theorem~\ref{thm:closed-of-qquot}
on page~\pageref{thm:closed-of-qquot}).  As we already mentioned, we show that
a generalised quotient \(Q\) is closed if and only if \(H/H^{co\,Q}\) is
a \(Q\)-Galois extension.  Furthermore, we note that a left coideal
subalgebra~\(K\) is closed if and only if \(H\sir H/K^+H\) is a \(K\)-Galois
coextension (see Definition~\ref{defi:coGalois} on
page~\pageref{defi:coGalois}).  This gives an answer to the question when
a bijective correspondence between generalised quotients over which~\(H\) is
faithfully coflat and coideal subalgebras over which~\(H\) is faithfully flat
extends to a bijective correspondence without (co)flatness assumptions.  We
also note that for any \(H\)-extension \(A/B\) over a field a generalised
quotient \(Q\) which is not of the form \(H/K^+H\) cannot be closed.

There are two preprints of mine on the arxiv which are related with this
thesis.  These are both listed as co-authored with Dorota Marciniak, since we
started investigating this subject together, but she was not much involved in
the subsequent developments and went on to do her PhD in game theory.  The
first preprint (\href{http://arxiv.org/abs/0912.0291}{Galois Theory of Hopf
    Galois Extensions}, \texttt{arXiv:0912.0291}) is about Hopf-Galois
extensions over fields, and does not include the generalisation of
Schauenburg's correspondence between admissible objects. The more recent
preprint (\href{http://arxiv.org/abs/0912.1795}{Galois Theory for H-extensions
    and H-coextensions}, \texttt{arXiv:0912.1785}) will be harmonized with
this thesis, though it does not contain the results of the last chapter (pages
\pageref{chap:coring_approach}--\pageref{chap:coring_approach_end}).  The
latter preprint has a long list of arxiv versions since my first attempt to
build the Galois corres- pondence for extensions over rings turned out to
contain a mistake. The results of Chapter 3 (pages
\pageref{chap:lattices}--\pageref{chap:lattices_end}) are not included in
either of the preprints, except for the completeness of lattices of
generalised quotients and subalgebras of a Hopf algebra.

\chapter{Preliminaries}\label{subsec:basics}
In this chapter we introduce the basic tools and notions we will use.  We
start with an introduction to \emph{posets} (i.e. partially ordered sets),
\emph{lattices}, and \emph{Galois connections}.  Lattices are probably one of
the most ubiquitous objects in mathematics.  We recall the basics of the
theory of \emph{complete} and \emph{algebraic lattices}.  Complete lattices
are the ones that allow for arbitrary infima and suprema, rather than just
finite ones.  We will use the existence theorem of Galois connections which
requires completeness, and thus later on we will prove that the lattices that
appear in the theory of Hopf algebras are complete.  Algebraic lattices can be
characterised as lattices of subalgebras of \textit{universal algebras}. Also
lattices of congruences of \textit{universal algebras} are algebraic.  For
example all lattices of sub/quotient structures of all classical algebraic
structures like: modules or groups, have algebraic lattices of sub/quotient
structures.  In one of the next chapters we are going to show that also some
of the lattices which appear in the theory of Hopf algebras and their
(co)actions are algebraic, usually under strong exactness properties of their
underlying module (flat Mittag--Leffler modules, the property which will be
studied in the next chapter).  Every algebraic lattice is isomorphic to
a lattice of subsets of a set which is closed under infinite intersections and
directed sums (see Theorem~\ref{thm:algebraic_structures} on
page~\pageref{thm:algebraic_structures}).  We will use this theorem later, in
order to show that some of the lattices are algebraic.

In section~\ref{sec:Galois_connections} (on
page~\pageref{sec:Galois_connections}) we introduce the notion of a Galois
connection between posets.  The most important results of this section are:
Proposition~\ref{prop:properties-of-adjunction} (on
page~\pageref{prop:properties-of-adjunction}), where we show basic, but very
useful, properties of Galois connections and the existence theorem of Galois
correspondences between complete lattices
(Theorem~\ref{thm:existence-of-adjunction} on
page~\pageref{thm:existence-of-adjunction}).  The main references for lattice
theory and Galois connections are: \cite{gb-of:represantations_of_lattices},
\cite{gg:lattice-theory}, \cite{bd-hp:introduction-to-lattices} and
\cite{sr:lattices}.

In the final section~\ref{sec:intro_hopf_algebras} we collect basic
definitions in the theory of Hopf algebras and their actions and coactions.
You will find there definitions of a \emph{coalgebra}, a \emph{bialgebra} and
a \emph{Hopf algebra}.  We also introduce the theory of \emph{comodules} in
Subsection~\ref{ssec:comodules} (on page~\pageref{ssec:comodules}).  Then we
pass to the main object of our study \emph{comodule algebras}
(Subsection~\ref{ssec:comodule_algberas} on
page~\pageref{ssec:comodule_algberas}).  Here we also recall the theory of
crossed products.  Their Galois theory we will study later on.  Module
algebras over a bialgebra or a Hopf algebras are recalled in
Subsection~\ref{ssec:module_algebras} (on
page~\pageref{ssec:module_algebras}).  We close this section with a discussion
(Subsection~\ref{ssec:finite_galois_theory_of_fields} on
page~\pageref{ssec:finite_galois_theory_of_fields}) of finite Galois field
extensions in terms of group Hopf algebras.

\section{Posets and Lattices}
\index{poset}
\index{partially ordered set|see{poset}}
\begin{definition}
    A \bold{partially ordered set}, \bold{poset} for short, is a set
    \(P\) together with an order relation \(\preceq\) which is reflexive,
    transitive and antisymmetric.

    \noindent A \bold{down-set} \(D\) of \(P\) is a subset \(D\subseteq
	P\) which satisfies: \(d'\preceq d\in D\Rightarrow d'\in D\),
    dually an \bold{upper-set} \(U\) of \(P\) is a subset \(U\subseteq
	P\) with the property \(U\ni u\preceq u'\Rightarrow u'\in U\).

    \noindent Let \((P,\preceq)\) and \((Q,\leq)\) be two posets. Then
    a map \(\phi:P\sir Q\) is called \bold{monotonic} if \(p\preceq p'\)
    implies \(\phi(p)\leq \phi(p')\) for \(p,p'\in P\).
\end{definition}
Let \(B\) be a subset of \(P\). The smallest (greatest) element \(b\) of
\(B\), if it exists, is defined as the element of \(B\) such that for any
\(b'\in B\) \(b\preceq b'\) (respectively, \(b\succeq b'\)). They are unique
by the antisymmetry of an order relation.  An element \(p\in P\) is called an
upper (lower) bound of \(B\) if for all \(b\in B\) \(p\succeq b\)
(respectively, \(p\preceq b\)). We let \(\sup B\) denote the least upper bound
(supremum for short) of \(B\) and \(\inf B\) the greatest lower bound
(infimum) of \(B\).

\begin{definition}
    Let \((P,\preceq)\) be a poset, then the \bold{opposite poset}
    \((P^\op,\preceq^\op)\) is defined by \(P^\op\coloneq P\) (as sets) and
    \(p\preceq^\op q\iff q\preceq p\) for \(p,q\in P\). We will often write
    \(P^\op\) to denote the opposite poset of a poset \(P\).

    \noindent Let \((P,\preceq)\) and \((Q,\leq)\) be two posets. A map
    \(\phi:P\sir Q\) is called \bold{antimonotonic} if \(\phi\) is a map of
    posets \(\phi:(P^\op,\preceq^\op)\sir (Q,\leq)\).
\end{definition}

\index{lattice}
A \bold{lattice} is a poset in which there exists the supremum and infimum of
any subset with two elements or equivalently of any finite nonempty subset.
A lattice can also be defined as an algebraic structure which has two binary
operations: join (an abstract supremum of two elements) denoted by $\vee$ and
meet (an abstract infimum of two elements) denoted by $\wedge$ which satisfy
the following equalities:
\[\begin{array}{lll} 
	a\wedge a=a              & a\vee a=a              & \text{\textsf idempotent laws}  \\
	a\vee b=b\vee a           & a\wedge b=b\wedge a           & \text{\textsf commutative laws} \\
	a\vee (b\vee c)=(a\vee b)\vee c & a\wedge (b\wedge c)=(a\wedge b)\wedge c & \text{\textsf associative laws} \\
	a\vee (a\wedge b)=a         & a\wedge (a\vee b)=a         & \text{\textsf absorption laws}
\end{array}\]
The correspondence between lattice operations and the lattice order is made
by: for a given order \(\preceq\) we define \(a\wedge b\coloneq\inf\{a,b\}\)
and \(a\vee b\coloneq\sup\{a,b\}\), while for given lattice operations
\((\wedge,\vee)\) one defines an order by \(a\preceq b\iff a=a\wedge b\iff
    b=a\vee b\). Let \((L,\wedge_L,\vee_L)\) and \((M,\wedge_M,\vee_M)\) be
two lattices. Then a map \(f:L\sir M\) is called a \bold{lattice homomorphism}
if for any \(k,l\in L\) we have \(f(k\wedge_L l)=f(k)\wedge_Mf(l)\) and
\(f(k\vee_L l)=f(k)\vee_Mf(l)\). Note that a lattice homomorphism is
necessarily an order preserving map, but the converse might not be true.
A \bold{lattice antihomomorphism} is a map \(f:L\rightarrow M\) such that
\(f\) is a homomorphism  of lattices \(L^\op\) and \(M\), i.e. for any
\(k,l\in L\) we have \(f(k\wedge_L l)=f(k)\vee_Mf(l)\) and \(f(k\vee_L
    l)=f(k)\wedge_Mf(l)\). 

\index{lattice!dual lattice}
If \((L,\vee,\wedge)\) is a lattice then its dual is \(L,\wedge,\vee\).  Note
that this definition is compatible with the definition of a dual poset.

We refer the reader to~\cite{gg:lattice-theory} for the theory of lattices.
\index{lattice!complete}
\begin{definition}
	A lattice $(L,\vee,\wedge)$ is \bold{complete} if for every non empty
	set $B\subseteq L$ there exists $\sup B$ and $\inf B$.
\end{definition}
\begin{remark}\label{rem:complete_lattice}
    In a lattice \(L\) there exists arbitrary infima (suprema) if and only if
    there are arbitrary suprema (infima) of non empty subsets. Clearly, if a poset is closed under
    arbitrary infima (or suprema) then it is a complete lattice, since the
    following formulas hold 
    \begin{alignat*}{2}
	l\vee k & =\bigwedge\{l'\in L:l'\geq l,\,l'\geq k\} & \quad & \text{if
    }L\text{ is closed under infima,}\\
	l\wedge k & =\bigvee\{l'\in L:l'\leq l,\,l'\leq k\} &       & \text{if }L\text{ is closed under suprema}
    \end{alignat*}
    for \(l,k\in L\).  We use the lattice notation for infima and suprema:
    \(\wedge S\) denotes infimum of a subset \(S\subseteq L\) and
    \(\bigwedge S\) its supremum.
\end{remark}
\begin{definition}
    \begin{enumerate}
	\index{lower semilattice}
	\index{filter}
	\item A poset \(P\) which has finite infima is called a \bold{lower
	      semilattice}.  A \bold{filter} of a lower semilattice \(P\) is
	  an upper-set which is closed under finite infima.  The set of
	  filters will be denoted by \(\mathcal{F}(P)\).
	\index{upper semilattice}
	\index{ideal}
	\item A poset \(P\) which has finite suprema is called an \bold{upper
	      semilattice}.  An \bold{ideal} of an upper semilattice \(P\) is
	  a down-set which is closed under finite suprema.  The set of ideals
	  will be denoted by \(\mathcal{I}(P)\). 
    \end{enumerate}
\end{definition}
\index{sublattice}
\index{upper subsemilattice}
\index{lower subsemilattice}
\index{join-semilattice|see{upper semilattice}}
\index{meet-semilattice|see{lower semilattice}}
An upper semilattice is also called join-semilattice, and lower semilattice is
called meet-semilattice.  A \bold{sublattice} of a lattice is a subset closed
under meet and join. An \bold{upper (lower) subsemilattice} of an upper
(lower) semilattice is a subset which is closed under join (meet
respectively). 

Next we state a lemma which will be extensively used in the first part of this
work.
\begin{lemma}\label{lem:lk}
	Let $(M,\wedge_M,\vee_M)$ be a complete lattice and let there be two
	complete lattices which are upper sub-semilattices of $M$:
	$(K,\wedge_K,\vee_M)$ and $(L,\wedge_L,\vee_M)$. Let $K$ and $L$ have
	the same smallest element. Then $(K\cap L,\wedge_{K\cap L},\vee_{K\cap
	    L})$, where
	\begin{align*}
	    a\vee_{K\cap L}b   & \coloneq a\vee_M b\\
	    a\wedge_{K\cap L}b & \coloneq\bigvee_M\{c\in K\cap L:\ c\leq a\wedge_M b\}
	\end{align*}
	is a \bold{complete lattice}.
\end{lemma}
\begin{proof}
The join \(\wedge_{K\cap L}\) is well defined, because $K$ and $L$ are upper
subsemilattices of the lattice $M$.  Now let us prove that the meet is well
defined as well.  The lattices $K$ and $L$ have the same smallest element, so
the set \mbox{$\{c\in K\cap L:\ c\leq a\ and\ c\leq b\}$}, for any
\mbox{$a,b\in K\cap L$}, is non-empty. The supremum \mbox{$\bigvee_M\{c\in
	K\cap L:\ c\leq a\ and\ c\leq b\}$} exists and belongs to $K\cap L$,
because both $K$ and $L$ are complete. The axioms of lattice operations are
trivially satisfied.

It remains to show that the lattice $(K\cap L,\wedge_{K\cap L},\vee_{K\cap L})$
is complete. Let $B\subseteq K\cap L$. Then
\[\sup_{K\cap L}B=\sup_MB\quad \mathrm{and}\quad \inf_{K\cap L}B=\sup_M\{x\in K\cap L:\forall_{b\in B}\,x\leq b\}.\] 
This infimum and this supremum exist, because $M$ is complete and they belong to
$K\cap L$ because $K$ and $L$ are complete upper sub-semilattices of $M$.
\end{proof}
A subset \mbox{$\{x\in L \,|\;a\leq x\leq b\}$} of the lattice $L$ is called
an \bold{interval} and it will be denoted by \([a,b]\).
\index{lattice!compact element}
\begin{definition}\label{defi:compact}
	An element $z$ of a lattice $L$ is called \bold{compact} if for any
	subset $S\subseteq L$ such that $z\leq\bigvee S$ there exists a finite
	subset $S_f$ of $S$ with the property $z\leq\bigvee S_f$. 
\end{definition}
\begin{example}\label{ex:vector-subspace}
    \begin{enumerate}
	\item Let $V$ be a $\k$-vector space. Then a subspace $W$ of \(V\) is
	      a compact element of the lattice of subspaces
	      $\Sub_{\textit{Vect}}(V)$ if and only if it is finite
	      dimensional.
	\item Let us consider the lattice \(\bN\cup\{\infty\}\):
	    \begin{center}
		\begin{tikzpicture}
		    \fill (0cm, 3cm) circle (2pt) node[left]{\(i_\infty\)};
		    \node at (0cm, 2.6cm) {\(\vdots\)};
		    \fill (0cm, 2cm) circle (2pt) node[left]{\(i_3\)};
		    \fill (0cm, 1cm) circle (2pt) node[left]{\(i_2\)};
		    \fill (0cm, 0cm) circle (2pt) node[left]{\(i_1\)};
		    \draw[-] (0cm,0cm) -- (0cm,2.2cm);
		    \draw[-] (0cm,2.8cm) -- (0cm, 3cm);
		\end{tikzpicture}
	    \end{center}
	    Then each element \(i_k\) for \(k<\infty\) is compact and
	    \(i_\infty\) is not compact though it is a supremum of compact
	    elements.
    \end{enumerate}

\end{example}
\index{lattice!algebraic}
\index{lattice!dually algebraic}
\index{algebraic lattice|see{lattice!algebraic}}
\index{dually algebraic lattice|see{lattice!dually algebraic}}
\begin{definition}\label{defi:alglat}
	A lattice is \bold{algebraic} if it is complete and every element is
	a supremum of compact elements.  A lattice is \bold{dually algebraic}
	if its dual is algebraic.
\end{definition}
\begin{examplebr}
    \begin{enumerate}
	\item Let \(A\) be an \(\R\)-algebra, where \(\R\) is a commutative
	    ring, then the lattices of ideals, left (right) ideals and
	    subalgebras are algebraic. 
	\item The lattice from Example~\ref{ex:vector-subspace}(ii) is
	    an algebraic lattice.
        \item The following complete lattice is not algebraic:
	    \begin{center}
		\begin{tikzpicture}
		    \fill (1cm, 2cm) circle (2pt) node[right]{\(x\)};
		    \fill (0cm, 3cm) circle (2pt) node[left]{\(i_\infty\)};
		    \node at (0cm, 2.6cm) {\(\vdots\)};
		    \fill (0cm, 2cm) circle (2pt) node[left]{\(i_2\)};
		    \fill (0cm, 1cm) circle (2pt) node[left]{\(i_1\)};
		    \draw[-] (0cm, 1cm) -- (1cm, 2cm) -- (0cm, 3cm);
		    \draw[-] (0cm,1cm) -- (0cm,2.2cm);
		    \draw[-] (0cm,2.8cm) -- (0cm, 3cm);
		\end{tikzpicture}
	    \end{center}
	    Then \(x\) is not a compact element, since
	    \(x\leq\bigvee_{k<\infty}i_k\) but for any \(k_0<\infty\)
	    \(\bigvee_{k\leq k_0}i_k=i_{k_0}\) is not comparable with \(x\).
	    The only compact elements are \(i_k\) for \(k<\infty\).
	    Furthermore, \(x\) is not a supremum of compact elements. 
    \end{enumerate}
\end{examplebr}
\index{lattice!algebraic!examples}
It is a well known theorem of Universal Algebra that lattices of subalgebras
(see~\cite[Cor.~3.3]{sb-hs:universal_algebra}) and lattices of congruences
(quotient structures, see~\cite[Thm.~5.5]{sb-hs:universal_algebra}) of any
algebraic structure are algebraic.  In particular, the lattices of sub-objects
and quotient objects of classical algebraic structures like groups,
semi-groups, rings, modules, etc.  are algebraic.  Furthermore, Birkhoff and
Frink proved that any algebraic lattice can be represented as a lattice of
subalgebras of a universal algebra
(see~\cite{gb-of:represantations_of_lattices}).
\begin{remark}\label{rem:alg}
    Let $\Sub_\textit{Vect}(V)$ be the lattice of subspaces of a finite
    dimensional $\k$-vector space $V$.  By Example~\ref{ex:vector-subspace}
    every subspace of a finite dimensional vector space is a compact element
    of $\Sub_\textit{Vect}(V)$.  Thus any sublattice of $\Sub_\textit{Vect}(V)$
    is algebraic.  It is dually algebraic as well, since we have a dual
    isomorphism:
    \[\Sub_\textit{Vect}(V)\simeq\Sub_\textit{Vect}(V^*)\]
    Thus a sublattice of $\Sub_\textit{Vect}(V)$ is anti-isomorphic to an
    algebraic sublattice of the lattice $\Sub_\textit{Vect}(V^*)$.
\end{remark}
\begin{definition}
    Let \(P\) be a poset.  Then a non-empty set \(D\) of \(P\) is
    (upwards) \bold{directed} if for every \(a,b\in D\) there exists \(c\in D\) such
    that \(c\geq a\) and \(c\geq b\).
\end{definition}
\index{\(\cap\overrightarrow{\cup}\)-structure}
\begin{definition}
    Let \(X\) be a set and let \(\mathcal{P}(X)\) denote the power set of
    \(X\).  Note that \(\mathcal{P}(X)\) is a poset with the inclusion order
    of subsets of \(X\).  A subset \(\mathcal{M}\subseteq\mathcal{P}(X)\) is
    called a \(\cap\overrightarrow{\cup}\)-\bold{structure} if \(\mathcal{M}\)
    is closed under arbitrary intersections and unions of directed sets of its
    elements.
\end{definition}
\index{directed unions}
Directed unions of subsets we will denote by \(\overrightarrow{\cup}\).  The
following theorem is an important characterisation of algebraic lattices.
\index{lattice!algebraic!\(\cap\overrightarrow{\cup}\)-structure}
\begin{theorem}[{\cite[Thm.~7.5]{sr:lattices}}]\label{thm:algebraic_structures}
    Let \(L\) be a lattice.  Then \(L\) is algebraic if and only if it is
    isomorphic as a poset to a \(\cap\overrightarrow{\cup}\)-structure.
\end{theorem}
The proof is after~\cite[Thm.~7.5]{sr:lattices}.  We put it here since
it gives an important point of view on the structure of algebraic
lattices.
\begin{proof} 
    First let us observe that a \(\cap\overrightarrow{\cup}\)-structure
    \(\mathcal{M}\) on a set \(X\) is an algebraic lattice.  Since
    \(\mathcal{M}\) is closed under arbitrary intersections it is a complete
    lattice with meet \(\cap\) and join \(\cup\).  Now, let \(S\subseteq X\)
    then we let \(\langle S\rangle\coloneq\bigcap\{Y\in
	    \mathcal{M}:\,S\subseteq Y\}\).  We first show that \(\langle
	S\rangle\) is a compact element of \(\mathcal{M}\) whenever \(S\) is
    finite.  If \(\langle
	S\rangle\subseteq\overrightarrow{\bigcup}\{M_i:\,M_i\in\mathcal{M},i\in
	    I\}\), for a finite set \(S\), then for some \(i\in I\) we have
    \(S\subseteq M_i\) and thus \(\langle S\rangle\subseteq M_i\).  Conversely,
    let us assume that \(K\in\mathcal{M}\) is a compact element.  The family
    \(\{\langle S\rangle:\,S\in\mathcal{M},\,S\subseteq K,\,S\text{ is
		finite}\}\) is a directed set, since for \(\langle
	S_i\rangle\) (\(i=1,2\)) we have \(\langle S_i\rangle\subseteq\langle
	S_1\cup S_2\rangle\subseteq K\) for \(i=1,2\). Now, we have
    \[K\subseteq\overrightarrow{\bigcup}\{\langle S\rangle:\,S\subseteq K,S\text{ is finite}\}\]
    And thus there exists a finite set \(S\) such that \(K\subseteq\langle
	S\rangle\). The other inclusion is clear, since \(S\subseteq K\). Thus
    we obtain \(K=\langle S\rangle\) for some finite set \(S\).
    
    Now, consider the set \(\mathcal{K}(L)\) of all compact elements of \(L\).
    Then \(\mathcal{K}(L)\) is an upper semilattice which inherits joins from
    the lattice \(L\) (i.e.  supremum in \(L\) of compact elements is
    a compact element). Thus we can consider the set of ideals
    \(\mathcal{I}(\mathcal{K}(L))\) of \(\mathcal{K}(L)\), which is an
    \(\cap\overrightarrow{\cup}\)-structure. Let \((J_i)_{i\in I}\) be
    a family of ideals then the intersection \(\bigcap_{i\in I}J_i\) is an
    ideal: for if \(l\in\bigcap_{i\in I}J_i\) then for every \(i\in I\) \(l\in
	J_i\) and thus if \(l'\leq l\) then \(l'\in J_i\) for every \(i\in
	I\), and thus \(l'\in\bigcap_{i\in I}J_i\). This shows that
    \(\bigcap_{i\in I}J_i\) is a down set.  If \(l\) and \(l'\) belong to
    \(\bigcap_{i\in I}J_i\) then for every \(i\in I\) we have \(l,l'\in J_i\)
    and thus \(l\vee l'\in J_i\) for all \(i\in I\), hence \(l\vee
	l'\in\bigcap_{i\in I}J_i\).  This proves that the set of ideals is
    closed under all intersections.  Now, let \((J_i)_{i\in I}\) be a directed
    system of ideals of \(\mathcal{K}(L)\).  We show that \(\bigcup_{i\in
	    I}J_i\) is an ideal.  For this let \(l\in{\bigcup}_{i\in I}J_i\)
    then there exists \(i\) such that \(l\in J_i\). It follows that
    \({\bigcup}_{i\in I}J_i\) is a down set, since each \(J_i\) is.  Let
    \(l,l'\in{\bigcup}_{i\in I}J_i\). Then there exist \(i,i'\in I\) such that
    \(l\in J_i\) and \(l'\in J_{i'}\). Furthermore, there exists \(i''\in I\)
    such that \(J_i\cup J_{i'}\subseteq J_{i''}\), and thus \(l,l'\in
	J_{i''}\).  In consequence, \(l\vee l'\in
	J_{i''}\subseteq{\bigcup}_{i\in I}J_i\).  Hence \({\bigcup}_{i\in
	    I}J_i\) is an ideal. This shows that the set of ideals of
    \(\mathcal{K}(L)\) is indeed a \(\cap\overrightarrow{\cup}\)-structure. 

    Note that \(\emptyset\in\mathcal{K}(L)\) is the least element of
    \(\mathcal{K}(L)\).  We have a map: \(L\ni
	a\selmap{}\{k\in\mathcal{K}(L):\ k\leq
	    a\}\in\mathcal{I}(\mathcal{K}(L))\). It is clear that this map
    preserves the order. Let us take \(a,a'\in L\), such that
    \(\{k\in\mathcal{K}(L):\ k\leq a\}=\{k\in\mathcal{K}(L):\ k\leq a'\}\).
    Then we get
    \[a=\sup\{k\in\mathcal{K}(L):\ k\leq a\}=\sup\{k\in\mathcal{K}(L):\ k\leq
	    a'\}=a'\]
    where the first and last equalities follow since \(L\) is an algebraic
    lattice. It remains to show that the considered map is an epimorphism. For
    this let \(I\in\mathcal{I}(\mathcal{K}(L))\). Now let \(a\coloneq
	\bigvee I\). Then
    \(I\subseteq\{k\in\mathcal{K}(L):\,k\leq a\}\).
    Now, if \(x\) belongs to the right hand side, then since
    \(x\leq\bigvee I\), since \(x\) is compact, there exists
    \(i\in I\) such that \(x\subseteq i\), and thus \(x\in I\) (since \(I\) is
    an ideal). Since the constructed map is a poset isomorphism it is also
    a lattice isomorphism. This ends the proof.
\end{proof}
The following remark was proven at the beginning of the previous theorem.
\index{lattice!compact element}
\begin{remark}\label{rem:compact_elements}
    Let us stress that the compact elements of
    a \(\cap\overrightarrow{\cup}\)-structure \(X\) are exactly the elements
    of the form \(\langle S\rangle\) where \(S\subseteq X\) is a finite
    subset.
\end{remark}

\section{Galois connections}\label{sec:Galois_connections}

We start with a definition:
\index{Galois connection}
\index{Galois connection!closed element}
\begin{definition}[Galois connection]\label{defi:Galois-connection}
	Let \((P,\preceq)\) and \((Q,\leq)\) be two partially ordered sets.
	Antimonotonic morphisms of posets \(\phi:P\mpr{}Q\) and
	\(\psi:Q\mpr{}P\) establish a \bold{Galois connection} if 
	\begin{equation}\label{eq:galprop}
		\mathop{\forall}\limits_{p\in P\,q\in Q}\ p\preceq \psi\circ \phi(p)\
		and\ q\leq \phi\circ \psi(q)
	\end{equation} 
	We refer to this property as the \bold{Galois property}. An element of
	\(P\) (or \(Q\)) will be called \bold{closed} if it is fixed by
	\(\psi\phi\) (\(\phi\psi\) respectively).  The sets of closed elements
	of \(P\) and \(Q\) will be denoted by \(\ov P\) and \(\ov Q\)
	respectively.  A standard notation for a Galois connection is 
	\[P\galois{\phi}{\psi}Q\] 
	Another name which appear in the literature for this notion is
	\bold{Galois correspondence}, which we use interchangeably.
\end{definition}
Note that the Galois property~\eqref{eq:galprop} is equivalent to:
\begin{equation}\label{eq:adjointness}
    \phi(p)\geq q\iff p\preceq\psi(q)
\end{equation}
for any \(p\in P\) and \(q\in Q\).  Categorically speaking, Galois connections
are the same as contravariant adjunctions between posets, which can be
understood as categories in a straightforward manner.  Note that in the
definition of a Galois connection \(\phi\) and \(\psi\) are not assumed to be
lattice antihomomorphisms.
\index{Galois connection!properties}
\begin{proposition}\label{prop:properties-of-adjunction}
	Let \((\phi,\psi)\) be a Galois connection between posets
	\((P,\preceq)\) and \((Q,\leq)\). Then:
	\begin{enumerate}
		\item $\ov P=\psi(Q)$ and $\ov Q=\phi(P)$,
		\item The restrictions $\phi|_{\ov P}$ and $\psi|_{\ov Q}$ are
		      \textsf{inverse bijections} of $\ov P$ and $\ov Q$ and
		      \(\ov P\) and \(\ov Q\) are the largest subsets such
		      that \(\phi\) and \(\psi\) restricts to inverse
		      bijections.
		\item The map \(\psi\) is \textsf{unique} in the sense that if
		      \((\phi,\psi)\) and \((\phi,\psi')\) form Galois
		      connections then necessarily \(\psi=\psi'\). In
		      a similar way \(\phi\) is \textsf{unique}. 
		\item The map \(\phi\) is \textsf{mono} (\textsf{onto}) if and
		      only if the map \(\psi\) is \textsf{onto}
		      (\textsf{mono}). 
		\item If one of the two maps \(\phi,\psi\) is an
		      \textsf{isomorphism} then the second is its
		      \textsf{inverse}.
		\item The map \(\phi\) is an injection if and only if
		      \(P=\overline P\).
	\end{enumerate}
\end{proposition}
\begin{proof}
    \begin{enumerate}
	\item It is clear that \(\overline{P}\subseteq\psi(Q)\), now let
	      \(p=\psi(q)\) for \(q\in Q\). Then
	      \(\psi\phi(p)=\psi\phi\psi(q)\preceq\psi(q)=p\), since
	      \(\phi\psi(q)\geq q\), furthermore
	      \(\psi\phi(p)=\psi\phi\psi(q)\succeq\psi(q)=p\) by the Galois
	      property~\eqref{eq:galprop} when applied to \(\psi(q)\). The
	      other equality \(\overline{Q}=\phi(P)\) follows in the same way.
	\item We showed that \(\psi\phi\psi(q)=\psi(q)\); a similar argument
	      shows that the following equality holds
	      \(\phi\psi\phi(p)=\phi(p)\).  This together with~(i)
	      proves~(ii).
	\item Let both \((\phi,\psi)\) and \((\phi,\widetilde\psi)\) be Galois
	      connections. We have \(\psi\geq\psi\phi\widetilde\psi\), since
	      \(\id_Q\leq\phi\widetilde\psi\) and
	      \(\psi\phi\widetilde\psi\geq\widetilde\psi\) since
	      \(\psi\phi\succeq\id_P\). Thus \(\psi\geq\widetilde\psi\).
	      Changing the roles of \(\psi\) and \(\widetilde\psi\) we will
	      get \(\widetilde\psi\geq\psi\) and thus indeed
	      \(\psi=\widetilde\psi\).
	\item[\textit{(iv,v)}] Let us assume that \(\phi\) is
	     a monomorphism. Then \(\psi\phi=\id_P\) since we have
	     \(\phi\psi\phi=\phi\). Thus \(\psi\) is a split epimorphism.  If
	     \(\phi\) is an epimorphism, then it follows that
	     \(\phi\psi=\id_Q\) and thus \(\psi\) is a split monomorphism.
	     This proves both~(iv) and~(v).
	\item[\textit{(vi)}]  The last assertion is a consequence of~(iv).
    \end{enumerate}
\end{proof}
\begin{lemma}\label{lem:continuity}
    Let \((\phi,\psi)\) be a Galois connection between two posets
    \((P,\preceq)\) and \((Q,\leq)\).  Then \(\phi\) and \(\psi\) reflect all
    existing suprema into infima.
\end{lemma}
\begin{proof}
    Let $p_i\in P$ be such that $\bigvee p_i$ exists. Then \(\phi(\bigvee
	p_i)\) is a lower bound of \(\{\phi(p_i):i\in I\}\). Furthermore, if
    \(b\) is a lower bound of \(\{\phi(p_i):i\in I\}\), i.e. \(\phi(p_i)\geq
	b\) for all \(i\in I\), then by~\eqref{eq:adjointness}
    \(p_i\preceq\psi(b)\) and thus \(\bigvee p_i\preceq\psi(b)\). It follows
    that \(\phi(\bigvee p_i)\geq b\), and so \(\phi(\bigvee p_i)=\bigvee
	\phi(p_i)\).
\end{proof}
\begin{corollary}
    Let \((\phi,\psi)\) be a Galois connection between two complete lattices
    \((P,\succeq)\) and \((Q,\geq)\).  Then the posets of closed elements are
    also complete.
\end{corollary}
\begin{proof}
    It is enough to show that \(\overline{P}\) is closed under arbitrary
    infima. Let \(S\subseteq \overline{P}\), then \(S=\psi(W)\) for some
    \(W\subseteq Q\). Now we have: \(\inf_P S=\inf_P\psi(W)=\psi(\sup_Q
	W)\in\overline{P}\).
\end{proof}
Below we construct an example of a Galois connection in which infima are not
preserved, even if one of the maps is an antihomomorphism of lattices.
\index{Galois connection!example}
\begin{example}
    Let us consider the following two lattices:
    \begin{center}
	\begin{tikzpicture}
	    \fill (0cm, 1cm) circle (2pt) node[above]{1};
	    \fill (0cm, 0cm) circle (2pt) node[left]{x};
	    \fill (0cm, -1cm) circle (2pt) node[below]{0};
	    \draw[-] (0cm,1cm) -- (0cm,-1cm);
	\end{tikzpicture}
	\hspace{2cm}
	\begin{tikzpicture}
	    \fill (0cm, 1cm) circle (2pt) node[above]{1};
	    \fill (-1cm, 0cm) circle (2pt) node[left]{a};
	    \fill (1cm, 0cm) circle (2pt) node[right]{b};
	    \fill (0cm, -1cm) circle (2pt) node[below]{0};
	    \draw[-] (0cm,1cm) -- (-1cm, 0cm) -- (0cm,-1cm);
	    \draw[-] (0cm,1cm) -- (1cm, 0cm) -- (0cm,-1cm);
	\end{tikzpicture}
    \end{center}
    Then the maps \(\phi(x)=a\), \(\phi(1)=0\), \(\phi(0)=1\), and
    \(\psi(a)=x\), \(\psi(b)=0\), \(\psi(1)=0\), \(\psi(0)=1\)
    define a Galois correspondence. The closed elements are \(\{1,x,0\}\) and
    \(\{1,a,0\}\). Furthermore \(\phi\) is an antihomomorphism of lattices
    (reflects suprema (infima) into infima (suprema)), but \(\psi(a\wedge
	b)=\psi(0)=1\), while \(\psi(a)\vee\psi(b)=x\vee 0=x\).
\end{example}
\index{Galois connection!existence}
\begin{theorem}\label{thm:existence-of-adjunction}
	Let \(P\) and \(Q\) be two posets. Let \(\phi:P\sir Q\) be an
	anti-monotonic map of posets. If \(P\) is complete then there
	\textsf{exists a Galois connection} \((\phi,\psi)\) if and only if
	\textsf{\(\phi\) reflects all suprema into infima}.
\end{theorem}
\begin{proof}
	After Lemma~\ref{lem:continuity}, it remains to prove that if $\phi$
	reflects infinite suprema and $P$ is complete then $\phi$ has a right
	adjoint. 
	
	Let $\phi$ be an antimonotonic map which reflects infinite suprema.
	Then $\psi$ can be defined by the formula:
	\begin{equation}\label{eq:right-adjoint}
	    \psi(q)=\bigvee\{p\in P:\phi(p)\geq q\}\quad\forall_{q\in Q}.
	\end{equation}
	One then easily checks that:
	\[\phi(p)\geq q\iff p\preceq\bigvee\{\tilde p\in P:\phi(\tilde p)\geq q\}:=\psi(q)\]
\end{proof}

\section{Algebras, coalgebras and Hopf algebras}\label{sec:intro_hopf_algebras}
The unadorned tensor product \(\otimes\) will denote the tensor product of
modules over a commutative base ring~\(\R\).
\subsection{Algebras}
\index{algebra}
\begin{definition}
    An \bold{algebra} \(A\) over a ring \(\R\) is an \(\R\)-module together
    with a \(\R\)-bilinear associative multiplication: \(m_A:A\times A\sir
	A\).  We consider all algebras to be unital, that is there exists
    a unit element of \(A\), denoted by \(1_A\), such that \(1_A\cdot
	a=a=a\cdot 1_A\) for all \(a\in A\), where \(\cdot\) is the usual
    notation for a multiplication \(m_A\). Furthermore, the homomorphism
    \(1:R\sir A,r\selmap{} r1_A\) factors through the center of \(A\).  The
    associativity and unitality conditions can be expressed diagrammatically
    by imposing that the following diagrams commute:  
    \begin{center}
	\begin{tikzpicture}
	    \matrix[column sep=1cm, row sep=1cm]{
		\node (A1) {\(A\otimes A\otimes A\)}; & \node (A2) {\(A\otimes A\)};\\
		\node (B1) {\(A\otimes A\)};    & \node (B2) {\(A\)};\\
	    };
	    \draw[->] (A1) --node[above]{\(m_A\otimes\id_A\)} (A2);
	    \draw[->] (A1) --node[left]{\(\id_A\otimes m_A\)} (B1);
	    \draw[->] (A2) --node[right]{\(m_A\)} (B2);
	    \draw[->] (B1) --node[below]{\(m_A\)} (B2);
	\end{tikzpicture}
	\begin{tikzpicture}
	    \matrix[column sep=1cm,row sep=1cm]{
		\node (A1) {\(\R\otimes A\)}; & \node (A2) {\(A\)};     & \node (A3) {\(A\otimes \R\)}; \\
		                       & \node (B) {\(A\otimes A\)};\\
	    };
	    \draw[->] (A1) --node[above]{\(\cong\)} (A2);
	    \draw[->] (A3) --node[above]{\(\cong\)} (A2);
	    \draw[->] (A1) --node[below left]{\( 1\otimes\id_A\)} (B);
	    \draw[->] (A3) --node[below right]{\(\id_A\otimes 1\)} (B);
	    \draw[<-] (A2) --node[right]{\(m_A\)} (B);
	\end{tikzpicture}
    \end{center}
    where \(m_A:A\otimes A-\sir A\) is the \(\R\)-linear map defined by the
    \(\R\)-bilinear multiplication \(\cdot:A\times A\sir A\).

    A morphism of \(\R\)-algebras \(f:(A,m_A)\sir(B,m_B)\) is an \(\R\)-module
    homomorphism from \(f:A\sir B\) such that \(f(a\cdot a')=f(a)\cdot f(a')\).
\end{definition}
Let us note that indeed if the above diagrams commute then the image of
\(1:R\sir A\) lies in the center of \(A\), since:
\begin{equation*}
    \begin{split}
	(r1_A)\cdot a&=m_A(r1_A\otimes a)=m_A(1_A\otimes ra)\\
	&=m_A(ra\otimes 1_A)=m_A(a\otimes r1_A)=a\cdot (r1_A)
    \end{split}
\end{equation*}

For \(a,b\in A\) we will  often write \(ab\) instead of \(a\cdot b\) or
\(m_A(a\otimes b)\).
\subsection{Coalgebras}
\index{coalgebra}
\index{coalgebra!comultiplication}
\index{colagebra!counit}
\begin{definition}\label{defi:coalgebra}
    A \bold{coalgebra} \(C\) over a ring \(\R\), an \(\R\)-coalgebra for
    short, is an \(\R\)-module, together with two maps: comultiplication
    \(\Delta:C\sir C\otimes C\) and counit \(\epsilon:C\sir \R\) such that the
    following diagrams commute:
    \begin{center}
	{\hfill\begin{tikzpicture}
	    \matrix[column sep=1cm, row sep=1cm]{
		\node (A1) {\(C\)};    & \node (A2) {\(C\otimes C\)};\\
		\node (B1) {\(C\otimes C\)}; & \node (B2) {\(C\otimes C\otimes C\)};\\
	    };
	    \draw[->] (A1) --node[above]{\(\Delta\)} (A2);
	    \draw[->] (A1) --node[left]{\(\Delta\)} (B1);
	    \draw[->] (A2) --node[right]{\(\Delta\otimes\id_C\)} (B2);
	    \draw[->] (B1) --node[below]{\(\id_C\otimes\Delta\)} (B2);
	\end{tikzpicture}
	\begin{tikzpicture}
	    \matrix[column sep=1cm,row sep=1cm]{
		\node (A1) {\(\R\otimes C\)}; & \node (A2) {\(C\)}; & \node (A3) {\(C\otimes \R\)}; \\
		                    & \node (B) {\(C\otimes C\)};\\
	    };
	    \draw[<-] (A1) --node[above]{\(\cong\)} (A2);
	    \draw[<-] (A3) --node[above]{\(\cong\)} (A2);
	    \draw[<-] (A1) --node[below left]{\(\epsilon\otimes\id_C\)} (B);
	    \draw[<-] (A3) --node[below right]{\(\id_C\otimes\epsilon\)} (B);
	    \draw[->] (A2) --node[right]{\(\Delta\)} (B);
	\end{tikzpicture}
	\hfill\refstepcounter{equation}\raisebox{11mm}{\normalfont{(\theequation)}}\label{diag:counit}}
    \end{center}
    \index{Sweedler notation}
    \index{coalgebra!Sweedler notation}
    Where \(C\otimes \R\cong C\cong \R\otimes C\) are the canonical
    isomorphisms.  We use the following version of the sumless Sweedler
    notation for the coproduct: \(\Delta(c)=c_{(0)}\otimes c_{(1)}\):
    \(\Delta(c)\) is an element of the tensor product \(C\otimes C\). Thus it
    is a sum of simple tensors \(\sum_{i=1}^nc_i'\otimes c_i''\). In the
    Sweedler notation this sum is written as \(\sum_c c_{(1)}\otimes c_{(2)}\).
    We will also drop the summation sign and we will simply write
    \(c_{(1)}\otimes c_{(2)}\).  Note that particular simple tensors or even
    elements of \(C\) which appears in the sum are not unique. 

    A morphism \(f:C\sir D\) of coalgebras \((C,\Delta_C,\epsilon_C)\) and
    \((D,\Delta_D,\epsilon_D)\) is a map of \(\R\)-modules such that
    \(\Delta_D\circ f = f\otimes f\circ \Delta_C\) and \(\epsilon_D\circ
	f=\epsilon_C\).
\end{definition}
\subsection{Comodules}\label{ssec:comodules}
\index{comodule}
\begin{definition}
    A (\bold{right}) \(C\)-\bold{comodule} \(M\) is an \(\R\)-module, together with
    a coaction defined by a map: \(\delta:M\sir M\otimes C\) such that:
    \begin{center}
	\begin{tikzpicture}
	    \matrix[column sep=1cm, row sep=1cm]{
		\node (A1) {\(M\)};    & \node (A2) {\(M\otimes C\)};\\
		\node (B1) {\(M\otimes C\)}; & \node (B2) {\(M\otimes C\otimes C\)};\\
	    };
	    \draw[->] (A1) --node[above]{\(\delta\)} (A2);
	    \draw[->] (A1) --node[left]{\(\delta\)} (B1);
	    \draw[->] (A2) --node[right]{\(\delta\otimes\id_C\)} (B2);
	    \draw[->] (B1) --node[below]{\(\id_M\otimes\Delta\)} (B2);
	\end{tikzpicture}
	\begin{tikzpicture}
	    \matrix[column sep=1cm,row sep=1cm]{
		\node (A1) {\(M\)}; & \node (A2) {\(M\otimes C\)}; \\
		                    & \node (B) {\(M\otimes \R\)};\\
	    };
	    \draw[->] (A1) --node[above]{\(\delta\)} (A2);
	    \draw[->] (A1) --node[below left]{\(\cong\)} (B);
	    \draw[->] (A2) --node[right]{\(\id_M\otimes\epsilon\)} (B);
	\end{tikzpicture}
    \end{center}
    A morphism \(f:M\sir N\) of right \(C\)-comodules \((M,\delta_M)\) and
    \((N,\delta_N)\) is an \(\R\)-linear homomorphism such that \(\delta_N\circ
	f=f\otimes\id_C\circ\delta_M\).

    In a similar way one defines left \(C\)-comodules and morphisms of them.

    The category of right \(C\)-comodules with the above morphism 
    will be denoted by \(\Mod^C\), the category of left \(C\)-comodules we
    will denote by \(^C\Mod\).
\end{definition}
\index{comodule!Sweedler notation}
We will use the Sweedler's notation: for a right \(C\)-comodule \(M\), \(m\in
    M\), \(\delta(m)=m_{(0)}\otimes m_{(1)}\) and for a left \(C\)-comodule
\(N\), \(n\in N\), \(\delta(n)=n_{(-1)}\otimes n_{(0)}\). 

\index{convolution}
Then \(C^*=\Hom_\R(C,\R)\) is a unital \(\R\)-algebra, with \bold{convolution}
as multiplication: \((f\ast g)(c)\coloneq f(c_{(1)})g(c_{(2)})\) for \(f,g\in
C^*\). Clearly, \(\epsilon:C\sir\R\) is the unit, since
\((f\ast\epsilon)(c)=f(c_{(1)})\epsilon(c_{(2)})=f(c_{(1)}\epsilon(c_{(2)}))=f(c)\)
and similarity \(\epsilon\ast f=f\). Associativity is a consequence of
coassociativity of the comultiplication in \(C\):
\begin{align*}
    \bigl((f\ast g)\ast h\bigr)(c) & = \bigl(f\ast g\bigr)(c_{(1)})h(c_{(2)}) \\
                   & = f({c_{(1)}}_{(1)})g({c_{(1)}}_{(2)})h(c_{(2)}) \\
                   & = f({c_{(1)}})g({c_{(2)}})h(c_{(3)}) \\
                   & = f(c_{(1)})g({c_{(2)}}_{(1)})h({c_{(2)}}_{(2)}) \\
                   & = f(c_{(1)})\bigl(g\ast h\bigr)(c_{(2)}) \\
                   & = \bigl(f\ast (g\ast h)\bigr)(c)
\end{align*}

Let \(M\) be a right \(C\)-comodule.  Then \(M\) is a left
\(C^*=\Hom_\R(C,\R)\) module, with the action \(c^*\cdot m\coloneq
    m_{(0)}c^*(m_{(1)})\) where \(m\in M\) and \(c^*\in C^*\). Associativity of
this action follows from coassociativity of the \(C\)-coaction on \(M\):
\(\delta:M\sir M\otimes C\).  Since the use of the dual space \(C^*\) we will
assume that \(C\) is a projective \(\R\)-module.

Let \(M\) be a left \(C^*\)-module and let \(\eta_M:C^*\otimes M\sir M\),
\(\eta_M(c^*\otimes m)=c^*m\) for \(c^*\in C^*\) and \(m\in M\), be its module
structure map.  We also define
\[\rho_M:M\ir\Hom_\R(C^*,M),\,\rho_M(m)(c^*)\coloneq c^*m.\]
Let \(\mathit{ev}:C\sir C^{**}\) be the map
\(\mathit{ev}(c)(c^*)\coloneq c^*(c)\), and
\[f_M:M\otimes C^{**}\ir\Hom_\R(C^*,M),\,f_M(m\otimes c^{**})(c^*)=c^{**}(c^*)m\]
where \(m\in M\), \(c^*\in C^*\) and \(c^{**}\in C^{**}\).  Finally we define:
\[\mu_M:M\otimes C\ir\Hom_\R(C^*,M),\,\mu_M(m\otimes c)(c^*)=c^*(c)m\]
for \(m\in M\), \(c\in C\) and \(c^*\in C^*\).  It follows that \(\mu_M\) is an
injective map.  This follows from projectivity of \(C\): since we can choose
a basis \(\{d_i\}\) of \(C\) and a dual basis \(\{d_i^*\}\) of \(C^*\) such
that for every \(c\in C\) \(\sum_id_i^*(c)d_i=c\).  Now let \(\mu_M(\sum_\alpha
    m_\alpha\otimes c_\alpha)=0\), hence for every \(d_j^*\) we have
\(0=\mu_M(\sum_\alpha m_\alpha\otimes c_\alpha)(d_j^*)=d_j^*(c_\alpha)m_\alpha\).
Thus:
\[0=\sum_{\alpha,j} d_j^*(c_\alpha)m_\alpha\otimes d_j = \sum_{\alpha}m_\alpha\otimes c_\alpha\]
\index{comodule!rational}
\index{rational comodule|see{comodule!rational}}
\begin{definition}
    Let \(M\) be a left \(C^*\)-module with action \(\eta_M:C^*\otimes M\sir
    M\). We call it \bold{rational} if the following inclusion holds:
    \[\rho_M(M)\subseteq\mu_M(M\otimes C)\]
    We let \(\mathsf{Rat}(_{C^*}\Mod)\) denote the full subcategory of
    \(_{C^*}\Mod\) consisting of all rational modules. 
\end{definition}
\index{comodule!rational}
\begin{remarkbr}\label{rem:rational_moduiles}
    \begin{enumerate}
	\item Let \(M\) be a left \(C^*\)-module then it is rational if and
	      only if there exist two finite families of \(m_i\in M\) and
	      \(c_i\in C\) (\(i=1,\dots,n\)) such that \(c^*m=\sum_i
		  m_ic^*(c_i)\) for every \(m\in M\) and any \(c\in C\).
	\item Furthermore, if \(M\) is a \(C^*\)-rational left module and
	      \(\{m_i\}_{i=1,\dots,n}\), \(\{c_i\}_{i=1,\dots,n}\) and
	      \(\{m_i'\}_{i=1,\dots,n}\), \(\{c_i'\}_{i=1,\dots,n}\) are two
	      such families then \(\sum_{i=1}^nm_i\otimes
		  c_i=\sum_{i=1}^nm_i'\otimes c_i'\) in \(M\otimes C\), since
	      \(\mu_M(\sum_{i=1}^nm_i\otimes c_i)=\mu_M(\sum_{i=1}^nm_i'\otimes
		  c_i')\) and \(\mu_M\) is injective.
    \end{enumerate}
\end{remarkbr}
By point (i) every \(C^*\)-module which comes from a \(C\)-comodule \(M\) is
rational. The two families \(\{m_i\}_{i=1,\dots,n}\) and
\(\{c_i\}_{i=1,\dots,n}\) of elements of \(M\) and \(C\) respectively, we get
by imposing \(\delta(m)=\sum_{i=1}^nm_i\otimes c_i\), where \(\delta\) is the
structure map of the \(C\)-comodule \(M\). It turns out that the above remark
allows to construct a \(C\)-comodule structure on a rational \(C^*\)-module.
This leads to the following result.
\begin{theorem}
    Let \(C\) be a \(\k\)-coalgebra. Then the functor which sends
    \(M\in\Mod^C\) to the corresponding rational \(C^*\)-module \(M\) is an
    isomorphism of categories
    \[\Mod^C\cong\mathsf{Rat}(_{C^*}\Mod).\]
\end{theorem}
\begin{proof}
    See~\cite[Thm~2.2.5]{sd-cn-sr:hopf-alg}.
\end{proof}
For a corresponding statement for \(C\)-comodules over a commutative ring see
Theorem~\ref{thm:Wisbauer_category_and_comodules} on
page~\pageref{thm:Wisbauer_category_and_comodules}.

Let us note that if \(M\) is a rational \(C^*\)-module then so is any
submodule and quotient module of it. See~\cite[Prop.~2.2.6]{sd-cn-sr:hopf-alg}
for the proof of this statement.

\index{cotensor product}
\index{comodule!cotensor product}
Finally, we will need the cotensor product of two \(\R\)-comodules.  Let us
consider a right \(C\)-comodule \((M,\delta_M)\) and a left
\(C\)-comodule\((N,\delta_N)\).  Then \(M\cotensor_CN\) is the equaliser (in
the category of \(\R\)-modules) of the following diagram:
\[M\cotensor_CN\sir M\otimes N\lmprr{\delta_M\otimes\id_N}{\id_M\otimes\delta_N}M\otimes C\otimes N\]

\subsection{Bialgebras and Hopf algebras}
\index{bialgebra}
\begin{definition}
    A \bold{bialgebra} \(B\) is a tuple \((B,m,u,\Delta,\epsilon)\) such that
    \((B,m,u)\) is an associative \(\R\)-algebra with multiplication
    \(m:B\otimes B\sir B\) and unit \(u:R\sir B\), \((B,\Delta,\epsilon)\) is
    an \(R\)-coalgebra with comultiplication \(\Delta\) and counit
    \(\epsilon\). Furthermore, the following compatibility conditions has to be
    satisfied, i.e. the following diagrams commute: 
    \begin{center}
	\begin{tikzpicture}
	    \matrix[column sep=2cm,row sep=1cm]{
		\node (A1) {\(B\otimes B\)}; & \node (A2) {\(B\otimes B\otimes B\otimes B\)}; & \node (A3) {\(B\otimes B\otimes B\otimes B\)};  \\
		\node (B1) {\(B\)};    &                              & \node (B3) {\(B\otimes B\)};\\
	    };
	    \draw[->] (A1) --node[above]{\(\Delta\otimes\Delta\)} (A2);
	    \draw[->] (A2) --node[above]{\(\id_B\otimes\tau\otimes\id_B\)} (A3);
	    \draw[->] (A1) --node[left]{\(m\)} (B1);
	    \draw[->] (B1) --node[below]{\(\Delta\)} (B3);
	    \draw[->] (A3) --node[right]{\(m\otimes m\)} (B3);
	\end{tikzpicture}
    \end{center}
    where \(\tau:B^{\otimes 2}\sir B^{\otimes 2}\) is the flip:
    \(\tau(a\otimes b)=b\otimes a\), 
    \begin{center}
	\begin{tikzpicture}
	    \matrix[column sep=1cm,row sep=1cm]{
		\node (A1) {\(B\)}; & \node (A2) {\(B\otimes B\)};\\
		\node (B1) {\(\R\)}; & \node (B2) {\(\R\otimes \R\)};\\
	    };
	    \draw[->] (A1) --node[above]{\(\Delta\)} (A2);
	    \draw[->] (B1) --node[left]{\(u\)} (A1);
	    \draw[->] (B2) --node[right]{\(u\otimes u\)} (A2);
	    \draw[->] (B1) --node[below]{\(\cong\)} (B2);
	\end{tikzpicture}
	\begin{tikzpicture}
	    \matrix[column sep=1cm,row sep=1cm]{
		\node (A1) {\(B\otimes B\)}; & \node (A2) {\(B\)};\\
		\node (B1) {\(\R\otimes \R\)}; & \node (B2) {\(\R\)};\\
	    };
	    \draw[->] (A1) --node[above]{\(m\)} (A2);
	    \draw[->] (A1) --node[left]{\(\epsilon\otimes\epsilon\)} (B1);
	    \draw[->] (A2) --node[right]{\(\epsilon\)} (B2);
	    \draw[->] (B1) --node[below]{\(\cong\)} (B2);
	\end{tikzpicture}
    \end{center}
    and finally: \(\epsilon\circ u=\id_\R\).
\end{definition}
The commutative diagrams show that \(\Delta\) and \(\epsilon\) are
morphisms of \(\R\)-algebras or equivalently that \(m\) and \(u\) are
morphisms of \(R\)-coalgebras \(B\) and \(B\otimes B\).
\index{Hopf algebra}
\index{anitpode}
\begin{definition}\label{defi:Hopf_algebra}
    A \bold{Hopf algebra} \(H\) is a tuple \((H,m,u,\Delta,\epsilon,S)\)
    such that \((H,m,u,\Delta,\epsilon)\) is a bialgebra and \(S:H\sir H\) is
    the antipode for which the following diagram commutes:
    \begin{center}
	\begin{tikzpicture}
	    \node (A2) at (0,0) {\(H\)};
	    \node (B1) at (-2,-1.5) {\(H\otimes H\)}; 
	    \node (B2) at (0,-2.5) {\(\R\)}; 
	    \node (B3) at (2,-1.5) {\(H\otimes H\)};
	    \node (C1) at (-2,-3.5) {\(H\otimes H\)}; 
	    \node (C3) at (2,-3.5) {\(H\otimes H\)};
	    \node (D2) at (0,-5) {\(H\)}; 
	    \draw[->] (A2) --node[above left]{\(\Delta\)} (B1);
	    \draw[->] (A2) --node[above right]{\(\Delta\)} (B3);
	    \draw[->] (A2) --node[right]{\(\epsilon\)} (B2);
	    \draw[->] (B1) --node[left]{\(S\otimes\id_H\)} (C1);
	    \draw[->] (B3) --node[right]{\(\id_H\otimes S\)} (C3);
	    \draw[->] (B2) --node[right]{\(u\)} (D2);
	    \draw[->] (C1) --node[below left]{\(m\)} (D2);
	    \draw[->] (C3) --node[below right]{\(m\)} (D2);
	\end{tikzpicture}
    \end{center}
    or in Sweedler notation:
    \(S(h_{(0)})h_{(1)}=\epsilon(h)1_H=h_{(0)}S(h_{(1)})\). 
\end{definition}
Let us note that the axioms of a Hopf algebra are self dual.  If
\((H,m,u,\Delta,\epsilon,S)\) is a finite dimensional Hopf algebra then its
dual \((H^*,\Delta^*,\epsilon^*,m*,u*,S^{-1})\)is a Hopf algebra.  The
antipode of \(H^*\) is inverse of \(S\), and it is well defined since for any
finite dimensional Hopf algebra \(S\) is a bijection.
\index{Hopf algebra!examples}
\begin{examplesbr}[Hopf algebras]\label{ex:Hopf_algebras}
    \begin{enumerate}
	\index{group algebra}
	\item Let \(G\) be a group and let \(\k\) be a field. Then the group
	      algebra \(\k[G]\), which is a vector spaces spanned by \(G\) with
	      multiplication induced from the group multiplication, is a Hopf
	      algebra, where \(\Delta(g)=g\otimes g\) and \(\epsilon(g)=1\).
	      Elements \(c\) of a coalgebra \(C\) which satisfy the above
	      conditions, i.e. \(\Delta(c)=c\otimes c\) and \(\epsilon(c)\) are
	      called group-like.  The antipode \(S\) in \(\k[G]\) is given by
	      \(S(g)=g^{-1}\) for every \(g\in G\). 
	\item Let \(G\) be a finite group, then \(\k[G]^*\) is a Hopf algebra
	      with basis \(\delta_g\in \k[G]^*\), which is a dual basis to the
	      basis of \(\k[G]\) given by elements of \(G\). The multiplication
	      is given by: \(\delta_g\cdot\delta_h=\delta_{g,h}\delta_g\),
	      where $\delta_{g,h}$ is the Kronecker symbol given by
	      \(\delta_{g,k}=\left\{
		\begin{smallmatrix}
		    1 &\text{iff }g=h\hfill\\ 0 &\text{otherwise}\hfill
		\end{smallmatrix}\right.\).  The unit of the algebra structure
	    is \(\sum_{g\in G}\delta_g\) (note that we can write the unit in
	    this form since \(G\) is finite). The comultiplication is set by:
	    \[\Delta(\delta_g)=\sum_{\substack{(h,k)\in G\times G\\
			hk=g}}\delta_h\otimes\delta_k\] and
	    \(\epsilon(\delta_g)\coloneq\delta_g(1)=\delta_{g,1}\).  The
	    antipode \(S\) is given by \(S(\delta_g)=\delta(g^{-1})\). 
	\index{affine group scheme}
	\item An affine group scheme (over a field \(\k\)) as a representable
	      functor from the category of commutative algebras over \(\k\)
	      \(\mathsf{cAlg}_\k\) to the category of groups \(\mathsf{Gr}\).
	      It turns out that the category of affine group schemes is
	      equivalent to the category of commutative \(\k\)-Hopf algebras
	      denoted by \(\mathsf{cHopf}_\k\) via{\footnote{This just follows
		      from the Yoneda lemma. We refer
		      to~\cite{md-pg:groupes-algebriques} for the theory of
		      affine group schemes.}}:
	     \[\mathsf{cHopf}_\k\ni H\selmap{}\Bigl(\mathsf{cAlg}_\k\ni A\selmap{}\Hom_{\mathsf{cAlg}_\k}(H,A)\in\mathsf{Gr}\Bigr)\]
	     The group structure on \(\Hom_{\mathsf{cAlg}_\k}(H,A)\) is the
	     convolution product: for \(f,g\in\Hom_{\mathsf{cAlg}_\k}(H,A)\)
	     \((f\ast g)(h)\coloneq f(h_{(1)})\cdot g(h_{(2)})\) (where
	     "\(\cdot\)"~denotes the multiplication in the algebra \(A\)) and
	     \(f^{-1}(h)=f(S(h))\).  Here \(S\) denotes the antipode of \(H\).
	     Affine group schemes form a category with morphisms being natural
	     transformations which consist of group homomorphisms for every
	     commutative algebra. This category we denote by
	     \(\mathsf{GrSch}_\k\). 
	    
	     Let us note that the group scheme represented by a Hopf algebra
	     \(H\) is commutative (i.e. has values in commutative groups) if
	     and only if \(H\) is a cocommutative Hopf algebra. The category
	     of commutative group schemes, which is a full subcategory of
	     \(\mathsf{GrSch}_\k\), we denote by \(\mathsf{cGrSch}_\k\).
	\index{universal enveloping algebra}
	\item\label{ex:Hopf_algebras_env} Let \(\mathfrak{g}\) be a Lie
	     algebra over a field \(\k\).  Then the enveloping algebra
	     \(\mathcal{U}(\mathfrak{g})\) is defined as a quotient of the
	     unital free algebra \(F(\mathfrak{g})\) on the set of basis
	     elements of \(\mathfrak{g}\) by the ideal generated by \(x\cdot
		 y-y\cdot x-[x,y]\) for \(x,y\in\mathfrak{g}\) where
	     "\(\cdot\)" denotes multiplication in \(F(\mathfrak{g})\) and
	     \([-,-]\) is the Lie bracket of the Lie algebra \(\mathfrak{g}\),
	     i.e.
	    \[\mathcal{U}(\mathfrak{g})\coloneq F(\mathfrak{g})/\langle x\cdot y-y\cdot x-[x,y]:x,y\in\mathfrak{g}\rangle\]
	    There is a standard Hopf algebra structure on
	    \(\mathcal{U}(\mathfrak{g})\) given by
	    \(\Delta(x)=x\otimes1+1\otimes x\) for all \(x\in\mathfrak{g}\).
	    Indeed, this rule defines a map
	    \(\mathfrak{g}\sir\mathfrak{g}\otimes\mathfrak{g}\), which
	    uniquely determines a map \(\Delta':F(\mathfrak{g})\sir
		\mathcal{U}(\mathfrak{g})\otimes\mathcal{U}(\mathfrak{g})\)
	    since it is an algebra map and \(\mathfrak{g}\) generates
	    \(F(\mathfrak{g})\). This map in turn uniquely determines
	    \(\Delta:\mathcal{U}(\mathfrak{g})\sir\mathcal{U}(\mathfrak{g})\otimes\mathcal{U}(\mathfrak{g})\)
	    since:
	    \[\Delta'(x\cdot y-y\cdot x)=(x\cdot y-y\cdot x)\otimes1+1\otimes(x\cdot y-y\cdot x)\]
	    and \(\Delta'([x,y])=[x,y]\otimes1+1\otimes[x,y]\) since
	    \([x,y]\in\mathfrak{g}\).  The antipode \(S\) is uniquely
	    determined by the rule \(S(x)=-x\) for every \(x\in\mathcal{g}\). 
    \end{enumerate}
\end{examplesbr}

\subsection{$H$-comodule algebras}\label{ssec:comodule_algberas}
\index{comodule algebra}
\index{comodule algebra!coinvariants}
\index{H-extension|see{comodule algebra}}
\begin{definition}\label{defi:comodule_algebra}
    Let \(H\) be a \(\R\)-Hopf algebra, and \(A\) a \(\R\)-algebra and a right
    (left) \(H\)-comodule, i.e. there is a counital and coassociative map
    \(\delta:A\sir A\otimes H\) for the right \(H\)-comodule version. The
    algebra \(A\) is called an \bold{\(H\)-comodule algebra} if the structure
    map \(\delta\) is an algebra homomorphism.
    
    \bold{Coinvariants} of the \(H\)-coaction (\(H\)-comodule) are defined
    by \(A^{co\,H}:=\{a\in A:\delta(a)=a\otimes 1_H\}\). They form
    a subalgebra. An extension of the type \(A/A^{co\,H}\) we will call an
    \bold{\(H\)-extension}. Coinvariants of left \(H\)-comodule algebras are
    defined similarly.
\end{definition}
Let us note that the coinvariants fit into the following equaliser diagram:
\[A^{co\,H}\lmonoir A\lmprr{\delta}{\id_A\otimes1_H}A\otimes H.\]
\begin{examplesbr}[Comodule algebras]\label{ex:comodule_algebras}
    \begin{enumerate}
	\index{comodule algebra!examples!Hopf algebra}
	\item Every Hopf algebra \(H\) (over a commutative ring \(\R\)) is
	    an \(H\)-comodule algebra both left and right. The
	    \(H\)-comodule structure is set by the comultiplication \(\Delta\)
	    of \(H\).  Furthermore, \(H^{co\,H}=\R\). For this let us assume
	    that \(h\in H^{co\,H}\), then \(\Delta(h)=h_{(1)}\otimes
		h_{(2)}=h\otimes 1_H\). Now computing \(\epsilon\otimes\id_H\)
	    of both sides we get \(h=\epsilon(h)1\in \R\). The other inclusion
	    is clear.
	\index{comodule algebra!examples!G-graded algebra}
	\item Let \(G\) be a group and \(A\) a \(G\) graded \(\R\)-algebra.
	    That is \(A=\oplus_{g\in G} A_g\), and furthermore \(A_g\cdot
		A_h\subseteq A_{gh}\).  For \(e\in G\) the unit element of
	    \(G\), \(A_e\) is a subalgebra of \(A\). Then \(A\) is an
	    \(\R[G]\)-comodule algebra via: \(\delta:A\sir A\otimes \k[G]\),
	    \(\delta(a)=a\otimes g\) for all \(a\in A_g\) for \(g\in G\).
	    Clearly, we have \(A^{co\,\R[G]}=A_e\) since \(e\in \k[G]\) is the
	    unit element of the Hopf algebra \(\k[G]\).

	    It is easy to observe that every \(\R[G]\)-comodule algebra is of
	    this form. For this one defines \(A_g\coloneq\{a\in
		    A:\delta(a)=a\otimes g\}\). Since \(\delta\) is
	    a monomorphism we get that \(A_g\cap A_h=\{0\}\) if \(g\neq h\)
	    (for \(g,h\in G\)). The inclusion \(A_g\cdot A_h\subseteq A_{gh}\)
	    holds since \(\delta\) is an algebra homomorphism. The equality
	    \(A=\oplus_{g\in G}A_g\) follows now from coassociativity of the
	    coaction \(\delta\): let \(a\in A\), then \(\delta(a)=\sum_{g\in
		    G}a_g\otimes g\) for some \(a_g\in A\) and \(g\in G\),
	    where only finitely many \(a_g\neq 0\). Then by coassociativity of
	    \(\delta\) we get \(\sum_{g\in G}\delta(a_g)\otimes g=\sum_{g\in
		    G}a_g\otimes  g\otimes g\). Now since \(G\) form an
	    \(\R\)-basis of \(\R[G]\) we get \(\delta(a_g)=a_g\otimes g\),
	    i.e.  \(a_g\in A_g\) for all \(g\in G\). In this way
	    \(A=\oplus_{g\in G}A_g\).
	\index{crossed product}
	\item\label{itm:crossed_product} Let \(B\) be an \(\R\)-algebra and \(H\) a Hopf algebra. Let us
	    assume that \(H\) \bold{acts weakly} on \(B\) that is there is a morphism
	    \(H\otimes B\sir B,\,h\otimes a\selmap{}h\cdot a\) such that:
	    \begin{enumerate}[label=\textbf{(\alph{enumii})}]
		\item \(h\cdot(ab)=(h_{(1)}\cdot a)(h_{(2)}\cdot b)\),
		\item \(h\cdot 1=\epsilon(h)1\),
		\item \(1\cdot a=a\).
	    \end{enumerate}
	    Note that \(B\) might not to be an \(H\)-module, however if this
	    is the case a weak action is called an \bold{action}. Furthermore,
	    let \(\sigma:H\otimes H\sir B\) be an \(\R\)-linear, convolution
	    invertible map\footnote{That is there exists a map \(\nu:H\otimes
		    H\sir B\) such that for all \(h\otimes k\in H\otimes H\)
		we have \(\sigma(h_{(1)}\otimes k_{(1)})\nu(h_{(2)}\otimes
		    k_{(2)})=\epsilon(h)\epsilon(k)1_B=\nu(h_{(1)}\otimes
		    k_{(1)})\sigma(h_{(2)}\otimes k_{(2)})\)}. The
	    \bold{crossed product} \(B\#_\sigma H\) is \(B\otimes H\) as an
	    \(\R\)-module, with multiplication: \[a\#_\sigma h\cdot b\#_\sigma
		k:= a\left(h_{(1)}\cdot b\right)\sigma\left(h_{(2)}\otimes
		    k_{(1)}\right)\#_\sigma h_{(3)}k_{(2)}\]
	    for all \(a,b\in B\) and \(h,k\in H\), where \(a\#_\sigma h\)
	    denotes \(a\otimes h\) as an element of \(B\#_\sigma H\).  If the
	    cocycle \(\sigma\) is trivial, i.e.
	    \(\sigma(-,-)=\epsilon(-)\otimes\epsilon(-)1_B\), then the crossed
	    product \(B\#_\sigma H\) is called a \bold{smash product} and it
	    is denoted by \(B\# H\). 
	    \index{crossed product!associativity}
	    \begin{lemma}\label{lem:associativity_of_crossed_products}
		The crossed product \(B\#_\sigma H\) (over a commutative ring \(\R\)) is an
		associative algebra with unit \(1_B\# 1_H\) if and only if the following
		two conditions are satisfied:
		\begin{enumerate}[label=\textbf{(\roman{enumii})}]
		    \item  \(1_H\cdot a=a\), for all \(a\in B\), and
			   \[\bigl(h_{(1)}\cdot(k_{(1)}\cdot a)\bigr)\sigma(h_{(2)},k_{(2)})=\sigma(h_{(1)},k_{(1)})\bigl((h_{(2)}k_{(2)})\cdot a\bigr)\]
			   for all \(h,k\in H\) and \(a\in B\).  \(B\) is then
			   called a \textsf{twisted \(H\)-module}.
		    \item \(\sigma(h,1_H)=\sigma(1_H,h)=\epsilon(h)1_B\), for
			  all \(h\in H\), and 
			  \[\bigl(h_{(1)}\cdot\sigma(k_{(1)},l_{(1)})\bigr)\sigma(h_{(2)},k_{(2)}l_{(2)})=\sigma(h_{(1)},k_{(1)})\sigma(h_{(2)}k_{(2)},l)\]
			  for all \(h,k,l\in H\).  \(\sigma\) is then called
			  a \textsf{cocycle}.
		\end{enumerate}
	    \end{lemma}
	    See~\cite[Lem.~10]{yd-mt:cleft-comodule-algebras}
	    or~\cite{rb-mc-sm:crossed_products}.  In the rest of this work we
	    assume that all crossed products are associative with identity
	    \(1_B\#_\sigma1_H\), and thus that the conditions of the lemma are
	    satisfied.

	    Now, \(B\#_\sigma H\) becomes an \(H\)-comodule algebra by
	    setting:
	    \[\delta(a\#_\sigma h)=a\#_\sigma h_{(1)}\otimes h_{(2)},\text{ for }a\in B,h\in H.\]
    \end{enumerate}
\end{examplesbr}
\index{Hopf Galois extension|see{\(H\)-Galois extension}}
\index{H-Galois extension}
\begin{definition}\label{defi:Hopf-Galois_extension}
	An $H$-extension $A/A^{co\,H}$ is called \bold{$H$-Hopf--Galois
	    extension} (\bold{$H$-Galois extension}) if the \bold{canonical
	    map} of right $H$-comodules and left $A$-modules:
	\begin{equation}\label{eq:canonical-map}
	    \can:A\otimes_{A^{co\,H}} A\sir A\otimes H,\ a\otimes b\eli{}ab_{(0)}\otimes b_{(1)}
	\end{equation}
	is an \emph{isomorphism}.
\end{definition}
\index{H-Galois extension!examples}
\begin{examplesbr}[Hopf--Galois extensions]
    \begin{enumerate}
	\item[Ex.~\ref{ex:comodule_algebras} \textbf{(i)}\phantom{ii}] Let
	     \(H\) be a Hopf algebra, then it is an \(H\)-Galois extension of
	     the base ring \(\R\) with \(\can^{-1}(h\otimes k)\coloneq
		 hS(k_{(1)})\otimes k_{(2)}\).
	\item[Ex.~\ref{ex:comodule_algebras} \textbf{(ii)}\phantom{i}]
	     A \(G\)-graded \(\k\)-algebra \(A\) is a \(\k[G]\)-Galois
	     extension if and only if \(A\) is strongly graded, that is
	     \(A_g\cdot A_h=A_{gh}\). This was first discovered
	     by~\cite{ku:galoiserweiterungen}.
	\item[Ex.~\ref{ex:comodule_algebras} \textbf{(iii)}] A crossed product
	     \(B\#_\sigma H\) over a ring \(\R\) is a Hopf--Galois extension
	     as we show below.
    \end{enumerate}
\end{examplesbr}
\index{cleft extension}
\index{normal basis property}
\begin{definition}\label{defi:cleft_&_normal_basis}
    \begin{itemize}[noitemsep]
	\item An \(H\)-extension \(A/B\) (over a ring \(\R\)) is called
	    \bold{cleft} if there exists a convolution invertible
	    \(H\)-comodule map \(\gamma:H\sir A\).
	\item An \(H\)-extension \(A/B\) has the \bold{normal basis
	    property} if and only if $A$ is isomorphic to $B\otimes H$ as
	    a left $B$-module and right $H$-comodule.
    \end{itemize}
\end{definition}
Let us note that the normal basis property is equivalent to the classical
notion if \(H\) is finite dimensional. We defer the explanation until
introducing \(H\)-module algebras.  We say that a \(R\)-algebra \(A\) is flat
if it is flat \(R\)-module, that is the functor \(A\otimes
    -:\mod{R}\sir\mod{R}\) preserves exact sequences.
\index{cleft extension!characterisation}
\begin{theorem}\label{thm:characterisation_of_cleftness}
    Let \(A\) be an \(H\)-comodule algebra over \(\R\) such that \(B=A^{co\,H}\) is a flat
    \(\R\)-algebra. Then the following are equivalent:
    \begin{enumerate}
	\item the extension \(B\subseteq A\) is equivalent to a crossed
	    product \(B\subseteq B\#_\sigma H\) for some weak action of
	    \(H\) on \(B\) and some invertible cocycle \(\sigma\) satisfying
	    the twisted module condition;
	\item \(B\subseteq A\) is cleft;
	\item \(B\subseteq A\) is \(H\)-Galois with normal basis property.
    \end{enumerate}
\end{theorem}
This theorem generalises~\cite[Thm~1.18]{rb-sm:crossed-products} where it is
proven for algebras over a field \(\k\). Let us note that the proof
of~\cite{rb-sm:crossed-products} also works in the more general context,
however we show a different proof.
\begin{proof}
    The equivalence between (ii) and (iii) was shown
    in~\cite{yd-mt:cleft-comodule-algebras}. The implication
    \(\text{(ii)}\Rightarrow\text{(i)}\) follows
    from~\cite[Thm~11]{yd-mt:cleft-comodule-algebras} as shown in the proof
    of~\cite[Thm~1.18]{rb-sm:crossed-products}. Instead of proving
    \(\text{(i)}\Rightarrow\text{(ii)}\) as it is done by
    \citeauthor{rb-sm:crossed-products}, we show that
    \(\text{(i)}\Rightarrow\text{(iii)}\).

    First, if \(B\) is flat over \(\R\) we have \((B\#_\sigma
	H)^{co\,H}=B\#_\sigma 1_H\subseteq B\#_\sigma H\): the diagram 
    \[\R=H^{co\,H}\lmonoir H\lmprr{\Delta}{\id_H\otimes 1_H}H\otimes H\]
    is an equaliser and \(B\) is flat over \(\R\) thus the fork
    \[B=B\otimes \R\lmonoir B\#_\sigma H\lmprr{\delta}{\id_{B\#_\sigma H}\otimes 1_H}B\#_\sigma H\otimes H\]
    is an equaliser as well. Hence, indeed, \((B\#_\sigma
	H)^{co\,H}=B\#_\sigma 1_H\cong B\).  Now it easily follows that
    \(B\#_\sigma H\) has the normal basis property, since \(a\#_\sigma1\cdot
	b\#_\sigma h=(ab)\#_\sigma h\), for \(a,b\in B,\ h\in H\), and
    \(\delta=\id_B\otimes\Delta:B\#_\sigma H\sir B\#_\sigma H\otimes H\).  To
    prove the Galois property let us observe that we have a commutative
    diagram:
    \begin{center}
	\begin{tikzpicture}
	    \matrix[matrix of nodes, column sep=1cm, row sep=1cm]{
		|(A1)| \((B\#_\sigma H)\otimes_B(B\#_\sigma H)\) & |(A2)| \(B\#_\sigma H\otimes H\) \\
		|(B1)| \(B\otimes H\otimes H\)                    & \\
	    };
	    \draw[->] (A1) --node[above]{\(\can_{B\#_\sigma H}\)} (A2);
	    \draw[->] (A1) --node[left]{\(\alpha\)} (B1);
	    \draw[->] (B1) --node[below right]{\(\beta\)} (A2);
	\end{tikzpicture}
    \end{center}
    where \(\alpha:(B\#_\sigma H)\otimes_B(B\#_\sigma H)\sir B\#_\sigma
	H\otimes H\) is defined by: \(\alpha(a\#_\sigma h\otimes_B b\#_\sigma
	k)=(a\#_\sigma h\cdot b\#_\sigma 1_h)\otimes k=a(h_{(1)}\cdot
	b)\#_\sigma h_{(2)}\otimes k\).  It is an isomorphism with inverse
    \(B\#_\sigma H\otimes H\ni a\otimes h\otimes k\selmap{}a\#_\sigma h\otimes
	1_A\#_\sigma k\in(B\#_\sigma H)\otimes_B(B\#_\sigma H)\). The map
    \(\beta\) is explicitly given by:
    \begin{align*}
B\otimes H\otimes H\ni a\otimes h\otimes k\selmap{\beta} & a\sigma(h_{(1)},k_{(1)})\otimes\can_H(h_{(2)}\otimes k_{(2)}) \\
		      & =a\sigma(h_{(1)},k_{(1)})\otimes h_{(2)}k_{(2)}\otimes k_{(3)}\in B\#_\sigma H\otimes H
    \end{align*}
    It is invertible with inverse \(\beta^{-1}\) the composition:
    \begin{align*}
	B\#_\sigma H\otimes H\ni a\#_\sigma h\otimes k\selmap{} & a\otimes\can_H^{-1}(h\otimes k)\\
	                                & =a\otimes hS(k_{(1)})\otimes k_{(2)} \\
	\hfill\selmap{}                 & a\sigma^{-1}(h_{(1)}S(k_{(2)}),k_{(3)})\otimes h_{(2)}S(k_{(1)})\otimes k_{(4)}\\
	                                & \hspace{1cm}\in B\otimes H\otimes H
    \end{align*}
    where \(\sigma^{-1}\) is the convolution inverse of \(\sigma\). It now
    follows that \(\can_{B\#_\sigma H}\) is invertible and thus \(B\subseteq
	B\#_\sigma H\) is \(H\)-Galois with the normal basis property.
\end{proof}
\begin{corollary}
    Let \(A/B\) be an \(H\)-extension, such that both \(A\) and \(H\) are
    finite dimensional.  Furthermore, let us assume that \(A\) has the normal
    basis property.  Then \(A\) is a crossed product if and only if \(\can_A\)
    is an epimorphism or monomorphism.
\end{corollary}
\begin{proof}
    Since \(A\cong B\otimes H\) as a left \(B\)-module and a right
    \(H\)-comodule we have \(A\otimes_BA\cong A\otimes_B(B\otimes H)\cong
	A\otimes H\).  Thus \(\dim_kA\otimes_BA=\dim_kA\otimes H\).  It
    follows that \(\can_A:A\otimes_BA\ir A\otimes H\) is an isomorphism (and
    hence \(A\) is a crossed product) if and only if it is a monomorphism or
    an epimorphism.
\end{proof}

\subsection{$H$-module algebras}\label{ssec:module_algebras}
Another set of examples of \(H\)-comodule algebras can be built from
\(H\)-module algebras.
\index{module algebra}
\index{H-module algebra|see{module algebra}}
\begin{definition}\label{defi:H-module-algabra}
    Let \(A\) be an algebra and \(H\) a Hopf algebra. We say that \(A\) is an
    \(H\)-module algebra if \(H\) measures \(A\), that is there is an action
    \(H\otimes A\sir A\) such that \(h\cdot (ab)=(h_{(1)}\cdot a)(h_{(2)}\cdot
	b)\) which makes \(A\) an \(H\)-module, or equivalently the following
    diagram commutes:
    \begin{center}
	{\hfill\begin{tikzpicture}
	    \matrix[matrix of nodes, column sep=1cm, row sep=1cm]{
		|(A1)|	\(H\otimes A\otimes A\) & |(A2)| \(H\otimes A\) \\
		|(B1)|	\(A\otimes A\) & |(B2)|	\(A\) \\
	    };
	    \draw[->] (A1) --node[above]{\(m\)} (A2);
	    \draw[->] (A2) -- (B2);
	    \draw[->] (A1) -- (B1);
	    \draw[->] (B1) --node[below]{\(m\)} (B2);
	\end{tikzpicture}
	\hfill\refstepcounter{equation}\raisebox{10mm}{(\theequation)}\label{diag:H-module_algebra}\\
	}
    \end{center}
    where \(m:A\otimes A\sir A\) is the multiplication of \(A\) and the
    vertical maps are the \(H\)-action.
    The \bold{invariant} subalgebra \(A^H\) is defined as
    \(A^H\coloneq\{a\in A:\forall_{h\in H}h\cdot a=\epsilon(h)a\}\). 
\end{definition}
Let us recall the following theorem:
\begin{proposition}\label{prop:H-comodule_algebras}
    Let \(H\) be a finite dimensional Hopf algebra (over a field \(\k\)), and
    let \(A\) be a \(\k\)-algebra.  Then \(A\) is a right \(H\)-comodule
    algebra if and only if \(A\) is a left \(H^*\)-module algebra.
    Furthermore, in this case we have \(A^{co\,H}=A^{H^*}\). 
\end{proposition}
\begin{proof}
    The proof can be found in~\cite[Thm~6.2.4]{sd-cn-sr:hopf-alg}.  We only
    sketch the constructions.  Let \(A\) be an \(H\)-comodule algebra with the
    \(H\)-comodule structure map \(\delta:A\sir A\otimes H\).  Then \(A\) is
    an \(H^*\)-module algebra by:
    \[f\cdot a\coloneq \id_A\otimes f\circ\delta(a)=a_{(0)}f(a_{(1)})\]
    Conversely, if \(A\) is an \(H^*\)-module algebra, then we pick a dual
    basis of \(H^*\) and \(H\): \(\{e_i^*\}\) and \(\{e_i\}\) respectively,
    where \(i=1,\dots,n=\dim H\).  Then the \(H\)-comodule structure map is
    set by:
    \[\delta:A\sir A\otimes H,\quad\delta(a)=\sum_{i=1}^{n}(e_i^*\cdot a)\otimes e_i.\]
\end{proof}
We are now ready to explain the connection between the non-classical and
classical normal basis properties (see
Definition~\ref{defi:cleft_&_normal_basis}).
\index{normal basis property!comparison with the classical version}
\begin{lemma}
    Let \(A\) be an \(H\)-comodule algebra over a field \(\k\), with \(H\)
    a finite dimensional Hopf algebra. Let us consider \(A\) as an
    \(H^*\)-module algebra.  Then \(A\) has the normal basis property if and
    only if there exists \(a\in A\) and a basis \(\{f_i\}_{i=1,\dots,n}\) of
    \(H^*\) such that \(\{f_i\cdot a\}_{i=1,\dots,n}\) is a basis of 
    \(A\) as a left \(H\)-module.
\end{lemma}
\begin{proof}
    A sketch of the proof can be found
    in~\cite[Lemma~3.5]{sm:hopf-galois-survey}.
\end{proof}
\index{module algebra!examples}
\begin{examplesbr}[\(H\)-module algebras]
    \begin{enumerate}
	\item Let \(G\) be a group acting by automorphisms on an
	      \(\R\)-algebra \(A\).  Then \(A\) is an \(\R[G]\)-module algebra
	      via the induced action: \(\R[G]\otimes A\sir A\) from the
	      \(G\)-action.  Furthermore, \(\R[G]\)-module structures on \(A\)
	      are in one to one correspondence with actions of \(G\) on \(A\)
	      via the automorphism group of the algebra \(A\). In this case
	      \(A^{\k[G]}\) is the subalgebra of elements of \(A\) which are
	      fixed by all elements of \(G\) (sum of all one element orbits of
	      the \(G\)-action).
	\item If in the previous example we assume that \(G\) is a finite
	      group and we assume that \(\R=\k\) is a field, we get that \(A\)
	      is a \(\k[G]^*\)-comodule algebra (by
	      Proposition~\ref{prop:H-comodule_algebras}) and
	      \(\k[G]^*\)-comodule algebra structures on \(A\) are in one to
	      one correspondence with \(G\)-actions on \(A\) via the
	      automorphism group of \(A\). It also follows that
	      \(A^{co\,\k[G]^*}\) is the algebra of fixed elements of the
	      corresponding \(G\)-action.
	\item This is a dual case of Example~\ref{ex:comodule_algebras}(ii).
	      Let, as before, \(A\) be a \(G\) graded \(\k\)-algebra, where
	      \(G\) is a finite group. Let \(H=\k[G]^*\), and let \(\{e_g:g\in
		      G\}\) be the dual basis of \(\k[G]^*\) to the basis
	      \(\{g:g\in G\}\) of \(\k[G]\). The comultiplication of
	      \(\k[G]^*\) can be described by the following formula:
	      \[\Delta(e_g)=\sum_{h\in G}e_{gh^{-1}}\otimes e_h\]
	      the counit is set by \(\epsilon(e_g)=0\) if \(g\neq e\), where
	      \(e\in G\) is the unit element of \(G\) and \(\epsilon(e)=1\).
	      The antipode is set by \(S(e_g)=e_{g^{-1}}\). For \(a\in A\) we
	      have \(a=\sum_{g\in G}a_g\) where \(a_g\in A_g\) for all \(g\in
		  G\).  The \(\k[G]^*\)-module algebra structure on
	      \(A=\oplus_{g\in G}A_g\) is set by \(e_g\cdot a=a_g\), where
	      \(a_g\) is the \(A_g\) component of \(a\). Indeed, the
	      \(H\)-module algebra condition is satisfied, since
	      \[e_g\cdot (ab)=(ab)_{g}=\sum_{h\in G}a_{gh^{-1}}b_{h}=\sum_{h\in G}(e_{gh^{-1}}\cdot a)(e_h\cdot b)\]
	      The subalgebra of invariants is \(A_e\) (where \(e\in G\) is the
	      unit element of the group).
	\item Let \(\mathfrak{g}\) be a \(\k\)-Lie algebra and let \(A\) be
	      a \(\k\)-algebra such that \(\mathfrak{g}\) acts on \(A\) by
	      derivations, i.e. there exists a homomorphism
	      \(\alpha:\mathfrak{g}\sir\Der_\k(A)\).  Here \(\Der_\k(A)\)
	      denotes the set of \(\k\)-linear endomorphisms of \(A\), such
	      that \(f(ab)=f(a)b+af(b)\).  Let \(\mathcal{U}(\mathfrak{g})\)
	      be the universal enveloping (Hopf) algebra of \(\mathfrak{g}\).
	      Since \(\mathcal{U}(\mathfrak{g})\) is generated (as a vector
	      space) by monomials \(g_1\dots g_n\), \(g_i\in\mathfrak{g}\), we
	      put
	      \begin{equation}\label{eq:U(g)-action}
		    (g_1\dots g_n)\cdot a=\alpha(g_1)(\dots\alpha(g_n)(a)\dots)
	      \end{equation}
	      In this way \(A\) becomes a \(\mathcal{U}(\mathfrak{g})\)-module
	      algebra. Furthermore, this sets up a one to one correspondence
	      between \(\mathcal{U}(\mathfrak{g})\)-comodule algebra
	      structures on \(A\) and \(\mathfrak{g}\) actions through
	      derivations on \(A\).  If \(\mathcal{U}(\mathfrak{g})\otimes
		  A\sir A\) is a \(\mathcal{U}(\mathfrak{g})\)-module algebra
	      structure then for each \(g\in\mathfrak{g}\) we have:
	      \(g\cdot(ab)=(g_{(1)}\cdot a)(g_{(2)}\cdot b)=(g\cdot a)b+
		  a(g\cdot b)\) since \(\Delta(g)=g\otimes 1+1\otimes g\) in
	      \(\mathcal{U}(\mathfrak{g})\). Thus
	      a \(\mathcal{U}(\mathfrak{g})\)-module algebra structure
	      restricts to an action of \(\mathfrak{g}\) through derivations,
	      which is the one we started with. Conversely, if we have an
	      \(\mathcal{U}(\mathfrak{g})\)-module algebra structure, then it
	      must satisfy the formula~\ref{eq:U(g)-action}, since it is
	      associative (it is a \(\mathcal{U}(\mathfrak{g})\)-module). Thus
	      a \(\mathcal{U}(\mathfrak{g})\)-module algebra structure is
	      uniquely determined by the underlying \(\mathfrak{g}\)-action.
	\item There is an adjoint action of \(H\) on itself, which makes \(H\)
	      an \(H\)-module algebra:
	      \[h\cdot k\coloneq\ad(h)(k)\coloneq h_{(1)}kS(h_{(2)}).\]
	      Indeed:
	      \begin{equation*}
		  \begin{split}
		    \ad(h)(kl)\coloneq h_{(1)}klS(h_{(2)})  & = h_{(1)}kS(h_{(2)})h_{(3)}lS(h_{(4)})\\
							    & = \ad(h_{(1)})(k)\ad(h_{(2)})(l)
		  \end{split}
	      \end{equation*}
	      In the case \(H=\k[G]\) it is given by \(h\cdot k\coloneq
		  hkh^{-1}\) for \(h,k\in G\), and in the case
	      \(H=\mathcal{U}(\mathfrak{g})\) it is determined by the adjoint
	      action: \(\ad(g)(h)\coloneq gh-hg\) for \(g\in\mathfrak{g},h\in
		  H\). We have \(H^H=Z(H)\) the center of \(H\). Indeed, for
	      \(h\in H^H\) and \(k\in H\) we have
	      \[kh=k_{(1)}\epsilon(k_{(2)})h=k_{(1)}hS(k_{(2)})k_{(3)}=\epsilon(k_{(1)})hk_{(2)}=hk.\]
	      The other inclusion easily follows from the antipode axiom (see
	      Definition~\ref{defi:Hopf_algebra}).
    \end{enumerate}
\end{examplesbr}

\subsection{Finite Galois field extensions}\label{ssec:finite_galois_theory_of_fields}
In this subsection we introduce Galois field extensions and show how they can
be understood using the Hopf--Galois approach.
\index{Galois extension}
\begin{definition}\label{defi:Galois_field_extention}
    Let \(\bE/\bF\) be an algebraic field extension. It is called
    \bold{Galois} if there exists a subgroup \mbox{\(G<\Aut(\bE)\)} such that
    \(\bF=\bE^G\coloneq\{e\in\bE:\forall_{g\in G}g(e)=e\}\). 
\end{definition}
If \(\bE/\bF\) is a Galois extension then it turns out that the group
\(G=\Aut(\bE/\bF)\), i.e. the group of automorphisms of \(\bE\) which preserve
the subfield \(\bF\). Let \(\bE/\bF\) be a field extension. Let us denote by
\(\Gal(\bE/\bF)\) the subgroup of \(\Aut(\bE)\) of all automorphisms \(\phi\)
such that for all \(f\in\bF\) \(\phi(f)=f\). We let
\([\bE:\bF]\coloneq\dim_\bF\bE\) denote the \bold{degree} of the field
extension \(\bE/\bF\). An extension is called \bold{finite} if the degree is
finite. It turns out that an algebraic extension is Galois if and only if
\(\bE^{\Gal(\bE/\bF)}=\bF\) and as a consequence \(\bE/\bF\) is a finite
Galois extension if and only if \(\bigl[\bE:\bF\bigr]=|\Gal(\bE/\bF)|\).
Furthermore, let us note that Galois extensions have the classical basis
property (see~\cite[chapter~4.14]{nj:basic_algebra_I}), i.e. there exists
\(e\in\bE\) such that \(\{g(e):g\in\Gal(\bE/\bF)\}\) is a basis over \(\bF\).

Now let us assume that \(\bE/\bF\) is a finite field extension, and thus
\(\Gal(\bE/\bF)\) is a finite group. We already know that the
\(\Gal(\bE/\bF)\)-action on \(\bE\) induces a \(\bF[\Gal(\bE/\bF)]\)-module
algebra structure on \(\bE\). Furthermore, let \(\k\subseteq\bF\) be a finite
field extension. Then \(\bE\) possesses
a \(\k\bigl[\Gal(\bE/\bF)\bigr]\)-module algebra structure.
\index{H-Galois extension!Galois theory for finite field extensions}
\begin{proposition}\label{prop:can_H_for_field_ext}
    Let \(\bE/\bF\) be a finite field extension and \(\k\subseteq\bF\) be
    as above. Then  the extension \(\bE/\bF\) is Galois if and only if it is
    a \(\k[G]^*\)-Hopf--Galois extension for some group \(G\), i.e.
    \(\bF=\bE^{co\,\k[G]^*}\) and the canonical map
    \[\can_\bE:\bE\otimes_\bF\bE\sir\bE\otimes \k[G]^*\]
    is bijective. Furthermore, if this is the case, then \(G=\Gal(\bE/\bF)\).
\end{proposition}
Though the proof is well known (see for example~\cite{sm:hopf-galois-survey}),
for the sake of completeness we put it below.
\begin{proof}
    First assume that \(\bE/\bF\) is a Galois extension
    (Definition~\ref{defi:Galois_field_extention}).  Let
    \(n=|\Gal(\bE,\bF)|\), let \(\Gal(\bE/\bF)=\{g_1,\dots,g_n\}\) and let
    \(\{\phi_i\}_{i=1,\dots,n}\) be the dual basis of
    \(\k\bigl[\Gal(\bE/\bF)\bigr]\).  Since \(\bE/\bF\) is Galois it has an
    \(\bF\)-basis \(\{f_i\}_{i=1,\dots,n}\) of \(n\) elements.  The
    \(\Gal(\bE/\bF)\)-action on \(\bE\) determines a coaction as shown in
    Proposition~\ref{prop:H-comodule_algebras}, that is
    \(\delta:\bE\sir\bE\otimes \k\bigl[\Gal(\bE/\bF)\bigr]^*\),
    \(\delta(e)=\sum_{i=1}^ng_i(e)\otimes\phi_i\).  Furthermore,
    \(\bF=\bE^{\Gal(\bE/\bF)}=\bE^{co\,\k[\Gal(\bE/\bF)]^*}\).  The canonical
    map is given by \(\can_\bE(e\otimes_\bF
	e')=\sum_{i=1}^neg_i(e')\otimes\phi_i\).  Thus if \(\sum_{i=1}^n
	e_i\otimes f_i\in\ker\can_\bE\), then \(\sum_{i=1}^ne_ig_j(f_i)=0\)
    for all \(j=1,\dots,n\), since \(\{\phi_j\}_{j=1,\dots,n}\) are all
    independent.  Now let us consider the \(n\times n\)-matrix \(\mathbb{A}\)
    with entries \(\mathbb{A}_{ij}\coloneq g_j(e_i)\).  The Dedekind lemma on
    independence of automorphisms implies that its columns are independent,
    and thus it is invertible.  It follows that \(e_i=0\) for all
    \(i=1,\dots,n\).  Thus \(\can_\bE\) is injective.  Now,
    \(\dim_\bF(\bE\otimes_\bF\bE)=\bigl[\bE:\bF\bigr]^2\), hence
    \(\dim_\k(\bE\otimes_\bF\bE)=\bigl[\bE:\bF\bigr]^2\bigl[\bF:\k\bigr]\) and
    \(\dim_\k(\bE\otimes
	\k\bigl[\Gal(\bE/\bF)\bigr])=[\bE:\k]|\Gal(\bE/\bF)|=\bigl[\bE:\bF\bigr]^2\bigl[\bF:\k\bigr]\).
    The canonical map \(\can_\bE\) is an isomorphism since both
    \(\bE\otimes_\bF\bE\) and \(\bE\otimes \k\bigl[\Gal(\bE/\bF)\bigr]\) are
    of the same finite dimension.

    Now, let us  assume that \(\bE/\bF\) is \(\k[G]\)-Galois. Then by
    Proposition~\ref{prop:H-comodule_algebras} there exists a \(\k[G]\)-module
    algebra structure on \(\bE\), with \(\bE^{\k[G]}=\bE^{co\,\k[G]^*}=\bF\).
    It is uniquely determined by a \(G\)-action on \(\bE\) and moreover
    \(\bE^G=\bE^{\k[G]}=\bF\). Thus \(\bE/\bF\) is a Galois extension. Since
    \(\can_\bE\) is an isomorphism comparing the \(\k\)-dimensions of
    \(\bE\otimes_\bF\bE\) and \(\bE\otimes \k[G]\) we get
    \(\bigl[\bE:\bF\bigr]=|G|\).
    By~\cite[Thm~1.7]{hk-mt:hopf-algebras-and-galois-extensions} we get that
    \[\bE\otimes\k[G]\sir\End_\bF(\bE),\quad e\otimes g\elmap{}\bigl(\bE\ni e'\selmap{}eg(e')\in\bE\bigr)\] 
    is an isomorphism. This implies that the map \(G\ni
	g\selmap{}\bigl(e\selmap{}g(e)\bigr)\in\Aut(\bE)\) is a monomorphism.
    In consequence we can write \(G\leq\Aut(\bE)\). Furthermore,
    \(G\leq\Gal(\bE/\bF)\), since \(\bE^G=\bF\). Now, since
    \(|G|=[\bE:\bF]=|\Gal(\bE/\bF)|\) (since \(\bE/\bF\) is a Galois extension)
    we get that \(G=\Gal(\bE/\bF)\).
\end{proof}
\index{Galois extension!crossed product}
\index{crossed product!finite Galois field extension}
\begin{corollary}
    Let \(\bE/\bF\) be a finite Galois extension of fields.  Then \(\bE\) is
    a crossed product, i.e. there exists a cocycle \(\sigma\) and a weak
    action of \(\k\bigl[\Gal(\bE/\bF)\bigr]^*\) on \(\bF\) such that
    \[\bE\cong\bF\#_\sigma \k\bigl[\Gal(\bE/\bF)\bigr]^*.\]
\end{corollary}
\begin{proof}
    Finite Galois extensions are
    \(\k\bigl[\Gal(\bE/\bF)\bigr]^*\)-Hopf--Galois and have a normal basis,
    thus the result follows from
    Theorem~\ref{thm:characterisation_of_cleftness}.
\end{proof}

\chapter{Modules with the intersection property}\label{chap:modules_with_int_property}
\setcounter{section}{1}

In this chapter we investigate some properties of the functor \(M\otimes-\)
for an \(\R\)-module \(M\). The result obtained will play a crucial role in
further studies.  We first recall a result that for a flat \(\R\)-module \(M\)
the functor \(M\otimes-\) preserves all finite intersections.  Then we
consider a new property, which we call the \emph{intersection property}
(Definition~\ref{defi:intersection_property} on
page~\pageref{defi:intersection_property}).  A module has this property if the
tensor functor preserves all intersections, not just finite ones.  This
property will be used later to show that the lattice of subcomodules of
a \(C\)-comodule is algebraic if the coalgebra \(C\) is a flat Mittag--Leffler
module.  Moreover, this property will play a vital role in the construction of
a Galois correspondence for comodule algebras over a ring.  We show that
direct sums of modules which have the intersection property also have it, thus
free modules have it and we show that direct summands of modules with this
property also share it.  Hence projective modules have the intersection
property. In Proposition~\ref{prop:Mittag-Leffler_intersection} (on
page~\pageref{prop:Mittag-Leffler_intersection}) we show that flat
Mittag--Leffler modules also have the intersection property.  The
Mittag--Leffler condition was first studied
in~\citep{mr-lg:platitude-et-projectivte}.  The authors have shown that flat
Mittag--Leffler modules preserve filtered limits of monomorphisms
(Proposition~\ref{prop:ML_and_limits} on page~\pageref{prop:ML_and_limits}).
Since intersection is this kind of limit flat Mittag--Leffler modules have the
intersection property.  We give here an independent proof of this fact.  Based
on the results of \citeauthor{mr-lg:platitude-et-projectivte} we note that
flat modules have this property if and only if they satisfy the
Mittag--Leffler condition (Corollary~\ref{cor:mittag-leffler} on
page~\pageref{cor:mittag-leffler}).  \citet{dh-jt:mittag-leffler} showed that
\(\kappa\)-projective modules (see Definition~\ref{defi:kappa_projective} on
page~\pageref{defi:kappa_projective}) are flat Mittag--Leffler and hence they
possess the intersection property.  We also prove that locally projective
modules (Definition~\ref{prop:int_prop_for_locally_projectives} on
page~\pageref{prop:int_prop_for_locally_projectives}) have this property.
Then we show that the intersection property is stable under pure submodules
(Proposition~\ref{prop:int_prop_pure-submodules} on
page~\pageref{prop:int_prop_pure-submodules}).  The very last
Theorem~\ref{thm:intersection_prop_for_all_modules} (on
page~\pageref{thm:intersection_prop_for_all_modules}) shows a condition that
must be fulfilled in order for every flat \(R\)-module to have the
intersection property.  Though we rather tend to think that this theorem might
be useful the other way around.  Knowing a ring for which all the flat modules
have the intersection property we gain another tool (a class of exact
sequences) to study its modules.  For example every flat right \(R\)-module is
projective (and hence has the intersection property) if (and only if) the ring
is right perfect (by a theorem of Bass).

\bigskip
\index{intersection property!flat modules}
\begin{proposition}[{\cite[40.16(2)]{tb-rw:corings-and-comodules}}]\label{prop:flat_intersection}
	    Let \(M',M''\subseteq M\) be \(\R\)-submodules and let
	    \(K\) be a flat \(\R\)-module. Then 
	    \[\bigl(M'\otimes K\bigr)\cap\bigl(M''\otimes K\bigr)=(M'\cap M'')\otimes K\]
\end{proposition}
Since the proof in~\cite{tb-rw:corings-and-comodules} is omitted we put it
here.
\begin{proof}
	    Consider the following commutative diagram:
		\begin{center}
		    \hspace*{-1cm}\begin{tikzpicture}
			\matrix[column sep=1cm,row sep=1cm]{
    \node (B1) {\(0\)}; & \node (B2) {\(\bigl(M'\cap M''\bigr)\otimes K\)}; & \node (B3) {\(M'\otimes K\)}; & \node (B4) {\(M'/(M'\cap M'')\otimes K\)}; & \node (B5) {\(0\)}; \\
    \node (C1) {\(0\)}; & \node (C2) {\(M''\otimes K\)};       & \node (C3) {\(M\otimes K\)};  & \node (C4) {\(M/M''\otimes K\)};  & \node (C5) {\(0\)}; \\
			};
			\draw[->] (B2) --node[left]{\(i_2\)} (C2);
			\draw[->] (B3) --node[right]{\(i'\otimes\id_K\)} (C3);
			\draw[->] (B4) -- (C4);
			\draw[->] (B1) -- (B2);
			\draw[->] (B2) --node[above]{\(i_1\)} (B3);
			\draw[->] (B3) --node[above]{\(p'\otimes\id_K\)} (B4);
			\draw[->] (B4) -- (B5);
			\draw[->] (C1) -- (C2);
			\draw[->] (C2) --node[below]{\(i''\otimes\id_K\)} (C3);
			\draw[->] (C3) --node[below]{\(p''\otimes\id_K\)}(C4);
			\draw[->] (C4) -- (C5);
		    \end{tikzpicture}
		\end{center}
		where both rows are exact and every vertical map is
		a monomorphism (monomorphisms in the category of modules are
		just injective homomorphisms).  Thus the left commutative
		square is a pullback: let as assume that we have two
		homomorphisms: \(f:N\sir M'\otimes K\) and \(g:N\sir
		    M''\otimes K\) such that \((i'\otimes\id_K)\circ
		    f=(i''\otimes\id_K)\circ g\). Then one can easily see that
		\((p'\otimes\id_K)\circ f=0\) (since the far most vertical
		homomorphism is a monomorphism). Thus there exists a map
		\(h:N\sir (M'\cap M'')\otimes K\) such that \(f=i_1\circ h\).
		By commutativity of the left square and since
		\(i''\otimes\id_K\) is a monomorphism it follows that
		\(i_2\circ h=g\). Uniqueness of \(h\) is a consequence of the
		fact that \(i_1\) is an injection. The proposition follows
		since pullbacks in the category of modules along monomorphisms
		are given by intersection.
\end{proof}
Furthermore, it is not hard to show that tensoring with a flat module
preserves all finite limits. Indeed, tensoring with a flat module preserves
equalisers and finite products (which in the category of \(\R\)-modules are
equal to finite sums, which are preserved by the tensor functor since it has
a right adjoint).  In this chapter we show that there is a large class of
modules for which the tensor product functor preserves arbitrary
intersections. We will also construct examples of flat and faithfully flat
modules without this property.

Let $N'$ be a submodule of $N$, $i:N'\subseteq N$, and let $M$ be an
$\R$-module. Then the \bold{canonical image} of $M\otimes N'$ in $M\otimes N$
is the image of \(M\otimes N'\) under the map $\id_M\otimes i$. It will be
denoted by $\im(M\otimes N')$.
\index{intersection property}
\begin{definition}\label{defi:intersection_property}
    Let \(M,N\) be \(\R\)-modules, and let \((N_\alpha)_{\alpha\in I}\) be
    a family of \(\R\)-submodules of \(N\).  We say that a module \(M\) has
    the \bold{intersection property with respect to \(N\)} if the
    homomorphism:
       \[\im\bigl(M\otimes\bigl(\bigcap_{\alpha\in I}N_\alpha\bigr)\bigr)\ir \bigcap_{\alpha\in I}\im\bigl(M\otimes N_\alpha\bigr)\] 
    is an isomorphism for any family of submodules \((N_\alpha)_{\alpha\in
	    I}\). We say that \(M\) has the \bold{intersection property} if
    the above condition holds for any \(\R\)-module~\(N\).
\end{definition}
Note that if \(M\) is flat then it has the intersection property if and only
if the map \(M\otimes(\bigcap_{\alpha\in I}N_\alpha)\ir \bigcap_{\alpha\in
	I}(M\otimes N_\alpha)\) is an isomorphism for any family
\(\{N_\alpha\}\) of submodules of a module \(N\).

\index{intersection property!direct sums}
\begin{proposition}\label{prop:ip_direct_sums}
    The intersection property is closed under direct sums.
\end{proposition}
\begin{proof}
    Let \(X=\oplus_{i\in I}X_i\), be a direct sum of modules with the
    intersection property. We let \(\pi_i:\oplus_{i\in I}X_i\sir X_i\) be the
    canonical projection on the \(i\)-th factor and \(s_i:X_i\sir\oplus_{i\in
	    I}X_i\) be the canonical section. Let \((N_\alpha)_{\alpha\in J}\)
    be a family of submodules of an \(\R\)-module \(N\). Then we have a split
    epimorphism \((\oplus_{i\in I}\pi_i)\otimes\id_N\) with a right inverse
    \((\oplus_{i\in I}s_i)\otimes\id_N\). Also for each \(\alpha\in J\) the
    map \((\oplus_{i\in I}\pi_i)\otimes\id_{N_\alpha}\) is a split epimorphism
    with right inverse \((\oplus_{i\in I}s_i)\otimes\id_{N_\alpha}\).
    Furthermore, we have a family of split epimorphisms, whose sections are
    jointly surjective:
    \begin{center}
	\begin{tikzpicture}
	    \matrix[column sep=2cm]{
		\node (A) {\(\im\left((\oplus_{i\in I}X_i)\otimes N_\alpha\right)\)}; & \node (B) {\(\im(X_i\otimes N_\alpha)\)};\\
	    };
	    \draw[->>] (A) --node[below]{\(\pi_i\otimes\id_{N}\)} (B);
	    \draw[<-<] ($(A)+(1.4cm,3mm)$) .. controls +(8mm,5mm) and +(-8mm,5mm) ..node[above]{\(s_i\otimes\id_{N}\)} ($(B)+(-1.0cm,3mm)$);
	\end{tikzpicture}
    \end{center}
    They induce the following family of projections with jointly surjective
    sections:
    \begin{center}
	\begin{tikzpicture}
	    \matrix[column sep=2cm]{
		\node (A) {\(\bigcap_{\alpha\in J}\im\left((\oplus_{i\in I}X_i)\otimes N_\alpha\right)\)}; & \node (B) {\(\bigcap_{\alpha\in J}\im(X_i\otimes N_\alpha)\)};\\
	    };
	    \draw[->>] (A) --node[below]{\(\pi_i\otimes\id_{N}\)} (B);
	    \draw[<-<] ($(A)+(1.8cm,4mm)$) .. controls +(8mm,5mm) and +(-8mm,5mm) ..node[above]{\(s_i\otimes\id_{N}\)} ($(B)+(-1.4cm,4mm)$);
	\end{tikzpicture}
    \end{center}
    For this let \(x\in\bigcap_{\alpha\in J}\im\left((\oplus_{i\in
		I}X_i)\otimes N_\alpha\right)\). Then, for each \(\alpha\in
	J\), there exists \(y_\alpha\in\oplus_{i\in I}\im(X_i\otimes
	N_\alpha)\) such that \(\left(\oplus_{i\in
		I}s_i\otimes\id_N\right)(y_\alpha)=\sum_{i\in
	    I}s_i\otimes\id_{N}(y_\alpha)=x\). Since \(\oplus_{i\in
	    I}s_i\otimes\id_N\) is a monomorphism we get
    \(y=y_\alpha\in\bigcap_{\alpha\in J}\oplus_{i\in I}\im(X_i\otimes
	N_\alpha)\) for all \(\alpha\in J\).  It follows that:
    \[\bigcap_{\alpha\in J}\im\left((\oplus_{i\in I}X_i)\otimes N_\alpha\right)\lmpr{\oplus_{i\in I}\pi_i\otimes\id_N}\oplus_{i\in I}\bigcap_{\alpha\in J}\im(X_i\otimes N_\alpha)\] 
    is an isomorphism with inverse \(\oplus_{i\in I}s_i\otimes\id_N\). 
    Now the proposition will follow from the commutativity of the diagram:
    \begin{center}
	{\hfill\begin{tikzpicture}
	    \matrix[column sep=-1.2cm,row sep=1.3cm]{
	    \node (A0) {\(\im\left((\oplus_{i\in I}X_i)\otimes(\bigcap_{\alpha\in J}N_\alpha)\right)\)}; &                                                  & \node (B0) {\(\bigcap_{\alpha\in J}\im\left((\oplus_{i\in I}X_i)\otimes N_\alpha\right)\)};\\
	    \node (A1) {\(\im(\oplus_{i\in I}(X_i\otimes(\bigcap_{\alpha\in J}N_\alpha)))\)}; & & \node (B2) {\(\oplus_{i\in I}\bigcap_{\alpha\in J}\im(X_i\otimes N_\alpha)\)};\\
	    & \node (A2) {\(\oplus_{i\in I}\im(X_i\otimes(\bigcap_{\alpha\in J}N_\alpha))\)}; &\\
	    };
	    \draw[->] (A0) -- (B0);
	    \draw[->] (A0) --node[left]{\(\simeq\)} (A1);
	    \draw[->] (A1) --node[below]{\(=\)} (A2);
	    \draw[->] (A2) --node[below]{\(\simeq\)} (B2);
	    \draw[->] (B2) --node[right]{\(\simeq\)} (B0);
	\end{tikzpicture}\nolinebreak[4]
	\hfill\refstepcounter{equation}\raisebox{23mm}{(\theequation)}\label{diag:proof-sums-and-intersection}}
    \end{center}
    We will go around this diagram from the top left corner to the top
    right one and prove that all the maps on the way are isomorphisms. The
    first map is an isomorphism since tensor product commutes with
    colimits. 
    Clearly the second map is an isomorphism as well. 
    The bottom right arrow in~\eqref{diag:proof-sums-and-intersection} is an
    isomorphism since all \(X_i\) (\(i\in I\)) have the intersection property
    and we already showed that the last homomorphism is an isomorphism.
\end{proof}
\index{intersection property!direct summands}
\begin{proposition}\label{prop:ip_direct_summands}
    The intersection property is stable under taking direct summands.
\end{proposition}
\begin{proof}
    Let \(M\) be a direct summand in a module \(P\) which has the intersection
    property. Let \(M'\) be a complement of \(M\) in \(P\). Then we have
    a chain of isomorphisms:
    \begin{align*}
	\im\Bigl(M\otimes\bigcap_{\alpha\in J}N_\alpha\Bigr)\oplus\im\Bigl(M'\otimes\bigcap_{\alpha\in J}N_\alpha\Bigr) & \simeq \im\Bigl((M\oplus M')\otimes(\bigcap_{\alpha\in J}N_\alpha)\Bigr) \\
	                                                   & \hspace{-6mm}=\bigcap_{\alpha\in J}\im\Bigl((M\oplus M')\otimes N_\alpha\Bigr) \\
	                                                   & \hspace{-6mm}\simeq\bigcap_{\alpha\in J}\im\left(M\otimes N_\alpha\oplus M'\otimes N_\alpha\right) \\
	                                                   & \hspace{-6mm}\simeq\bigcap_{\alpha\in J}\im\left(M\otimes N_\alpha\right)\oplus\bigcap_{\alpha\in J}\im(M'\otimes N_\alpha)
    \end{align*}
    Since every isomorphism in the above diagram commutes with projection onto
    the first and second factor the composition also does. Thus it is the direct
    sum of the two natural maps: 
    \[\im\Bigl(M\otimes\bigcap_{\alpha\in J}N_\alpha\Bigr)\sir\bigcap_{\alpha\in J}\im\left(M\otimes N_\alpha\right),\] 
    \[\im\Bigl(M'\otimes\bigcap_{\alpha\in J}N_\alpha\Bigr)\sir\bigcap_{\alpha\in J}\im\left(M'\otimes N_\alpha\right).\] 
    It follows that both maps are isomorphisms, hence both \(M\) and \(M'\)
    have the intersection property. 
\end{proof}
Let us note that \(\R\) itself has the intersection property (since
\(\R\otimes -\) is naturally isomorphic to the identity functor). Every free
\(\R\)-module has the intersection property by
Proposition~\ref{prop:ip_direct_sums} and by
Proposition~\ref{prop:ip_direct_summands} every projective \(\R\)-module has
it too.  Under one of the following conditions on the ring~\(\R\): \(\R\) is
a left noetherian local ring, or \(\R\) is a domain satisfying the strong rank
condition, i.e. for any \(n\in\bN\), any set of \(n+1\) elements of \(\R^n\)
is linearly dependent, then any finitely generated flat left \(\R\)-module is
projective (see~\cite[Thm~4.38]{tl-modules}). Thus it has the
\emph{intersection property}.
\index{Mittag--Leffler module}
\begin{definition}\label{defi:Mittag-Leffler}
    Let \(M\) be an \(R\)-module.  It is a \bold{Mittag--Leffler module} if for
    any family of \(R\)-modules \(M_\lambda\) (\(\lambda\in\Lambda\)) the
    canonical map
    \[M\otimes\prod_{\lambda\in\Lambda}M_\lambda\sir\prod_{\lambda\in\Lambda}\bigl(M\otimes M_\lambda\bigr),\ m\otimes(m_\lambda)_{\lambda\in\Lambda}\selmap{}(m\otimes m_\lambda)_{\lambda\in\Lambda}\]
    where \(m\in M\) and \(m_\lambda\in M_\lambda\) for each
    \(\lambda\in\Lambda\), is a monomorphism. The above condition is also
    called the Mittag--Leffler condition. 
\end{definition}
\begin{proposition}[{\cite[Cor.~2.1.7]{mr-lg:platitude-et-projectivte}}]\label{prop:ML_and_limits}
    Let \(M\) be a Mittag--Leffler \(\R\)-module (where \(\R\) is a not
    necessarily commutative ring).  For every projective filtered system of
    \(\R^\op\)-modules \((Q_r,u_{rs})\) the canonical map \((\lim
	Q_r)\otimes_\R M\ir\lim(Q_r\otimes_\R M)\) is injective.  Furthermore,
    it is bijective if all the maps \(u_{rs}\) are injective. 
\end{proposition}
\begin{proof}
    The first part follows from the definition of Mittag--Leffler modules that
    we took. Note that~\citeauthor{mr-lg:platitude-et-projectivte} use
    another equivalent definition of Mittag--Leffler modules
    (see~\cite[Prop.~2.1.5]{mr-lg:platitude-et-projectivte}). The second claim
    follows from the snake lemma applied to the following exact diagram:
    \begin{center}
	\begin{tikzpicture}
	    \matrix[column sep=9mm,row sep=1cm]{
		\node (A0) {\(0\)}; & \node (A1) {\((\lim Q_r)\otimes M\)}; & \node (A2) {\(Q_r\otimes M\)}; & \node (A3) {\(\bigl(\lim(Q_r/Q_s)\bigr)\otimes_\R M\)}; & \node (A4) {\(0\)}; \\
		\node (B0) {\(0\)}; & \node (B1) {\(\lim(Q_r\otimes M)\)};  & \node (B2) {\(Q_r\otimes M\)}; & \node (B3) {\(\lim\bigl((Q_r/Q_s)\otimes M\bigr)\)}; & \\
	    };
	    \draw[->] (A0) -- (A1);
	    \draw[->] (A1) -- (A2);
	    \draw[->] (A2) -- (A3);
	    \draw[->] (A3) -- (A4);
	    \draw[->] (B0) -- (B1);
	    \draw[->] (B1) -- (B2);
	    \draw[->] (B2) -- (B3);
	    \draw[->] (A1) -- (B1);
	    \draw[->] (A2) -- (B2);
	    \draw[->] (A3) -- (B3);
	\end{tikzpicture}
    \end{center}
    where \(r\) is a fixed index.
\end{proof}
\begin{proposition}\label{prop:Mittag-Leffler_intersection}
    Let \(M\) be a flat Mittag--Leffler \(\R\)-module.  Then \(M\) has the
    intersection property.
\end{proposition}
The above result follows from Proposition~\ref{prop:ML_and_limits} by
\citeauthor{mr-lg:platitude-et-projectivte} but we present here another proof.
\begin{proof}
    Let \(N_\alpha\) for \(\alpha\in I\) be a family of submodules of an
    \(\R\)-module \(N\). Let us consider the following diagram:
    \begin{center}
	\begin{tikzpicture}
	    \matrix[column sep=9mm,row sep=1cm]{
		\node (A) {\(0\)}; & \node (B) {\(M\otimes\left(N/\bigcap_{\alpha\in I}N_\alpha\right)\)};          & \node (C) {\(M\otimes\left(\prod_{\alpha\in I}N/N_\alpha\right)\)}; \\
				   &                                             & \node (D) {\(\prod_{\alpha\in I}\bigl(M\otimes (N/N_\alpha)\bigr)\)}; \\
		\node (G) {\(0\)}; & \node (F) {\((M\otimes N)/\bigcap_{\alpha\in I}\im(M\otimes N_\alpha)\)}; & \node (E) {\(\prod_{\alpha\in I}(M\otimes N)/\im(M\otimes N_\alpha)\)}; \\
	    };
	    \draw[->] (A) -- (B);
	    \draw[->] (B) --node[above]{\(i\)} (C);
	    \draw[->] (C) --node[right]{\(f\)} (D);
	    \draw[->] (D) --node[right]{\(g\)} (E);
	    \draw[<-] (E) --node[below]{\(j\)} (F);
	    \draw[<-] (F) -- (G);
	    \draw[->,dashed] (B) --node[left]{\(G\)} (F);
	\end{tikzpicture}
    \end{center}
    Where \(i\) and \(j\) are the canonical embeddings:
    \begin{align*}
            i\Bigl(m\otimes(n+\bigcap_{\alpha\in I}N_\alpha)\Bigr) & \coloneq m\otimes(n+N_\alpha)_{\alpha\in I}\\
            \intertext{and}
            j\Bigl((m\otimes n)+\bigcap_{\alpha\in I}\im\left(M\otimes N_\alpha\right)\Bigr) & \coloneq\left(m\otimes n+\im(M\otimes N_\alpha)\right)_{\alpha\in I}
    \end{align*}
    for \(m\in M\) and \(n\in N\).  While \(f\) sends
    \(m\otimes(n_\alpha+N_\alpha)_{\alpha\in I}\) to \(\left(m\otimes
	    (n_\alpha+N_\alpha)\right)_{\alpha\in I}\) and \(g\) is the
    canonical isomorphism. Note that \(\im(gfi)\subseteq\im(j)\) and hence if
    \(M\) is a flat Mittag--Leffler module, then \(G:=gfi\) can be considered
    an embedding \(G:M\otimes\left(N/\bigcap_{\alpha\in
		I}N_\alpha\right)\sir(M\otimes N)/\bigcap_{\alpha\in
	    I}\im(M\otimes N_\alpha)\). Hence we get the exact diagram:
    \begin{center}
	\hspace*{-1cm}\begin{tikzpicture}
	    \matrix[column sep=5mm,row sep=1cm]{
		    & \node (A1) {\(0\)};                    & \node (A2) {\(0\)};    & \node (A3) {\(0\)};                          & \\
\node (B0) {\(0\)}; & \node (B1) {\(\im\left(M\otimes\left(\bigcap_{\alpha\in I}N_\alpha\right)\right)\)}; & \node (B2) {\(M\otimes N\)}; & \node (B3) {\(M\otimes\left(N/\bigcap_{\alpha\in I}N_\alpha\right)\)};          & \node (B4) {\(0\)}; \\
\node (C0) {\(0\)}; & \node (C1) {\(\bigcap_{\alpha\in I}\im\left(M\otimes N_\alpha\right)\)};  & \node (C2) {\(M\otimes N\)}; & \node (C3) {\((M\otimes N)/\bigcap_{\alpha\in I}\im\left(M\otimes N_\alpha\right)\)}; & \node (C4) {\(0\)}; \\
		    & \node (D1) {\(\coker(H)\)};            & \node (D2) {\(0\)};    & \node (D3) {\(\coker(G)\)};                  & \\
	    };
	    \draw[->] (B0) -- (B1);
	    \draw[->] (B1) -- (B2);
	    \draw[->] (B2) -- (B3);
	    \draw[->] (B3) -- (B4);
	    \draw[->] (C0) -- (C1);
	    \draw[->] (C1) -- (C2);
	    \draw[->] (C2) -- (C3);
	    \draw[->] (C3) -- (C4);

	    \draw[->] (A1) -- (B1);
	    \draw[->] (A2) -- (B2);
	    \draw[->] (A3) -- (B3);

	    \draw[->] (B1) --node[left]{\(H\)} (C1);
	    \draw[->] (B2) --node[left]{\(=\)} (C2);
	    \draw[->] (B3) --node[right]{\(G\)} (C3);

	    \draw[->] (C1) -- (D1);
	    \draw[->] (C2) -- (D2);
	    \draw[->] (C3) -- (D3);
	\end{tikzpicture}
    \end{center}
    where \(H\) is the canonical embedding of
    \(\im\left(M\otimes\left(\bigcap_{\alpha\in I}N_\alpha\right)\right)\)
    into the module \(\im\left(\bigcap_{\alpha\in I}\left(M\otimes
		N_\alpha\right)\right)\) as submodules of \(M\otimes N\). By
    the snake lemma we get the following short exact sequence: 
    \[0=\ker(G)\sir\coker(H)\sir 0\]
    Thus the monomorphism~\(H\) is onto.
\end{proof}
There is also another class of modules with the intersection property. To
define it we need the following
\index{dense system of submodules}
\begin{definition}
    Let \(\R\) be a ring, \(M\) an \(\R\)-module and \(\kappa\) be a regular
    uncountable cardinal.  A direct system \(C\) of submodules of \(M\) is
    said to be a \(\kappa\)-\bold{dense system} in \(M\) if:
    \begin{enumerate}
	\item  \(C\) is closed under unions of well-ordered ascending
	    chains of length \(<\kappa\),
	\item every subset of \(M\) of cardinality \(<\kappa\) is
	    contained in an element of \(C\).
    \end{enumerate}
\end{definition}
\index{\(\kappa\)-projectve module}
\begin{definition}\label{defi:kappa_projective}
    Let \(\R\), \(M\) and \(\kappa\) be as above. Then \(M\) is
    \(\kappa\)-\bold{projective} if it has a \(\kappa\)-dense system \(C\) of
    \(<\kappa\)-generated projective modules.
\end{definition}
\index{Mittag--Leffler module!\(\aleph_1\)-projective}
A module is flat Mittag--Leffler if and only if it is \(\aleph_1\)-projective
as is shown in~\cite[Thm.~2.9]{dh-jt:mittag-leffler}. Thus we get the
following 
\index{intersection property!\(\aleph_1\)-projective}
\begin{corollary}
    Any \(\aleph_1\)-projective module has the intersection property.
\end{corollary}
We would like now to recall a result
of~\citeauthor{mr-lg:platitude-et-projectivte} which is relevant. 
\begin{proposition}[{\cite[Prop.~2.1.8]{mr-lg:platitude-et-projectivte}}]\label{prop:ML_and_coefficient_modules}
    Let \(M\) be a flat \(\R\)-module, such that for every finitely generated
    free \(\R^\op\)-module \(L\) and for each \(x\in L\otimes_\R M\), the set
    of submodules \(Q\) of \(L\) such that \(x\in Q\otimes_\R M\) has
    a smallest element.  Then \(M\) is a Mittag--Leffler module.
\end{proposition}
Now using the above result we obtain:
\index{intersection property!flat Mittag--Leffler module}
\begin{corollary}\label{cor:mittag-leffler}
    A flat module has the intersection property if and only if it is
    Mittag--Leffler.
\end{corollary}
\begin{proof}
    If a module is flat and Mittag--Leffler then it has the intersection
    property by Proposition~\ref{prop:Mittag-Leffler_intersection}. On the
    other hand if \(M\) is flat and has the intersection property then it
    satisfies the Mittag--Leffler condition by the preceding result of
    \citeauthor{mr-lg:platitude-et-projectivte}: take the set of submodules
    \(Q\) of \(L\) (where \(L\) and \(M\) are as in the previous proposition)
    such that for a given \(x\in L\otimes_\R M\), \(x\in Q\otimes_\R M\). Then
    the smallest such \(Q\), due to the intersection property, is just the
    intersection of all such submodules \(Q\).
\end{proof}

\index{intersection property!flat module without the intersection property}
\begin{example}\label{ex:flat_without_ip}
    Let \(p\) be a prime ideal of \(\bZ\). Then \(\bigcap_ip^i=\{0\}\).  We
    let \(\bZ_p\) denote the ring of fractions of \(\bZ\)  with respect to
    \(p\).  The \(\bZ\)-module \(\bZ_p\) is flat, and
    \(\bZ_p\otimes_\bZ\bigcap_ip^i=\{0\}\). On the other hand
    \(\bigcap_i\bZ_p\otimes_\bZ p^i\cong\bZ_p\).  By a similar argument
    \(\bQ\) doesn't posses the intersection property, even though it is flat
    over~\(\bZ\).  The problem is that, the intersection property is stable
    under arbitrary direct sums but not under cokernels.  Now it is easy to
    construct a faithfully flat module which does not have the intersection
    property.  The \(\bZ\)-modules \(\bZ\oplus\bZ_p\) and \(\bZ\oplus\bQ\) are
    the examples.
\end{example}
In the proof of Propositions~\ref{prop:ip_direct_sums}
and~\ref{prop:ip_direct_summands} we showed that the intersection property is
stable under split exact sequences.  However, the above examples show that the
intersection property is not stable under pure (exact)
sequences \cite[Def.~4.83]{tl-modules}, i.e. whenever \(0\sir M'\sir M\sir
    M''\sir 0\) is a pure exact sequence and \(M\) has the intersection
property then~\(M''\) might not have it.  It is well known that if \(M''\) is
flat then \(M'\subseteq M\) is pure~\cite[Thm~4.85]{tl-modules},
hence~\ref{ex:flat_without_ip} is indeed a source of counter examples.
However, we can show the following proposition:
\index{intersection property!pure submodules}
\begin{proposition}\label{prop:int_prop_pure-submodules}
    Let \(M'\) be a pure submodule of a module \(M\) with the intersection
    property.  Then \(M'\) has the intersection property.
\end{proposition}
\begin{proof}
    We have the following commutative diagram with exact rows:
    \begin{center}
	\hspace*{-2cm}
	\begin{tikzpicture}[style={>=angle 60,thin},cross line/.style={preaction={draw=white, -, line width=5pt}}]
	    \node (B0) at (0,0) {\(0\)};
	    \node (B1) at (3.2,0) {\(\im\left(M\otimes\left(\bigcap_{\alpha\in I}N_\alpha\right)\right)\)};
	    \node (B2) at (7,0) {\(M\otimes N\)};
	    \node (B3) at (11,0) {\(M\otimes\left(N/\bigcap_{\alpha\in I}N_\alpha\right)\)};
	    \node (B4) at (14.5,0) {\(0\)};
	    \node (C0) at (0,-4) {\(0\)};
	    \node (C1) at (3.2,-4){\(\bigcap_{\alpha\in I}\im\left(M\otimes N_\alpha\right)\)};
	    \node (C2) at (7,-4) {\(M\otimes N\)};
	    \node (C3) at (11,-4){\((M\otimes N)/\bigcap_{\alpha\in I}\im\left(M\otimes N_\alpha\right)\)};
	    \node (C4) at (14.5,-4) {\(0\)};
	    \draw[->] (B0) -- (B1);
	    \draw[->] (B1) -- (B2);
	    \draw[->] (B2) --node[fill=white]{\(g\)} (B3);
	    \draw[->] (B3) -- (B4);
	    \draw[->] (C0) -- (C1);
	    \draw[->] (C1) -- (C2);
	    \draw[->] (C2) -- (C3);
	    \draw[->] (C3) -- (C4);

	    \draw[->] (B1) --node[left]{\(H\)}node[above,rotate=-90]{\(\simeq\)} (C1);
	    \draw[->] (B2) --node[above,rotate=-90]{\(=\)} (C2);
	    \draw[->] (B3) --node[right]{\(G\)} (C3);

	    \node[xshift=-2cm,yshift=-1.4cm] (D0) at (0,0) {\(0\)};
	    \node[xshift=-2cm,yshift=-1.4cm] (D1) at (3.2,0) {\(\im\left(M'\otimes\left(\bigcap_{\alpha\in I}N_\alpha\right)\right)\)};
	    \node[xshift=-2cm,yshift=-1.4cm] (D2) at (7,0) {\(M'\otimes N\)};
	    \node[xshift=-2cm,yshift=-1.4cm] (D3) at (11,0) {\(M'\otimes\left(N/\bigcap_{\alpha\in I}N_\alpha\right)\)};
	    \node[xshift=-2cm,yshift=-1.4cm] (D4) at (14.5,0) {\(0\)};
	    \node[xshift=-2cm,yshift=-1.4cm] (E0) at (0,-4) {\(0\)};
	    \node[xshift=-2cm,yshift=-1.4cm] (E1) at (3.2,-4){\(\bigcap_{\alpha\in I}\im\left(M'\otimes N_\alpha\right)\)};
	    \node[xshift=-2cm,yshift=-1.4cm] (E2) at (7,-4) {\(M'\otimes N\)};
	    \node[xshift=-2cm,yshift=-1.4cm] (E3) at (11,-4){\((M'\otimes N)/\bigcap_{\alpha\in I}\im\left(M'\otimes N_\alpha\right)\)};
	    \node[xshift=-2cm,yshift=-1.4cm] (E4) at (14.5,-4) {\(0\)};
	    \draw[->,cross line] (D0) -- (D1);
	    \draw[->,cross line] (D1) -- (D2);
	    \draw[->,cross line] (D2) --node[fill=white]{\(f\)} (D3);
	    \draw[->,cross line] (D3) -- (D4);
	    \draw[->,cross line] (E0) -- (E1);
	    \draw[->,cross line] (E1) -- (E2);
	    \draw[->,cross line] (E2) -- (E3);
	    \draw[->,cross line] (E3) -- (E4);

	    \draw[->,cross line] (D1) --node[left]{\(H'\)} (E1);
	    \draw[->,cross line] (D2) --node[above,rotate=-90]{\(=\)} (E2);
	    \draw[->] (D3) --node[right]{\(G'\)} (E3);

	    \draw[>->] (D1) -- (B1);
	    \draw[>->] (D2) -- (B2);
	    \draw[>->] (D3) -- (B3);

	    \draw[>->] (E1) -- (C1);
	    \draw[>->] (E2) -- (C2);
	    \draw[->] (E3) -- (C3);
	\end{tikzpicture}
    \end{center}
    It easily follows that \(H'\) is a monomorphism.  Let
    \(x\in\bigcap_{\alpha\in I}\im\left(M'\otimes N_\alpha\right)\). To prove
    that \(x\) is in the image of \(H'\) it is enough to show that it goes to
    \(0\) under \(f\). Now since \(H\) is an isomorphism it goes to \(0\)
    under \(g\) and thus it belongs to the kernel of \(f\).
\end{proof}
Note that for pure submodules of flat Mittag--Leffler modules the above
Proposition follows easily since they necessarily are flat Mittag--Leffler.
\index{intersection property!locally projective modules}
\begin{proposition}\label{prop:int_prop_for_locally_projectives}
    Let \(M\) be a locally projective \(\R\)-module. Then \(M\) is a flat
    Mittag--Leffler module, so it has the intersection property.
\end{proposition}
\begin{proof}
    We first show that \(M\) is flat. Let \(i:N\subseteq N'\) be an
    \(\R\)-submodule. And let \(\sum_i m_i\otimes n_i\in\ker\id_M\otimes i\).
    Let \(\pi:F\sir M\) be an epimorphism, where \(F\) is a free module. Let
    us put \(M_0\subseteq M\) the submodule generated by \(\{m_i\}\). Since
    \(M\) is locally--projective we have
    \begin{center}
	{\hfill\begin{tikzpicture}
	    \matrix[column sep=1cm,row sep=1cm]{
		\node (A0) {\(0\)}; & \node (A1) {\(M_0\)}; & \node (A2) {\(M\)}; & \\
		                    & \node (B1) {\(F\)}; & \node (B2) {\(M\)}; & \node (B3) {\(0\)}; \\
	    };
	    \draw[->] (A0) -- (A1);
	    \draw[->] (A1) --node[above]{\(i_0\)} (A2);
	    \draw[->] (B1) --node[below]{\(\pi\)} (B2);
	    \draw[->] (B2) -- (B3);
	    \draw[->] (A2) --node[right]{\(\id_M\)} (B2);
	    \draw[->,dashed] (A2) --node[left]{\(\exists h\)} (B1);
	\end{tikzpicture}
	\hfill\refstepcounter{equation}\raisebox{1.25cm}{(\theequation)}\label{diag:locally_projective_flat}}
    \end{center}
    and we have \(\pi\circ h|_{M_0}=\id_M|_{M_0}\), hence \(\pi\circ
    h(m_i)=m_i\). We have a commutative diagram:
    \begin{center}
	\begin{tikzpicture}
	    \matrix[column sep=1.3cm,row sep=1cm]{
		\node (A1) {\(M\otimes N\)}; & \node (A2) {\(M\otimes N'\)}; \\
		\node (B1) {\(F\otimes N\)}; & \node (B2) {\(F\otimes N'\)}; \\
		\node (C1) {\(M\otimes N\)}; & \node (C2) {\(M\otimes N'\)}; \\
	    };
	    \draw[->] (A1) --node[above]{\(\id_M\otimes i\)} (A2);
	    \draw[>->] (B1) --node[above]{\(\id_F\otimes i\)} (B2);
	    \draw[->] (C1) --node[below]{\(\id_M\otimes i\)} (C2);
	    \draw[->] (A1) --node[left]{\(h\otimes\id_N\)} (B1);
	    \draw[->] (A2) --node[right]{\(h\otimes\id_{N'}\)} (B2);
	    \draw[->] (B1) --node[left]{\(\pi\otimes\id_N\)} (C1);
	    \draw[->] (B2) --node[right]{\(\pi\otimes\id_{N'}\)} (C2);
	\end{tikzpicture}
    \end{center}
    Now we have \((\pi\circ h)\otimes\id_N\bigl(\sum_im_i\otimes
	n_i\bigr)=\sum_im_i\otimes n_i\) because \(\sum_im_i\otimes
	n_i\in\im(i_0\otimes\id_N)\). Since \(\sum_i m_i\otimes
	n_i\in\ker(\id_M\otimes i)\) and \(\id_F\otimes i\) is a monomorphism
    (a free module is flat) thus \((h\otimes\id_N)\bigl(\sum_im_i\otimes
	n_i\bigr)=0\) so 
    \[\sum_im_i\otimes n_i=(\pi\otimes\id_N)\circ( h\otimes\id_N)\bigl(\sum_im_i\otimes n_i\bigr)=0\]
    Hence \(\ker\id_M\otimes i=\{0\}\).

    Now we show that \(M\) has the Mittag--Leffler property. Let \(M_i,i\in
	I\) be a family of \(\R\)-modules and let us consider the canonical
    map
    \[M\otimes\prod_{i}M_i\sir\prod_i(M\otimes M_i)\]
    Let \(\sum_i m_i\otimes(m_i^j)_{j\in I}\) be in the kernel of this map,
    where \(m_i\in M\) and \(m^j_i\in M_j\). As before let \(M_0\) be the
    submodule of \(M\) generated by all the elements \(m_i\) and we set
    \(\pi\), \(h\) as in~\eqref{diag:locally_projective_flat}. We conclude as
    before with the commutative diagram:
    \begin{center}
	\begin{tikzpicture}
	    \matrix[column sep=1.3cm,row sep=1cm]{
		\node (A1) {\(M\otimes\prod_i M_i\)}; & \node (A2) {\(\prod_i(M\otimes M_i)\)}; \\
		\node (B1) {\(F\otimes\prod_i M_i\)}; & \node (B2) {\(\prod_i(F\otimes M_i)\)}; \\
		\node (C1) {\(M\otimes\prod_i M_i\)}; & \node (C2) {\(\prod_i(M\otimes M_i)\)}; \\
	    };
	    \draw[->] (A1) -- (A2);
	    \draw[>->] (B1) -- (B2);
	    \draw[->] (C1) -- (C2);
	    \draw[->] (A1) --node[left]{\(h\otimes\id_{\prod_iM_i}\)} (B1);
	    \draw[->] (A2) --node[right]{\(\prod_ih\otimes\id_{M_i}\)} (B2);
	    \draw[->] (B1) --node[left]{\(\pi\otimes\id_{\prod_iM_i}\)} (C1);
	    \draw[->] (B2) --node[right]{\(\prod_i\pi\otimes\id_{M_i}\)} (C2);
	\end{tikzpicture}
    \end{center}
\end{proof}
\index{intersection property!when every flat module has the intersection
    property}
\begin{theorem}\label{thm:intersection_prop_for_all_modules}
    Every flat \(\R\)-module \(M\) has the intersection property (or
    equivalently has the Mittag--Leffler property) if and only if for any
    exact sequence: 
    \[0\sir M'\sir M\sir M''\sir 0\] 
    with \(M'\), \(M\) projective, \(M''\) flat and any family of submodules
    \((N_\alpha)_{\alpha\in I}\) of an \(\R\)-module \(N\) the sequence:
    \begin{equation}\label{eq:every_flat}
            0\sir \bigcap_\alpha \left(M'\otimes N_\alpha\right)\sir\bigcap_\alpha \left(M\otimes N_\alpha\right)\sir\bigcap_\alpha \left(M''\otimes N_\alpha\right)\sir 0
    \end{equation}
    is exact.
\end{theorem}
\begin{proof}
    Every flat module is a colimit of projective modules. Any colimit of
    projective modules can be computed as a cokernel of a map between
    projective modules, by~\cite[Thm~1, Chap.~V,
    \S2]{smc:categories-for-the-working-mathematician}.  So let \(M''\) be
    a flat module and let 
    \[\mathcal{E}:\ 0\sir M'\sir M\sir M''\sir 0\]
    be an exact sequence, where \(M'\) and \(M\) are projective modules. The
    extension \(\mathcal{E}\) is pure, since \(M''\) is
    flat~\cite[Thm~4.85]{tl-modules}. Let \((N_\alpha)_{\alpha\in I}\) be
    a family of submodules of an \(\R\)-module \(N\). We have a commutative
    diagram:
    \begin{center}
	\hspace*{-3mm}\begin{tikzpicture}
	    \matrix[matrix, column sep=1cm, row sep=1cm]{
		\node(A) {\(0\)}; & \node(B) {\(M'\otimes\bigcap_\alpha N_\alpha\)};   & \node(C) {\(M\otimes\bigcap_\alpha N_\alpha\)};   & \node(D) {\(M''\otimes\bigcap_\alpha N_\alpha\)};   & \node(E) {\(0\)};\\
		\node(U) {\(0\)}; & \node(W) {\(\bigcap_\alpha\left(M'\otimes N_\alpha\right)\)}; & \node(X) {\(\bigcap_\alpha\left(M\otimes N_\alpha\right)\)}; & \node(Y) {\(\bigcap_\alpha\left(M''\otimes N_\alpha\right)\)}; & \node(Z) {\(0\)};\\
	    };
	    \draw[->] (A) -- (B);
	    \draw[->] (B) -- (C);
	    \draw[->] (C) -- (D);
	    \draw[->] (D) -- (E);
	    \draw[->] (U) -- (W);
	    \draw[->] (W) -- (X);
	    \draw[->] (X) -- (Y);
	    \draw[->] (Y) -- (Z);
	    \draw[->] (B) --node[left]{\rotatebox{90}{\(\cong\)}} (W);
	    \draw[->] (C) --node[left]{\rotatebox{90}{\(\cong\)}} (X);
	    \draw[->] (D) --node[right]{\(f\)} (Y);
	\end{tikzpicture}
    \end{center}
    The upper row is exact by purity of \(M'\subseteq M\). Thus the lower row
    is exact if and only if the canonical map \(f\) is an isomorphism.
\end{proof}
The exactness of~\eqref{eq:every_flat} is rather difficult to obtain but it
might be useful once we know that every flat module is Mittag--Leffler.  For
example, by the classical Bass results, every flat left module is projective if
(and only if) the ring is left perfect.

\chapter{Lattices of quotients and subobjects}\label{chap:lattices}
This chapter contains important lattice--theoretical results.  We focus on
lattices that appear in the theory of bi- or Hopf algebras and their
(co)actions.  We examine two properties: \textit{completeness} and being an
(dually) \textit{algebraic lattice}\footnote{Dually algebraic lattices are the
    ones which opposite lattices are algebraic.}.  Completeness of many
lattices ought to be known, though the other property seems not to be
addressed by any previous studies.  Note that in complete lattices the join
(the meet) is uniquely determined by the meet (the join).  Some of the lattice
operations might seem counter-intuitive, since everything seems to be dual to
the classical structures.  For example the infimum of subalgebras is given by
intersection while the supremum is the intersection of all subalgebras which
contain both of them.  However, the coideals of a coalgebra behave
differently: the supremum of two coideals is their sum and their meet is the
sum of all coideals contained in the intersection.  The appearance of the
\emph{intersection property} shall be explained by
Theorem~\ref{thm:algebraic_structures}.  Let us note that generally we work
over a commutative ring, though some of the results are only proved over
a field.  We begin by showing that a lattice of coideals of a coalgebra,
ordered by inclusion, is complete.  Then we recall the duality between the
quotients of a coalgebra \(C\) and closed subalgebras of the topological
algebra \(A=C^*\).  From this duality we derive that the lattice of coideal is
dually algebraic (Proposition~\ref{thm:quot(C)-algebraic} on
page~\pageref{thm:quot(C)-algebraic}), by showing that the lattice of coideals
is dually isomorphic to the lattice of closed subalgebras of the dual algebra
\(C^*\).  We also construct a simple example
(Example~\ref{ex:lattices_of_coideals} on
page~\pageref{ex:lattices_of_coideals}) of coalgebras (finite dimensional
algebras) whose lattices of quotients are not distributive or not modular.  It
turns out that the lattice of of subcoalgebras of a coalgebra over a field is
algebraic: the meet is given by intersection and join is the sum
(Proposition~\ref{thm:subcoalgebras_algebraic} on
page~\pageref{thm:subcoalgebras_algebraic}).  It turns out that subcoalgebras
of a \(\k\)-coalgebra form also a dually algebraic lattice which is isomorphic
to the lattice of closed ideals of the topological algebra \(A=C^*\).  From
the decomposition theorem of commutative coalgebras over a field \(\k\), we
derive a decomposition theorem for their lattices of subcoalgebras
(Corollary~\ref{cor:decomposition_of_subcoalgebras} on
page~\pageref{cor:decomposition_of_subcoalgebras}).  Then we discuss the
theory of subcomodules of a \(C\)-comodule over a ring \(R\).  We show that
the lattice of subcomodules is complete and if the coalgebra \(C\) is a flat
Mittag--Leffler module then it is also algebraic.  The join is simply given by
the sum, and in the latter case the meet (also the infinite one) is given by
intersection.  Our further results heavily depend on this statement.  Wisbauer
showed that the category of comodules is equivalent to the smallest
Grothendieck subcategory of \(C^*\)-modules which contain \(C\) (with its
natural \(C^*\)-modules structure) if and only if \(C\) is locally projective.
In this case it follows that the lattice of subcomodules of a comodule \(M\)
is isomorphic to the algebraic lattice of \(C^*\)-submodules of the
\(C^*\)-module \(M\).

In section~\ref{sec:lattices-bi_and_Hopf_algebras} we analyse the lattice of
quotients and substructures of bialgebras and Hopf algebras, including  normal
and conormal coideals.  We also define generalised quotients and generalised
subalgebras of bialgebras
(Definition~\ref{defi:generalised_quotients_and_subalgebras} on
page~\pageref{defi:generalised_quotients_and_subalgebras}), and we show that
their lattices are complete (Proposition~\ref{prop:qquot_complete}).  This is
important for the construction of the Galois correspondence for comodule
algebras over bialgebras in the following chapter.  The lattice of generalised
subalgebras is algebraic by Proposition~\ref{prop:qsub_algebraic} (on
page~\pageref{prop:qsub_algebraic}).  Note that the lattice of generalised
quotients and generalised subbialgebras are isomorphic in some cases, for
example when the bialgebra is finite dimensional over a field
(Theorem~\ref{thm:newTakeuchi}\ref{itm:newTakeuchi_finite} on
page~\pageref{thm:newTakeuchi}) 

We end the chapter with some important examples of the lattices of generalised
quotients and generalised subalgebras.  For the group Hopf algebra \(\k[G]\)
we show that the lattice of generalised quotients is anti-isomorphic to the
lattice of subgroups of the group \(G\), while the lattice of Hopf algebra
quotients is isomorphic to the lattice of normal subgroups
(Proposition~\ref{prop:group_algebra_quotients} on
page~\pageref{prop:group_algebra_quotients}).  If the group is finite it
follows that generalised quotients of \(k[G]^*\) is anti-isomorphic to the
lattice of subgroups (Proposition~\ref{prop:dual_gruop_algebra_quotients} on
page~\pageref{prop:dual_gruop_algebra_quotients}).  Both propositions we
formulate using \(G\)-sets rather than subgroups. The reason is that we get an
isomorphism rather than an anti-isomorphism, secondly and most importantly,
\(G\)-sets appear very naturally in the Grothendieck approach to Galois
theory: the equivalence of categories of \(G\)-sets and split algebras over
a Galois extension of fields.

The final example is the lattice of generalised quotients of the enveloping
algebra of a Lie algebra \(\mathfrak{g}\).  It turns out that the result is
very similar to the case of a group algebra.  Here the lattice of generalised
quotients turns out to be anti-isomorphic to the lattice of Lie subalgebras of
\(\mathfrak{g}\), while the lattice of Hopf algebra quotients is
anti-isomorphic to the lattice of Lie ideals of \(\mathfrak{g}\). 

In the last section we note that the lattice of \(H\)-comodule subalgebras
of an \(H\)-comodule algebra \(A\) is algebraic if \(H\) is a flat
Mittag--Leffler module (Proposition~\ref{prop:comodule_subalg_algebraic} on
page~\pageref{prop:comodule_subalg_algebraic}).

We end the chapter with a helpful table which lists all the fifteen lattices
that we discuss, with the list of properties we prove together with references
to the statements in this chapter.

\section{Coalgebras}
\index{coalgebra!coideal}
\index{coalgebra!poset of coideals}
\index{coalgebra!left coideal}
\index{coalgebra!right coideal}
\begin{definition}\label{defi:coideal}
    Let \(C\) be an \(\R\)-coalgebra.  A \bold{coideal} \(I\) is a submodule
    of \(C\) such that \(C/I\) is a coalgebra and the natural epimorphism
    \(C\sir C/I\) is a map of  coalgebras.  Furthermore, we say that
    \(I\subseteq C\) is a \bold{right} (\bold{left}) \bold{coideal} if
    \(C/I\) is a right (left) comodule and the quotient map \(C\sir C/I\) is
    a map of right (left) \(C\)-comodules.
\end{definition}
Note that~\citeauthor{tb-rw:corings-and-comodules} define coideals as kernels
of a surjective coalgebra homomorphism and then they show that these two
definitions are equivalent.  Let us cite this result here:
\begin{proposition}[{\cite{tb-rw:corings-and-comodules}}]
    Let \(C\) be an \(\R\)-coalgebra and let \(I\) be an \(\R\)-submodule of
    \(C\). We let \(\pi:C\sir C/I\) be the quotient \(\R\)-module map. Then the
    following conditions are equivalent: 
    \begin{enumerate}
	\item \(I\) is a coideal in the sense of Definition~\ref{defi:coideal},
	\item \(I\) is a kernel of a surjective coalgebra map,
	\item \(\Delta(I)\subseteq\ker(\pi\otimes\pi)\), and \(I\subseteq\ker\epsilon\).
    \end{enumerate}
    Furthermore, if \(I\subseteq C\) is a pure submodule, then (i)-(iii) are
    equivalent to:
    \begin{enumerate}
	\item \(\Delta(I)\subseteq I\otimes C+C\otimes I\).
    \end{enumerate}
    If (i) holds then \(C/I\) is cocommutative provided \(C\) is.
\end{proposition}
\begin{proof}
    For the proof see~\cite[Prop.~2.4]{tb-rw:corings-and-comodules}.
\end{proof}
\index{coalgebra!poset of coideals!complete}
\begin{proposition}\label{prop:coid-complete}
    For a coalgebra \(C\) over a commutative ring~\(\R\) the set of all
    coideals, denoted by \(\coId(C)\), forms a \textsf{complete lattice} with
    inclusion as the order relation.  The lattice operations in \(\coId(C)\)
    are given by
    \begin{align*}
	I_1\vee I_2 & =I_1+I_2\\
	I_1\wedge I_2 & =\sum_{\substack{I\in \coId(C)\\I\subseteq I_1\cap I_2}}I
    \end{align*}
\end{proposition}
\begin{proof}
    It is enough to show that any infinite suprema exist.  Let
    \(\pi_\lambda:C\sir C_\lambda\) (\(\lambda\in\Lambda\)) be a family of
    coalgebra epimorphisms with kernels \(I_\lambda\).  Let us take
    \(I=\sum_{\lambda\in\Lambda}I_\lambda\) and let \(p:C\sir C/I\) be the
    natural projection.  By the previous proposition it is enough to show that
    \(p\otimes p\circ\Delta(I)=0\), since \(I\subseteq\ker\epsilon\).  Let us
    take \(c\in I\).  It is a sum \(c=\sum_{\lambda\in\Lambda}c_\lambda\),
    where each \(c_\lambda\in I_\lambda\) and only finitely many of them are
    non--zero.  Then 
    \[p\otimes p\circ\Delta(c)=\sum_{\lambda\in\Lambda}p({c_\lambda}_{(1)})p({c_\lambda}_{(2)})=0\]
    The last equality follows since \(p\) factorises through each
    \(p_\lambda\) and for each \(\lambda\) we have 
    \[p_\lambda({c_\lambda}_{(1)})p_\lambda({c_\lambda}_{(2)})=0\]
    since \(I_\lambda\) is a coideal.
\end{proof}

\index{coalgebra!dual algebra}
If \(C\) is a \(\k\)-coalgebra then the dual space \(C^*\) is a unital algebra
with the following multiplication (the \textit{convolution product}):
\[f\ast g(c)\coloneq f(c_{(1)})g(c_{(2)}),\quad f,g\in C^*\]
The unit of the convolution product is the counit \(\epsilon:C\sir\k\).  Let
\(V\) be a \(\k\)-vector space and let \(V^*\) be its dual.  For a subspace
\(W\subseteq V\) we let \(W^\perp\coloneq\{f\in V^*:\,f|_W=0\}\), while for
\(W\subseteq V^*\) we will write \(W^\perp\coloneq\bigcap_{f\in W}\ker\,f\).
Note that if \(V\) is a finite dimensional space then the above maps define
a bijective correspondence (which reverses the inclusion order) between the
subspaces of \(V\) and the subspaces of its dual.
\begin{proposition}\label{prop:right_coideals-right_ideals}
    Let \(C\) be a coalgebra over a field \(\k\). Then
    \begin{enumerate}
	\item \(J\subseteq C\) is a right (left) coideal if and only if
	    \(J^\perp\) is a right (left) ideal in \(C^*\).
	\item if \(I\subseteq C^*\) is a right (left) ideal in \(C^*\),
	    then \(I^\perp\) is a right (left) coideal in \(C\);
    \end{enumerate}
\end{proposition}
\begin{proof}
    See~\cite[Prop.~1.4.5]{ms:hopf-alg}.
\end{proof}
If follows that for a finite dimensional coalgebra \(C\) the poset of right
(left) coideals of \(C\) is anti-isomorphic to the poset of right (left)
ideals of \(C^*\).
\begin{proposition}\label{prop:coideals-subalgebras}
    Let \(C\) be a coalgebra over a field \(\k\).  Then
    \begin{enumerate}
	\item \(J\subseteq C\) is a coideal if and only if
	    \(J^\perp\subseteq C^*\) is a subalgebra.
	\item if \(S\subseteq C^*\) is a subalgebra of \(C^*\), then
	    \(S^\perp\) is a coideal of \(C\);
    \end{enumerate}
\end{proposition}
\begin{proof}
    See~\cite[Prop.~1.4.6]{ms:hopf-alg}.
\end{proof}
In the finite dimensional case we have the following result:
\index{coalgebra!poset of coideals!dually algebraic}
\begin{proposition}\label{prop:coid_subag}
    Let \(C\) be a finite dimensional coalgebra over a field \(\k\). Then we
    have a dual isomorphism of lattices:
    \[(\coId(C),\wedge,+)\cong(\Sub_\Alg(C^*),\cap,\vee)\]
    and thus \(\coId(C)\) is a dually algebraic lattice.
\end{proposition}
\begin{proof}
    The isomorphism of lattices is given by \(\coId(C)\ni
	I\selmap{}I^\perp\coloneq\{f\in C^*:f|_I=0\}\in\Sub_\Alg(C^*)\). We
    have \((I_1+I_2)^\perp=I_1^\perp\cap I_2^\perp\) for \(I_i\in\coId(C)\)
    (\(i=1,2\)), but furthermore this formula works for infinite joins of
    coideals. From this we get
    \begin{alignat*}{2}
	(I_1\wedge I_2)^\perp & = \bigl(\sum_{\substack{I\in\coId(C)\\I\subseteq I_1\cap I_2}}I\bigr)^\perp & \,=\, & \bigl(\bigcap_{\substack{I\in\coId(C)\\ I\subseteq I_1\cap I_2}}I^\perp\bigr)\\
	& = \bigl(\bigcap_{\substack{A\in\Sub_\Alg(C^*)\\ I_1^\perp\cup
		I_2^\perp\subseteq A}}A\bigr) & & {\small\text{(since
	    }\coId(C^*)^\op\cong\Sub_\Alg(C^*)\text{)}}\\
	& = I_1^\perp\vee I_2^\perp & &\\
    \end{alignat*}
\end{proof}
Thus \(\Quot(C)\) is an algebraic lattice (see Definition~\ref{defi:alglat} on
page~\pageref{defi:alglat}) if \(C\) is finite dimensional. We are going to
show that it is algebraic regardless of the dimension of \(C\).  Let us note
that every complete upper subsemilattice of the lattice of subvector spaces of
a finite dimensional vector space (like \(\coId(C)\) for a finite dimensional
coalgebra \(C\)) is algebraic, since every vector subspace is a compact
element of the lattice of subvector spaces.  
\begin{remark}\label{rem:sublattices_of_sub_vect}
    Every complete sublattice of the lattice of subspaces of a finite
    dimensional vector space is dually algebraic, since the lattice of
    subvector spaces of a finite dimensional vector space \(V\) is
    anti-isomorphic to the lattice of subvector spaces of the dual vector
    space \(V^*\).  
\end{remark}
The lattices of: \(\k\)-subcoalgebras, \(\k\)-subbialgebras, \(\k\)-subHopf
algebras, as we will see later, are sublattices of the lattice of subvector
spaces.

In order to show that the lattice of coideals is algebraic regardless of the
dimension we will need some finer tools to study the dual algebra.  The first
of these is the fundamental theorem of comodules.
\index{fundamental theorem of comodules}
\begin{theorem}[{\textbf{Fundamental Theorem of Comodules}}]\label{thm:fundamental_thm_of_comodules}
    Let \(C\) be a coalgebra over a field~\(\k\) and let \(M\) be a right
    \(C\)-comodule. Any element \(m\in M\) belongs to a finite dimensional
    subcomodule.
\end{theorem}
The proof can be found in many text books: it follows
from~\cite[Cor.~2.1.4]{ms:hopf-alg} or is proved
in~\cite[Thm.~2.1.7]{sd-cn-sr:hopf-alg}.
\index{coalgebra!poset of quotients}
\begin{definition}
    Let \(C\) be an \(\R\)-coalgebra.  We let \(\Quot(C)=\{C/I:\ I-\text{
		a coideal of }C\}\) with order relation \(C/I_1\succcurlyeq
	C/I_2\Leftrightarrow I_1\subseteq I_2\). 
\end{definition}
\index{coalgebra!poset of quotients!lattice}
Clearly \(\Quot(C)\) is anti-isomorphic to the lattice \(\coId(C)\) and thus,
by Proposition~\ref{prop:coid-complete}, it is a complete lattice. 

We will now study the dual algebra \(C^*\) in some more detail.  It turns out
that it is a topological algebra.  Let \(X\) and \(Y\) be non empty sets. The
\bold{finite topology} on the mapping space \(Y^X\) is the product topology
when we view \(Y^X\) as a product of \(Y_x\coloneq Y\) for \(x\in X\), where
each \(Y_x\) is regarded as a discrete space. A basis for open sets in this
topology is given by the sets of the form
\[\mathcal{U}_{g,x_1,\ldots,x_n}=\left\{f\in Y^X:\,f(x_i)=g(x_i), i=1,\dots,n\right\}\]
where \(g\in Y^X\) and \(\{x_i:\,i=1,\dots,n\}\subseteq X\) is a finite
subset.  Every open set is a union of open sets of this form. Now, if \(X\)
and \(Y\) are \(\k\)-vector spaces then \(\Hom_\k(X,Y)\) is a subspace of
\(Y^X\) and we will consider the topology induced by the finite topology of
\(Y^X\). This topology on \(\Hom_\k(X,Y)\) is also called the \bold{finite
    topology}.  The following theorem holds:
\begin{theorem}\label{thm:duality_for_subspace}
    Let \(V\) be a vector space. Then the maps \(\Sub(V)\ni
	W\selmap{}W^\perp\in\Sub(V^*)\) and \(\Sub(V^*)\ni
	W\selmap{}W^\perp\in\Sub(V)\) form a Galois connection. Furthermore,
    the map \(\Sub(V)\ni W\selmap{}W^\perp\in\Sub(V^*)\) is a monomorphism and
    \(W\in\Sub(V^*)\) is closed in this Galois connection if and only if \(W\)
    is closed in the finite topology on \(V^*=\Hom_\k(V,\k)\).
\end{theorem}
\begin{proof}
    The proof can be found in~\cite[Thm~1.2.6]{sd-cn-sr:hopf-alg}.
\end{proof}
\begin{corollary}\label{cor:dualit_for_subspaces}
    There is a bijection between subspaces of \(V\) and closed subspaces of
    \(V^*\).
\end{corollary}
In the following definition we restrict ourselves only to discrete
topological fields.
\index{topological vector space}
\index{topological algebra}
\begin{definition}
    Let \(\k\) be a field considered with the discrete topology.
    \begin{itemize}
	\item Let \(V\) be a \(\k\)-vector space. It is called a
	    \bold{topological vector space} if it is given together with
	    a topology such that the addition of vectors and the scalar
	    multiplication are continuous operations.  
	\item Let \(A\) be an \(k\)-algebra. We say that \(A\) is
	    a \bold{topological algebra} if \(A\) is a topological vector
	    space and the multiplication and the unit are continuous.
    \end{itemize}
\end{definition}
\begin{example}
    Let \(V\) be a vector space. Then \(V^*\) together with the finite
    topology is a topological vector space.
    See~\cite[Prop.~1.2.1]{sd-cn-sr:hopf-alg}.
\end{example}
\index{coalgebra!dual algebra!topological algebra}
\begin{lemma}\label{lem:topological_algebra}
    Let \(C\) be a \(\k\)-coalgebra. Then \(C^*\) together with the finite topology
    is a topological algebra.
\end{lemma}
\begin{proof}
    The open sets of the form \(\mathcal{O}_V\coloneq\{f\in C^*:f|_V=0\}\),
    where \(V\subseteq C\) is a finite dimensional subspace, form a basis of
    neighbourhoods of \(0\) and thus it is enough to show that the preimage of
    an open set \(\mathcal{O}_V\) is open. For this let \((f,g)\) be such that
    \(f\ast g\in\mathcal{O}_V\). Using the \textit{Fundamental Theorem of
	Comodules}~\ref{thm:fundamental_thm_of_comodules} there exists \(V_r\)
    a right subcomodule of \(C\) such that \(V\subseteq V_r\), such that
    \(\dim V_r<\infty\), and a finite dimensional left subcomodule of \(C\),
    denoted by \(V_l\), such that \(V\subseteq V_l\).  Then
    \(\mathcal{O}_{V_r}\times\mathcal{O}_{V_l}\) is an open neighbourhood of
    \((f,g)\in C^*\times C^*\) such that
    \(\mathcal{O}_{V_r}\ast\mathcal{O}_{V_l}\subseteq\mathcal{O}_V\). That is,
    for \(a\in\mathcal{O}_{V_r}\) and \(b\in\mathcal{O}_{V_l}\) and \(v\in V\)
    we have \(\Delta(v)\in V_r\otimes C\cap C\otimes V_l=V_r\otimes V_l\) and
    hence \((a\ast b)(v)=a(v_{(1)})b(v_{(2)})=0\), and thus \(a\ast
	b\in\mathcal{O}_V\). Since, \(\k\) is considered as a discrete space
    the unit map \(\k\sir C^*\) is continuous. This shows that \(C^*\) is
    a topological \(\k\)-algebra.
\end{proof}
\index{coalgebra!poset of quotients!algebraic}
\begin{theorem}\label{thm:quot(C)-algebraic}
    Let \(C\) be a coalgebra over a field \(\k\). Then the lattice \(\Quot(C)\)
    is algebraic.
\end{theorem}
\begin{proof}
    First let us observe that \(\Quot(C)\) is isomorphic to the lattice of
    closed subalgebras of \(C^*\). This follows from
    Lemma~\ref{lem:topological_algebra},
    Theorem~\ref{thm:duality_for_subspace} and
    Proposition~\ref{prop:coideals-subalgebras}. Let \(X\subseteq C^*\) be
    a finite set. Then the smallest closed subalgebra of \(C^*\) which
    contains \(X\) (denoted by \(S_X\)) is a compact element of the lattice of
    closed subalgebras of \(C^*\). For a closed subalgebra \(S\) we have
    \(S=\bigvee_{\substack{X\subseteq S\hfill\\ X\text{-finite}\hfill}}S_X\).
\end{proof}
\index{lattice!modular}
\index{modular lattice}
\begin{definition}
    Let \((L,\vee,\wedge)\) be a lattice. We say that it is \bold{modular}
    if for all \(a,b,c\in L\):
    \[a\geq c\Rightarrow a\wedge(b\vee c)=(a\wedge b)\vee c\]
    A lattice \(L\) is \bold{distributive} if for all \(a,b,c\in L\):
    \begin{align*}
	a\wedge(b\vee c) &=(a\wedge b)\vee(a\wedge c)\\
	a\vee(b\wedge c) &= (a\vee b)\wedge(a\vee c)
    \end{align*}
\end{definition}
Let us note that if one of the distributive laws holds then the other is
satisfied. If a lattice is distributive then it is modular. Let us introduce
the following two lattices:
\begin{center}
    \begin{tikzpicture}
	\fill (0cm, 1cm) circle (2pt) node[above]{1};
	\fill (-1cm, 0cm) circle (2pt) node[left]{a};
	\fill (0cm, 0cm) circle (2pt) node[left]{b};
	\fill (1cm, 0cm) circle (2pt) node[right]{c};
	\fill (0cm, -1cm) circle (2pt) node[below]{0};
	\draw[-] (0cm,1cm) -- (0cm,-1cm);
	\draw[-] (0cm,1cm) -- (-1cm, 0cm) -- (0cm,-1cm);
	\draw[-] (0cm,1cm) -- (1cm, 0cm) -- (0cm,-1cm);
	\node at (0cm,-2cm) {\(M_3\)};
    \end{tikzpicture}
    \hspace{1.5cm}
    \begin{tikzpicture}
	\fill (0cm,1cm) circle (2pt) node[above]{1};
	\fill (-1cm, 5mm) circle (2pt) node[left]{a};
	\fill (-1cm, -5mm) circle (2pt) node[left]{c};
	\fill (1cm, 0cm) circle (2pt) node[right]{b};
	\fill (0cm, -1cm) circle (2pt) node[below]{0};
	\draw[-] (0cm,1cm) -- (-1cm,5mm) -- (-1cm,-5mm) -- (0cm, -1cm);
	\draw[-] (0cm,1cm) -- (1cm, 0cm) -- (0cm, -1cm);
	\node at (0cm,-2cm) {\(N_5\)};
    \end{tikzpicture}
\end{center}
One can easily verify that \(M_3\) is not distributive and \(N_5\) is not
modular.
\begin{theorem}
    Let \((L,\vee,\wedge)\) be a lattice. 
    \begin{enumerate}
	\item The lattice \(L\)  is modular if and only if \(L\) does not have
	    a sublattice isomorphic to \(N_5\).
	\item The lattice \(L\)  is distributive if and only if \(L\) does not
	    have a sublattice isomorphic to either \(N_5\) or \(M_3\).
    \end{enumerate}
\end{theorem}
\begin{proof}
    See~\cite[Thm.~4.7]{sr:lattices}.
\end{proof}
The lattices of coideals are in general neither modular nor distributive,
since the lattice of submodules in general does not possess these properties.
Below we present some examples.
\index{coalgebra!poset of coideals!non distributive}
\index{coalgebra!poset of coideals!non modular}
\begin{examplesbr}\label{ex:lattices_of_coideals}
    \begin{enumerate}
	\item Let \(V\) be a vector space such that \(\dim V\geq 2\). Then the
	      lattice \(\Sub_{\mathit{Vect}}(V)\) is not distributive.  It is
	      modular, as every lattice of submodules is.
	\item Let \(C\) be a coalgebra such that \(C=C_0\oplus k1\).  For
	      every \(c\in C_0\) we set \(\Delta(c)=c\otimes 1+1\otimes c\),
	      and \(\Delta(1)=1\otimes 1\).  The counit is set by
	      \(\epsilon(c)=0\) for all \(c\in C_0\) and \(\epsilon(1)=1\).
	      The coalgebra \(C\) is cocommutative.  Every subspace of
	      \(C_0=ker\epsilon\) is a coideal of \(C\), i.e.
	      \(\coId(C)=\Sub_{\mathit{Vect}}(C_0)\).  Thus the lattice of
	      coideals is not distributive if \(\dim C_0\geq2\).

	      Let us note that this coalgebra is dual to the commutative
	      unital algebra \(A=A_0\oplus k1\), with unit \(1\) and such that
	      for all
	      \(a,b\in A_0\) we have \(a\cdot b=0\) (with \(\dim A_0=\dim C_0\)).
	\item Let \(A\) be a finite dimensional commutative and unital algebra
	      generated by two elements \(a\) and \(x\) with relations:
	      \[a^2x=a,\ x^2=1,\ a^4=1.\]
	      Then the lattice of subalgebras contains \(N_5\) as
	      a sublattice:
	      \begin{center}
	          \begin{tikzpicture}
	              \fill (0cm,1cm) circle (2pt) node[above]{\(A\)};
	              \fill (-1cm, 5mm) circle (2pt) node[left]{\(\langle a\rangle\)};
	              \fill (-1cm, -5mm) circle (2pt) node[left]{\(\langle a^2\rangle\)};
	              \fill (1cm, 0cm) circle (2pt) node[right]{\(\langle x\rangle\)};
	              \fill (0cm, -1cm) circle (2pt) node[below]{0};
	              \draw[-] (0cm,1cm) -- (-1cm,5mm) -- (-1cm,-5mm) -- (0cm, -1cm);
	              \draw[-] (0cm,1cm) -- (1cm, 0cm) -- (0cm, -1cm);
	          \end{tikzpicture}
	      \end{center}
	      where \(\langle y\rangle\) denotes the subalgebra generated by
	      \(y\in A\).  Hence the lattice \(\Quot(A^*)\cong\Sub_\Alg(A)\)
	      is not modular, where \(A^*\) is a (cocommutative)
	      \(\k\)-coalgebra since \(A\) is a finite dimensional
	      (commutative) \(\k\)-algebra.
    \end{enumerate}
\end{examplesbr}
\index{coalgebra!subcoalgebra}
\index{coalgebra!poset of subcoalgebras}
\index{subcoalgebra|see{coalgebra!subcoalgebra}}
\begin{definition}
    Let \(C\) be an \(\R\)-coalgebra. An \(\R\)-coalgebra \(C'\) is called
    a \bold{subcoalgebra} if 
    \begin{enumerate} 
	\item \(C'\) is a coalgebra,
        \item \(C'\) is an \(\R\)-submodule of \(C\), and 
	\item the inclusion map \(C'\subseteq C\) is a coalgebra map.  
    \end{enumerate}
    The set of all subcoalgebras of a coalgebra \(C\) we will denote by
    \(\Sub_\Coalg(C)\). It is a poset under the following relation: for
    subcoalgebras \(C'\) and \(C''\) of a coalgebra \(C\), \(C'\preceq C''\)
    if and only if \(C'\) is a subcoalgebra of \(C''\).
\end{definition}
If \(C'\) is a pure submodule (for example if \(\R\) is a field) then \(C'\)
is a subcoalgebra of a coalgebra \(C\) if and only if \(\Delta_C(C')\subseteq
    C'\otimes C'\), since by the purity \(C'\otimes C'\subseteq C\otimes C\).
\index{coalgebra!poset of subcoalgebras!complete}
\begin{theorem}\label{thm:subcoalgebras_algebraic}
    Let \(C\) be a coalgebra over a field \(\k\).  Then the poset of
    subcoalgebras \(\Sub_\Coalg(C)\) is a complete poset with lattice
    operations:
    \[D_1\vee D_2\coloneq D_1+D_2,\ D_1\wedge D_2\coloneq D_1\cap D_2\]
    for \(D_i\subseteq C\) (\(i=1,2\)) subcoalgebras of \(C\). Furthermore, it
    is closed under arbitrary intersections and directed (set-theoretic) sums
    and thus it is an algebraic lattice.
\end{theorem}
\begin{proof}
    Note that if \(D,D'\) are subcoalgebras of a coalgebra \(C\) then \(D+D'\)
    is a subcoalgebra of \(C\). Moreover, if \(\mathcal{O}\) is a family of
    subcoalgebras then \(\sum_{D\in\mathcal{O}}D\) is a subcoalgebra of \(C\).
    Thus \(\Sub_\Coalg(C)\) is a complete lattice. Now let us show that if
    \(C_i\) (for \(i\in I\)) is a collection of subcoalgebras, then
    \(\bigcap_{i\in I}C_i\) is a subcoalgebra. We have \(\bigcap_{i\in
	    I}C_i=\bigcap_{i\in I}(C_i^{\perp\perp})=(\sum_{i\in
	    I}C_i^\perp)^\perp\). The first equality follows from
    Theorem~\ref{thm:duality_for_subspace} and
    Proposition~\ref{prop:properties-of-adjunction}(vi) and the second from
    Lemma~\ref{lem:continuity}. The sum \(\sum_{i\in I}C_i^\perp\) is an
    ideal, since \(C_i^\perp\) are ideals of the algebra \(C^*\). Thus
    \(\bigcap_{i\in I}C_i=(\sum_{i\in I}C_i^\perp)^\perp\) is a subcoalgebra
    in \(C\). It follows that \(\Sub_\Coalg(C)\) is
    a \(\cap\overrightarrow{\cup}\)-structure and by
    Theorem~\ref{thm:algebraic_structures} it is an algebraic lattice. 
\end{proof}
Hence \(\Sub_\Coalg(C)\) (for a \(\k\)-coalgebra \(C\)) is a sublattice of
\(\Sub_{\k_\Vect}(C)\) and thus by
Remark~\ref{rem:sublattices_of_sub_vect}, if \(C\) is finite dimensional
this lattice is also dually algebraic. It turns out that this property holds
for any \(\k\)-coalgebra. 
\index{coalgebra!poset of subcoalgebras!dually algebraic}
\begin{theorem}\label{thm:subcoalgebras_dualy_algebraic}
    Let \(C\) be a \(\k\)-coalgebra. Then the lattice \(\Sub_\Coalg(C)\) is
    anti-isomorphic to the lattice of closed ideals of the algebra \(C^*\) and thus
    it is a dually algebraic lattice.
\end{theorem}
\begin{proof}
    The map \(\Sub_\Coalg(C)\ni D\selir
    D^\perp\in\{I\in\Id(A):I\text{-closed}\}\) is a bijection
    by~\cite[Prop.~1.4.3]{ms:hopf-alg} and
    Corollary~\ref{cor:dualit_for_subspaces}. Since the lattice of closed
    ideals of a topological algebra is algebraic (with compact elements:
    closures of finitely generated ideals) the theorem follows.
\end{proof}

Let \(C\) be a coalgebra over a field \(\k\). For \(c\in C\) there exists a
smallest subcoalgebra, denoted by \(C(c)\), such that \(c\in C(c)\).
Furthermore, by the \textit{Fundamental Theorem of Coalgebras} it is finite
dimensional.
\index{fundamental theorem of coalgebras}
\begin{theorem}[{\textbf{Fundamental Theorem of
	    Coalgebras}}]\label{thm:fundamental_theorem_of_coalgebras}
    Let \(C\) be a coalgebra over a field \(\k\) and let \(c\in C\). Then there
    exists a finite dimensional subcoalgebra of \(C\) which contains \(c\).
\end{theorem}
\begin{proof}
    See~\cite[Thm~1.4.7]{sd-cn-sr:hopf-alg}.
\end{proof}
Let \(V\subseteq C\) be a subset. We let \(C(V)\) be the smallest coalgebra
which contains \(V\). Clearly, \(C(V)=C(\Span(V))\), where \(\Span(V)\)
denotes the vector subspace spanned by \(V\). Furthermore, \(C(V)=\sum_{v\in
	V}C(v)\). The compact elements of \(\Sub_\Coalg(C)\) are precisely the
subcoalgebras \(C(V)\) where \(V\) is a finite subset of \(C\), by
Remark~\ref{rem:compact_elements}. Since these subcoalgebras are all finite
dimensional and clearly all finite dimensional subcoalgebras are compact we
conclude with
\index{coalgebra!poset of subcoalgebras!compact elements}
\begin{proposition}
    A subcoalgebra \(B\) of a \(\k\)-coalgebra \(C\) (where \(\k\) is a field)
    is a compact element of \(\Sub_\Coalg(C)\) if and only if \(\dim B<\infty\).
\end{proposition}
\begin{example}
    Let \(C=C_0\oplus k1\) be the \(\k\)-coalgebra from
    Example~\ref{ex:lattices_of_coideals}(ii).  Then a subspace \(V\subseteq
	C\) is a subcoalgebra if and only if \(1\in V\).  It is easy to observe
    that every such subspace is indeed a subcoalgebra.  Now, let us assume that
    \(D\subseteq C\) is a subcoalgebra.  Take \(d\in D\) and write it as
    \(d=d_0+\lambda1\) where \(d_0\in C_0\) is non zero and \(\lambda\in \k\).
    Let \(d_0^*\) be an element of \(C^*\) such that \(d_0^*(d_0)=1\) and
    \(d_0(1)=0\). Then \(\Delta(d)=\lambda 1\otimes1+d_0\otimes1+1\otimes
	d_0\in D\otimes D\).  Now we apply \(d_0^*\otimes\id_C\) and we obtain:
    \(1\in D\).  Thus we have an isomorphism of lattices:
    \[\Sub_\Coalg(C)\cong\Sub_\Vect(C_0)\]
    Hence \(\Sub_\Coalg(C)\) is modular and it is not distributive if \(\dim
	C_0\geq 2\).
\end{example} 
More can be said about the lattice of subcoalgebras of a cocommutative
coalgebra. For this we need the following notions:
\index{coalgebra!simple}
\index{coalgebra!irreducible}
\index{coalgebra!irreducible components}
\index{coalgebra!pointed}
\index{simple coalgebra|see{coalgebra!simple}}
\index{irreducible coalgebra|see{coalgebra!irreducible}}
\index{pointed coalgebra|see{coalgebra!pointed}}
\begin{definition}\label{defi:simpe_irreducible_coalgebras}
    Let \(C\) be a coalgebra. It is called:
    \begin{enumerate}
	\item \bold{simple} if it has no proper subcoalgebras, i.e. the
	    only subcoalgebras are \(\{0\}\) and \(C\);
        \item \bold{irreducible} if it has a unique simple subcoalgebra;
        \item \bold{pointed} if every of its simple subcoalgebras is one
	    dimensional.
    \end{enumerate}
    Since a sum of irreducible subcoalgebras which contain a common simple
    subcoalgebra is an irreducible subcoalgebra there exists maximal
    irreducible subcoalgebras. These are called \bold{irreducible components}.
    An irreducible component that is pointed is called a \bold{pointed
	irreducible component}. 
\end{definition}
\index{universal enveloping algebra!pointed}
For example the coalgebra \(\mathcal{U}(\mathfrak{g})\), where
\(\mathfrak{g}\) is a Lie algebra, is a pointed irreducible coalgebra, with
the unique simple subcoalgebra \(\k1\subseteq\mathcal{U}(\mathfrak{g})\). 

Let us note that every subcoalgebra contains a nontrivial simple subcoalgebra.
By the \textit{Fundamental Theorem of Coalgebras}
(Theorem~\ref{thm:fundamental_theorem_of_coalgebras}) it contains a finite
dimensional subcoalgebra. If it is not simple, it contains a nontrivial
subcoalgebra of smaller dimension. There must be a nonzero simple subcoalgebra
by a finite induction. The \textit{Fundamental Theorem of Coalgebras} shows
also that simple coalgebras are all finite dimensional.
\index{coalgebra!irreducible components}
\begin{theorem}[{\cite[Thm~2.4.7]{ea:hopf_algebras}}]\label{thm:irreducible_subcoalgebras}
    Let \(C\) be a coalgebra over a field \(\k\). Then
    \begin{enumerate}
	\item an arbitrary irreducible subcoalgebra of \(C\) is
	    contained in an irreducible component of \(C\);
	\item a sum of distinct irreducible components of \(C\) is
	    a direct sum;
	\item if \(C\) is cocommutative then it is a direct sum of its
	    irreducible components.
    \end{enumerate}
\end{theorem}
Let \(L,K\) be two lattices. Then \(L\times K\) is a lattice with
component wise operations and the order given by: \((l,k)\geq_{L\times
K}(l',k')\) if and only if \(l\geq_L l'\) and \(k\geq_K k'\).  The lattice
\(L\times K\) is called the \textsf{product lattice} of \(L\) and \(K\).
\index{lattice!indecomposable}
\index{indecomposable lattice|see{lattice!indecomposable}}
\begin{definition}
    Let \(L\) be a lattice.  It is called \bold{indecomposable} if it is not
    isomorphic to a product of two lattices.
\end{definition}
As a corollary of Theorem~\ref{thm:irreducible_subcoalgebras}(iii) we get.
\index{coalgebra!poset of subcoalgebras!decomposition (for cocommutative
    coalgebras)}
\begin{corollary}\label{cor:decomposition_of_subcoalgebras}
    Let \(C\) be a cocommutative coalgebra over a field \(\k\).  Then the
    lattice \(\Sub_\Coalg(C)\) has a direct product decomposition into
    indecomposable sublattices.  Let \(C_i\) (\(i\in I\)) be the set of all
    irreducible components of \(C\).  Then the indecomposable components of
    \(\Sub_\Coalg(C)\) are the sublattices \(\Sub_\Coalg(C_i)\).
\end{corollary}
\begin{proof}
    It only remains to show that for each irreducible component \(C_i\) of
    \(C\) the lattice \(\Sub_\Coalg(C_i)\) is indecomposable.  For this let
    \(M_i\subseteq C_i\) be the unique simple subcoalgebra of \(C_i\).  Then for
    every \(0\neq D\subseteq C_i\), where \(D\) is a subcoalgebra, we have
    \(M_i\leq D\).  Now let us assume that \(\Sub_\Coalg(C_i)\cong L\times K\),
    where \(L,K\) are sublattices.  Since \(\Sub_\Coalg(C_i)\) is bounded and
    complete so are the sublattices \(L\) and \(K\).  Then we must have
    \((0_K,1_L)\geq M_i\) and \((1_K,0_L)\geq M_i\) and thus \(0_{K\times
	    L}=(0_K,1_L)\wedge(1_K,0_L)\geq M_i\), hence \(M_i=0_{K\times L}\)
    which is a contradiction.
\end{proof}

Let us now pass to \(C\)-comodules.  The following theorem is an important
step for us, since it will allow for implications when we mix algebraic
structures like subalgebras with subcomodules (for example \emph{generalised
    subalgebras} of bialgebras).  It also will be used in the proof of the
construction of a Galois connection for \(H\)-extensions
(Theorem~\ref{thm:existence} on page~\pageref{thm:existence}) and also when we
compare our Galois connection with an earlier result of Schauenburg
(Remark~\ref{rem:comparison_of_adjunctions} on
page~\pageref{rem:comparison_of_adjunctions}).
\index{comodule!poset of subcomodules!complete}
\begin{theorem}\label{thm:lattice_of_subcomodules}
    Let \(M\) be a (right) \(C\)-comodule, for a coalgebra \(C\) over
    a commutative ring \(\R\).  Then the poset of subcomodules of \(M\),
    denoted by \(\Sub_{\Mod^C}(M)\), is a complete lattice.  For
    \(N_i\in\Sub_{\Mod^C}(M)\), \(i\in I\) we have:
    \[\bigvee_{i\in I}N_i=\sum_{i\in I}N_i\]
    Furthermore, if \(C\) is flat as an \(\R\)-module, then \(N_1\wedge
	N_2=N_1\cap N_2\) and the lattice of subcomodules of \(M\) is modular.
    The lattice \(\Sub_{\Mod^C}(M)\) is algebraic if \(C\) is a flat
    Mittag--Leffler \(\R\)-module.  In the latter case we thus have:
    \[\bigwedge_{i\in I}N_i=\bigcap_{i\in I}N_i\]
    for a family of subcomodules \(N_i\subseteq M\), \(i\in I\).
\end{theorem}
For \(M=C\), for a coalgebra \(C\) over a field, one can prove this theorem in
the same way as Theorem~\ref{thm:quot(C)-algebraic}, since right coideals of
\(C\) correspond to closed right ideals of \(C^*\) (by
Theorem~\ref{thm:duality_for_subspace}, Lemmas~\ref{lem:topological_algebra}
and~\ref{prop:right_coideals-right_ideals}), which form an algebraic lattice.
Let us note that if the ground ring is a field then the category of
\(C\)-comodules is equivalent to the category of rational \(C^*\)-modules.
Furthermore, a quotient module, and a submodule as well, of a rational module
is rational by~\cite[Thm.~2.2.6]{sd-cn-sr:hopf-alg}.  Hence the lattice of
subcomodules of a \(C\)-comodule \(M\) is isomorphic to the lattice of
submodules of the rational module \(M\) with the induced \(C^*\)-module
structure.
\begin{proofof}{{Theorem~\ref{thm:lattice_of_subcomodules}}}
    First we note that if \((N_i)_{i\in I}\) is a family of subcomodules
    of a \(C\)-comodule \(M\) then 
    \[\bigvee_{i\in I}N_i=\sum_{i\in I}N_i\in\Sub_{\Mod^C}(M)\]
    Thus the poset of subcomodules is a complete lattice.  Let us assume that
    \(C\) is flat.  Then \(N_i\subseteq M\) is a \(C\)-subcomodule if
    \(\delta(N_i)\subseteq N_i\otimes C\subseteq M\otimes C\), where
    \(\delta:M\sir M\otimes C\) is the \(C\)-comodule structure map.  Then by
    Proposition~\ref{prop:flat_intersection} we have \(\delta(N_1\cap
    N_2)\subseteq (N_1\otimes C)\cap(N_2\otimes C)=(N_1\cap N_2)\otimes C\).
    Thus \(N_1\cap N_2\) is a subcomodule and \(N_1\wedge N_2=N_1\cap N_2\).

    Now, if \(C\) is a flat Mittag--Leffler module then for any family of its
    subcomodules \((N_i)_{i\in I}\), we have
    \[\delta\Bigl(\bigcap_{i\in I}N_i\Bigr)\subseteq\bigcap_{i\in I}\bigl(N_i\otimes C\bigr)=\Bigl(\bigcap_{i\in I}N_i\Bigr)\otimes C\]
    by Proposition~\ref{prop:Mittag-Leffler_intersection} and thus
    \(\bigcap_{i\in I}N_i\) is a \(C\)-subcomodule of \(M\).  Now it follows
    that \(\Sub_{\Mod^C}(C)\) is a \(\cap\overrightarrow{\cup}\)-structure and
    thus is an algebraic lattice.

    \citeauthor[3.13]{tb-rw:corings-and-comodules} show that the category of
    of comodules is a Grothendieck category if \(C\) is flat as an
    \(\R\)-module.  In Grothendieck categories the set of subobjects always
    forms a complete modular lattice.
\end{proofof}
Let us note that the lattice of subcomodules in general is not atomic.  An atom
of a lattice \(L\), with the smallest element \(0\),  is an element \(l>0\)
such that if \(l'\in L\) is such that \(l'<l\) then \(l'=0\).  A lattice is
called atomic if every element is a supremum  of a subset of the set of atoms.
Atoms of the lattice of subcomodules are exactly the simple subcomodules.  Thus
the lattice of subcomodules is atomic if and only if \(M\) is semisimple.  For
example let us consider \(C\) as a right \(C\)-comodule.  Then 
\[C\cong\bigoplus_{\substack{N\subseteq C\\N\text{-simple subcomodule}}}E(N)\] 
where \(E(N)\) is the injective envelope of \(N\) (we refer
to~\cite[Thm.~2.4.16]{sd-cn-sr:hopf-alg} for injective envelopes in the
categories of \(C\)-comodules). Since in general \(N\subsetneq E(N)\) and
\(N\) is the unique simple subcomodule which is contained in \(E(N)\) we see
that \(E(N)\) cannot be a sum of simple subcomodules. This shows that the
lattice of right coideals of \(C\) is atomic if and only if every simple right
coideal of \(C\) is an injective \(C\)-comodule.

There is another case in which we can say something about the lattice of
subcomodules. For this we need some module theoretic notions.
\begin{definition}
    \begin{itemize}[noitemsep]
	\item Let \(M,N\) be \(\R\)-modules. We say that \(N\) is
	    \bold{generated} by \(M\) if there is an epimorphism
	    \(\oplus_{\lambda\in\Lambda}M\sir N\) for some set \(\Lambda\).
	\item We say that \(N\) is \bold{subgenerated} by \(M\) if it is
	    isomorphic to a submodule of an \(M\)-generated module.
    \end{itemize}
    The full subcategory of \(\R\)-modules which are subgenerated by \(M\) we
    denote by~\(\sigma_\R\bigl[M\bigr]\).
\end{definition}
\index{locally projective module}
\begin{definition}[{\cite[sec.~4.6]{rw:coalgebras-and-bialgebras}}]\label{defi:locally_projective}
    Let \(M\) be an \(\R\)-module. It is a \bold{locally projective} if for any
    diagram with exact rows of the form:
    \begin{center}
	\begin{tikzpicture}
	    \matrix[column sep=1.2cm,row sep=1.2cm]{
		\node (A0) {\(0\)}; & \node (A1) {\(F\)}; & \node (A2) {\(M\)}; & \\
		                    & \node (B1) {\(L\)}; & \node (B2) {\(N\)}; & \node (B3) {\(0\)}; \\
	    };
	    \draw[->] (A0) -- (A1);
	    \draw[->] (A1) --node[above]{\(i\)} (A2);
	    \draw[->] (B1) --node[below]{\(g\)} (B2);
	    \draw[->] (B2) -- (B3);
	    \draw[->] (A2) --node[right]{\(f\)} (B2);
	    \draw[->,dashed] (A2) --node[fill=white]{\(\exists h\)} (B1);
	\end{tikzpicture}
    \end{center}
    where \(F\) is a finitely generated module, there exists \(h:M\sir L\),
    such that \(g\circ h\circ i = f\circ i\).
\end{definition}
If \(M\) is a right \(C\)-comodule, then it becomes a left \(C^*\)-module with
the action \(f\cdot m=\bigl((\id\otimes f)\circ
    \delta\bigr)(m)=m_{(0)}f(m_{(1)})\). It turns out that this construction
is functorial and in some cases its image in the category of \(C^*\)-modules
is the \(\sigma_{C^*}\bigl[C\bigr]\) subcategory. 
\begin{theorem}[{\cite[sec.~8.3]{rw:coalgebras-and-bialgebras}}]\label{thm:Wisbauer_category_and_comodules}
    Let \(C\) be an \(\R\)-coalgebra. Then the following conditions are
    equivalent:
    \begin{enumerate}
	\item \(\Mod^C\cong\sigma_{C^*}\bigl[C\bigr]\);
	\item \(\Mod^C\) is a full subcategory of \(_{C^*}\Mod\);
	\item \(C\) is locally projective as a left \(\R\)-module;
	\item every left \(C^*\)-submodule of \(C\) is a \(C\)-subcomodule.
    \end{enumerate}
\end{theorem}
\index{comodule!poset of subcomodules!algebraic}
\begin{corollary}
    Let \(C\) be a coalgebra over a ring \(\R\), such that \(C\) is locally
    projective left \(\R\)-module and let \(M\) be a (right) \(C\)-comodule.
    Then we have an isomorphism of posets
    \(\Sub_{\Mod^C}(M)\cong\Sub_{_{C^*}\Mod}(M)\) and thus
    \(\Sub_{\Mod^C}(M)\) is an algebraic lattice.
\end{corollary}
\begin{proof}
    Since the category \(\sigma_{C^*}\bigl[C]\) is closed under subobjects we
    have that the poset of \(C^*\)-submodules of \(M\) is equal to the poset
    of \(C^*\)-subgenerated submodules of \(M\). Now the corollary follows
    from the above theorem.
\end{proof}

\section{Bialgebras and Hopf Algebras}\label{sec:lattices-bi_and_Hopf_algebras}
\index{bialgebra!biideal}
\index{biideal|see{bialgebra!biideal}}
\index{bialgebra!poset of biideals}
\index{poset of biideals|see{bialgebra!poset of biideals}}
\index{bialgebra!subbialgebra}
\index{subbialgebra|see{bialgebra!subbialgebra}}
\index{bialgebra!poset of subbialgebras}
\index{poset of subbialgebras|see{bialgebra!poset of subbialgebras}}
\index{subbialgebra|see{bialgebra!subbialgebra}}
\index{Hopf algebra!Hopf subalgebra}
\index{subHopf algebra|see{Hopf algebra!Hopf subalgebra}}
\index{Hopf subalgebra|see{Hopf algebra!Hopf subalgebra}}
\index{Hopf algebra!poset of Hopf subalgebras}
\index{poset of Hopf subalgebras|see{Hopf algebra!poset of Hopf subalgebras}}
\begin{definition}
    Let \(B\) be a bialgebra and let \(H\) be a Hopf algebra, both over a commutative ring~\(\R\).
    Then:
    \begin{enumerate}
	\item a \bold{biideal} of \(B\) is a kernel of a surjective morphism
	    of bialgebras with domain~\(B\); the poset of biideals (under
	    inclusion) we will denote by \(\Id_\mathit{bi}(B)\); 
	\item a \bold{subbialgebra} of \(B\) is a subalgebra and
	    a subcoalgebra of \(B\); the poset of subbialgebras (under
	    inclusion) we will denote by \(\Sub_{\mathit{bi}}(B)\);
	\item a \bold{Hopf ideal} of a Hopf algebra~\(H\) is a kernel of
	    a surjective morphism of Hopf algebras with domain~\(H\); the
	    poset of Hopf ideals (under inclusion) we will denote by
	    \(\Id_{\mathit{Hopf}}(H)\); 
	\item a \bold{subHopf algebra} (or \emph{Hopf subalgebra}) \(K\) of
	    a Hopf algebra~\(H\) is a subbialgebra such that \(S_H(K)\subseteq
		K\), where \(S_H\) is the antipode of \(H\); the poset of
	    subHopf algebras (under inclusion) we will denote by
	    \(\Sub_{\mathit{Hopf}}(H)\);.
    \end{enumerate}
\end{definition}
Let \(A\) be a \(\k\)-algebra and \(B\subseteq A\) a subset of \(A\).  Then by
\(\langle B\rangle\) we will denote the smallest subalgebra of \(A\) which
contains \(B\).  Let us note that if \(A\) is a bialgebra, then \(\langle
B\rangle\) is a subbialgebra.  If \(A\) is a Hopf algebra then \(\langle
B\cup S(B)\rangle\) is the smallest Hopf subalgebra of \(A\) which contains \(B\).
\begin{proposition}\label{prop:bi_and_Hopf_algebra_lattices}
    Let \(H\) be a Hopf algebra and \(B\) a bialgebra over a ring \(\R\).
    \begin{enumerate}
	\index{bialgebra!poset of biideals!complete}
	\item\label{prop:biideals_complete} The poset of biideals of \(B\) is
	    a complete lattice. The join and meet are given by
	    \[I\vee J\coloneq I+J,\quad I\wedge J\coloneq\sum_{\substack{K\in\Id_\mathit{bi}(B)\\ K\subseteq I\cap J}}K.\]
	\index{Hopf algebra!poset of Hopf ideals!complete}
	\item\label{prop:Hopf_ideals_complete} The poset of Hopf ideals of \(H\) form
	    a complete lattice with operations: 
	    \[I\vee J\coloneq I+J,\quad I\wedge J\coloneq\sum_{\substack{K\in\Id_\mathit{Hopf}(H)\\ K\subseteq I\cap J}}K.\]
    \end{enumerate}
    Furthermore, if \(\R\) is a field then 
    \begin{enumerate}
	\index{bialgebra!poset of subbialgebras!algebraic}
	\item\label{prop:subbialg_algebraic} the poset of subbialgebras is an algebraic lattice, where meet and
	    join have the form: 
	    \[B_1\vee B_2\coloneq \langle B_1+B_2\rangle,\quad B_1\wedge B_2\coloneq B_1\cap B_2\]
	    for \(B_i\in\Sub_\mathit{bi}(B)\) (\(i=1,2\));
	\index{Hopf algebra!poset of Hopf subalgebras!algebraic}
	\item\label{prop:subHopfalg_algebraic} the poset of subHopf algebras is an algebraic lattice with
	    operations:
	    \[H_1\vee H_2\coloneq \langle H_1+H_2\rangle,\quad H_1\wedge H_2\coloneq H_1\cap H_2\]
	    for \(H_i\in\Sub_\mathit{bi}(H)\) (\(i=1,2\)).
    \end{enumerate}
\end{proposition}
\begin{proof}
    \begin{enumerate}
	\item A sum of biideals (Hopf ideals) is a biideal (Hopf
	    ideal, respectively) and thus it is their join in
	    \(\Id_\mathit{bi}(B)\) (\(\Id_\mathit{Hopf}(H)\)). Furthermore,
	    both \(\Id_\mathit{bi}(B)\) and \(\Id_\mathit{Hopf}(H)\) are
	    closed under arbitrary joins (sums of \(\R\)-submodules) and hence
	    they are complete lattices. The formulas for the infimum follow
	    from Remark~\ref{rem:complete_lattice}.
    \item The last two statements follow from
	Theorem~\ref{thm:algebraic_structures}, since subbialgebras and
	subHopf algebras (over a field) are closed under intersections and
	directed sums (see Theorem~\ref{thm:subcoalgebras_algebraic}).
    \end{enumerate}
\end{proof}
Let us note that the lattices of subbialgebras (subHopf algebras) are
sublattices of the lattice of subspaces of the underlying vector space. Thus,
by Remark~\ref{rem:sublattices_of_sub_vect}, these lattices are also dually
algebraic if the bialgebra (Hopf algebra) is finite dimensional.
\index{Hopf algebra!normal homomorphism}
\index{Hopf algebra!normal subalgebra}
\index{Hopf algebra!normal Hopf ideal}
\index{Hopf algebra!conormal quotient}
\begin{definition}
    Let \(f:K\sir H\) be a Hopf algebra homomorphism.  Then:
    \begin{enumerate}
	\item \(f\) is called \bold{normal} if for all \(k\in K\) and
	    \(h\in H\) we have:
	    \[h_{(1)}f(k)S(h_{(2)})\in f(K)\text{ and } S(h_{(1)})f(k)h_{(2)}\in f(K)\]
	\item \(f\) is called \bold{conormal} if for all \(k\in \ker\,f\) we
	    have:
	    \[k_{(2)}\otimes S(k_{(1)})k_{(3)}\in\ker\,f\otimes K\text{ and }k_{(2)}\otimes k_{(1)}S(k_{(3)})\in\ker\,f\otimes K\]
    \end{enumerate}
    A Hopf subalgebra \(K\subseteq H\) is called \bold{normal} if the
    inclusion map is normal.  A Hopf ideal \(I\subseteq H\) is called
    \bold{normal} if the quotient map \(H\sir H/I\) is conormal.
\end{definition}
Using Proposition~\ref{prop:bi_and_Hopf_algebra_lattices} we obtain the
following:
\index{Hopf algebra!poset of normal subalgebras!algebraic}
\begin{corollary}
    Let \(H\) be a Hopf algebra over a field \(\k\).  Then the poset of normal
    Hopf subalgebras is an algebraic lattice.  The poset of normal Hopf ideals
    is a complete lattice (for this it is enough that \(\k\) is a commutative
    ring).
\end{corollary}
\begin{proof}
    The lattice of normal Hopf subalgebras is a lower subsemilattice of the
    lattice of Hopf subalgebras (which is closed under infinite meets).
    Furthermore, the lattice of normal Hopf ideals is an upper subsemilattice
    of the complete lattice of Hopf ideals (which is closed under infinite
    joins). Thus both lattices of normal subalgebras/ideals are complete.
    A directed sum of normal Hopf subalgebras is a normal Hopf subalgebra and
    also an intersection of two normal Hopf subalgebras is a normal Hopf
    subalgebra (see Theorem~\ref{thm:subcoalgebras_algebraic} on
    page~\pageref{thm:subcoalgebras_algebraic}). In this way normal subHopf
    algebras  form a \(\cap\overrightarrow{\cup}\)-structure.  The poset of
    subHopf algebras is algebraic by Theorem~\ref{thm:algebraic_structures} .

\end{proof}

\index{bialgebra!generalised quotient}
\index{bialgebra!poset of generalised quotients}
\index{generalised quotient of a bialgebra|see{bialgebra!generalised
	quotient}}
\index{poset of generalised quotients of a bialgebra|see{bialgebra!poset of generalised
	quotients}}
\index{bialgebra!generalised subalgebra}
\index{bialgebra!poset of generalised subalgebras}
\index{generalised subalgebra of a bialgebra|see{bialgebra!generalised subalgebra}}
\index{poset of generalised subalgebras of a bialgebra|see{bialgebra!poset of generalised subalgebras}}
\begin{definition}\label{defi:generalised_quotients_and_subalgebras}
    \begin{enumerate}
	\item A \bold{generalised quotient} \(Q\) of a bialgebra \(B\) is
	    a quotient by a right ideal coideal.  The {poset of generalised
		quotients} will be denoted by \(\qquot(B)\).  The order
	    relation of \(\qquot(B)\) we will denote by \(\succcurlyeq\):
	    \[B/I\succcurlyeq B/J\Leftrightarrow I\subseteq J\]
	    for \(B/I,B/J\in\qquot(B)\).
	\item A \bold{generalised subalgebra}~\(K\) of a bialgebra~\(B\)
	    is a left coideal subalgebra.  The poset of generalised subalgebras
	    will be denoted by~\(\qsub(B)\).
    \end{enumerate}
\end{definition}
The poset \(\qquot(B)\) is dually isomorphic to the poset of right ideals
coideals of \(B\), which will be denoted as \(\qid(B)\).  We define only the
right version of generalised quotients and the left version of generalised
subalgebras since we will consider right \(B\)-comodules (right \(B\)-comodule
algebras).
\index{bialgebra!poset of generalised quotients!complete}
\begin{proposition}\label{prop:qquot_complete}
    Let \(B\) be a bialgebra over a ring \(\R\).  Then the poset \(\qquot(B)\)
    is a complete lattice. 
\end{proposition}
\begin{proof}
    There is a canonical isomorphism of posets \(\qquot(B)\simeq\qid(B)^\op\).
    The supremum in \(\qid(B)\) is given by the sum of submodules, while the
    infimum is given by the formula:
    \begin{equation}\label{eq:meet_in_qid}
	I\wedge J=\mathop\sum\limits_{\substack{K\subseteq I\cap J\\ K\in\qid(B)}}K
    \end{equation}
    where \(I,J\in\qid(B)\). 
\end{proof}
\index{bialgebra!poset of generalised subbialgebras!algebraic}
\begin{proposition}\label{prop:qsub_algebraic}
    Let \(B\) be a bialgebra over a ring \(\R\) such that \(B\) is a flat
    Mittag--Leffler \(R\)-module.  Then the poset \(\qsub(B)\) is an
    algebraic lattice.
\end{proposition}
\begin{proof}
    The poset \(\qsub(B)\) is closed under (set theoretic) intersections (see
    the proof of Proposition~\ref{thm:lattice_of_subcomodules}), thus it is
    a complete lattice. It is also closed under directed sums and thus it is
    a \(\cap\overrightarrow{\cup}\)-structure and by
    Theorem~\ref{thm:algebraic_structures} it is an algebraic lattice.
\end{proof}

\subsection{Quotients of $k[G]$ and its dual Hopf algebra}
Before the next theorem we need the following 
\index{G-set}
\index{G-set!transitive}
\index{G-set!morphism}
\begin{definition}
    Let \(G\) be a group. A left (right) \(G\)\bold{-set} \(S\) is a set
    together with a group homomorphism (anti-homomorphism respectively)
    \(\phi:G\mpr{}\mathsf{Bij}(S)\), where \(\mathsf{Bij}(S)\) is the group of
    bijections of \(S\).  We will write \(gs=g(s)\coloneq\phi(g)(s)\) for left
    \(G\)-sets, and \(sg\coloneq\phi(g)(s)\) for right \(G\)-sets.  
    
    A \(G\)-set \(S\) is called \bold{transitive} if for any two \(s,s'\in S\)
    there exists \(g\in G\) such that \(gs=s'\) (\(sg=s'\) respectively).

    A morphism of \(G\)-sets \(f:S\sir S'\) is a map  of sets such that
    \(f(gs)=gf(s)\) (\(f(sg)=f(s)g\) respectively) for every \(s\in S\).
\end{definition}
The next two propositions describe the poset \(\qquot(H)\) for \(\k[G]\) and
its dual \(k[G]^*\) (if \(G\) is a finite group). It turns out that they are
isomorphic to the poset of isomorphism classes of transitive \(G\)-sets (which
is anti-isomorphic to the poset of subgroups of \(G\)), while \(\Quot(H)\),
the poset of Hopf algebra quotients, is isomorphic to the poset of quotient
subgroups of \(G\).  In the Hopf--Galois theory considered in
Chapter~\ref{chap:Hopf-Galois_theory} we consider \(\qquot(H)\) rather than
\(\Quot(H)\). This parallels the use of subgroups of the Galois group in the
classical Galois theory (or Grothendieck-Galois theory).
\index{group algebra!generalised quotients}
\begin{proposition}\label{prop:group_algebra_quotients}
    Let \(H=\k[G]\) where \(G\) is a group. Then
    \begin{enumerate}
	\item \(\Quot(\k[G])\cong\Quot(G)\), where \(\Quot(G)\) is the poset
	    of quotient groups of \(G\) (which is anti-isomorphic to the poset
	    of normal subgroups);
	\item \(\qquot(\k[G])\cong\Quot_{G\textit{-set}}(G)\), where
	    \(\Quot_{G\textit{-set}}(G)\) is the poset of quotient \(G\)-sets
	    of the free \(G\)-set \(G\) (which is anti-isomorphic to the
	    poset of all subgroups of \(G\)).
    \end{enumerate}
\end{proposition}
Let us note that a transitive \(G\)-set \(S\) is of the form \(G/G_0\) where
\(G_0\) is a subgroup of \(G\), with the action induced by multiplication from
the left for left \(G\)-sets and from the right for right \(G\)-sets.

We fix a notation. If \(C\) is a coalgebra, then a nonzero element \(c\in C\)
is called \bold{group-like} if \(\Delta(c)=c\otimes c\). The set  of group-like elements we denote by \(\mathbf{G}(C)\). If the base ring is a field,
then the group-like elements are linearly independent. This fails for rings
with idempotents: for example if the base ring \(\R\) contains an idempotent
\(p\) then if \(c\in C\) is a group-like element then \(pc\) is a group-like
element as well. Furthermore, if \(B\) is a bialgebra over a field, then
\(\mathbf{G}(B)\) is a monoid with unit \(1_B\in B\), since a product of two
group-like element is group-like. If \(H\) is a Hopf algebra and \(g\in H\) is
a group-like element, then \(S(g)\) is a group-like element, and moreover it
is an inverse  of \(g\) in \(\mathbf{G}(H)\). In other words, \(\mathbf{G}(H)\)
is a group. It turns out that this gives rise to a functor from the category
of bialgebras (Hopf algebras) to the category of monoids (groups). In the case
of coalgebras it gives a functor from coalgebras to sets,
adjoint to the functor \(G\selmap{}\k[G]\).
\begin{proofof}{Proposition~\ref{prop:group_algebra_quotients}}
    \begin{enumerate}
	\item First let us note that if \(N\) is a normal subgroup,
	    then \(\k[G/N]\) is a quotient Hopf algebra of the Hopf algebra
	    \(\k[G]\) by the map induced by the projection \(G\sir G/N\). Now,
	    let \(\pi:\k[G]\sir\k[G]/I\) be a Hopf algebra projection, i.e.
	    let \(I\) be a Hopf ideal. Then there exists a set \(G_0\subseteq
		G\) such that \(G'\coloneq\{\pi(g):g\in G_0\}\) is a basis of
	    \(\k[G]/I\). It follows, that \(\mathbf{G}(\k[G]/I)=G'\) is
	    a quotient group of \(G=\mathbf{G}(\k[G])\) via the map
	    \(g\selmap{}\pi(g)\).  Furthermore, since \(\k[G]/I\) has a basis
	    of group-like elements it is a group algebra, and \(\k[G]/I\cong
		\k[G']\). Note that \(G'\) is a quotient group of \(G\). This
	    shows that the following two maps are inverses of each other:
	    \begin{center}
		\begin{tikzpicture}
		    \matrix[matrix of nodes, column sep=2cm, row sep=1cm]{
			|(A1)| \(G\)  & |(A2)| \(\k[G]\) \\
			|(B1)| \(G'\) & |(B2)| \(\k[G']\) \\
		    };
		    \draw[->>] (A1) --node[left]{\(\Quot(G)\ni p\)} (B1);
		    \draw[->>] (A2) --node[right]{\(\k[p]\in\Quot(k[G])\)} (B2);
		    \draw[|->] ($(A1)+(0.3cm,-0.8cm)$) -- +(2.3cm,0);
		\end{tikzpicture}
	    \end{center}
	    \begin{center}
		\begin{tikzpicture}
		    \matrix[matrix of nodes, column sep=2cm, row sep=1cm]{
			|(A1)| \(\k[G]\)  & |(A2)| \(G\) \\
			|(B1)| \(\k[G]/I\) & |(B2)| \(G(\k[G]/I)\) \\
		    };
		    \draw[->>] (A1) --node[left]{\(\Quot(k[G])\ni\pi\)} (B1);
		    \draw[->>] (A2) --node[right]{\(G(\pi)\in\Quot(G)\)} (B2);
		    \draw[|->] ($(A1)+(0.3cm,-0.85cm)$) -- +(3cm,0);
		\end{tikzpicture}
	    \end{center}

	\item Let \(\k[G]/I\) be a quotient of \(\k[G]\) by a coideal
	    right ideal \(I\). The set of group-like elements
	    \(\mathbf{G}\bigl(\k[G]/I\bigr)\) of the coalgebra \(\k[G]/I\) is
	    a \(\mathbf{G}\bigl(\k[G]\bigr)\)-set and hence a \(G\)-set,
	    because \(\mathbf{G}\bigl(\k[G]\bigr)=G\). Since the map
	    \(\k[G]\epr{}\k[G]/I\) is an epimorphism it follows that:
	    \begin{enumerate}
		\item \(\k[G]/I\) is spanned by group-like elements (as
		    a \(\k\)-vector space),
		\item \(\mathbf{G}\bigl(\k[G]/I\bigr)\) is
		    a transitive (right) \(G\)-set.
	    \end{enumerate}
	    On the other hand, if \(S\) is a (right) transitive \(G\)-set then
	    \(\k[S]\) is a right \(\k[G]\)-module through the right
	    \(G\)-action on \(S\) and a coalgebra quotient only if we set each
	    \(s\in S\) to be a group-like element, since \(S\cong G/G_0\),
	    where \(G_0\) is a subgroup of \(G\). Now, let us observe that
	    these two constructions: 
	    \[\qid(\k[G])\ni I\selmap{}\mathbf{G}(\k[G]/I)\in\Quot_{G\textit{-set}}(G)\]
	    and
	    \[\Quot_{G\textit{-set}}(G)\ni S\selmap{}\ker\left(\k[G]\epr{}\k[S]\right)\in\qid(\k[G])\]
	    are inverse to each other, and thus the claim follows. First, let
	    us show that
	    \[I=\ker\left(\k[G]\epr{}\k\bigl[\mathbf{G}\bigl(\k[G]/I\bigr)\bigr]\right).\]
	    This follows since we have a commutative diagram:
		\begin{center}
		\begin{tikzpicture}
		    \matrix[column sep=5mm, row sep=1cm]{
						     & \node (A) {\(\k[G]\)}; & \\
	     \node (B1) {\(\k\bigl[\mathbf{G}\bigl(\k[G]/I\bigr)\bigr]\)}; &                       & \node (B2) {\(\k[G]/I\)}; \\
		    };
		    \draw[->>] (A) -- (B1);
		    \draw[->>] (A) -- (B2);
		    \draw[->] (B1) --node[above]{\(\cong\)}node[below]{\(\alpha\)} (B2);
		\end{tikzpicture}
	    \end{center}
	    where \(\alpha(x)=x\) for all \(x\in\mathbf{G}(\k[G]/I)\), and thus
	    it is a monomorphism. It is an epimorphism by (\textbf{G}1). The
	    remaining equality \(S=\mathbf{G}(\k[S])\) is straightforward and
	    it shows that
	    \(S=\mathbf{G}\bigl(\k[G]/\ker\bigl(\k[G]\eir\k[S]\bigr)\bigr)\) 
    \end{enumerate}
\end{proofof}
\begin{proposition}\label{prop:dual_gruop_algebra_quotients}
    Let \(G\) be a finite group. Then the following map
    \begin{equation}\label{eq:qquot_affine_cordinate_ring_of_G}
	\Quot_{G\textit{-Set}}(G)\ni G/G_0\selmap{}\k[G_0]^*\in\qquot(\k[G]^*)
    \end{equation}
    is an anti-isomorphism of posets.
\end{proposition}
Let us note that if \(G\) is finite then \(\qquot(\k[G])\cong\qsub(\k[G])\),
by Theorem~\ref{itm:newTakeuchi_finite}.  In a consequence we have
\(\qquot(\k[G])\cong\qquot(\k[G]^*)\) by
Propositions~\ref{prop:right_coideals-right_ideals}
and~\ref{prop:coideals-subalgebras}. Hence the above proposition follows from
the previous result. However, we present a direct construction.
\begin{proof}
    Let us choose a basis of $\k[G]^*$ consisting of $\delta_g$ given by
    $\delta_g(h)=\delta_{g,h}$, where $\delta_{g,h}$ is the Kronecker symbol
    given by \(\delta_{g,k}=\left\{
	\begin{smallmatrix}
	    1 &\text{iff }g=k\hfill\\ 0 &\text{otherwise}\hfill
	\end{smallmatrix}\right.\). First let us observe that any right ideal
    \(I\) of $\k[G]^*$ must be generated by some subset of this basis. We have
    the equality $\delta_g\cdot\delta_h=\delta_{g,h}\delta_g$ for all \(g,h\in
	G\). Let \(\sum_{l=1}^n\lambda_l\delta_{g_l}\in I\) for some
    coefficients \(\lambda_l\in \k\) and some \(g_l\in G\) (\(l=1,\dots,n\)).
    Then 
    \[\Big( \sum_{l=1}^{n}\lambda_l\delta_{g_l}\Big)\cdot\delta_{g_i}=\lambda_i\delta_{g_i}\]
     and thus $\delta_{g_i}\in I$ for all \(1\leq i\leq n\) such that
     $\lambda_i\neq 0$. A right ideal \(I\) of \(\k[G]^*\) is a coideal if and
     only if the set \mbox{$M_I:=\{g:\delta_g\notin I\}$} is a submonoid of
     $G$, i.e. it is closed under multiplication and contains the unit of $G$
     (since \(G\) is finite it is a subgroup of \(G\)). This is because $I$ is
     a coideal of $\k[G]^*$ if and only if $I^\perp$ is a subalgebra of
     $\k[G]$ (by Proposition~\ref{prop:coideals-subalgebras}(iii)) and $M_I$
     is a basis of $I^\perp$.  On the other hand, a submonoid $M$ of $G$
     defines a right ideal coideal \(I_M\) of $\k[G]^*$.  The right ideal
     coideal $I_M$ is spanned by all the $\delta_g$ for $g\notin M$. We have
     \(I_M=\k[M]^\perp\), thus \(I_M\) is a coideal by
     Proposition~\ref{prop:coideals-subalgebras} and it is a right ideal by
     Proposition~\ref{prop:right_coideals-right_ideals}, since \(\k[M]\) is
     a right coideal subalgebra of \(\k[G]\). We have a bijective
     correspondence between $\qid(\k[G]^*)$ and $\Sub_{\textit{group}}(G)$,
     which is given by \(\qid(\k[G]^*)\ni
	 I\elmap{}M_I\in\Sub_{\textit{group}}(G)\) and
     \(\Sub_{\mathit{group}}(G)\ni G_0\selmap{}I_{G_0}\in\qid(\k[G]^*)\).
     Indeed, this is a pair of inverse bijections: for \(G_0\leq G\)
     a subgroup we have \(M_{I_{G_0}}=G_0\) since all \(\delta_g\), for \(g\in
	 G\), are linearly independent. Moreover, for \(I\in\qid(\k[G]^*)\),
     \(I_{M_I}=I\), since we showed that \(I\) is spanned by the elements
     \(\delta_g\) which form the basis of \(I_{M_I}\) (by definition of
     \(I_{M_I}\) its basis is \(\{\delta_g:g\notin
	     M_I\}=\{\delta_g:\delta_g\in I\}\)). Now, the claim follows:
     \[\Quot_{G\textit{-set}}(G)^\op\cong\Sub_{\mathit{group}}(G)\cong\qid(\k[G]^*)^\op\cong\qquot(\k[G]^*)\]
     This is indeed the map~\eqref{eq:qquot_affine_cordinate_ring_of_G}, since
     \(\k[G]^*/I_{G_0}\cong\k[G_0]^*\). 
\end{proof}

\subsection{Quotients of $\mathcal{U}(\mathfrak{g})$}

In this section we will consider the universal enveloping Lie algebra
\(\mathcal{U}(\mathfrak{g})\) where \(\mathfrak{g}\) is a finite dimensional
Lie algebra.  We will show how to construct all the generalised quotients and
we shall prove the Poincar\'{e}--Birkhoff--Witt theorem for them.  First let
us state the Poincar\'{e}--Birkhoff--Witt theorem for universal enveloping
algebras:
\index{Poincar\'{e}--Birkhoff--Witt theorem}
\begin{theorem}[Poincar\'{e}--Birkhoff--Witt]\label{thm:pbw_theorem}
    Let \(\mathfrak{g}\) be a finite dimensional \(\k\)-Lie algebra and
    \(\{X_\lambda:\lambda\in\Lambda\}\) a totally ordered basis of
    \(\mathfrak{g}\) indexed by a set \(\Lambda\). Let
    \(\mathcal{U}(\mathfrak{g})\) be the enveloping algebra of
    \(\mathfrak{g}\). Then the set 
    \[\{1\}\cup\bigl\{X_{\lambda_1}\cdot\ldots\cdot X_{\lambda_k}\in\mathcal{U}(\mathfrak{g}): X_{\lambda_1}\leq\dots\leq X_{\lambda_n},\;\lambda_1,\ldots,\lambda_n\in\Lambda,\;n\in\bN^+\bigr\}\] 
    is a \(\k\)-linear basis of \(\mathcal{U}(\mathfrak{g})\). 
\end{theorem}
\index{universal enveloping algebra!generalised quotient}
Now let us give a construction of a generalised quotient of
\(\mathcal{U}(\mathfrak{g})\).  Let \(\mathfrak{h}\) be a \(\k\)-Lie
subalgebra of \(\mathfrak{g}\) and let
\(\pi:\mathfrak{g}\eir\mathfrak{g}/\mathfrak{h}\) be the \(\k\)-linear
quotient map.  Choose a linear basis \(\mathsf{C}\) of the quotient space
\(\mathfrak{g}/\mathfrak{h}\).  Define
\(\mathcal{Q}(\mathfrak{g}/\mathfrak{h})\) to be the quotient of the free
\(\mathcal{U}(\mathfrak{g})\)-module \(\mathsf{F}(\k
    u\oplus\mathfrak{g}/\mathfrak{h})\) generated by the basis \(\mathsf{C}\)
by the following relation:
\begin{equation}\label{eq:ggenquot_relation}
    u\cdot X = \pi(X)
\end{equation}
for all \(X\in\mathfrak{g}\) such that \(\pi(X)\in\mathsf{C}\).  It then
follows that the above relation is satisfied for all \(X\in\mathfrak{g}\).
There exists a right \(\mathcal{U}(\mathfrak{g})\)-module homomorphism:
\[\mathcal{Q}(\pi):\mathcal{U}(\mathfrak{g})\sir\mathcal{Q}(\mathfrak{g}/\mathfrak{h})\]
which is uniquely determined by \(\mathcal{Q}(\pi)(X)=\pi(X)\) for
\(X\in\mathfrak{g}\), where \(\pi(X)\in\mathfrak{g}/\mathfrak{h}\) is treated
as an element of \(\mathcal{Q}(\mathfrak{g}/\mathfrak{h})\) and
\(\mathcal{Q}(\pi)(1)=u\).  The map \(\mathcal{Q}(\pi)\) is well defined by
the above Poincar\'{e}--Birkhoff--Witt Theorem and the following simple
computation:
\begin{align*}
    \mathcal{Q}(\pi)(XY-YX) & = \mathcal{Q}(\pi)(X)Y-\mathcal{Q}(\pi)(Y)X \\
                          & = \mathcal{Q}(\pi)(1)XY-\mathcal{Q}(\pi)(1)YX \\
                          & = \mathcal{Q}(\pi)(1)(XY-YX) \\
                          & = \mathcal{Q}(\pi)(1)([X,Y]) \\
			  & = \mathcal{Q}(\pi)([X,Y])
\end{align*}
\begin{note}\label{note:U(G)_gen_quotients}
    Note that the following relation is satisfied in
    \(\mathcal{Q}(\mathfrak{g}/\mathfrak{h})\) (which is a consequence
    of~\ref{eq:ggenquot_relation}):
    \begin{equation}\label{eq:ggenquot_consequence}
	\mathcal{Q}(\pi)(X)Y-\mathcal{Q}(\pi)(Y)X = \mathcal{Q}(\pi)([X,Y])
    \end{equation}
    for all \(X,Y\in\mathfrak{g}\). This relation holds for every generalised
    quotient of \(\mathcal{U}(\mathfrak{g})\). From this relation it follows that
    for any generalised quotient \(p:\mathcal{U}(\mathfrak{g})\sir Q\) the
    subspace \(\ker p\cap\mathfrak{g}\) is a \(\k\)-Lie subalgebra of
    \(\mathfrak{g}\).  This explains why we have assumed that
    \(\mathfrak{h}\subseteq\mathfrak{g}\) is a Lie subalgebra. 
\end{note}
Now we show, that
\(\mathcal{Q}(\mathfrak{g}/\mathfrak{h})\) is indeed a generalised quotient of
\(\mathcal{U}(\mathfrak{g})\). 
\index{universal enveloping algebra!generalised quotient!coalgebra structure}
\begin{lemma}
    Let \(\mathfrak{g}\) be a finite dimensional \(\k\)-Lie algebra,
    \(\mathfrak{h}\) its Lie subalgebra, and
    \(\pi:\mathfrak{g}\sir\mathfrak{g}/\mathfrak{h}\) the \(\k\)-linear
    quotient map. Then
    \(\mathcal{Q}(\pi):\mathcal{U}(\mathfrak{g})\sir\mathcal{Q}(\mathfrak{g}/\mathfrak{h})\)
    is a generalised quotient. The coalgebra structure of
    \(\mathcal{Q}(\mathfrak{g}/\mathfrak{h})\) is uniquely determined by:
    \begin{equation}\label{eq:coalgebra_U(h)}
            \Delta_{\mathcal{Q}(\mathfrak{g}/\mathfrak{h})}\bigl(\mathcal{Q}(\pi)(X)\bigr)=\mathcal{Q}(\pi)(X)\otimes u+u\otimes\mathcal{Q}(\pi)(X)
    \end{equation}
    for all \(X\in\mathfrak{g}\) and the requirement that
    \(u\in\mathcal{Q}(\mathfrak{h})\) is a group-like element. 
\end{lemma}
\begin{proof}        
    The map \(\mathcal{Q}(\pi)\) is \(\mathcal{U}(\mathfrak{g})\)-module
    surjection by the construction of
    \(\mathcal{Q}(\mathfrak{g}/\mathfrak{h})\).  We set the comultiplication
    of \(\mathcal{U}(\mathfrak{h})\) by:
    \[\Delta_{\mathcal{Q}(\mathfrak{g}/\mathfrak{h})}\bigl(\mathcal{Q}(\pi)(X_0\cdot\ldots\cdot X_k)\bigr)\coloneq\bigl(\mathcal{Q}(\pi)(X_0)\otimes u+u\otimes\mathcal{Q}(\pi)(X_0)\bigr)\cdot\bigl(\Delta_{\mathcal{U}(\mathfrak{g})}(X_1\cdot\ldots\cdot X_k)\bigr)\]
    and \(\Delta_{\mathcal{Q}(\mathfrak{g}/\mathfrak{h})}(u)=u\otimes u\). It
    is well defined since the relation~\ref{eq:ggenquot_relation} is
    preserved:
    \begin{align*}
	\Delta_{\mathcal{Q}(\mathfrak{g}/\mathfrak{h})}(u\cdot X) & \coloneq (u\otimes u)\cdot (X\otimes 1+1\otimes X)\\
	                                                 & = (u\cdot X)\otimes u+u\otimes(u\cdot X)\\
	                                                 & = \mathcal{Q}(\pi)(X)\otimes u+u\otimes\mathcal{Q}(\pi)(X)\\
	                                                 & \eqcolon\Delta_{\mathcal{Q}(\mathfrak{g}/\mathfrak{h})}\bigl(\mathcal{Q}(\pi)(X)\bigr)
    \end{align*}
    Clearly,
    \(\mathcal{Q}(\pi):\mathcal{U}(\mathfrak{g})\sir\mathcal{Q}(\mathfrak{g}/\mathfrak{h})\)
    is a coalgebra map. Formula~\eqref{eq:coalgebra_U(h)} uniquely determines
    the coalgebra structure of
    \(\mathcal{Q}(\mathfrak{g})/\mathfrak{h})\) since for any Hopf
    algebra \(H\) and its generalised quotient \(\pi:H\sir Q\) we have:
    \begin{align*}
	\Delta_{Q}(\pi(h)k) & =\Delta_{Q}(\pi(hk)) \\
					     & = \pi\otimes\pi\circ\Delta_{H}(hk) \\
					     & = \pi\otimes\pi\circ\bigl(\Delta_{H}(h)\cdot \Delta_{H}(k)\bigr) \\
					     & = (\pi\otimes\pi\circ\Delta_{H}(h))\cdot \Delta_{H}(k) \\
					     & = \Delta_{Q}(h)\cdot\Delta_{H}(k)
    \end{align*}
    for any \(h,k\in H\). 
\end{proof}

\index{generalised Poincar\'{e}--Birkhoff--Witt theorem}
\index{Poincar\'{e}--Birkhoff--Witt theorem (generalised)}
\index{universal enveloping algebra!generalised quotient!Poincar\'{e}--Birkhoff--Witt basis}
\begin{theorem}[generalised Poincar\'{e}--Birkhoff--Witt]\label{thm:gen_pbw_theorem}
    Let \(\mathfrak{h}\subseteq\mathfrak{g}\) and
    \(\pi:\mathfrak{g}\sir\mathfrak{g}/\mathfrak{h}\) be as above (in
    particular \(\mathfrak{g}\) is finite dimensional).  Let \(\mathsf{B}\) be
    a totally ordered basis of \(\mathfrak{g}\). Then the set
    \begin{equation}\label{eq:gen_pbw_basis}
	\left\{\begin{array}{rl}
	\pi(Z_0)\cdot Z_1\cdot\ldots\cdot Z_n\in\mathcal{Q}(\mathfrak{g}/\mathfrak{h}):& Z_0,\ldots, Z_n\in\mathsf{B} \\
	                                              & Z_0\leq\dots\leq Z_n,\;n\in\bN \\
						      & \pi(Z_0)\cdot Z_1\cdot\ldots\cdot Z_n\neq 0
	\end{array}\right\}
    \end{equation} 
    together with the single element \(u\) forms a \(\k\)-linear basis of
    \(\mathcal{Q}(\mathfrak{g}/\mathfrak{h})\). 
\end{theorem}
\begin{proof}
    Let us denote the free \(\mathcal{U}(\mathfrak{g})\)-module on \(\k
	u\oplus\mathfrak{g}/\mathfrak{h}\) by \(\mathsf{F}\). By definition of
    \(\mathcal{U}(\mathfrak{g}/\mathfrak{h})\) we have an epimorphism
    \(p:\mathsf{F}\sir\mathcal{U}(\mathfrak{g}/\mathfrak{h})\). Since
    \(\mathsf{F}\) is just a sum of copies of \(\mathcal{U}(\mathfrak{g})\),
    for which Theorem~\ref{thm:pbw_theorem} holds, we can construct a map:
    \(L:\mathsf{F}\sir\mathsf{F}\) such that
    \begin{align*}
            L(\pi(Z_0)\cdot Z_1\cdot\ldots\cdot Z_n)    & =\pi(Z_0)\cdot Z_1\cdot\ldots\cdot Z_n \\
            L(u\cdot Z_0\cdot\ldots\cdot Z_n) & =\pi(Z_0)\cdot Z_1\cdot\ldots\cdot Z_n
    \end{align*}
    where \(Z_0,\ldots,Z_n\in\mathsf{B}\) and satisfy the condition
    \(Z_0\leq\dots\leq Z_n\), and \(L(u)=u\). The kernel of
    \(p:\mathsf{F}\sir\mathcal{Q}(\mathfrak{g}/\mathfrak{h})\) is generated as
    a \(\mathcal{U}(\mathfrak{g})\)-submodule by the elements of the form:
    \(u\cdot X-\pi(X)\) for \(X\in\mathsf{B}\) and hence it is spanned as
    a \(\k\)-vector space by the elements of the form: \(u\cdot
	Z_0\cdot\ldots\cdot Z_n-\pi(Z_0)\cdot Z_1\cdot\ldots\cdot Z_n\) where
    \(Z_0,\ldots,Z_n\in\mathsf{B}\) and \(Z_0\leq\cdots\leq Z_n\). By
    definition of \(L\) we get
    \(L|_{\ker(\mathsf{F}\sir\mathcal{U}(\mathfrak{h}))}=0\).  This proves
    that the elements~\eqref{eq:gen_pbw_basis} are linearly independent.
    Suppose on the contrary that there exists a linear relation:
    \[\sum_{i=1}^N\lambda_i p\bigl(\pi(Z_0^i)\cdot Z_1^i\cdot\ldots\cdot Z_{n_i}^i\bigr)=0\]
    where all the summands satisfy the conditions in~\ref{eq:gen_pbw_basis}.
    Then \(\sum_{i}\lambda_i\pi(Z_0^i)\cdot Z_1^i\cdot\ldots\cdot
	Z_{n_i}^i\in\ker p\) and hence \(\sum_{i}\lambda_i\pi(Z_0^i)\cdot
	Z_1^i\cdot\ldots\cdot
	Z_{n_i}^i=L\bigl(\sum_{i}\lambda_i\pi(Z_0^i)\cdot
	Z_1^i\cdot\ldots\cdot Z_{n_i}^i\bigr)=0\) in \(\mathsf{F}\).  But the
    set:
    \begin{equation*}
	\left\{\begin{array}{rl}
		\pi(Z_0)\cdot Z_1\cdot\ldots\cdot Z_n\in\mathsf{F}: & Z_0,\ldots, Z_n\in\mathsf{B} \\
		                               & Z_1\leq\dots\leq Z_n,\;n\in\bN
	\end{array}\right\}
    \end{equation*} 
    is linearly independent by Theorem~\ref{thm:pbw_theorem} and hence
    \(\lambda_i=0\) for \(i=1,\ldots,N\).  The claim that the
    set~\eqref{eq:gen_pbw_basis} and \(u\) span
    \(\mathcal{Q}(\mathfrak{g}/\mathfrak{h})\) is straightforward. 
\end{proof}
\begin{corollary}
    Let \(\mathfrak{g}\sir\mathfrak{g}/\mathfrak{h}\) be as above.  Then the
    natural map
    \(\mathfrak{g}/\mathfrak{h}\sir\mathcal{Q}(\mathfrak{g}/\mathfrak{h})\) is
    injective.
\end{corollary}
\index{universal enveloping algebra!generalised quotient!pointed}
\index{universal enveloping algebra!generalised quotient!irreducible}
\begin{corollary}\label{cor:Q_irr_pointed}
    Let \(\mathfrak{g}\sir\mathfrak{g}/\mathfrak{h}\) be as above.  Then
    \(\mathcal{Q}(\mathfrak{g}/\mathfrak{h})\) is an irreducible pointed
    coalgebra. 
\end{corollary}
\begin{proof}
    The coalgebra \(\mathcal{Q}(\mathfrak{g}/\mathfrak{h})\) is a quotient of
    the (irreducible) pointed coalgebra \(\mathcal{U}(\mathfrak{g})\).  Hence
    it is pointed.  Furthermore, it is irreducible, since it has only one
    group-like element \(\mathcal{Q}(\pi)(1)\).  Let us assume that the
    following element is group-like:
    \[x=\lambda_uu+\sum_{b\in\mathsf{B}}\lambda_bb\]
    where \(\mathsf{B}\) is the set~\eqref{eq:gen_pbw_basis}, only finitely
    many \(\lambda_b\in\k\) are nonzero, and \(\lambda_u\in\k\).  All the
    elements of \(\mathsf{B}\) belong to
    \(\ker\,\epsilon_{\mathcal{Q}(\mathfrak{g}/\mathfrak{h})}\).  We get
    \(x=\lambda_u u\), by computing
    \(\epsilon\otimes\id\circ\Delta_{\mathcal{Q}(\mathfrak{g}/\mathfrak{h})}(x)\).
    We must have \(x=u\) since distinct group-like elements are linearly
    independent.
\end{proof}
Now we are going to show that if \(\mathfrak{h}\) is a Lie ideal then the
enveloping algebras \(\mathcal{U}(\mathfrak{g}/\mathfrak{h})\) and
\(\mathcal{Q}(\mathfrak{g}/\mathfrak{h})\) coincide, i.e. are isomorphic as
generalised quotients of \(\mathcal{U}(\mathfrak{g})\).
\index{universal enveloping algebra!generalised quotient!Hopf algebra quotient}
\begin{proposition}\label{prop:U_and_Q}
    Let \(\mathfrak{g}\) be a finite dimensional Lie algebra, \(\mathfrak{h}\)
    an ideal of \(\mathcal{g}\) and
    \(\pi:\mathfrak{g}\sir\mathfrak{g}/\mathfrak{h}\) the quotient Lie algebra
    homomorphism.  Then we have an isomorphism (of generalised quotients):
    \begin{center}
	{\hfill\begin{tikzpicture}
	    \matrix[matrix of nodes, column sep=1cm,row sep=1cm]{
						 & |(A)| \(\mathcal{U}(\mathfrak{g})\) & \\
|(C)| \(\mathcal{Q}(\mathfrak{g}/\mathfrak{h})\) &                                    & |(B)| \(\mathcal{U}(\mathfrak{g}/\mathfrak{h})\) \\
 	    };
	    \draw[->] (A) --node[above left]{\(\mathcal{Q}(\pi)\)} (C);
	    \draw[->] (A) --node[above right]{\(\mathcal{U}(\pi)\)} (B);
	    \draw[->] (C) --node[below]{\(\cong\)} (B);
	\end{tikzpicture}
	\hfill\refstepcounter{equation}\raisebox{12mm}{(\theequation)}\label{diag:gquot_of_ug}\\
	}
    \end{center}
    where \(\mathcal{U}(\pi)\) is the algebra homomorphism induced by the
    map~\(\pi\). 
\end{proposition}
\begin{proof}
    First we want to show that we have a natural epimorphism:
     \[\mathcal{Q}(\mathfrak{g}/\mathfrak{h})\sir\mathcal{U}(\mathfrak{g}/\mathfrak{h})\]
     For this we use the definition of
     \(\mathcal{Q}(\mathfrak{g}/\mathfrak{h})\) as a quotient of a free
     \(\mathcal{U}(\mathfrak{g})\)-module \(\mathsf{F}\) generated by \(\k
	 u\oplus\mathfrak{g}/\mathfrak{h}\). We have
     a \(\mathcal{U}(\mathfrak{g})\)-module homomorphism
     \(\mathsf{F}(\pi):\mathsf{F}\sir\mathcal{U}(\mathfrak{g}/\mathfrak{h})\) which sends
     a generator \(X\) of \(\mathsf{F}\), i.e.
     \(X\in\mathfrak{g}/\mathfrak{h}\), to its image in
     \(\mathcal{U}(\mathfrak{g}/\mathfrak{h})\), and
     \(\mathsf{F}(\pi)(u)=\mathcal{U}(\pi)(1)\). This map descends to the
     quotient \(\mathcal{Q}(\mathfrak{g}/\mathfrak{h})\). For this it is
     enough to check that:
     \begin{align*}
	 \mathsf{F}(\pi)(u\cdot X) & = \mathsf{F}(\pi)(u)\cdot X \\
	                     & = \mathcal{U}(\pi)(1)\cdot X \\
	                     & = \mathcal{U}(\pi)(X) \\
			     & = \mathsf{F}(\pi)(X)
     \end{align*}
     Hence we get the map
     \(\mathcal{Q}(\mathfrak{g}/\mathfrak{h})\sir\mathcal{U}(\mathfrak{g}/\mathfrak{h})\).
     This map is an epimorphism since the diagram~\eqref{diag:gquot_of_ug}
     commutes. It remains to shows that it is a monomorphism. For this let us
     choose a totally ordered basis \(\mathsf{B}\) of \(\mathfrak{g}\) such
     that a subset of \(B\) spans the Lie ideal \(\mathfrak{h}\). Let us
     consider the associated Poincare--Birkhoff--Witt basis of
     \(\mathcal{U}(\mathfrak{g})\). The kernel of \(\mathcal{U}(\pi)\) is
     spanned by all the PBW basis elements of the form \(Z_1\cdot\ldots\cdot
	 Z_n\) such that \(Z_1\leq\ldots\leq Z_n\),
     \(Z_1,\ldots,Z_n\in\mathsf{B}\) and there is at least one \(i\) such that
     \(Z_i\in\mathfrak{h}\) (this follows from the
     Poincar\'{e}--Birkhoff--Witt theorem for
     \(\mathcal{U}(\mathfrak{g}/\mathfrak{h})\)). We claim that, for such an
     element, \(\pi(Z_1)\cdot Z_2\cdot\ldots\cdot Z_n=0\) as an element of
     \(\mathcal{Q}(\mathfrak{g}/\mathfrak{h})\).  We will prove this by
     induction under the smallest index \(i\) such that
     \(Z_i\in\mathfrak{h}\). If \(Z_1\in\mathfrak{h}\) then clearly
     \(\pi(Z_1)\cdot Z_2\cdot\ldots\cdot Z_n=0\) in
     \(\mathcal{Q}(\mathfrak{g}/\mathfrak{h})\). Also if
     \(Z_2\in\mathfrak{h}\) then 
     \[\pi(Z_1)\cdot Z_2\cdot\ldots\cdot Z_n=\pi(Z_2)\cdot Z_1\cdot\ldots\cdot Z_n+\pi([Z_1,Z_2])\cdot Z_3\cdot\ldots\cdot Z_n=0\]
     since \(\mathfrak{h}\) is a Lie ideal. If the induction hypothesis holds
     for \(i\) and the smallest index \(j\) such that \(Z_j\in\mathfrak{h}\)
     is equal to \(i+1\) then by commuting \(Z_i\) and \(Z_{i+1}\) we get the
     claim under the induction hypothesis:
     \begin{align*}
      \pi(Z_1)\cdot Z_2\cdot\ldots\cdot Z_i\cdot Z_{i+1}\cdot\ldots & =\pi(Z_1)\cdot Z_2\cdot\ldots\cdot Z_{i+1}\cdot Z_i\cdot\ldots \\
				    & \phantom{=}+\pi(Z_1)\cdot Z_2\cdot\ldots\cdot [Z_i,Z_{i+1}]\cdot\ldots \\
				    & =0
     \end{align*}
\end{proof}
An element \(p\) of a Hopf algebra \(H\) is called \bold{primitive} if
\(\Delta(p)=p\otimes1+1\otimes p\). The set of primitive elements  of a Hopf
algebra \(H\), which we denote by \(P(H)\), forms a Lie algebra with the usual
bracket \([p,q]=pq-qp\). Furthermore, \(H\selmap{}P(H)\) defines a functor
from the category of Hopf algebras to the category of Lie algebras, that is if
\(f:H\sir K\) is a Hopf algebra morphism then it restricts to a morphism of
Lie algebras \(P(f):P(H)\sir P(K)\). It is the right adjoint functor to
\(\mathcal{U}\) which assigns the enveloping Hopf algebra of a Lie algebra
discussed in Example~\ref{ex:Hopf_algebras}. Let us also note that in
characteristic zero we have \(P(\mathcal{U}(\mathfrak{g}))=\mathfrak{g}\),
while in characteristic \(p\) \(X^p\in P(\mathcal{U}(\mathfrak{g}))\) for any
\(X\in\mathfrak{g}\). 

There is another important consequence of the PBW basis for generalised
quotients.
\index{universal enveloping algebra!generalised quotient!primitive elements}
\begin{proposition}\label{prop:primitive_elems_of_Q}
    Let \(\k\) be a field of characteristic \(0\).  Let \(\mathfrak{g}\) be
    a finite dimensional \(\k\)-Lie algebra and \(\mathfrak{h}\) a Lie
    subalgebra of \(\mathcal{g}\).  Then
    \(P(\mathcal{Q}(\mathfrak{g}/\mathfrak{h}))=\mathfrak{g}/\mathfrak{h}\). 
\end{proposition}
\begin{proof}
    The proof goes exactly the same way as for \(\mathcal{U}(\mathfrak{g})\).
    We include it here for completeness.  Let \(\mathsf{B}\) be a totally
    ordered basis of \(\mathfrak{g}\) and let us consider the
    basis~\eqref{eq:gen_pbw_basis}.  For a basis element \(b=\pi(Z_0)\cdot
	Z_1\cdot\ldots\cdot Z_m\) of
    \(\mathcal{Q}(\mathfrak{g}/\mathfrak{h})\) we have:
    \begin{align}\label{eq:delta_on_b}
\Delta_{\mathcal{Q}(\mathfrak{g}/\mathfrak{h})}(b)= & b\otimes1+1\otimes b+\\
& \hspace{-2cm}+\sum_{p=1}^{m-1}\sum_{\sigma\textit{ - }(p,m)\textit{-shuffle}}\pi(Z_{\sigma(1)})Z_{\sigma(2)}\cdots Z_{\sigma(p)}\otimes\pi(Z_{\sigma(p+1)})Z_{\sigma(2)}\cdots Z_{\sigma(m)} \notag
    \end{align}
    where a \((p,m)\)-shuffle \(\sigma\) is a permutation such that:
    \[\sigma(1)<\sigma(2)<\cdots<\sigma(p)\quad\textrm{and}\quad\sigma(p+1)<\sigma(p+2)<\cdots<\sigma(m)\]
    Note that since \(Z_0<\cdots<Z_m\) both factors of the tensor
    \(\pi(Z_{\sigma(1)})Z_{\sigma(2)}\cdots Z_{\sigma(p)}\) and
    \(\pi(Z_{\sigma(p+1)})Z_{\sigma(2)}\cdots Z_{\sigma(m)}\) are basis
    elements (they are non zero, see~\eqref{eq:gen_pbw_basis}).  By definition
    of the comultiplication
    \(\Delta_{\mathcal{Q}(\mathfrak{g}/\mathfrak{h})}\) it follows that
    \(\mathfrak{g}/\mathfrak{h}\) is a subset of the set of primitive elements
    of \(\mathcal{Q}(\mathfrak{g}/\mathfrak{h})\).  Now, let
    \[u=\sum_{i=1}^l\alpha_iZ_i\]
    where \(\alpha_i\in\k\) and each \(Z_i=\pi(Z_{i,1})Z_{i,2}\cdot\ldots\cdot
	Z_{i,m_i}\) is a basis element which belongs
    to~\eqref{eq:gen_pbw_basis}.  Let \(v=\sum_{i:\alpha_i=1}
	Z_i\in\pi(\mathfrak{g})\).  Then \(u-v\) is also a primitive element.
    Hence without loss of generality we can assume that \(m_i>1\) for all
    \(i=1,\dots,l\).  For such a primitive element \(u\) we will show that
    \(u=0\).  By equation~\eqref{eq:delta_on_b} and since \(u\) is a primitive
    element we have:
    {\small\begin{align}\label{eq:primitive_condition}
	\sum_{i=1}^l\sum_{p=1}^{m_i-1}\sum_{\sigma\textit{ - }(p,m)\textit{-shuffle}} \pi(Z_{i,\sigma(1)})Z_{i,\sigma(2)}\cdots Z_{i,\sigma(p)}\otimes\pi(Z_{i,\sigma(p+1)})Z_{i,\sigma(p+2)}\cdots Z_{i,\sigma(m_i)} & \\
	&\hspace*{-8cm}= \Delta_{\mathcal{Q}(\mathfrak{g}/\mathfrak{h})}(u)-u\otimes1-1\otimes u=0\notag
    \end{align}}
    Let \(b\) and \(b'\) be basis elements which belong
    to~\eqref{eq:gen_pbw_basis} and let
    \[I(b,b')\coloneq\left\{(i,p,\sigma):\begin{array}{l}
	    \sigma\textit{ is a }(p,m_i)\textit{ shuffle},\\
	    b=\pi(Z_{i,\sigma(1)})Z_{i,\sigma(2)}\cdots Z_{i,\sigma(p)},\\
	    b'=\pi(Z_{i,\sigma(p+1)})Z_{i,\sigma(p+2)}\cdots Z_{i,\sigma(m_i)}
	\end{array}\right\}\]
    Then the condition~\eqref{eq:primitive_condition} is satisfied if and only
    if for any pair of basis elements \(b,b'\) we have:
    \[\sum_{(i,p,\sigma)\in I(b,b')}\alpha_i=0\]
    We use the convention that the sum over an empty set is zero.  Let us take
    \((i,p,\sigma),(j,q,\tau)\in I(b,b')\).  Then we have:
    \begin{align*}
	\pi(Z_{i,\sigma(1)})Z_{i,\sigma(2)}\cdots Z_{i,\sigma(p)}=&b=\pi(Z_{j,\tau(1)})Z_{j,\tau(2)}\cdots Z_{j,\tau(q)}\\
	\pi(Z_{i,\sigma(p+1)})Z_{i,\sigma(p+2)}\cdots Z_{i,\sigma(m_i)}=&b'=\pi(Z_{j,\tau(q+1)})Z_{j,\tau(q+2)}\cdots Z_{j,\tau(m_j)}
    \end{align*}
    It follows that \(p=q\) and \(m_i=m_j\) and for all \(1\leq t\leq m_i\):
    \[Z_{i,\sigma(t)}=Z_{j,\tau(t)}\]
    Now we have:
    \begin{equation}\label{eq:primitive_comp}
	\#\{t: Z_{i,t}=X\} = \#\{t: Z_{j,t}=X\} 
    \end{equation}
    Now, if \(Z_i\neq Z_j\) then there exists a minimal index \(t'\) such that
    \(Z_{i,t'}\neq Z_{j,t'}\) and \(Z_{i,t}=Z_{j,t}\) for all \(t<t'\).
    Without loss of generality we can assume that \(Z_{i,t'}<Z_{j,t'}\).
    Since \(Z_{j,s}\leq Z_{j,s+1}\) for all \(s\) we have:
    \begin{align*}
	\#\{t:Z_{j,t}=Z_{i,t'}\} & =\#\{t: Z_{j,t}=Z_{i,t'}\text{ and }t<t'\}\\
	                         & =\#\{t:Z_{i,t}=Z_{i,t'}\text{ and }t<t'\}\\
				 & <\#\{t:Z_{i,t}=Z_{i,t'}\}
    \end{align*}
    But this contradicts equation~\eqref{eq:primitive_comp}.  Thus
    \(Z_{i,t}=Z_{j,t}\) for all \(t\) and hence \(Z_i=Z_j\).  In consequence
    \(i=j\).  This shows that if \(I(b,b')\) is non-empty, then there exist
    uniquely determined numbers \(i\) and \(p\) such that any element of
    \(I(b,b')\) looks like \((i,p,\sigma)\) for some \((p,m_i)\)-shuffle
    \(\sigma\).  Let us denote this number by \(i_{b,b'}\).
    Equation~\eqref{eq:primitive_condition} is equivalent to
    \(\alpha_{i_{b,b'}}\cdot\#I(b,b')=0\) for all pairs \(b,b'\) of basis
    elements of \(\mathcal{Q}(\mathfrak{g}/\mathfrak{h})\).  Since
    \(\#I_{b,b'}\) is a natural number and the characteristic of the field
    \(\k\) is zero we get \(\alpha_i=0\) for all \(i\).  Thus \(u=0\).  This
    proves that all the primitive elements of
    \(\mathcal{Q}(\mathfrak{g}/\mathfrak{h})\) are those of
    \(\mathfrak{g}/\mathfrak{h}\). 
\end{proof}

For a (one-sided) coideal subalgebra \(K\) of a Hopf algebra \(H\) we let
\(K^+\coloneq K\cap\ker\epsilon\).
\index{universal enveloping algebra!generalised quotient!Takeuchi correspondence}
\begin{proposition}\label{prop:enveloping-algebra-quotients_}
    Let \(\mathfrak{g}\) be a finite dimensional Lie algebra over a field
    \(\k\) of characteristic zero and let \(\mathfrak{h}\) be a Lie
    subalgebra.  Then we have isomorphism of generalised quotients: 
    \begin{center}
	{\hfill\begin{tikzpicture}
	    \matrix[matrix of nodes, column sep=1cm,row sep=1cm]{
		& |(A)| \(\mathcal{U}(\mathfrak{g})\) & \\
|(C)| \(\mathcal{U}(\mathfrak{g})/\mathcal{U}(\mathfrak{h})^+\mathcal{U}(\mathfrak{g})\) & & |(B)| \(\mathcal{U}(\mathfrak{g}/\mathfrak{h})\) \\
 	    };
	    \draw[->] (A) --node[above left]{\(p\)} (C);
	    \draw[->] (A) --node[above right]{\(\mathcal{Q}(\pi)\)} (B);
	    \draw[->] (B) --node[below]{\(\cong\)} (C);
	\end{tikzpicture}
	\hfill\refstepcounter{equation}\raisebox{12mm}{(\theequation)}\label{diag:gquot_of_ug_}\\
	}
    \end{center}
    where
    \(p:\mathcal{U}(\mathfrak{g})\sir\mathcal{U}(\mathfrak{g})/\mathcal{U}(\mathfrak{h})^+\mathcal{U}(\mathfrak{g})\)
    is the natural projection.
\end{proposition}
\begin{proof}
    Let choose a Poincar\'{e}--Birkhoff--Witt basis for
    \(\mathcal{Q}(\mathfrak{g}/\mathfrak{h})\) as given in
    equation~\eqref{eq:gen_pbw_basis}.  Then we define a \(\k\)-linear map
    \(\phi:\mathcal{Q}(\mathfrak{g}/\mathfrak{h})\sir\mathcal{U}(\mathfrak{g})/\mathcal{U}(\mathfrak{h})^+\mathcal{U}(\mathfrak{g})\),
    by \(\phi(\pi(Z_0)\cdot Z_1\cdot\ldots\cdot Z_n)=p(Z_0)\cdot
	Z_1\cdot\ldots\cdot Z_m\), where \(\pi(Z_0)\cdot Z_1\cdot\ldots\cdot
	Z_n\) is a basis element which belongs to~\eqref{eq:gen_pbw_basis}.
    Clearly, \(\phi\) is a morphism of generalised quotients and hence, by
    commutativity of diagram~\eqref{diag:gquot_of_ug_}, it is an epimorphism.
    Now let us consider
    \(\phi|_{P(\mathcal{Q}(\mathfrak{g}/\mathfrak{h}))}:P(\mathcal{Q}(\mathfrak{g}/\mathfrak{h}))\sir
	P(\mathcal{U}(\mathfrak{g})/\mathcal{U}(\mathfrak{h})^+\mathcal{U}(\mathfrak{g}))\).
    We already know that
    \(P(\mathcal{Q}(\mathfrak{g}/\mathfrak{h}))=\mathfrak{g}/\mathfrak{h}\).
    Furthermore, \(\mathfrak{g}/\mathfrak{h}\subseteq
	P(\mathcal{U}(\mathfrak{g})/\mathcal{U}(\mathfrak{h})^+\mathcal{U}(\mathfrak{g}))\),
    which follows from the Poincar\'{e}--Birkhoff--Witt theorem for
    \(\mathcal{U}(\mathfrak{g})\).  It follows that
    \(\phi|_{P(\mathcal{Q}(\mathfrak{g}/\mathfrak{h}))}:P(\mathcal{Q}(\mathfrak{g}/\mathfrak{h}))\sir
	P(\mathcal{U}(\mathfrak{g})/\mathcal{U}(\mathfrak{h})^+\mathcal{U}(\mathfrak{g}))\)
    is a monomorphism.  Now since \(\mathcal{Q}(\mathfrak{g}/\mathfrak{h})\)
    is a pointed irreducible coalgebra by Corollary~\eqref{cor:Q_irr_pointed}
    it follows that \(\phi\) is an injective homomorphism (by~\cite[Thm
    2.4.11]{ea:hopf_algebras}).
\end{proof}
\index{universal enveloping algebra!generalised quotient!characterisation theorem}
\begin{theorem}\label{thm:enveloping-algebra-quotients}
    Let \(\mathfrak{g}\) be a finite dimensional Lie algebra over a field
    \(\k\) of characteristic \(0\). Then
    \(\Quot(\mathcal{U}(\mathfrak{g}))\cong\Quot_\mathit{Lie}(\mathfrak{g})\),
    where \(\Quot_{\mathit{Lie}}(\mathfrak{g})\) denotes the poset of quotient
    Lie algebras of the Lie algebra \(\mathfrak{g}\). Furthermore, there is an
    order reversing isomorphism
    \(\qquot(\mathcal{U}(\mathfrak{g}))\cong\Sub_{\mathit{Lie}}(\mathfrak{g})\),
    where \(\Sub_{\mathit{Lie}}(\mathfrak{g})\) is the poset of Lie
    subalgebras of the Lie algebra \(\mathfrak{g}\). 
\end{theorem}
\begin{proof}
    We will show that the following two order reversing maps:
    \begin{gather*}
	\qquot(\mathcal{U}(\mathfrak{g}))\ni Q\selmap{}\ker(\mathfrak{g}\sir P(Q))\in\Sub_{\mathit{Lie}}(\mathfrak{g})\\
	\Sub_{\mathit{Lie}}(\mathfrak{g})\ni\mathfrak{h}\selmap{}\mathcal{Q}(\mathfrak{g}/\mathfrak{h})\in\qquot(\mathcal{U}(\mathfrak{g}))
    \end{gather*}
    are inverse bijections.  Let \(\pi:\mathcal{U}(\mathfrak{g})\epr{}Q\) be
    a generalised quotient.  Since \(\pi(\mathfrak{g})\) generates \(Q\) as
    a \(\mathcal{U}(\mathfrak{g})\)-module we have an epimorphism
    \(\phi:\mathcal{Q}(\pi(\mathfrak{g}))\sir Q\) induced by the inclusion
    \(\pi(\mathfrak{g})\subseteq P(Q)\).  The coalgebra
    \(\mathcal{Q}(\pi(\mathfrak{g}))\) is pointed irreducible (see
    Corollary~\ref{cor:Q_irr_pointed} on page~\pageref{cor:Q_irr_pointed}).
    The map \(\phi|_{P(\mathcal{U}(\pi(\mathfrak{g})))}\) is the inclusion
    \(P(\mathcal{Q}(\pi(\mathfrak{g})))=\pi(\mathfrak{g})\subseteq Q\) when
    \(\k\) has characteristic zero (see
    Proposition~\ref{prop:primitive_elems_of_Q} on
    page~\pageref{prop:primitive_elems_of_Q}).  It follows that \(\phi\) is
    a monomorphism by~\cite[Thm~2.4.11]{ea:hopf_algebras}.  Thus \(\phi\) is
    an isomorphism.  We get:
    \(P(Q)=P(\mathcal{Q}(\pi(\mathfrak{g})))=\pi(\mathfrak{g})\), and hence
    \(P(Q)\) is indeed a quotient Lie algebra of \(\mathfrak{g}\) (or
    a quotient of \(\mathfrak{g}\) by a Lie subalgebra, see
    Note~\ref{note:U(G)_gen_quotients}) and
    \(\mathcal{Q}(P(Q))=\mathcal{Q}(\pi(\mathfrak{g}))\cong Q\) as generalised
    quotients of \(\mathcal{U}(\mathfrak{h})\).  On the other hand if
    \(\mathfrak{h}\) is a quotient of \(\mathfrak{g}\) by a Lie subalgebra,
    then \(\mathcal{Q}(\mathfrak{h})\in\qquot(\mathcal{U}(\mathfrak{g}))\) and
    \(P(\mathcal{Q}(\mathfrak{h}))=\mathfrak{h}\), since \(\k\) has
    characteristic zero (see Proposition~\ref{prop:primitive_elems_of_Q}).

    Once we have shown the result for generalised quotients the claim for Hopf
    algebra quotients follows from Proposition~\ref{prop:U_and_Q} (on
    page~\pageref{prop:U_and_Q}).  Note that if \(Q\) is a Hopf algebra
    quotient then the map \(\mathfrak{g}=P(\mathcal{U}(\mathfrak{g}))\sir
	P(Q)\) is an epimorphism of Lie algebras.  Thus its kernel is a Lie
    ideal.
\end{proof}

\subsection{Quantum enveloping algebras}

\citeauthor{ih-ck:homogeneous_right_coideal_subalgebras} classify all
homogeneous right coideal subalgebras of a quantised enveloping algebra of
a complex semisimple Lie algebra \(\mathfrak{g}\).  A (one-sided) coideal
subalgebra is called homogeneous if it contains the Hopf subalgebra of all
group-like elements.  Interestingly the classification is based on the Weyl
group \(W\) of \(\mathfrak{g}\).  The first hint that such a classification
was possible was given by~\cite{vk-as:right_coideal_subalgebras_of_Uq(sl)}.
They provided a complete classification of homogeneous right coideal
subalgebras of the multiplier version of \(U_q(\mathfrak{sl}_{n+1})\) (with
\(q\) not a root of unity).  This led to the conjecture that the number of
homogeneous right coideal subalgebras of the Borel Hopf subalgebra of
\(U_q(\mathfrak{g})\) coincides with the order of the Weyl group of
\(\mathfrak{g}\).  Kharchenko proved this conjecture for \(\mathfrak{g}\) of
type \(A_n\) and \(B_n\)
(\cite{vk-as:right_coideal_subalgebras_of_Uq(sl),vk:right_coideal_subalgebras_of_Uq(so)}).
The classification of right coideal subalgebras of \(U(\mathfrak{g})\) done by
\citeauthor{ih-ck:homogeneous_right_coideal_subalgebras} is based on the
results of~\cite{ih-hs:right_coideal_subalgebras_of_Nichols_algebras}.  Let us
note that in the classical case \(\mathcal{U}(\mathfrak{h})\) is a Hopf
subalgebra of \(\mathcal{U}(\mathfrak{g})\) for any Lie subalgebra
\(\mathfrak{h}\) of \(\mathfrak{g}\) (and
\(\mathcal{U}(\mathfrak{h})=\mathcal{U}(\mathfrak{g})^{\co\,\mathcal{Q}(\mathfrak{g}/\mathfrak{h})}\),
which follows from Example~\ref{ex:Takeuchi_for_U(g)}).  But in the quantum
case this is no longer true: \(\mathcal{U}_q(\mathfrak{h})\), when defined,
might not be isomorphic to a Hopf subalgebra of
\(\mathcal{U}_q(\mathfrak{g})\).  This actually led to the investigation of
one sided coideal subalgebras, since there are analogs of
\(\mathcal{U}(\mathfrak{h})\) in \(\mathcal{U}_q(\mathfrak{g})\) which are one
sided coideal subalgebras
(see~\cite{gl:coideal_subalgebras_and_quantum_symmetric_pairs}). 

Let us note that in the classical situation: for a separable field extension
\(\bE/\bF\) which is a Hopf Galois with Hopf algebra \(H\) (over a subfield
\(\k\subseteq\bF\), there exists an extension \(\bL\) of \(\k\) such that
\(\bL\otimes_\k H\) is isomorphic to a group Hopf algebra
(see~\cite{cg-bp:separable_field_extensions}).  It seems that group-like
elements are essential in the commutative situation,  though this might not be
the case for non commutative extensions.  Nevertheless let us note that the
coideal right ideals which correspond (via Takeuchi correspondence,
Theorem~\ref{thm:newTakeuchi}) to one sided coideal subalgebras classified by
\citeauthor{ih-ck:homogeneous_right_coideal_subalgebras} do not contain any
group-like elements other than the image of \(1\in H\).  It is so since all
the classified one sided coideal subalgebras contain the radical.

\section{Comodule subalgebras}
\index{comodule algebra!poset of comodule subalgebras}
Let \(A\) be an \(\R\)- algebra. The \emph{poset of all subalgebras} will be
denoted by \(\Sub_{\Alg}(A)\) and \(\Sub_{\Alg}(A/B)\) is the interval
\([B,A]\) in \(\Sub_{\Alg}(A)\). Both lattices are algebraic. The compact
elements of \(\Sub_\Alg(A/B)\) are the finitely generated subalgebras, and the
compact elements of \(\Sub_{\Alg}(A/B)\) are the algebras which are generated
by \(B\) and a finite set of elements of \(A\backslash B\).

Let \(Q\) be a generalised quotient of a bialgebra \(H\). Then \(A\) is
a \(Q\)-comodule with the structure map
\(\delta_Q:=(\id\otimes\pi_Q)\circ\delta_A\), where \(\pi_Q:H\sir Q\) is the
projection. If $Q$ is a bialgebra quotient then $A$ is a \(Q\)-comodule
algebra. The \bold{poset of \(H\)-comodule subalgebras} will be denoted by
\(\Sub_{{\Alg}^H}(A)\).  Let us define the following interval in
\(\Sub_{\Alg^H}(A)\):
\[\Sub_{\Alg^H}(A/A^{co\,H})\coloneq\{B\in\Sub_{\Alg^H}(A):A^{co\,H}\subseteq B\}\]
The lattice of \(H\)-comodule subalgebras of~\(A\) is an upper subsemilattice
of \(\Sub_\Alg(A)\) and thus it is complete. If \(H\) is flat as an
\(\R\)-module, then for \(S_i\in\Sub_{\Alg^H}(A)\) (\(i=1,2\)) we have
\(S_1\wedge S_2=S_1\cap S_2\), which follows from
Theorem~\ref{thm:lattice_of_subcomodules}.
\index{comodule algebra!poset of comodule subalgebras!algebraic}
\begin{proposition}\label{prop:comodule_subalg_algebraic}
    Let \(H\) be a bialgebra and let \(A\) be an \(H\)-comodule algebra such
    that \(H\) is a flat Mittag--Leffler \(\R\)-module.  Then the posets
    \(\Sub_{\Alg^H}(A)\) are algebraic lattices.
\end{proposition}
\begin{proof}
    It  follows from Theorem~\ref{thm:lattice_of_subcomodules} that both
    \(\Sub_\Alg(A)\) and \(\Sub_{\Alg^H}(A)\) are
    \(\cap\overrightarrow{\cup}\)-structures. 
\end{proof}

\section{Summary}
Let us bring together all the lattices considered:
\begin{center}
    \begin{longtable}{l|p{7cm}|p{4cm}}
	\textbf{\textsc{poset}}     & \textbf{\textsc{objects}}                                                                           & \textbf{\textsc{properties}}\\\hline
	\(\coId(C)\)                & coideals of a coalgebra \(C\) & \hyperref[prop:coid-complete]{complete} and \hyperref[thm:quot(C)-algebraic]{dually algebraic} (p.\pageref{thm:quot(C)-algebraic}) \\
	\(\Quot(C)\)                & quotient coalgebras of a coalgebra \(C\) & \hyperref[thm:quot(C)-algebraic]{algebraic} (p.\pageref{thm:quot(C)-algebraic}) \\
	\(\Sub_\Coalg(C)\)          & subcoalgebras of a coalgebra \(C\) & \hyperref[thm:subcoalgebras_algebraic]{algebraic} (p.\pageref{thm:subcoalgebras_algebraic}) and \hyperref[thm:subcoalgebras_dualy_algebraic]{dually algebraic} (p.\pageref{thm:subcoalgebras_dualy_algebraic}) {\footnotesize (over a field)} \\[3pt]
	\(\coId_r(C)\)              & right coideals of a coalgebra \(C\) & \hyperref[thm:lattice_of_subcomodules]{complete, algebraic} (p.\pageref{thm:lattice_of_subcomodules}) {\footnotesize (if \(C\) is a flat \textsf{ML} module)} \\
	\(\coId_l(C)\)              & left coideals of a coalgebra \(C\)                                                                  & \hyperref[thm:lattice_of_subcomodules]{as above} (p.\pageref{thm:lattice_of_subcomodules}) \\
	\(\Sub_{\Mod^C}(M)\)        & subcomodules of a right \(C\)-comodule \(M\)                                                        & \hyperref[thm:lattice_of_subcomodules]{as above} (p.\pageref{thm:lattice_of_subcomodules}) \\[3pt]
	\(\Id_{\mathit{bi}}(B)\)    & biideals of a bialgebra \(B\)                                                                       & \hyperref[prop:biideals_complete]{complete} (p.\pageref{prop:biideals_complete}) \\
	\(\Sub_{\mathit{bi}}(B)\)   & subbialgebras of a bialgebra \(B\)                                                                  & \hyperref[prop:subbialg_algebraic]{algebraic} (p.\pageref{prop:subbialg_algebraic}) {\footnotesize (over a field)} \\[3pt]
	\(\Id_{\mathit{Hopf}}(H)\)  & Hopf ideals of a Hopf algebra \(H\)                                                                 & \hyperref[prop:Hopf_ideals_complete]{complete} (p.\pageref{prop:Hopf_ideals_complete}) \\
	\(\Sub_{\mathit{Hopf}}(H)\) & subbialgebras of a Hopf algebra \(H\)                                                               & \hyperref[prop:subHopfalg_algebraic]{algebraic} (p.\pageref{prop:subHopfalg_algebraic}) {\footnotesize (over a field)} \\
	\(\Quot(B)\)                & bialgebra (Hopf algebra) quotients of a bialgebra (Hopf algebra) \(B\) & \hyperref[prop:biideals_complete]{complete} (p.\pageref{prop:biideals_complete}) \\
	\(\qquot(B)\)               & quotients of a bialgebra (Hopf algebra) \(B\) by a right ideal coideal                        & \hyperref[prop:qquot_complete]{complete} (p.\pageref{prop:qquot_complete})  \\
	\(\qsub(B)\)                & left coideal subalgebras of a bialgebra (or a Hopf algebra) \(B\) & \hyperref[prop:qsub_algebraic]{algebraic} (p.\pageref{prop:qsub_algebraic}) {\footnotesize (if \(B\) is a flat \textsf{ML} module)} \\[3pt]
	\(\Sub_\Alg(A)\)            & subalgebras of an algebra \(A\)                                                                     & algebraic \\
	\(\Sub_{\Alg^H}(A)\)        & \(H\)-comodule subalgebras of an \(H\)-comodule algebra \(A\) & \hyperref[prop:comodule_subalg_algebraic]{algebraic} (p.\pageref{prop:comodule_subalg_algebraic}) {\footnotesize(if \(H\) is a flat \textsf{ML} module)} \\
    \end{longtable}
\end{center}
In the above table an \textsf{ML} is a abbreviation for \textit{Mittag--Leffler}.
\label{chap:lattices_end}

\chapter{Galois Theory for Hopf--Galois extensions}\label{chap:Hopf-Galois_theory}
In this chapter we introduce Galois theory for Hopf--Galois extensions.  Let
us note that the most common argument to use \(\qquot(H)\) rather than
\(\Quot(H)\) is that a noncommutative Hopf algebra might have too few (or even
no non trivial) quotient Hopf algebras.  As we noted before, there is another
important argument, which shouldn't be missed. In the case of the group Hopf
algebra \(\k[G]\) (and its dual \(\k[G]^*\) if \(G\) is finite) the poset of
generalised quotients is isomorphic to the poset of isomorphism classes of
transitive \(G\)-sets (which is anti-isomorphic to the poset of subgroups of
\(G\)), while the Hopf quotients correspond to group quotients of \(G\)
(normal subgroups).  Since in Galois theory one considers the poset of
subgroups of \(G\) (or more generally the category of \(G\)-sets, as in the
Grothendieck approach to Galois Theory, see~\cite{fb-gj:galois-theories}) it
is more natural to use the poset of generalised quotients of Hopf algebras.  

We begin with a section about the Schauenburg approach towards the Galois
theory of Hopf Galois extensions of commutative rings.  He considers
a faithfully flat \(H\)-Galois extension of the commutative base ring and
constructs a Hopf algebra \(L(A,H)\) for which \(A\) is an
\(L(A,H)\)-\(H\)-bicomodule algebra and an \(L(A,H)\)-Galois extension of the
base ring.  Then he constructs a Galois connection between \(H\)-comodule
subalgebras of \(A\) and generalised quotients of \(L(A,H)\).  In the
following section we introduce a more general construction of a Galois
correspondence for \(H\)-extensions.  Our result, in the following section,
differs significantly from Schauenburg's.  First of all, we construct a Galois
connection between subalgebras (rather than \(H\)-comodule subalgebras) of an
\(H\)-comodule algebra \(A\) and generalised quotients of \(H\).  Secondly, we
drop the assumption that the coinvariants subalgebra must be equal to the
commutative base ring.  Importantly, we also leave the assumption that the
extension is Hopf--Galois (Definition~\ref{defi:Hopf-Galois_extension}).  We
will require that both the algebra and the Hopf algebra are flat
Mittag--Leffler modules over the base ring.  This extra module theoretical
condition is always satisfied over a field.  In this context we provide
another, new, construction of the Galois connection.  The formula obtained for
the adjoint will be used later. It resembles, to some extent, the known
formula for a Galois connection between left coideal subalgebras and
generalised quotients of a Hopf algebra (see Theorem~\ref{thm:newTakeuchi}).
Let us note that applying our construction to the left \(L(A,H)\)-comodule
algebra \(A\), as considered by Schauenburg, one will get the same adjunction,
which follows from the uniqueness of Galois correspondences (see
Remark~\ref{rem:comparison_of_adjunctions}).

In section~\ref{sec:closed_elements} (on page~\pageref{sec:closed_elements})
we give two interesting results on closed elements, one for subalgebras
(Theorem~\ref{thm:closed_subalgebras}) and one for generalised quotients
(Corollary~\ref{cor:Q-Galois_closed}).  The first one shows that for an
\(H\)-Galois extension satisfying some module theoretic assumptions,
a subalgebra \(S\) is closed if and only if the following map is a bijection:
\begin{equation*}
    \can_S: S\otimes_B A\ir A\cotensor_{\psi(S)} H,\quad\can_S(a\otimes b)=ab_{(1)}\otimes b_{(2)}
\end{equation*}
We make an attempt to show a similar result for generalised quotients.
Assuming that the \(H\)-extension \(A/A^{\co\,H}\) has epimorphic canonical
map, and satisfies some additional module theoretic conditions, we show
(Corollary~\ref{cor:Q-Galois_closed}) that a generalised quotient \(\pi:H\sir
    Q\) is closed if the following map is a bijection:
\begin{equation*}
    \can_Q:A\otimes_{A^{\co\,Q}}A\sir A\otimes Q,\quad\can_Q(a\otimes b)=ab_{(1)}\otimes\pi(b_{(2)}) 
\end{equation*}
In Theorem~\ref{thm:cleft-case} we will show that for crossed products this is
a sufficient and necessary condition.  We also show, that for \(H\)-Galois
extensions over a field, the above canonical map is an isomorphism if and only
if \(Q\) is closed and the map
\(\delta_A\otimes\delta_A:A\otimes_{A^{\co\,Q}}A\sir(A\otimes
    H)\otimes_{A\otimes H^{\co\,Q}}(A\otimes H)\) is an injective map (see
Theorem~\ref{thm:Q-Galois_closed_over_field} on
page~\pageref{thm:Q-Galois_closed_over_field}). 

In section~\ref{sec:admissibility} (on page~\pageref{sec:admissibility}) we
generalise Schauenburg's theory of admissible objects.  We show that
admissible subalgebras and admissible generalised quotients are in bijective
correspondence for Hopf Galois extensions (under some module theoretic
conditions).  Let us note that we redefine admissibility of subalgebras of
a comodule algebra.  We show that our definition is equivalent with
Schauenburg's when one considers faithfully flat \(H\)-Galois extensions of
the base ring with the assumption that \(H\) has a bijective antipode (which
is also considered by Schauenburg).  Schauenburg obtained a 1-1 correspondence
between admissible \(H\)-subcomodule algebras of \(A\) and admissible
generalised quotients of \(L(A,H)\).  We show that, in our broader context,
admissible subalgebras of \(A\) are in bijective correspondence with
admissible generalised quotients of \(H\).  We also show that Schauenburg
theory of admissible objects is a special case of our result (up to our
stronger module theoretical assumptions).

\section[Schauenburg's approach]{Schauenburg's approach to Hopf--Galois extensions of commutative rings} 
\label{sec:schauenburgs_approach} Our approach to closed elements of the
Galois connection between sub algebras of an \(H\)-extension and generalised
quotients of \(H\) generalises the results of Schauenburg.  His construction
of the correspondence is different than ours and requires additional
assumptions, nevertheless these assumptions and constructions allow one to use
Hopf algebraic methods more extensively than in our approach.  For the purpose
of completeness and a better understanding we discuss here, in this
preliminary section, the approach of Schauenburg.  It is based on two
articles~\cite{ps:hopf-bigalois} and \cite{ps:gal-cor-hopf-bigal}.
Schauenburg's construction was an adaptation of an earlier work
of~\citeauthor{fo-yz:gal-cor-hopf-galois}.

First let us note that Schauenburg's approach applies only to faithfully flat
\(H\)-Galois extensions of commutative rings.  In this section we will assume
that \(A\) is an \(R\)-algebra which is faithfully flat as an \(\R\)-module.
Furthermore, we will assume that it possess an \(H\)-comodule algebra
structure making it an \(H\)-Galois extension.

Before we start with the constructions we will introduce the following
Sweedler--type notation:
\begin{equation}\label{eq:kappa_-1}
    \can_H^{-1}(1\otimes h) \eqcolon  h_{[1]}\otimes h_{[2]}\in A\otimes A
\end{equation}
where \(h\in H\) and \(\can_H\) is the \textit{canonical}
map~\eqref{eq:canonical-map}.  The above map satisfies the following
relations:
\begin{subequations}
    \begin{align}
            (hk)_{[1]}\otimes(hk)_{[2]} & = k_{[1]}h_{[1]}\otimes h_{[2]}k_{[2]} \label{eq:kappa_-1:1}\\
	    h_{[1]}h_{[2]}  & = \epsilon(h)1_A \label{eq:kappa_-1:3} \\
	    a_{(0)}(a_{(1)})_{[1]}\otimes (a_{(1)})_{[2]}  & = 1_A\otimes a \label{eq:kappa_-1:2}\\
	    h_{[1]}\otimes(h_{[2]})_{(0)}\otimes(h_{[2]})_{(1)}  & =(h_{(1)})_{[1]}\otimes(h_{(1)})_{[2]}\otimes h_{(2)} \label{eq:kappa_-1:4}\\
	    (h_{[1]})_{(0)}\otimes h_{[2]}\otimes (h_{[1]})_{(1)} & = (h_{(2)})_{[1]}\otimes (h_{(2)})_{[2]}\otimes S(h_{(1)})
    \end{align}
\end{subequations}
where \(h\in H\) and \(a\in A\).  Equation~\ref{eq:kappa_-1:3} is just
a consequence of the following equality
\(\id_A\otimes\epsilon\circ\can_H(a\otimes b)=ab\); for~\eqref{eq:kappa_-1:2}
it is enough to compute \(can_H\) of both sides and conclude that they are
equal:
\begin{alignat*}{2}
    \can_H(a_{(0)}(a_{(1)})_{[1]}\otimes (a_{(1)})_{[2]}) & = a_{(0)}{a_{(1)}}_{[1]}{{a_{(1)}}_{[2]}}_{(0)}\otimes{{a_{(1)}}_{[2]}}_{(1)} &&\\
    & = a_{(0)}\otimes a_{(1)}&\quad&\textrm{by~\eqref{eq:kappa_-1}} \\
    & = \can_H(1\otimes a)
\end{alignat*}
where \(a\in A\).  Equation~\eqref{eq:kappa_-1:4} is a consequence of the fact
that \(\can_H\) is a right \(H\)-comodule map (where \(H\)-comodule structures
on \(A\otimes A\) and \(A\otimes H\) are induced from the right tensor
factor).  Finally, the first one can now be checked in the following way:
\begin{alignat*}{2}
    \can_H(b_{[1]}a_{[1]}\otimes a_{[2]}b_{[2]}) & = b_{[1]}a_{[1]}(a_{[2]})_{(0)}(b_{[2]})_{(0)}\otimes(a_{[2]})_{(1)}(b_{[2]})_{(1)} &&\\
    & = (b_{(1)})_{[1]}(a_{(1)})_{[1]}(a_{(1)})_{[2]}(b_{(1)})_{[2]}\otimes a_{(2)}b_{(2)} &\quad&\textrm{by~\eqref{eq:kappa_-1:4}}\\
    & = 1\otimes ab &&\textrm{by~\eqref{eq:kappa_-1:3}}\\
\end{alignat*}
Now we are ready to give the main  Schauenburg's construction:
\index{associated Hopf algebra L(H,A)}
\index{L(H,A)|see{associated Hopf algebra L(H,A)}}
\begin{proposition}[{\cite[Prop.~3.5]{ps:hopf-bigalois}}]
    Let \(H\) be a Hopf algebra over a commutative ring \(\R\) and let
    \(A/\R\) be a faithfully flat \(H\)-Galois extension.  Let
    \(L(A,H)\coloneq(A\otimes A)^{co H}\), where \(A\otimes A\) is considered
    with the codiagonal coaction of \(H\):
    \[A\otimes A\sir A\otimes A\otimes H,\quad a\otimes b\selmap{}a_{(0)}\otimes b_{(0)}\otimes a_{(1)}b_{(1)}\]
    Then:
    \begin{enumerate}
	\item \(L(A,H)\) is a subalgebra of \(A\otimes A^\op\);
	\item the coalgebra structure is given by\footnote{In the following
	      formulas we use simple tensors, which is an abuse of notation,
	      since the domain of the coalgebra structure maps is \((A\otimes
		  A)^{co H}\), which might not be spanned by simple tensors.
	      We do that for the sake of clarity.}:
	    \begin{align*}
		\Delta_{L(A,H)}(x\otimes y) & = x_{(0)}\otimes\can_H^{-1}(1\otimes x_{(1)})\otimes y \\
		\epsilon_{L(A,H)}(x\otimes y) & = xy\in A^{co H}=\R \\
		S_{L(A,H)}(x\otimes y) & = y_{(0)}\otimes (y_{(1)})_{[1]}x(y_{(1)})_{[2]}
	    \end{align*}
    \end{enumerate}
    The above structure turns \(L(A,H)\) into an \(\R\)-Hopf algebra.
\end{proposition}
For the proof that this indeed defines a Hopf algebra structure on \((A\otimes
    A)^{co H}\) see \cite[Lemma~3.3 and Theorem~3.5]{ps:hopf-bigalois}.  Now,
\(A\) becomes an \(L(A,H)\)-\(H\)-\textbf{biGalois extension}.  That means
\(A/\R\) is:
\index{bicomodule}
\index{bicomodule algebra}
\index{biGalois extension}
\begin{enumerate}
    \item a left \(L(A,H)\)-comodule algebra which is an \(L(A,H)\)-Galois
	extension of \(\R\), with the \(L(A,H)\)-comodule structure given by:
\[\delta_{L(A,H)}:A\sir L(A,H)\otimes A,\quad \delta_{L(A,H)}(a)\coloneq a_{(0)}\otimes\can_H^{-1}(1\otimes a_{(1)})\]
    \item a right \(H\)-comodule algebra which is an \(H\)-Galois extension of
	\(\R\).
\end{enumerate}
Moreover the two comodule structures: left \(L(A,H)\) and right \(H\)-comodule
commute (see the diagram below) making \(A\) an \(L(A,H)\)-\(H\)-bimodule:
\begin{center}
    \begin{tikzpicture}
	\matrix[matrix of nodes,column sep=2.3cm,row sep=1cm]{
	    |(A1)| \(A\) & |(A2)| \(A\otimes H\) \\
	    |(B1)| \(L(A,H)\otimes A\) & |(B2)| \(L(A,H)\otimes A\otimes H\) \\
	};
	\draw[->] (A1) --node[above]{\(\delta_H\)} (A2);
	\draw[->] (A2) --node[right]{\(\delta_{L(A,H)}\otimes\id_H\)} (B2); 
	\draw[->] (A1) --node[left]{\(\delta_{L(A,H)}\)} (B1);
	\draw[->] (B1) --node[below]{\(\id_{L(A,H)}\otimes\delta_H\)} (B2); 
    \end{tikzpicture}
\end{center}
We will use the following Sweedler notation for bicomodules:
\index{bicomodule!Sweedler notation}
\begin{align*}
    \delta_H(a) & \coloneq a_{(0)}\otimes a_{(1)} \in A\otimes H \\ 
    \delta_{L(A,H)}(a) & \coloneq a_{(-1)}\otimes a_{(0)} \in L(A,H)\otimes A \\ 
\end{align*}
where \(a\in A\).  We will show now that \(A\) is \(L(A,H)\)-Galois.  We will
prove slightly more, since we will need this result later:
\begin{lemma}\label{lem:Schauenburg_Galois_cond}
    Let \(H\) be a flat Hopf algebra over a commutative ring \(\R\) and let
    \(A/\R\) be a faithfully flat \(H\)-Galois extension.  Finally, let
    \(S\in\Sub_\Alg(A)\).  Then the following map is a bijection:
    \[\can_{S}:A\otimes_SA\sir(A\otimes_SA)^{co\,H}\otimes A,\can_{S}(a\otimes_Sb)=a_{(0)}\otimes_S(a_{(1)})_{[1]}\otimes(a_{(1)})_{[2]}b\]
\end{lemma}
If we take \(S=\R\) then \(\can_S\) is the canonical map associated with the
left \(L(A,H)\)-comodule algebra \(A\).
\begin{proof}
    This a consequence of the fundamental theorem for Hopf modules, but we
    will present a direct computation.  The inverse of \(\can_{S}\) is given
    by:
    \[\can_{S}^{-1}:(A\otimes_SA)^{co\,H}\otimes A\sir A\otimes_SA, \can_{S}^{-1}(a\otimes_S b\otimes c)=a\otimes_S bc\]
    Indeed:
    \begin{alignat*}{2}
\can_{S}^{-1}\circ\can_{S}(a\otimes b) & = a_{(0)}\otimes_S(a_{(1)})_{[1]}(a_{(1)})_{[2]}b&          &\\
			     & = a_{(0)}\otimes\epsilon(a_{(1)})b                            &\quad\quad&\textrm{by~\eqref{eq:kappa_-1:3}}\\
			     & = a\otimes b                                           &          &
    \end{alignat*}
   Now take \(a\otimes b\otimes c\in (A\otimes_SA)^{co\,H}\otimes A\) (again
   we abuse the notation, since \((A\otimes_SA)^{co\,H}\) might not be
   generated by simple tensors):
    \begin{alignat*}{2}
\can_{S}\circ\can_{S}^{-1}(a\otimes_S b\otimes c) & = \can_{S}(a\otimes_S bc)                                                                             & \quad & \\
                                        & \hspace{-1cm}= a_{(0)}\otimes_S(a_{(1)})_{[1]}\otimes(a_{(1)})_{[2]}bc                                      &       & \\
                                        & \hspace{-1cm}= a_{(0)}\otimes_S b_{(0)}(b_{(1)})_{[1]}(a_{(1)})_{[1]}\otimes(a_{(1)})_{[2]}(b_{(1)})_{[2]}c &       & \textrm{by~\eqref{eq:kappa_-1:2}}\\
                                        & \hspace{-1cm}= a_{(0)}\otimes_S b_{(0)}(a_{(1)}b_{(1)})_{[1]}\otimes(a_{(1)}b_{(1)})_{[2]}c                 &       & \textrm{by~\eqref{eq:kappa_-1:1}}\\
                                        & \hspace{-1cm}= a\otimes_S b\otimes c                                                                        &       & 
    \end{alignat*}
    Where the last equation follows since \(a\otimes b\in
	(A\otimes_SA)^{co\,H}\) and \({1_H}_{[1]}\otimes{1_H}_{[2]}=1_A\otimes 1_A\).
\end{proof}
A very important property of this construction is that the pair
\[(L(A,H),\delta_{L(A,H)})\]
is the final object in the category (see~\cite[Lem.~3.2]{ps:hopf-bigalois}) in
which objects are \(\R\)-modules \(M\) together with a right \(H\)-colinear
map \(\delta_M:A\sir M\otimes A\), and where morphisms are commutative
triangles:
\begin{center}
    \begin{tikzpicture}
	\matrix[matrix of nodes, column sep=1cm, row sep=1cm]{
	    & |(A)| \(A\) & \\
	    |(B1)| \(M\otimes A\) & & |(B2)| \(N\otimes A\) \\
	};
	\draw[->] (A) --node[above left]{\(\delta_M\)} (B1);
	\draw[->] (A) --node[above right]{\(\delta_N\)} (B2);
	\draw[->] (B1) --node[below]{\(f\otimes\id_A\)} (B2);
    \end{tikzpicture}
\end{center}
Based on the above universal property Schauenburg shows the following
\begin{proposition}[{\cite[Prop.~3.5]{ps:hopf-bigalois}}]\label{prop:Schaueburg_universal_prop}
    Let \(A\) be a faithfully flat \(H\)-Galois extension of the commutative
    base ring \(\R\).  Let \(B\) be a bialgebra which turns \(A\) into
    a \(B\)-\(H\)-biGalois extension.  Then there exists a unique isomorphism
    \(f:L(A,H)\sir B\) compatible with the \(B\)-comodule and
    \(L(A,H)\)-comodule structures on \(A\).
\end{proposition}
Let us now state the Schauenburg construction of the Galois correspondence,
which we are going to generalise in a subsequent section.
\index{Schauenburg correspondence}
\index{biGalois extension!Schauenburg correspondence}
\begin{proposition}[{\cite[Prop.~3.2]{ps:gal-cor-hopf-bigal}}]\label{prop:Schauenburg_existence}
    Let \(A/\R\) be a faithfully flat \(H\)-Galois extension and let us
    consider \(L=L(A,H)\).  Then there exists a Galois correspondence:
    \begin{equation}\label{eq:Schauenburg_galois_corr}
	\Sub_{\Alg^H}(A)\galois{}{}\qquot(L(A,H))
    \end{equation}
    where \(\Sub_{\Alg^H}(A)\) is the poset of \(H\)-subcomodule algebras of
    \(A\), a generalised quotient \(Q\in\qquot(L(A,H))\) is mapped to \(^{co
	    Q}A\coloneq\{a\in A:\delta_{L(A,H)}(a)=1_{L(A,H)}\otimes a\}\);
    while a subalgebra \(S\in\Sub_\Alg(A)\) is mapped to
    \((A\otimes_SA)^{co\,H}\in\qquot(L(A,H))\).
\end{proposition}
Further, Schauenburg defines interesting classes of subalgebras and
generalised quotients.  Then he proves that these objects embed into closed
elements of the above Galois correspondence.  Later on we will amend this
definition, so that it will work also for \(H\)-extensions with
a noncommutative subalgebra of coinvariants.
\index{bialgebra!generalised quotients!admissibility (Schauenburg)}
\index{admissible quotients!Schauenburg version}
\index{admissible subalgebras!Schauenburg version}
\begin{definition}\label{defi:Schauenburg_admissibility}
    \begin{enumerate}
	\item Let \(C\) be a coalgebra and let \(C\sir D\) be a coalgebra
	    quotient.  We call \(D\) \textbf{left} (\textbf{right})
	    \textbf{admissible} if it is flat as an \(\R\)-module and \(C\) is
	    right (left) faithfully flat as a \(D\)-comodule.
	\item Let \(S\in\Sub_\Alg(A)\), then \(S\) is called \textbf{left}
	      (\textbf{right}) \textbf{admissible} if \(A\) is faithfully flat
	      as a left (right) \(S\)-module.
    \end{enumerate}
    \noindent Admissible objects are defined as the ones which are both left and right
    admissible.
\end{definition}
The definition of admissibility for subalgebras will have to be altered in the
next section.  We will use the following result:
\begin{proposition}[{\cite[Prop.~3.5]{ps:gal-cor-hopf-bigal}}]\label{prop:Schauneburg_admissibility_and_op}
    Let \(H\) be an \(\R\)-flat Hopf algebra with a bijective antipode.  Then
    the map: 
    \[\qquot(H)\ni H/I\selmap{}(H/I)^\op\coloneq H/S_H(I)\in\qquot(H^\op)\] 
    is a bijection with inverse \(\qquot(H^\op)\ni
	H/I\selmap{}(H/I)^\op\coloneq H/S_{H^\op}(I)\in\qquot(H^\op)\).  The
    quotient \(Q^\op\in\qquot(H^\op)\) is right (left) admissible if and only
    if \(Q\in\qquot(H)\) is left (right) admissible. 
\end{proposition}

The main theorem concerning closed elements
of~\eqref{eq:Schauenburg_galois_corr} is the following:
\index{Schauenburg correspondence!admissibility}
\begin{theorem}[{\cite[Thm.~3.6]{ps:gal-cor-hopf-bigal}}]\label{thm:Schauenburg_Galois_closedness}
    Let \(H\) be a Hopf algebra with a bijective antipode and let \(A\) be
    a faithfully flat \(H\)-Galois extension of the base ring \(\R\).  Then
    the Galois correspondence~\eqref{eq:Schauenburg_galois_corr} restricts to
    bijection between (left, right) admissible generalised quotients of
    \(L(A,H)\) and \(H\)-subcomodule algebras \(S\) of \(A\) such that \(A\)
    is (left, right) faithfully flat over~\(S\).
\end{theorem}

\section{Galois correspondence for $H$-extensions}
We begin this section with a construction of a Galois connection between
generalised subalgebras and generalised quotients of a bialgebra which is flat
as an \(R\)-module.  This Galois connection, as reported
by~\citeauthor{ps:gal-cor-hopf-bigal} is folklore.  We put the proof here,
since it is not written elsewhere in the generality that we will need.
\index{Takeuchi correspondence!Galois connection}
\begin{proposition}
    Let \(H\) be a bialgebra, which is flat over a commutative base ring
    \(R\).  Then the maps:
	\begin{equation}\label{eq:galois-for-hopf-alg}
	    \begin{array}{ccc}
		\qsub(H)&\hspace{-.3cm}\galois{\psi}{\phi}&\qquot(H)
	    \end{array}
	\end{equation}
    where $\phi(Q)=H^{co\,Q},\psi(K)=H/K^+H$ form a \textsf{Galois
	connection}.
\end{proposition}
\begin{proof}
    First we show that the maps~\eqref{eq:galois-for-hopf-alg} are well
    defined and then we show that they form a Galois connection.  Let
    \(K\in\qsub(H)\) be a left coideal subalgebra, then \(K^+\) is a coideal:
    let us consider \(x=\Delta(k)-k\otimes 1\in H\otimes K\), where \(k\in
	K^+\).  Since \(\id\otimes\epsilon(x)=0\) and \(H\) is flat, we have
    \(x\in H\otimes K^+\).  Moreover \(x-1\otimes k\in H^+\otimes K^+\),
    because \(\epsilon\otimes\id(x)=k\).  Thus \(\Delta(k)-k\otimes1-1\otimes
	k\in H^+\otimes K^+\).  Hence \(K^+H\) is a coideal and a right
    \(H\)-ideal.  Now let \(Q\in\qquot(H)\) and let \(\pi\) be the canonical
    surjection \(\pi:H\sir Q\).  Then \(H^{\co\,Q}\) is a subalgebra: for
    \(x,y\in H^{\co\,Q}\) we have
    \begin{align*}
	(xy)_{(1)}\otimes\pi((xy)_{(2)}) & = x_{(1)}y_{(1)}\otimes\pi(x_{(2)}y_{(2)}) \\
	                         & = x_{(1)}y_{(1)}\otimes\pi(x_{(2)})y_{(2)} \\
	                         & = xy_{(1)}\otimes\pi(1)y_{(2)} \\
	                         & = xy_{(1)}\otimes\pi(y_{(2)})\\
	                         & = xy\otimes\pi(1)
    \end{align*}
    It remains to show that it is a left coideal.  For this let \(x\in
	H^{\co\,Q}\).  One easily gets that
    \[x_{(1)}\otimes {x_{(2)}}_{(1)}\otimes\pi({x_{(2)}}_{(2)})=x_{(1)}\otimes x_{(2)}\otimes\pi(1)\]
    Hence \(x_{(1)}\otimes x_{(2)}\in(H\otimes
	H)^{\co\,Q}=H\otimes(H^{\co\,Q})\), since \(H\) is a flat
    \(R\)-module.  This shows that \(H^{\co\,Q}\) is a left \(H\)-comodule
    subalgebra.

    Now let us show that these two maps indeed define a Galois connection.
    Let \(Q=H/I\in\qquot(H)\) and let
    \(\pi:H\sir Q\) be the projection map.  Then we want to show that
    \((H^{\co\,Q})^+H\subseteq I\).  For this let \(\sum_i x_iy_i\in
	(H^{\co\,Q})^+H\) be such that each \(x_i\in (H^{\co\,Q})^+\).  Then
    we have:
    \begin{equation*}
	\begin{split}
	    \pi\bigl(\sum_ix_iy_i\bigr) & =\epsilon({x_i}_{(1)}{y_i}_{(1)})\pi({x_i}_{(2)}{y_i}_{(2)}) \\
	                    & =\epsilon({x_i}_{(1)}{y_i}_{(1)})\pi({x_i}_{(2)}){y_i}_{(2)} \\
	                    & =\epsilon(x_i{y_i}_{(1)})\pi({y_i}_{(2)}) \\
	                    & =0
	\end{split}
    \end{equation*}
    Now let us take a \(K\in\qsub(H)\), and let \(\pi_K:H\sir H/K^+H\) be the
    natural projection.  Then for every \(k\in K\) we have:
    \[k_{(1)}\otimes\pi_K(k_{(2)})-k\otimes\pi_K(1)=k_{(1)}\otimes\pi_K(k_{(2)}-\epsilon(k_{(2)})1)=0\]
    Thus \(K\subseteq H^{\co\,H/K^+H}\), and indeed
    the maps~\eqref{eq:galois-for-hopf-alg} define a Galois connection.
\end{proof}

Now we prove existence of the Galois correspondence for comodule algebras
without the restriction of Schauenburg's approach.  Afterwards we will also
show how the result can be improved if the base ring happens to be a field.
In this case the correspondence has a very similar form to the correspondence
of the previous Proposition.
\index{comodule algebra!Galois correspondence}
\begin{theorem}[{Galois correspondence for \(H\)-comodule algebras}]\label{thm:existence} 
    Let \(H\) be a bialgebra and let \(A/B\) be an \(H\)-extension over
    a commutative ring \(\R\) such that \(A\) and \(H\) are flat
    Mittag--Leffler \(\R\)-modules.  Then we have a \textsf{Galois
	connection}:
    \begin{equation}\label{eq:galois-connection}
	\Sub_{\Alg}(A/B)\,\galois{\psi}{\phi}\,\qquot(H)
    \end{equation}
    where \(\phi(Q):=A^{co\,Q}\) and \(\psi\) is given by the following
    formula: 
    \begin{equation}\label{eq:psi}
	\psi(S)=\bigvee\{Q\in\qquot(H): S\subseteq A^{co\,Q}\}
    \end{equation}
    for \(S\in\Sub_\Alg(A/B)\).
\end{theorem}
\begin{proof}
    We shall show that $\phi$ reflects all suprema: \(A^{co\,\bigvee_{i\in
		I}\, Q_i}=\bigcap_{i\in I}\,A^{co\,Q_i}\).  From the set of
    inequalities: \(\bigvee_{i\in I}\,Q_i\succcurlyeq Q_j\ (\forall_{j\in
	    I})\) it follows that \(A^{co\,\bigvee_{i\in I}\,
	    Q_i}\subseteq\bigcap_{i\in I}\,A^{co\,Q_i}\).  Let us fix an
    element $a\in\bigcap_{i\in I}\,A^{co\,Q_i}$.  We let $I_i$ denote the
    coideal and right ideal such that $Q_i=H/I_i$.  We identify $A\otimes I_i$
    with a submodule of $A\otimes H$, which can be done under the assumption
    that $A$ is flat over \(\R\).  We want to show that:
    \begin{equation*}
	\begin{split}	    
	    \forall_{i\in I}\;a\in A^{co\,Q_i}\ &\Leftrightarrow\ \forall_{i\in I}\;\delta(a) - a\otimes 1\in A\otimes I_i \\
	    &\Leftrightarrow\ \delta(a)-a\otimes 1\in A\otimes\bigcap_{i\in I}\,I_i\\
	    &\Leftrightarrow\ a\in A^{co\,\bigvee_{i\in I} Q_i}
    \end{split}
    \end{equation*}
    The first equivalence is clear, the second follows from the equality:
    $\bigcap_{i\in I}A\otimes I_i=A\otimes \bigcap_{i\in I} I_i$ which holds
    since flat Mittag--Leffler modules have the intersection property (cf.
    Corollary~\ref{cor:mittag-leffler}).  It remains to show that if
    $\delta(a)-a\otimes 1\in A\otimes\bigcap_{i\in I}I_i$ then
    \mbox{$\delta(a)-a\otimes 1\in A\otimes\bigwedge_{i\in I}I_i$.} The other
    inclusion is clear, i.e.  \(A\otimes\bigwedge_{i\in I}I_{i}\subseteq
	A\otimes\bigcap_{i\in I}I_{i}\) (see formula~\eqref{eq:meet_in_qid} on
    page~\pageref{eq:meet_in_qid}).  We proceed in three steps.  We first
    prove this for \(H\), then for \(A\otimes H\) and finally for a general
    comodule algebra \(A\).
    \begin{enumerate}
	\item[\textbf{(i)}]\label{item:a} For \(A=H\) this follows using the
	     Galois connection~\eqref{eq:galois-for-hopf-alg}.  Let \(h\in H\)
	     be such that \(\Delta(h)-h\otimes 1\in H\otimes\bigcap I_i\).
	     Hence for all \(i\) we have \(\Delta(h)-h\otimes 1\in H\otimes
		 I_i\), since \(H\) is a flat module, and as a consequence
	     \(h\in \bigcap_{i\in I}H^{\co\,Q_i}\).  Let \(K=\bigcap_{i\in
		     I}H^{\co\,Q_i}\in\qsub(H)\).  Now by the Galois
	     connection property for all \(i\in I\) we have: \(Q_i\leq
		 H/{H^{\co\,Q_i}}^+H\leq H/K^+H\).  Thus \(\bigvee_{i\in
		     I}Q_i\leq H/K^+H\) and as a consequence:
	     \[\bigcap_{i\in I}H^{\co\,Q_i}=K\subseteq H^{\co\,H/K^+H}\subseteq H^{\co\,\bigvee Q_i}\] 
	     Hence \(\Delta(h)-h\otimes 1\in H\otimes\bigwedge_{i\in I}I_i\). 
	\item[\textbf{(ii)}] Now the proof for \(A\otimes H\): let
	     \(x\coloneq\sum_{k=1}^na_k\otimes h_k\in A\otimes H\) be such
	     that
	     \(y\coloneq\sum_{k=1}^na_k\otimes\Delta(h_k)-\sum_{k=1}^na_k\otimes
		 h_k\otimes 1_H\in A\otimes H\otimes \bigcap_{i\in I}I_i\).
	     Let us write \(y=\sum_{k=1}^m a_k'\otimes h_k'\otimes l_k'\)
	     where \(a_k'\in A\), \(h_k'\in H\), \(l_k'\in\bigcap_{i\in
		     I}I_i\) for each \(k=1,\dots,m\). Since \(A\) is a flat
	     Mittag--Leffler module, every finitely generated submodule of
	     \(A\) is contained in a projective
	     submodule~\cite[Thm~2.9]{dh-jt:mittag-leffler}.  Let \(A_0\) be
	     a projective submodule of \(A\) which contains both
	     \(\{a_k:k=1,\dots,n\}\) and \(\{a_k':k=1,\dots,m\}\).  By
	     projectivity of \(A_0\) we can choose a "dual basis", i.e. a set
	     of elements \(e_j\in A_0\) and their "duals" \(e^j\in A_0^*\),
	     such that \(e^j(a)\) is non zero only for finitely many \(j\in J\)
	     and \(\sum_{j\in J}e_je^j(a)=a\) for every \(a\in A_0\).
	     Using~(i) above we get:
	    \[\sum_{k=1}^ne^j(a_k)\Delta(h_k)-e^j(a_k)h_k\otimes1=\sum_{k=1}^me^j(a_k')h_k'\otimes l_k'\in H\otimes\bigwedge_{i\in I}I_i\]
	    and by the dual basis property we conclude that:
	    \[a_k\otimes\Delta(h_k)-a_k\otimes h_k\otimes1\in A_0\otimes H\otimes\bigwedge_{i\in I}I_i\]
	    It follows that \(y\in A\otimes H\otimes\bigwedge_{i\in I}I_i\),
	    as we claimed.  
	\item \textit{The general case}.  Let us observe that if
	    \(a_{\mathit{(0)}}\otimes a_{\mathit{(1)}}-a\otimes 1\in
    A\otimes\bigcap_{i\in I}I_i$ then $a_{\mathit{(0)}}\otimes
		a_{\mathit{(1)}}\otimes
		a_{\mathit{(2)}}-a_{\mathit{(0)}}\otimes
		a_{\mathit{(1)}}\otimes 1\in A\otimes H\otimes\bigcap_{i\in
		    I}I_i\).  Now by (ii): \(a_{\mathit{(0)}}\otimes
		a_{\mathit{(1)}}\otimes
		a_{\mathit{(2)}}-a_{\mathit{(0)}}\otimes
		a_{\mathit{(1)}}\otimes 1\in A\otimes H\otimes\bigwedge_{i\in
		    I}I_i\).  Computing \(\id_A\otimes\epsilon\otimes\id_H\)
	    we get \(\delta(a)-a\otimes 1\in A\otimes\bigwedge_{i\in I}I_i\). 
   \end{enumerate}
 
    The formula \(\psi(S)=\bigvee\{Q\in\qquot(H): S\subseteq A^{co\,Q}\}\) is
    an easy consequence of the Galois connection properties (cf.
    formula~\ref{eq:right-adjoint} on page~\pageref{eq:right-adjoint}).
\end{proof} 
Note that the above theorem also holds for left \(H\)-comodule algebras rather
than right ones.
\begin{remark}\label{rem:comparison_of_adjunctions}
    If \(B\) is commutative and \(H\) is a projective over the base ring then
    \(A/B\) is projective \((B\otimes H)\)-Hopf Galois extension
    by~\cite[Thm~1.7]{hk-mt:hopf-algebras-and-galois-extensions}.  Thus the
    above theorem applies as well as Schauenburg's result
    Proposition~\ref{prop:Schauenburg_existence} (page
    \pageref{prop:Schauenburg_existence}).  Note that Schauenburg constructs
    a Hopf algebra \(L\) such that \(A/B\) is an \(L\text{-}H\)-biGalois
    extension and he constructs a Galois connection between quotients of \(L\)
    and intermediate \(H\)-comodule subalgebras of \(A/B\).  The value of his
    approach is that the construction of \(L\) allows us to prove more about
    the correspondence (cf.~\cite[Thm.~3.8]{ps:gal-cor-hopf-bigal}).

    Let us note that if we apply the above theorem to the left
    \(L(A,H)\)-comodule algebra we get the Schauenburg adjunction.
    Coinvariants of quotients of \(L(A,H)\) are always \(H\)-subcomodule
    subalgebras, since \(A\) is an \(L(A,H)\)-\(H\)-bicomodule algebra. Now,
    since we assume that \(H\) is a flat Mittag--Leffler module, the inclusion
    of \(H\)-comodule subalgebras of \(A\) into subalgebras
    \(\Sub_{\Alg^H}(A)\subseteq\Sub_\Alg(A)\) preserves infinite intersections
    by Theorem~\ref{thm:lattice_of_subcomodules}.  We conclude that both
    adjunctions coincide by the uniqueness of Galois connections
    (Proposition~\ref{prop:properties-of-adjunction} on
    page~\pageref{prop:properties-of-adjunction}) .

    Finally, let us note that in the Schauenburg context, the coinvariants of
    generalised quotients of \(H\) are always \(L(A,H)\)-comodule subalgebras.
\end{remark}
\index{comodule algebra!Galois correspondence!over a field}
\begin{theorem}\label{thm:connection_over_field}
    Let \(A/B\) an \(H\)-extension over a base field \(\k\), with the comodule
    structure map \(\delta_A:A\sir A\otimes H\) and let
    \(S\in\Sub_\Alg(A/B)\).  Then \(\psi(S)=H/K_S^+H\), where \(K_S\) is the
    smallest element of \(\qsub(H)\) such that \(\delta_A(S)\subseteq A\otimes
	K_S\), i.e.
    \begin{equation}\label{eq:psi_over_field}
	K_S\coloneq\bigcap\bigl\{K\in\qsub(H): \delta_A(S)\subseteq A\otimes K\bigr\}
    \end{equation}
\end{theorem}
\begin{proof}
    It is enough to show that \(\phi:\qquot(H)\ni
	Q\selmap{}A^{co\,Q}\in\Sub_\Alg(A/B)\) and \(\psi':\Sub_\Alg(A/B)\ni
	S\selmap{}H/K_S^+H\in\qquot(H)\) form a Galois connection, by
    Proposition~\ref{prop:properties-of-adjunction}(iii).  Let us note that
    \(K_S\in\qsub(H)\), as given by formula~\eqref{eq:psi_over_field}, is well
    defined since the infimum in \(\qsub(H)\) is the set theoretic
    intersection, which follows from
    Proposition~\ref{thm:lattice_of_subcomodules} on
    page~\pageref{thm:lattice_of_subcomodules}.  First we show that
    \(S\subseteq A^{co\,H/K_S^+H}\).  Let \(s\in S\) and let \(\pi:H\sir
	H/K_S^+H\) be the canonical projection.  Now, by definition of \(K_S\),
    we have \(\delta_A(s)\in A\otimes K_S\).  We get the following equality in
    \(A\otimes H/K_S^+H\): \(s_{(0)}\otimes\pi(s_{(1)})=s_{(0)}\otimes
	\epsilon(s_{(1)})\pi(1)=s\otimes\pi(1)\).  Thus \(S\subseteq
	A^{co\,H/K_S^+H}\).  Now, let \(Q\in\qquot(H)\) and let
    \(K=H^{co\,Q}\).  Then \(\delta_A(A^{co\,Q})\subseteq A\otimes K\), by
    commutativity of the following diagram:
    \begin{center}
	\begin{tikzpicture}
	    \matrix[column sep=1.5cm, row sep=1cm]{
		\node (A1) {\(A^{co\,Q}\)}; & \node (A2) {\(A\otimes H\)}; & \node (A3) {\(A\otimes H\otimes H\)}; \\
		                            & \node (B2) {\(A\otimes Q\)}; & \node (B3) {\(A\otimes H\otimes Q\)};\\
	    };
	    \draw[>->] (A1) --node[above]{\(\delta_A\)} (A2);
	    \draw[->] ($(A2)+(6mm,1mm)$) --node[above]{\(\id_A\otimes\Delta\)} ($(A3)+(-1.1cm,1mm)$);
	    \draw[->] ($(A2)+(6mm,-1mm)$) --node[below]{\(\delta_A\otimes\id_H\)} ($(A3)+(-1.1cm,-1mm)$);
	    \draw[->] (A2) --node[left]{\(\id_A\otimes\pi\)} (B2);
	    \draw[->] (B2) --node[below]{\(\delta_A\otimes\id_Q\)} (B3);
	    \draw[->] (A3) --node[right]{\(\id_A\otimes\id_H\otimes\pi\)} (B3);
	\end{tikzpicture}
    \end{center}
    For \(a\in A^{co\,Q}\) we get \(a_{(0)}\otimes
	a_{(1)}\otimes\pi(a_{(2)})=a_{(0)}\otimes a_{(1)}\otimes\pi(1)\).
    Since we can choose all the elements \(a_{(0)}\) to be linearly
    independent we get that \(\delta_A(A^{\co\,Q})\subseteq A\otimes K\) and
    hence \(K_{A^{\co\,Q}}\subseteq K=H^{\co\,Q}\).  Now we get
    \(H/K_{A^{\co\,Q}}^+H\succcurlyeq H/K^+H\succcurlyeq Q\), where the last
    inequality follows from Theorem~\ref{thm:newTakeuchi} on
    page~\pageref{thm:newTakeuchi}.
\end{proof}
Note that if \(H\) is finite dimensional and \(Q\) is a Hopf algebra quotient
then \(A^{co\,Q}=\delta_A^{-1}(A\otimes H^{co\,Q})\)
(by~\cite[Thm.~1.7]{hs:hopf_galois_extensions}).  Let us note two similarities
in this proof and the previous one.  The existence of \(K_S\) is guaranteed by
flatness and the Mittag--Leffler property of \(A\) (which over fields holds
trivially).  This is due to Proposition~\ref{prop:ML_and_limits} (see also
Proposition~\ref{prop:ML_and_coefficient_modules} on
page~\pageref{prop:ML_and_coefficient_modules}).  Also note that in both proofs
we use the Galois connection~\ref{eq:galois-for-hopf-alg} (which can be proven
independently, cf.~\cite[Theorem~3.10]{ps:gal-cor-hopf-bigal}).
\index{comodule algebra!examples!G-graded algebra}
\begin{example}
    Let us take \(A\) to be the \(\k[G]\)-comodule \(\k\)-algebra defined in
    Example~\ref{ex:comodule_algebras}(iii).  The generalised quotients of
    \(\k[G]\) are of the form \(\pi:\k[G]\sir\k[G/G_0]\) for some subgroup
    \(G_0\) of \(G\) by Proposition~\ref{prop:group_algebra_quotients}
    (page~\pageref{prop:group_algebra_quotients}).  The correspondence:
    \[\Sub_\Alg(A/A_e)\,\galois{\psi}{\phi}\,\qquot(\k[G])\]
    where \(e\in G\) is the unit of \(G\), is given by
    \(\phi(\k[G/G_0])=\sum_{g\in G_0}A_g\) and \(\psi(B)\) corresponds to the
    smallest subgroup \(G_0\subseteq G\) such that \(B\subseteq\sum_{g\in
	    G_0}A_g\). 
    We can see that in this example \(\psi\phi=\id_{\qquot(\k[G])}\), but \(\phi\psi\) might
    not be an identity morphism.
\end{example}
\index{comodule algebra!Galois correspondence!field extension}
\begin{remark}
    Let \(\bE/\bF\) be a finite field extension.  Let
    \mbox{$\Sub_{\textit{field}}(\bF\subseteq \bE)$} denote the lattice of
    subfields of a field $\bE$ containing a subfield $\bF$ and
    \mbox{$\Sub_{\textit{ring}}(\bF\subseteq \bE)$} is the lattice of all
    intermediate subrings.  There is the following diagram of Galois
    connections in which the upper Galois connection is the classical one and
    the lower one is the Galois connection~\eqref{eq:galois-connection}. 
    \begin{center}
	\begin{tikzpicture}
	    \node[anchor=east] (A) at (0cm,0cm) {$\Sub_{\textit{field}}(\bE/\bF)$};
	    \node[anchor=west] (B) at (1cm,0cm) {$\Sub(\Gal\left(\bE/\bF\right))$};
	    \node[anchor=east] (C) at (0cm,-1.5cm) {$\Sub_{\textit{ring}}(\bE/\bF)$};
	    \node[anchor=west] (D) at (1cm,-1.5cm) {$\qquot(\bF\left[\Gal(\bE/\bF)\right]^*)$};
	    \draw[->] (-13mm,-3mm) -- (-13mm,-12mm);
	    \draw[->] (20mm,-3mm) --node[right]{\(\cong\)} (20mm,-12mm);
	    \draw[->,yshift=.6mm] (0cm,0cm) --node[above]{$\Gal$} (1cm,0cm);
	    \draw[<-,yshift=-.6mm] (0cm,0cm) --node[below]{$\Fix$} (1cm,0cm);
	    \draw[->,yshift=.6mm] (0cm,-1.5cm) --node[above]{$\psi$} (1cm,-1.5cm);
	    \draw[<-,yshift=-.6mm] (0cm,-1.5cm) --node[below]{$\phi$} (1cm,-1.5cm);
	\end{tikzpicture}
    \end{center}
    where the right vertical map is the isomorphism of
    Proposition~\ref{prop:dual_gruop_algebra_quotients} (on
    page~\pageref{prop:dual_gruop_algebra_quotients}) and
    \((\mathsf{Fix},\mathsf{Gal})\) is the classical adjunction of Galois
    theory of finite field extensions.  Commutativity of this diagram follows
    from the formula:
    \[\bE^{co\,\bF[G']^*}=\mathsf{Fix}(G')\] 
    It follows that $\phi$ factorises through the embedding
    $\Sub_{\textit{field}}(\bF\subseteq \bE)\subseteq
    \Sub_{\textit{ring}}(\bE/\bF)$, and thus the only closed elements of the
    Galois connection $(\phi,\psi)$ in $\Sub_{\textit{ring}}(\bE/\bF)$ are the
    ones coming from closed elements of the lattice
    $\Sub_{\textit{field}}(\bE/\bF)$.  Normal subgroups of the Galois group
    correspond to conormal quotients of $H$, i.e. quotients by a normal ideal,
    and also to normal subextensions. 
\end{remark}
\index{comodule algebra!Galois correspondence!example}
\begin{example}
    Let \(q\) be a primitive fourth root of unity and let \(\k=\bQ[q]\).  Let
    \(A\) be the \(\k\)-algebra generated by \(J,X,Z\) subject to the
    relations: \(XZ=qZX\), \(X^4=Z^{16}=1\),  \(J^2=-1\) (\(J\) commutes with
    both \(X\) and \(Z\)).  We let \(Y=Z^4\), and let \(B\) be the subalgebra
    generated by \(X\) and \(Y\).  Now let us construct the following Hopf
    algebra \(H\).  As an algebra it is generated by the elements \(c,s,t\)
    which satisfy the following relations:
    \begin{alignat*}{2}
            c^2+s^2 & = 1,  & \quad cs & = 0, \\
            ct      & = tc, & ts       & =-st \\
            t^2     & = 1
    \end{alignat*}
    The coalgebra structure is given by:
    \begin{align*}
	\Delta(c) & = c\otimes c-s\otimes s \\
	\Delta(s) & = s\otimes c+c\otimes s \\
	\Delta(t) & = t\otimes t \\
    \end{align*}
    Coassociativity of the comultiplication is a straightforward calculation:
    \begin{align*}
	\Delta\otimes\id\circ\Delta(c) & = c\otimes c\otimes c-s\otimes s\otimes c-s\otimes c\otimes s-c\otimes s\otimes s \\
	\id\otimes\Delta\circ\Delta(c) & = c\otimes c\otimes c -c\otimes s\otimes s-s\otimes c\otimes s -s\otimes s\otimes c
    \end{align*}
    and thus
    \(\Delta\otimes\id\circ\Delta(c)=\id\otimes\Delta\circ\Delta(c)\),
    similarly
    \begin{align*}
	\Delta\otimes\id\circ\Delta(s) & = c\otimes c\otimes s-s\otimes s\otimes s+s\otimes c\otimes c+c\otimes s\otimes c \\ 
	\id\otimes\Delta\circ\Delta(s) & = c\otimes c\otimes s+ c\otimes s\otimes c + s\otimes c\otimes c-s\otimes s\otimes s
    \end{align*}
    Thus \(\Delta\otimes\id\circ\Delta(s)=\id\otimes\Delta\circ\Delta(s)\).
    The counit is given by \(\epsilon(c)=1\), \(\epsilon(s)=0\) and
    \(\epsilon(t)=1\).  Finally, the antipode is defined by \(S(c)=c\),
    \(S(s)=-s\) and \(S(t)=t\).  Let us note that \(\dim_\k H=8\), the set
    \[1,t,c,s,c^2-s^2,tc,ts,t(c^2-s^2)\]
    is a linear basis.  There are four group-like elements in \(H\):
    \(1,t,c^2-s^2,t(c^2-s^2)\).  They form an abelian group isomorphic to
    \(\bZ_2\times\bZ_2\).  The Hopf subalgebra \(H\) generated by \(c\) and
    \(s\) is called \textit{circle Hopf algebra}.  The name
    \textit{trigonometric Hopf algebra} is also met in the literature.
    Furthermore, let us note that \(\k[i]\otimes_\k H\) is isomorphic to the
    group Hopf algebra \(\k[i][D_8]\), where \(D_8\) is the dihedral group of
    order \(8\), i.e. the group of symmetries of a rectangle.  For this note
    that \(c-is\) is a group-like element of order \(4\) such that
    \(t(c-is)t=(c-is)^3\).  The claim follows since \(\k[i]\otimes_\k H\) has
    \(8\) group-like elements and its dimension over \(\k[i]\) is \(8\).  In
    this identification \(t\) corresponds to a symmetry and \(c-is\) is
    a rotation.  The lattice of generalised subalgebras of \(H\) and the
    lattice of subgroups of \(D_8\) are drawn below.  We let \(\tau\) and
    \(\sigma\) be the generators of \(D_8\), such that \(\tau^2=1\),
    \(\sigma^4=1\), \(\tau\sigma\tau=\sigma^3\).
    \begin{center}
	\hspace*{-4cm}\hbox{
	    \vbox{
		\begin{tikzpicture}
		    \node (root) at (0,0) {\(1\)};
		    \node (a1) at (-3,1.5) {\(\langle\tau\rangle\)};
		    \node (a2) at (-1.5,1.5) {\(\langle\sigma^2\tau\rangle\)};
		    \node (a3) at (0,1.5) {\(\langle\sigma^2\rangle\)};
		    \node (a4) at (1.5,1.5) {\(\langle\sigma\tau\rangle\)};
		    \node (a5) at (3,1.5) {\(\langle\sigma^3\tau\rangle\)};
		    \draw[-] (root) -- (a1);
		    \draw[-] (root) -- (a2);
		    \draw[-] (root) -- (a3);
		    \draw[-] (root) -- (a4);
		    \draw[-] (root) -- (a5);

		    \node (b2) at (-1.75,3) {\(\langle \sigma^2,\tau\rangle\)};
		    \node (b3) at (0,3) {\(\langle\sigma\rangle\)};
		    \node (b4) at (1.75,3) {\(\langle \sigma^2,\sigma\tau\rangle\)};
		    \draw[-] (a1) -- (b2);
		    \draw[-] (a2) -- (b2);
		    \draw[-] (a3) -- (b2);
		    \draw[-] (a3) -- (b3);
		    \draw[-] (a3) -- (b4);
		    \draw[-] (a4) -- (b4);
		    \draw[-] (a5) -- (b4);

		    \node (D) at (0,4.5) {\(D_8\)};
		    \draw[-] (b2) -- (D);
		    \draw[-] (b3) -- (D);
		    \draw[-] (b4) -- (D);
		\end{tikzpicture}\\
		{\textit{Lattice of subgroups of \(D_8\)}}
	    }
	    \hspace*{-5.5cm}
	    \vbox{
		\begin{tikzpicture}
		    \node (root) at (0,0) {\(1\)};
		    \node (a1) at (-4,1.5) {\(\langle t\rangle\)};
		    \node (a2) at (-2,1.5) {\(\langle(c^2-s^2)t\rangle\)};
		    \node (a3) at (0,1.5) {\(\langle c^2-s^2\rangle\)};
		    \draw[-] (root) -- (a1);
		    \draw[-] (root) -- (a2);
		    \draw[-] (root) -- (a3);

		    \node (b2) at (-2.5,3) {\(\langle c^2-s^2,t\rangle\)};
		    \node (b3) at (0,3) {\(\langle c,s\rangle\)};
		    \node (b4) at (2.5,3) {\(\langle ct,st\rangle\)};
		    \draw[-] (a1) -- (b2);
		    \draw[-] (a2) -- (b2);
		    \draw[-] (a3) -- (b2);
		    \draw[-] (a3) -- (b3);
		    \draw[-] (a3) -- (b4);

		    \node (D) at (0,4.5) {\(H\)};
		    \draw[-] (b2) -- (D);
		    \draw[-] (b3) -- (D);
		    \draw[-] (b4) -- (D);
		\end{tikzpicture}\\
		{\textit{Lattice of generalised subalgebras of \(H\)}}
	    }
	}
    \end{center}
    The notation \(\langle\cdots\rangle\) denotes the subgroup/left coideal
    subalgebra generated by the listed elements\footnote{The meet in the
	lattice of right coideal subalgebras of a Hopf algebra, which is
	a flat Mittag--Leffler module, is the set theoretic intersection.
	Thus we can define the left coideal subalgebra of \(H\) generated by
	some set of elements as the intersection of all left coideal
	subalgebras which contains these elements.  This left coideal
	subalgebra will contain the elements by which it is generated.  Also
	note that if the span of generating elements is a left coideal then
	the smallest subalgebra which contains it is a left coideal
	subalgebra.}.  For example the coideal subalgebra which is generated
    by \(c^2-s^2\) is spanned by \(c^2\), \(s^2\) and \(1\).  The left coideal
    subalgebra \(\langle ct,st\rangle\) also contains \(c^2=(ct)^2\),
    \(s^2=st(-st)\).  Thus it contains \(c^2-s^2\).  The lattice of
    generalised subalgebras of \(H\) is anti-isomorphic with the lattice of
    generalised quotients of its dual \(H^*\).

    The action of \(H\) on \(A\) is defined in table~\ref{tab:H-action} below.
    \begin{table}[h]
	\begin{center}
	    \begin{tabular}{c|ccccccc}
			& \(1\) & \(X^j\) & \(J\)  &\(Z\)  & \(Z^2\)  & \(Z^3\)  & \(Y\) \\\midrule
		\(c\)   & \(1\) & \(X^j\) & \(J\)  &\(0\)  & \(-Z^2\) & \(0\)    & \(Y\) \\
		\(s\)   & \(0\) & \(0\)   & \(0\)  &\(-Z\) & \(0\)    & \(Z^3\)  & \(0\) \\
		\(t\)   & \(1\) & \(X^j\) & \(-J\) &\(Z\)  & \(Z^2\)  & \(Z^3\)  & \(Y\) \\
		\(c^2\) & \(1\) & \(X^j\) & \(J\)  &\(0\)  & \(Z^2\)  & \(0\)    & \(Y\) \\
		\(s^2\) & \(0\) & \(0\)   & \(0\)  &\(Z\)  & \(0\)    & \(Z^3\)  & \(0\) \\
	    \end{tabular}
	\end{center}
	\caption{\(H\)-action on \(A\)}
	\label{tab:H-action}
    \end{table}
    This action makes \(A\) an \(H\)-module algebra. The Hopf algebra \(H\) is
    finite dimensional, hence its dual coacts on \(A\).  Moreover, the
    subalgebra of coinvariants of the \(H^*\) coaction is equal to the
    subalgebra of invariants of the \(H\)-action.  We now show that \(A^H\) is
    the subalgebra generated by \(X\).  If \(a=a'+a''J\in A^H\) is an
    invariant element, where both \(a',a''\) belong to the subalgebra
    generated by \(X\) and \(Z\), then we must have \(a''=0\). Let
    \(a=\sum_{i=0}^{15} f_k Z^k\), where each \(f_k\) is a polynomial in
    \(X\).  Hence:
    \begin{equation*}
	0=s\cdot a = \sum_{k=0}^{7} (-1)^{k\,\modulo 2+1}f_{2k+1} Z^{2k+1}
    \end{equation*}
    In this way all the \(f_{2k+1}=0\) (\(0\leq k\leq7\)).  Moreover,
    \begin{equation*}
	a=c\cdot a = \sum_{k=0}^{3} f_{4k} Z^{4k} -\sum_{k=0}^{3} f_{4k+2}Z^{4k+2} 
    \end{equation*}
    It follows that \(f_{4k+2}=0\) for \(k=0,1,2,3\).  Thus \(A^H\subseteq
	B\).  Conversely, every element of \(B\) belongs to \(A^H\) and hence
    \(B=A^H=A^{\co\;H^*}\).  Furthermore, the extension \(A/B\) is
    a \(H^*\)-Galois.  In order to show this we compute the following matrix.
    The columns are values of \(\can:A\otimes_BA\sir\Hom_\k(H,A)\) on the
    elements written in the first row.\\[1ex]
    {
	\hfill\small\hspace*{-0.5mm}\begin{tabular}{c|c|c|c|c|c|c|c|c}
	     & \(1\otimes_B 1\) & \(1\otimes_B J\) & \(1\otimes_B Z\) & \(1\otimes_B JZ\) & \(1\otimes_B Z^2\) & \(1\otimes_BJZ^2\) & \(1\otimes_B Z^3\) & \(1\otimes_B JZ^3\) \\\bottomrule
    \(1_H\)  & \(1\)            & \(J\)            & \(Z\)            & \(JZ\)            & \(Z^2\)            & \(JZ^2\)           & \(Z^3\)            & \(JZ^3\) \\
    \(t\)    & \(1\)            & \(-J\)           & \(Z\)            & \(-JZ\)           & \(Z^2\)            & \(-JZ^2\)          & \(Z^3\)            & \(-JZ^3\) \\
    \(c\)    & \(1\)            & \(J\)            & \(0\)            & \(0\)             & \(-Z^2\)           & \(-JZ^2\)          & \(0\)              & \(0\) \\
    \(s\)    & \(0\)            & \(0\)            & \(-Z\)           & \(-JZ\)           & \(0\)              & \(0\)              & \(Z^3\)            & \(JZ^3\) \\
    \(c^2\)  & \(1\)            & \(J\)            & \(0\)            & \(0\)             & \(Z^2\)            & \(JZ^2\)           & \(0\)              & \(0\) \\
    \(ct\)   & \(1\)            & \(-J\)           & \(0\)            & \(0\)             & \(-Z^2\)           & \(JZ^2\)           & \(0\)              & \(0\) \\
    \(st\)   & \(0\)            & \(0\)            & \(-Z\)           & \(JZ\)            & \(0\)              & \(0\)              & \(Z^3\)            & \(-JZ^3\) \\
    \(c^2t\) & \(1\)            & \(-J\)           & \(0\)            & \(0\)             & \(Z^2\)            & \(-JZ^2\)          & \(0\)              & \(0\) \\
    \end{tabular}\hfill}\\[1ex]
    From this eight by eight matrix we get the following matrix (by
    normalising all the columns):
    {
	\small
    \[M=\left(\begin{array}{cccccccc}
	    1 & 1  & 1  &  1 &  1 &  1 &  1 &  1 \\
	    1 & -1 & 1  & -1 &  1 & -1 &  1 & -1 \\
	    1 & 1  & 0  &  0 & -1 & -1 &  0 &  0 \\
	    0 & 0  & -1 & -1 &  0 &  0 &  1 &  1 \\
	    1 & 1  & 0  &  0 &  1 &  1 &  0 &  0 \\
	    1 & -1 & 0  &  0 & -1 &  1 &  0 &  0 \\
	    0 & 0  & -1 & -1 &  0 &  0 &  1 & -1 \\
	    1 & 1  & 0  &  0 &  1 & -1 &  0 &  0 \\
	\end{array}\right)\]
    } Since \(\det M\neq 0\) (which we have verified using
    \textit{Python}\footnote{We used the
	\href{http://docs.scipy.org/doc/numpy}{\texttt{numpy.linalg}.} module.
	See \url{http://docs.scipy.org/doc/numpy}}) the map \(\can:
	A\otimes_BA\sir\Hom_\k(H,A)\) is injective.  It is an epimorphism
    since \(\dim_\k A\otimes_BA=\dim_\k\Hom_\k(H,A)=8\cdot\dim_\k A<\infty\).

    The image of \(\phi:\qquot(H)\sir\Sub_\Alg(A/B)\) is the following poset:
    \begin{center}
	\begin{tikzpicture}
	    \node (root) at (0,0) {\(A\)};
	    \node (a1) at (-4,1.5) {\(\langle X,Z\rangle\)};
	    \node (a2) at (-2,1.5) {\(\langle X,JZ\rangle\)};
	    \node (a3) at (0,1.5) {\(\langle J,X,Z^2\rangle\)};
	    \draw[-] (root) -- (a1);
	    \draw[-] (root) -- (a2);
	    \draw[-] (root) -- (a3);
    
	    \node (b2) at (-2.5,3) {\(\langle X,Z^2\rangle\)};
	    \node (b3) at (0,3) {\(\langle J,X,Z^4\rangle\)};
	    \node (b4) at (2.5,3) {\(\langle X,JZ^2\rangle\)};
	    \draw[-] (a1) -- (b2);
	    \draw[-] (a2) -- (b2);
	    \draw[-] (a3) -- (b2);
	    \draw[-] (a3) -- (b3);
	    \draw[-] (a3) -- (b4);
    
	    \node (D) at (0,4) {\(B\)};
	    \draw[-] (b2) -- (D);
	    \draw[-] (b3) -- (D);
	    \draw[-] (b4) -- (D);
	\end{tikzpicture}\\
	\textit{The image of \(\qquot(H^*)\) in \(\Sub_\Alg(A/B)\)}
    \end{center}
\end{example}
Note that \(A^{\langle c^2-s^2\rangle}=\langle J,X,Z^2\rangle\), \(A^{\langle
	ct,st\rangle}=A^{\langle c^2-s^2,ct,st\rangle}=\langle
    X,JZ^2\rangle\).  We will later see that the morphism
\(\phi:\qquot(H^*)\sir\Sub_\Alg(A/B)\) is a monomorphism for any finite
\(H^*\)-Galois extension. 

\section{Closed Elements}\label{sec:closed_elements}
The main aim of this section is to characterise the closed elements of the
correspondence~\eqref{eq:galois-connection}.  First we provide two new
important criteria for closedness of both: subalgebras of a comodule algebra
and generalised quotients of the Hopf algebra.  The criterion for generalised
quotients is obtained using new methods based on both poset and Hopf algebraic
methods.  As a consequence of this we get that if the Hopf algebra is finite
dimensional and the comodule algebra is a Hopf--Galois extension, then every
generalised quotient is a closed element.  Then we generalise the methods of
Schauenburg: we redefine admissibility (cf.
Definition~\ref{defi:Schauenburg_admissibility} on
page~\pageref{defi:Schauenburg_admissibility}), and we show that admissible
objects are closed elements of the Galois
connection~\eqref{eq:galois-connection}.  We were able to slightly simplify
the Schauenburg line of proof by proving a finer version
of~\cite[Prop.~3.5]{ps:gal-cor-hopf-bigal} (cf. Lemma~\ref{lem:galois_and_op}
on page~\pageref{lem:galois_and_op}).
\index{comodule algebra!Galois correspondence!closed subalgebras}
\begin{theorem}\label{thm:closed_subalgebras}
    Let \(A/B\) be an \(H\)-Galois extension such that \(A\) and \(H\) are
    flat Mittag--Leffler \(\R\)-modules, and let \(A\) be a faithfully flat
    left \(B\)-module.  Then a subalgebra \(S\in\Sub_\Alg(A/B)\) is a closed
    element of the Galois connection~\eqref{eq:galois-connection} if and only if
    the following canonical map is an isomorphism:
    \begin{equation}\label{eq:Schneider_canocial_map}
	\can_S: S\otimes_B A\ir A\cotensor_{\psi(S)} H,\quad\can_S(a\otimes b)=ab_{(1)}\otimes b_{(2)}
    \end{equation}
\end{theorem}
\begin{proof}
    First let us note that the map \(\can_S\) is well defined since it is
    a composition of \(S\otimes_B A\sir A^{co\,\psi(S)}\otimes_B A\), induced
    by the inclusion \(S\subseteq A^{co\psi(S)}\), with
    \(A^{co\,\psi(S)}\otimes_B A\sir A\cotensor_{\psi(S)}H\) which is an
    isomorphism by~\cite[Thm~1.4]{hs:normal-bases}.  If \(S\) is closed then
    \(S=A^{co\,\psi(S)}\), hence~\eqref{eq:Schneider_canocial_map} is an
    isomorphism.  Conversely, let us assume that \(\can_S\) is an isomorphism.
    Then we have a commutative diagram:
    \begin{center}
	\begin{tikzpicture}
	    \matrix[matrix, column sep=1.5cm, row sep=.75cm]{
		\node (A) {\(S\otimes_BA\)}; & \node (B) {\(A\cotensor_{\psi(S)} H\)};\\
		\node (C) {\(A^{co\,\psi(S)}\otimes_BA\)}; & \\
	    };
	    \draw[->] (A) --node[above]{\(\can_S\)} (B);
	    \node[rotate=-90] at ($(A)+(0cm,-7mm)$) {\(\subseteq\)};
	    \draw[->] (C) --node[below right]{\(\can_{A^{co\,\psi(S)}}\)} (B);
	\end{tikzpicture}
    \end{center}
    where both \(\can_S\) and \(\can_{A^{co\,\psi(S)}}\) are isomorphisms.
    Now, since \(A\) is a faithfully flat left \(B\)-module we must have
    \(S=A^{co\,\psi(S)}\).  Thus \(S\) is a closed element
    of~\eqref{eq:galois-connection}.
\end{proof}
\begin{theorem}\label{thm:mono}
	Let $A$ be an $H$-comodule algebra over a ring $\R$ with surjective
	canonical map and let $A$ be a $Q_1$-Galois and a $Q_2$-Galois
	extension, where $Q_1,Q_2\in \qquot(H)$.  Furthermore, let us assume
	that one of the following assumptions holds:
	\begin{enumerate}[topsep=0pt,noitemsep]
	    \item \(Q_1\) and \(Q_2\) are flat as \(\R\)-modules and
		\(\R\cong \R1_A\) as \(\R\)-modules, via the unit map
		\(1_A:\R\sir A,\;r\selr{}r1_A\) (hence \(A\) is a faithful
		\(\R\)-module); or
	    \item the unit map \(1_A:\R\sir A\) is a pure
		monomorphism (e.g.  it has a left inverse).
	\end{enumerate}
	Then the following implication holds:
	\[A^{co\,Q_1}=A^{co\,Q_2}\ \Rightarrow\ Q_1=Q_2\]
\end{theorem}
\begin{proof}
Let $B'=A^{co\,Q_1}=A^{co\,Q_2}$.  The following diagram commutes:
\begin{center}
\begin{tikzpicture}[>=angle 60,thick]
\matrix[matrix,column sep=7mm,row sep=7mm]{
   &                             & \node(AQ1) {\(A\otimes Q_1\)};\\
   \node(A) {\(A\otimes_{B'} A\)}; & \node(B) {\(A\otimes_{A^{co\,H}}A\)}; & \node(C)   {\(A\otimes H\)};\\
   &                             & \node(AQ2) {\(A\otimes Q_2\)};\\
};
\begin{scope}
\draw[->] (A) -- node[above]{$\can_{Q_1}$} (AQ1);
\draw[->] (A) -- node[below]{$\can_{Q_2}$} (AQ2);
\draw[<<-] (A) -- (B);
\draw[->>] (B) -- node[above,pos=.4]{$\can$}(C);
\draw[->] (C) -- node[right]{$id\otimes\pi_1$}(AQ1);
\draw[->] (C) -- node[right]{$id\otimes\pi_2$}(AQ2);
\end{scope}
\end{tikzpicture}
\end{center}
The maps \(\can_{Q_1}\) and \(\can_{Q_2}\) are isomorphisms.  Let
\(f:=(\can_{Q_1}\circ\can_{Q_2}^{-1})\circ(id\otimes\pi_2)\).  By
commutativity of the above diagram, \(f\circ\can\) and
\((id\otimes\pi_1)\circ\can\) are equal.  Now, surjectivity of \(\can\) yields
the equality \((\can_{Q_1}\circ
    \can_{Q_2}^{-1})\circ(id\otimes\pi_2)=(id\otimes\pi_1)\).  We are going to
construct  $\pi:\,Q_1\ir Q_2$ such that $\can_{Q_1}\circ
\can_{Q_2}^{-1}=id\otimes \pi$ and $\pi\circ\pi_2=\pi_1$.  Define
\(\pi:Q_2\sir Q_1\) by \(\pi(\pi_2(h))=\pi_1(h)\) for \(h\in H\).  If
\(\pi_2(h)=\pi_2(h')\) for \(h,h'\in H\) then
\begin{align*}
    1_A\otimes\pi_1(h) & =(\can_{Q_1}\circ\can_{Q_2}^{-1})(1_A\otimes\pi_2(h))\\
               & =(\can_{Q_1}\circ\can_{Q_2}^{-1})(1_A\otimes\pi_2(h'))\\
               & =1_A\otimes\pi_1(h')
\end{align*}
By flatness of \(Q_1\) and the condition \(\R1_A\cong \R\) or purity of the
unit map we have \(Q_1=\R1_A\otimes Q_1\subseteq A\otimes Q_1\).  Now, it
follows that \(\pi_1(h)=\pi_1(h')\).  This shows that \(\pi:Q_2\sir Q_1\) is
well defined.  Furthermore, $\pi$ is right $H$-linear and \(H\)-colinear:
\begin{align*}
    \Delta_{Q_1}(\pi(\pi_2(h))) & = \Delta_{Q_1}(\pi_1(h)) \\
    & = \pi_1\otimes\pi_1\circ\Delta_{H}(h) \\
    & = (\pi\circ\pi_2)^{\otimes 2}\circ\Delta_{H}(h) \\
    & = \pi\otimes\pi\Delta_{Q_2}(\pi_2(h))
\end{align*}
Thus $Q_2\succcurlyeq Q_1$.  In the same way, using
\((\can_{Q_2}\circ\can_{Q_1}^{-1})\circ(id\otimes\pi_1)\) instead of \(f\), we
obtain that $Q_1\succcurlyeq Q_2$.  Now, by antisymmetry of \(\succcurlyeq\),
we get \(Q_1=Q_2\).
\end{proof}
Before deriving a corollary from the above statement let us define an important
class of generalised quotients.  It depends on the comodule algebra and we are
on the way to show that they are closed elements of our Galois correspondence.
In some cases we will be able to show that this class is equal to the subset
of closed quotients of the Hopf algebra.
\index{comodule algebra!Q-Galois}
\index{Q-Galois extension|see{comodule algebra!Q-Galois}}
\begin{definition}
    Let \(A\) be an \(H\)-comodule algebra, where \(H\) is a Hopf algebra.
    Then for a generalised quotient \(\pi:H\sir Q\), \(A/A^{co\,Q}\) is called
    \(Q\)-\textbf{Galois} if the following map is an isomorphism:
    \[\can_Q:A\otimes_{A^{co\,Q}}A\sir A\otimes Q,\quad\can(a\otimes_{A^{co\,Q}}b)=ab_{(0)}\otimes\pi(b_{(1)})\]
\end{definition}
The following corollary explains why we rather tend to think that the above
defines a class of quotients of \(H\) than subalgebras of \(A\).  
\index{comodule algebra!Galois correspondence!closed generalised quotients}
\begin{corollary}\label{cor:Q-Galois_closed}
    Let \(A\) be an \(H\)-comodule algebra with epimorphic canonical map
    \(\can_H\) such that the Galois connection~\eqref{eq:galois-connection}
    exists and let the unit map \(1_A:\R\sir A\) be a pure monomorphism.  Then
    \(Q\in \qquot(H)\) is a \textsf{closed element} of the Galois
    connection~\eqref{eq:galois-connection} if \(A/A^{co\,Q}\) is
    \(Q\)-Galois.
\end{corollary}
This statement might be understood as a counterpart of
Theorem~\ref{thm:closed_subalgebras} for generalised quotients.
\begin{proof}
Fix a coinvariant subalgebra $A^{co Q}$ for some $Q\in \qquot(H)$ then
$\phi^{-1}(A^{co\,Q})$ is an upper-sublattice  of $\qquot(H)$ (i.e.  it is
a subposet closed under finite suprema) which has a greatest element, namely
$\widetilde Q=\psi(A^{co\,Q})$.  Moreover, $\widetilde Q$ is the only closed
element belonging to $\phi^{-1}(A^{co\,Q})$.  Both $Q\preccurlyeq\widetilde Q$
and the assumption that $A/A^{co\,Q}$ is $Q$-Galois imply that
\(A/A^{co\,\widetilde Q}\) is \(\widetilde Q\)-Galois.  To this end, we
consider the commutative diagram:
\begin{center}
\begin{tikzpicture}
    \node (A) at (0,0)	{\(A\otimes_B A\)};
    \node (B) at (3cm,0)	{\(A\otimes H\)};
    \node (C) at (0cm,-1.7cm) {\(A\otimes_{A^{co\,\widetilde Q}}A\)};
    \node (D) at (3cm,-1.7cm) {\(A\otimes \widetilde Q\)};
    \node (E) at (0cm,-3.4cm) {\(A\otimes_{A^{co\,Q}}A\)};
    \node (F) at (3cm,-3.4cm) {\(A\otimes Q\)};
    \begin{scope}[>=angle 60,thick]
	\draw[->>] (A) -- node[above]{\(\can_H\)} (B);
	\draw[->>] (A) -- (C);
	\draw[->>] (B) -- (D);
	\draw[->]  (C) -- node[below]{\(\can_{\widetilde Q}\)} (D);
	\draw[->]  (E) -- node[below]{\(\can_Q\)} node [above]{\(\simeq\)}  (F);
	\draw[->]  (C) -- node[left]{\(=\)}(E);
	\draw[->>] (D) -- (F);
    \end{scope}
\end{tikzpicture}
\end{center}
From the lower commutative square we get that $\can_{\widetilde Q}$ is
a monomorphism and from the upper commutative square we deduce that
$\can_{\widetilde Q}$ is onto.  Unless $\widetilde Q=Q$ we get a contradiction
with the previous proposition.
\end{proof}
The above result applies also to the Galois
correspondence~\eqref{eq:Schauenburg_galois_corr}, as it is the same as the
Galois connection~\ref{eq:galois-connection} for the left \(L(A,H)\)-comodule
algebra \(A\) (see Proposition~\ref{prop:properties-of-adjunction}(iii)).  For
a finite dimensional Hopf algebra \(H\) for every \(Q\) the extension
\(A/A^{co\,Q}\) is \(Q\)-Galois (see~\cite[Cor.~3.3]{ps-hs:gen-hopf-galois}).
Thus we get the following statement.
\begin{proposition}\label{prop:finite_case}
    Let \(H\) be a finite dimensional Hopf algebra over a field \(\k\).  Let
    \(A/B\) be an \(H\)-Galois extension.  Then every \(Q\in\qquot(H)\) is
    closed. 
\end{proposition}
Consequently, the map \(\phi\) of the Galois
connection~\eqref{eq:galois-connection} is always a monomorphism, when the
conditions of the above proposition are met.

\section[Correspondence between admissible objects]{Generalisation of Schauenburg's
    correspondence of admissible objects}\label{sec:admissibility}

In this section we generalise Theorem~\ref{thm:Schauenburg_Galois_closedness}
(page~\pageref{thm:Schauenburg_Galois_closedness}).  The line of proof is
essentially the same as Schauenburg's, though many details had to be changed,
and also some additional lemmas had to be proven.  We make one
assumption which did not appear in the original approach: in order to use
Corollary~\ref{cor:Q-Galois_closed} we assume that the unit map of the
comodule algebra \(A\) is a pure monomorphism.  The usage of
Corollary~\ref{cor:Q-Galois_closed} allows as to slightly simplify the proof.
Also in our situation we do not assume, as Schauenburg did, that the
coinvariants subalgebra is a central subalgebra, hence the single faithful
flatness condition of \(A\) over the coinvariants has to be changed into
faithful flatness as both a right and a left module.

\index{opposite algebra}
\index{algebra!opposite}
\index{opposite bialgebra}
\index{bialgebra!opposite}
\index{opposite generalised quotient}
\index{bialgebra!opposite generalised quotient}
In this section we will make heavy use of the opposite multiplication on both
an \(H\)-comodule algebra \(A\) and the Hopf algebra \(H\).  The opposite
algebra of \(A\) (or \(H\)) will be denoted by \(A^\op\) (\(H^\op\)
respectively).  Note that \(H^\op\) is a bialgebra, with the same
comultiplication as \(H\).  Furthermore, it is a Hopf algebra if the antipode
of \(H\) is bijective and \(S_{H^\op}=S_H^{-1}\), where \(S_H\) denotes the
antipode of \(H\).  The right \(H^\op\)-comodule structure of \(A^\op\) is
simply: \(A^\op\ni a\sir a_{(0)}\otimes a_{(1)}\in A^\op\otimes H^\op\).
Clearly, it is an algebra map, which makes \(A^\op\) into an
\(H^\op\)-comodule algebra.

We let \((\phi^{{\op}},\psi^{{\op}})\) denote the Galois
connection~\eqref{eq:galois-connection} for the \(H^{{\op}}\)-comodule algebra
\(A^{{\op}}\).  Note that \((\phi^{{\op}},\psi^{{\op}})\) is a Galois
correspondence if and only if \((\phi,\psi)\) is, by
Lemma~\ref{lem:galois_and_op}.  We recall the following notation introduced in
Proposition~\ref{prop:Schauneburg_admissibility_and_op}: for a generalised
quotient \(Q=H/I\in\qquot(H)\) we put \(Q^{{\op}}\coloneq
    H^{{\op}}/S_{H}(I)\in\qquot(H^{{\op}})\) and we also write
\(I^{{\op}}\coloneq S_H(I)\).  Now, we are ready to adapt
Definition~\ref{defi:Schauenburg_admissibility} to the more general context of
\(H\)-extensions with noncommutative coinvariant subrings.  We only amend the
definition of admissible subalgebras of \(A\), the coalgebra part is repeated.
\index{admissible quotients}
\index{admissible subalgebras}
\index{bialgebra!generalised quotients!admissibility}
\begin{definition}\label{defi:admissible_II}
    \begin{enumerate}
	\item Let \(C\) be an \(\R\)-coalgebra and let \(C\smpr{}{D}\) be
	    a coalgebra quotient.  Then \({D}\) is called \textbf{left}
	    (\textbf{right}) \textbf{admissible} if it is \(\R\)-flat (hence
	    faithfully flat) and \(C\) is left (right) faithfully coflat
	    over~\({D}\).  For \(I\in\coId(C)\) we will say that it is
	    \textbf{left} (\textbf{right}) \textbf{admissible} if and only if
	    the quotient \(C/I\) shares this property.
    \end{enumerate}
    \medskip
    \noindent Let \(A/B\) be an \(H\)-extension, where the Hopf algebra \(H\) has
    a bijective antipode.
    \begin{enumerate}
	\item Let \(S\) belong to \(\Sub_\Alg(A/B)\).  Then~\(S\) is
	    called \textbf{right admissible} if:
	    \begin{enumerate}[topsep=0pt,noitemsep,label={(\alph*)}]
		\item \(A\) is right faithfully flat over~\(S\),
		\item the composition: 
	    \begin{equation*}
		\begin{split}
		    &\can_S:A\otimes_SA\sir A\otimes_{A^{co\,\psi(S)}}A\mpr{\can_{\psi(S)}}A\otimes\psi(S)\\
		    &a\otimes_S b\selmap{}ab_{(0)}\otimes b_{(1)},
		\end{split}
	    \end{equation*}
		    is a bijection (it is well defined since \(S\subseteq
			A^{co\,\psi(S)}\)).  
	    \end{enumerate}
	\item We call \(S\) \textbf{left admissible} if:
	    \begin{enumerate}[topsep=0pt,noitemsep,label={(\alph*)}]
		\item \(A\) is left faithfully flat over~\(S\),
		\item the composition: 
		    \begin{equation*}
			\begin{split}
			    &\can_S:A\otimes_SA\sir A\otimes_{A^{co\,\psi(S)}}A\mpr{\can_{\psi(S)}}A\otimes\psi(S)\\
			    &a\otimes_S b\selmap{}ab_{(0)}\otimes b_{(1)},
			\end{split}
		    \end{equation*}
		    is a bijection.
	    \end{enumerate}
    \end{enumerate}
    An element is called admissible if it is both left and right admissible.
\end{definition}
\begin{remark}
    If \(S\) is left or right admissible and \(A\) is faithfully flat over the
    base ring \(\R\) then \(\psi(S)\) is flat as an \(\R\)-module.
\end{remark}
Later, in Lemma~\ref{lem:symmetry_od_admissibility}, we will show that
a subalgebra \(S\) is left (right) admissible if and only if \(S^\op\subseteq
    A^\op\) is right (left) admissible. The same holds for admissibility of
generalised quotients if \(H\) is a flat Hopf algebra with a bijective
antipode (see Proposition~\ref{prop:Schauneburg_admissibility_and_op} on
page~\pageref{prop:Schauneburg_admissibility_and_op}).

The above notion of right and left admissibility for intermediate subalgebras
of an extension \(A/B\) is equivalent to the original definition of
Schauenburg when one considers \(A/B\) as a left \(L(A,H)\)-Galois extension,
and when we assume that \(B=A^{co\,H}\) is equal to the ground ring \(\R\), as
is required in Schauenburg's approach.  It is so since then the conditions
\textit{(ii.b)} and \textit{(iii.b)} are automatically satisfied.  The
condition \textit{(ii.b)} is satisfied by the way \(L(A,H)\) is constructed.
This is shown in Lemma~\ref{lem:Schauenburg_Galois_cond}
(page~\pageref{lem:Schauenburg_Galois_cond}).  The condition \textit{(iii.b)}
follows in the same way under the assumption that \(H\) has a bijective
antipode, since then \(L(A^{\op},H^{\op})\cong L(A,H)^{\op}\) by the universal
property of \(L(A,H)\): if the antipode of \(H\) is bijective then \(A\) is
right \(H\)-Galois if and only if \(A^{\op}\) is right \(H^{\op}\)-Galois
(this is a special case of Lemma~\ref{lem:symmetry_od_admissibility} on
page~\pageref{lem:symmetry_od_admissibility}).  This also holds for \(L(A,H)\)
and \(L(A,H)^{\op}\), since \(L(A,H)\) has a bijective antipode whenever \(H\)
has.  In this way \(A^{\op}/\R\) is \(L(A,H)^{\op}\text{-}H^{\op}\)-Galois.
We have \(L(A^{\op},H^{\op})\cong L(A,H)^{\op}\), by the universal property of
\(L(A,H)\) (Proposition~\ref{prop:Schaueburg_universal_prop} on
page~\pageref{prop:Schaueburg_universal_prop}). 

We will need the following version of~\cite[Rem.~1.2]{hs:normal-bases}.
\begin{remark}\label{rem:faithfully_flatness}
    Let \(A/B\) be an \(H\)-extension, such that both \(R\)-modules \(A\) and
    \(H\) are flat Mittag--Leffler.  Let \(S\in\Sub_\Alg(A/A^{co\,H})\).  Then
    the following holds:
    \begin{enumerate}[noitemsep,nosep]
	\item if \(\can_S:A\otimes_SA\sir A\otimes\psi(S)\) is an isomorphism
	      and \(A\) is right or left faithfully flat over~\(S\) then~\(S\)
	      is a closed element of~\eqref{eq:galois-connection};
	\item if the natural projection \(A\otimes_SA\eir
	      A\otimes_{A^{co\,\psi(S)}}A\) is a bijection and \(A\) is right
              or left faithfully flat over \(S\) then \(S=A^{co\,\psi(S)}\),
              i.e.  \(S\) is a closed element of \(\Sub_\Alg(A/B)\)
              in~\eqref{eq:galois-connection}.
    \end{enumerate}
\end{remark}
\begin{proof}
    First let us prove (ii).  For this let us consider the commutative diagram:
    \begin{center}
	\begin{tikzpicture}
	    \matrix[column sep=1cm,row sep=0.7cm]{
		\node (A) {\(S\)}; & \node (B) {\(A\)}; & \node (C) {\(A\otimes_SA\)};\\
		\node (D) {\(A^{co\,\psi(S)}\)}; & \node (E) {\(A\)}; & \node (F) {\(A\otimes_{A^{co\,\psi(S)}}A\)};\\
	    };
	    \draw[->,dashed] (D) -- (A);
	    \node[yshift=-7mm,rotate=90] at (B) {\(=\)};
	    \draw[->] (C) --node[above,rotate=-90]{\(\simeq\)} (F);
	    \node[xshift=10mm] at (A) {\(\subseteq\)};
	    \node[xshift=10mm] at (D) {\(\subseteq\)};
	    \draw[->] ($(B)+(2mm,1mm)$) -- ($(C)+(-7mm,1mm)$);
	    \draw[->] ($(B)+(2mm,-1mm)$) -- ($(C)+(-7mm,-1mm)$);
	    \draw[->] ($(E)+(2mm,1mm)$) -- ($(F)+(-11mm,1mm)$);
	    \draw[->] ($(E)+(2mm,-1mm)$) -- ($(F)+(-11mm,-1mm)$);
	\end{tikzpicture}
    \end{center}
    The four maps \(A\smprr{}{} A\otimes_SA\) and \(A\smprr{}{}
	A\otimes_{A^{co\,\psi(S)}}A\) send \(a\in A\) to \(a\otimes 1_A\) or
    \(1_A\otimes a\) in the respective tensor products.  Since the diagram
    commutes and since the upper row is an equaliser (by faithfully flat
    descent), the dashed arrow exists, i.e.  \(A^{co\,\psi(S)}\subseteq S\).
    We get the equality since, \(S\subseteq A^{co\,\psi(S)}\) holds by the
    Galois connection property.

    The first claim follows from \textit{(ii)} when applied to \(S\subseteq
	A^{co\,\psi(S)}\), because the natural epimorphism \(A\otimes_SA\sir
	A\otimes_{A^{\co\;\psi(S)}}A\) is an isomorphism whenever
    \(\can_S:A\otimes_SA\sir A\otimes\psi(S)\) is an isomorphism:
    \begin{center}
	\begin{tikzpicture}
	    \matrix[matrix of nodes, column sep=1cm, row sep=1cm]{
		|(A1)| \(A\otimes_SA\) & |(A2)| \(A\otimes\psi(S)\) \\
		|(B1)| \(A\otimes_{A^{\co\;\psi(S)}}A\) & \\
	    };
	    \draw[->] (A1) --node[above]{\(\can_S\)}node[below]{\(\cong\)} (A2);
	    \draw[->>] (A1) -- (B1);
	    \draw[->] (B1) --node[below right]{\(\can_{A^{\co\,\psi(S)}}\)} (A2);
	\end{tikzpicture}
    \end{center}
\end{proof}

The following lemma is a finer version of
\cite[Prop.~3.5]{ps:gal-cor-hopf-bigal}, since we are able to prove the
equality \(\psi^{\op}(S^{{\op}})=\bigl(\psi(S)\bigr)^{{\op}}\) for any
\(S\in\Sub_\Alg(A/B)\): 
\begin{lemma}\label{lem:galois_and_op}
    Let \(H\) be a Hopf algebra with a bijective antipode.  Let \(A/B\) be an
    \(H\)-extension.  Then the Galois connection~\eqref{eq:galois-connection}
    exists for the \(H^{{\op}}\)-comodule algebra \(A^{{\op}}\) if and only if
    it exists for the \(H\)-extension \(A/B\).  Furthermore, if this is the
    case
    \[\phi^{{\op}}(Q^{\op})=\bigl(\phi(Q)\bigr)^{{\op}}\text{ and }\psi^{\op}(S^{{\op}})=\bigl(\psi(S)\bigr)^{{\op}}\]
    where \(S\in\Sub_\Alg(A/B)\) and \(Q\in\qquot(H)\).
\end{lemma} 
\begin{proof}
    The first equation: \(\phi^{{\op}}(Q^{\op})=\bigl(\phi(Q)\bigr)^{{\op}}\)
    is proved in~\cite[Prop.~3.5]{ps:gal-cor-hopf-bigal}. 
    Both maps: \(\Sub_\Alg(A/B)\ni
        S\selmap{}S^{{\op}}\in\Sub_\Alg(A^{{\op}}/B^{{\op}})\) and
    \(\qquot(H)\ni Q\selmap{}Q^{{\op}}\in\qquot(H^{{\op}})\) are isomorphisms
    of posets (see Proposition~\ref{prop:Schauneburg_admissibility_and_op} or
    \cite[Prop.~3.5]{ps:gal-cor-hopf-bigal}).  Thus \(\phi\) reflects suprema
    if and only if \(\phi^{{\op}}\) reflects them.  This shows that the Galois
    connection exists for \(A\) if and only if it exists for \(A^\op\) due to
    Theorem~\ref{thm:existence-of-adjunction}).  It remains to show that
    \(\psi^{{\op}}(S^{{\op}})=\bigl(\psi(S)\bigr)^{{\op}}\).  First, let us
    observe that for any set \(\mathcal{O}\) of right ideal coideals we have
    \(\bigwedge_{I\in \mathcal{O}}I^{{\op}}=\left(\bigwedge_{I\in
                \mathcal{O}}I\right)^{{\op}}\) in \(\qid(H^{\op})\):
    \begin{equation}\label{eq:wedge_and_op} 
	\bigwedge_{I\in \mathcal{O}}I^{{\op}} = \sum_{\substack{J\in\qid(H^{{\op}})\\
		J\subseteq\mathop\cap\limits_{I\in \mathcal{O}}I^{{\op}}}} J 
	= \sum_{\substack{J\in\qid(H)\\J^{{\op}}\subseteq(\mathop\cap\limits_{I\in \mathcal{O}}I)^{{\op}}}}J^{{\op}} 
	= \sum_{\substack{J\in\qid(H)\\J\subseteq\mathop\cap\limits_{I\in \mathcal{O}}I}}J^{{\op}} 
	= \Bigl(\bigwedge_{I\in \mathcal{O}}I\Bigr)^{{\op}} 
    \end{equation} 
    Let \(\overline{\psi}(S)=\ker(H\sir\psi(S))\).  Now using
    formula~\eqref{eq:psi} (on page~\pageref{eq:psi}) for \(\psi\) we get:
    \begin{align}\label{eq:phi_bar} 
	\overline{\psi}(S) & =\bigwedge\left\{I:\ S\subseteq A^{co\,H/I}\right\} \\
	\overline{\psi}^{{\op}}(S^{{\op}}) & =\bigwedge\left\{I^{{\op}}:\ S^{{\op}}\subseteq (A^{{\op}})^{co\,H^{{\op}}/I^{{\op}}}\right\} \notag
    \end{align} 
    Since
    \(\bigl(A^{co\,H/I}\bigr)^{{\op}}=(A^{{\op}})^{co\,H^{{\op}}/I^{{\op}}}\)
    and \(\qid(H)\ni I\selmap{}I^{{\op}}\coloneq S_H(I)\in\qid(H^{{\op}})\) is
    a bijection, the two sets in~\eqref{eq:phi_bar} are in a bijective
    correspondence via \(I\selmap{}I^{{\op}}\).  Now, the formula
    \(\psi^{{\op}}(S^{{\op}})=\bigl(\psi(S)\bigr)^{{\op}}\) follows from
    equation~\eqref{eq:wedge_and_op}.  
\end{proof}
Another way of proving the equation
\(\psi^{{\op}}(S^{{\op}})=\bigl(\psi(S)\bigr)^{{\op}}\) is to show that the
map \(\chi:\Sub_\Alg(A^{{\op}}/B^{{\op}})\sir\qquot(H^{{\op}})\) given by
\(\chi(S)\coloneq\bigl(\psi(S^{{\op}})\bigr)^{{\op}}\), where
\(S\in\Sub_\Alg(A^{{\op}}/B^{{\op}})\), is an adjoint for \(\phi^{{\op}}\) and
conclude the argument referring to uniqueness of Galois connections.
\begin{proposition}\label{prop:Schauenburg_admissibility}
    Let \(H\) be a Hopf algebra with a bijective antipode.  Let \(A\) be an
    \(H\)-extension of \(B\) such that \(A\) is \(H\)-Galois, \(A\) is
    a faithfully flat Mittag--Leffler module, \(H\) is a flat Mittag--Leffler
    \(\R\)-module, \(A\) is faithfully flat as both a left and right
    \(B\)-module, and also faithfully flat as an \(\R\)-module and finally let
    us assume that \(1_A:\R\sir A\) is pure.  Then:
    \begin{enumerate}[noitemsep,nosep]
	\item if \(S\in\Sub(A/B)\) is right admissible then so is
	    \(\psi(S)\) and \(\phi\psi(S)=S\),
	\item if \(Q\in\qquot(H)\) is left admissible then so is
	    \(\phi(Q)\coloneq A^{co\,Q}\) and moreover \(\psi\phi(Q)=Q\).  
    \end{enumerate}
\end{proposition}
\begin{proof}
    We first prove \textit{(ii)}: \(\phi(Q)\) is left admissible by
    applying~\cite[Thm~1.4]{hs:normal-bases}
    to \(A\).  The equality \(\psi\phi(Q)=Q^{{\op}}\) follows from
    Corollary~\ref{cor:Q-Galois_closed}.  Suppose, that \(S\in\Sub(A/B)\) is
    a right admissible subalgebra.  Then by
    Remark~\ref{rem:faithfully_flatness}(i) \(S=\phi\psi(S)\) and \(A/S\) is
    \(Q\)-Galois.  The proof that \(\psi(S)\) is a faithfully flat
    \(\R\)-module and \(H\) is faithfully coflat as a right
    \(\psi(S)\)-comodule remains the same as
    in~\cite[Prop.~3.4]{ps:gal-cor-hopf-bigal}.  For the sake of completeness
    let us recall these arguments.  First of all \(\psi(S)\) is a flat
    \(\R\)-module, since we have an isomorphism \(A\otimes_{S}A\cong
	A\otimes\psi(S)\) (by right admissibility of \(S\)) and \(A\) is
    faithfully flat as a right \(S\)-module and also as an \(\R\)-module.  For
    any left \(\psi(S)\)-comodule \(V\) we have a chain of isomorphisms:
    \[A\otimes (H\cotensor_{\psi(S)}V)\cong(A\otimes H)\cotensor_{\psi(S)}V\cong(A\otimes_BA)\cotensor_{\psi(S)}H\cong A\otimes_B(A\cotensor_{\psi(S)}V)\]
    Now, \(H\) is faithfully coflat as a right \(\psi(S)\)-comodule, since
    \(A\) is faithfully flat as an \(\R\)-module, faithfully flat as a right
    \(B\)-module and faithfully coflat as a right \(\psi(S)\)-comodule
    by~\cite[Rem.~1.2(2)]{hs:normal-bases}. 
\end{proof}

Now we will show that the property of being a left/right admissible subalgebra
is symmetric with respect to taking the opposite algebra.
\begin{lemma}\label{lem:symmetry_od_admissibility}
     Let \(H\) be a Hopf algebra with a bijective antipode and let \(A\) be an
     \(H\)-comodule algebra.  Then \(S\in\Sub_\Alg(A)\) is left (right)
     admissible if and only if \(S^\op\subseteq A^\op\) is right (left)
     admissible.
\end{lemma}
\begin{proof}
    Note that \(A\) is left (right) faithfully flat over \(S\) if and only if
    \(A^\op\) is right (left) faithfully flat over \(S^\op\). It remains to
    show that \(\can_S:A\otimes_SA\sir A\otimes\psi(S)\) is an isomorphism if
    and only if \(\can_{S^\op}:A^\op\otimes_{S^\op}A^\op\sir
	A^\op\otimes\psi^\op(S^\op)\) is. 

    Let \((\phi^\op,\psi^\op)\) be the Galois
    correspondence~\eqref{eq:galois-connection} for the \(H^\op\)-comodule
    algebra \(A^\op\).  Let us consider \(Q=H/I\in\qquot(H)\) with \(\pi:H\sir
	Q\) the natural projection.  Then
    \(Q^\op=H^\op/S_H(I)\in\qquot(H^\op)\), or equivalently \(Q^\op\) is
    a left generalised quotient of \(H\) (i.e. quotient by a left ideal
    coideal).  We let \(\pi^\op:H^\op\sir Q^\op\) denote the natural
    projection map.  The antipode \(S_H\) of \(H\) induces an isomorphism
    \(\overline{S_H}:Q\sir Q^\op\) which is given by
    \(\overline{S_H}(\pi(h))\coloneq\pi^\op(S_H(h))\), \(h\in H\).  Its
    inverse is given by \(Q^\op\ni\pi^\op(h)\seli\pi(S_H^{-1}(h))\in Q\).  Let
    us show that \(\overline{S_H}:Q\sir Q^\op\) is an \(H\)-module isomorphism
    (over the ring homomorphism: \(S_H:H\sir H^\op\)):
    \begin{align*}
	\overline{S_H}(\pi(h)k) & = \overline{S_H}\left(\pi(hk)\right) \\
	                      & = \pi^\op\left(S_H(hk)\right) \\
	                      & = \pi^\op(S_H(k)S_H(h)) \\
	                      & = S_H(k)\pi^\op\left(S_H(h)\right) \\
	                      & = S_H(k)\overline{S_H}(\pi(h))
    \end{align*}
    for all \(h,k\in H\).  Finally let us note that
    \(\psi^\op(S^\op)=(\psi(S))^\op\) by Lemma~\ref{lem:galois_and_op}.  We
    let \(Q=\psi(S)\), \(\pi:H\sir\psi(S)\). 
    
    We have the following commutative diagram: 
    \begin{center}
	{\hfill\begin{tikzpicture}
	    \matrix[matrix of nodes, column sep=1cm, row sep=1cm]{
		|(B1)| \(A\otimes A\) & |(B2)| \(A\otimes\psi(S)\) \\
		|(A1)| \(A^\op\otimes_{S^\op}A^\op\) & |(A2)| \(A^\op\otimes\psi^\op(S^\op)\) \\
	    };
	    \draw[->] (A1) --node[below]{\(\can_{S^\op}\)} (A2);
	    \draw[->] (B1) --node[above]{\(\can_{S}\)} (B2);
	    \draw[->] (B1) --node[left]{\(\tau\)} (A1);
	    \draw[->] (B2) --node[right]{\(\alpha\)} (A2);
	\end{tikzpicture}
	\hfill\refstepcounter{equation}\raisebox{12mm}{(\theequation)}\label{diag:symmetry_of_admissibility}\\
	}
    \end{center}
    where:
    \begin{subequations}
    \begin{alignat*}{2}
	&\can_{S}:A\otimes_{S}A\sir A\otimes\psi(S), &\quad& \can_{S}(a\otimes b)\coloneq ab_{(0)}\otimes b_{(1)} \\
	&\can_{S^\op}:A^\op\otimes_{S^\op}A^\op\sir A^\op\otimes\psi^\op(S^\op), &\quad& \can_{S^\op}(a\otimes b)\coloneq b_{(0)}a\otimes b_{(1)} \\
	&\tau:A\otimes_{S}A\sir A^\op\otimes_{S^\op}A^\op, &\quad& \tau(a\otimes b) = b\otimes a \\
	&\alpha:A\otimes\psi(S)\sir A^\op\otimes\psi^\op(S^\op), &\quad& \alpha(a\otimes\pi(h))\coloneq a_{(0)}\otimes a_{(1)}\overline{S_H}\left(\pi(h)\right)
    \end{alignat*}
    \end{subequations}
    Note that we use the convention that concatenation \(ab\) for \(a,b\in A\)
    always denotes the multiplication in \(A\).  The map \(\tau\) is a well
    defined isomorphism.  Also \(\alpha\) is well defined since
    \(\psi^\op(S^\op)=(\psi(S))^\op\) is a left \(H\)-module.  It is an
    isomorphism with inverse:
    \begin{align*}
            &\alpha^{-1}:A^\op\otimes\psi^\op(S^\op)\sir A\otimes\psi(S)\\
	    &\alpha^{-1}(a\otimes\pi^\op(h))\coloneq a_{(0)}\otimes \overline{S_H}^{-1}\left(\pi^\op(h)\right)a_{(1)}
    \end{align*}
    Indeed \(\alpha^{-1}\) is an inverse of \(\alpha\):
    \begin{align*}
	\alpha^{-1}\circ \alpha(a\otimes \pi(h)) & = \alpha^{-1}\left(a_{(0)}\otimes a_{(1)}\pi^\op(S_H(h))\right) \\
	& = \alpha^{-1}\left(a_{(0)}\otimes \pi^\op(a_{(1)}S_H(h))\right) \\ 
	& = a_{(0)}\otimes \pi\left(S_H^{-1}(a_{(2)}S_H(h))\right)a_{(1)} \\ 
	& = a_{(0)}\otimes \pi\left(hS_H^{-1}(a_{(2)})a_{(1)}\right) \\ 
        & = a\otimes\pi(h)
    \end{align*}
    The other equality \(\alpha\circ\alpha^{-1}=\id\) follows in a similar way.
    Now the claim readily follows from the above commutative diagram, since both
    \(\tau\) and \(\alpha\) are isomorphisms. 
\end{proof}

The following theorem is a generalisation of Schauenburg's
Theorem~\ref{thm:Schauenburg_Galois_closedness}.
\index{comodule algebra!Galois correspondence!admissibility}
\begin{theorem}\label{thm:Galois_closeness}
    Let \(H\) be a Hopf algebra with a bijective antipode.  Let \(A/B\) be an
    \(H\)-extension such that \(A/B\) is \(H\)-Galois and \(A\) is faithfully
    flat as both a left and right \(B\)-module.  Furthermore, suppose that
    \(A\) is a faithfully flat Mittag--Leffler \(\R\)-module and \(H\) is
    a flat Mittag--Leffler module.  Finally, let us assume that \(1_A:\R\sir
        A\) is pure.  Then the Galois connection~\eqref{eq:galois-connection}
    gives rise to a bijection between (left, right) admissible objects, hence
    (left, right) admissible objects are closed.
\end{theorem}
Note that here we consider a right \(H\)-comodule algebras, while
\citeauthor{ps:gal-cor-hopf-bigal} considers a left \(L(A,H)\)-comodule
algebra structures.  We make one additional assumption, that the unit map
\(1_A:\R\sir A\) is a pure monomorphism, since our proof makes use of
Corollary~\ref{cor:Q-Galois_closed}.  The Lemma~\ref{lem:galois_and_op} is
stronger than the corresponding part
of~\cite[Prop.~3.5]{ps:gal-cor-hopf-bigal} which simplifies the first part of
the following proof.  Finally, (left, right) admissible generalised quotients
of \(L(A,H)\) classify (left, right) admissible \(H\)-comodule algebras, while
(left, right) admissible quotients of \(H\) classify (left, right) admissible
subalgebras of \(A\).
\begin{proofof}{Theorem~\ref{thm:Galois_closeness}}
    We let \((\phi^{{\op}},\psi^{{\op}})\) denote the Galois
    connection~\eqref{eq:galois-connection} for \(A^{{\op}}\) a right
    \(H^{{\op}}\)-comodule algebra instead of \(A\) -- an \(H\)-comodule
    algebra.

    Let \(S\subseteq A\) be left admissible.  Then \(S^{{\op}}\subseteq
	A^{{\op}}\) is right admissible by
    Lemma~\ref{lem:symmetry_od_admissibility}.  Hence
    \(\psi^{{\op}}(S^{{\op}})\) is right admissible and
    \(\phi^{{\op}}\psi^{{\op}}(S^{{\op}})=S^{{\op}}\), by
    Proposition~\ref{prop:Schauenburg_admissibility}(1).  It follows that
    \(\phi\psi(S)=S\), by Lemma~\ref{lem:galois_and_op}.
    Using~\cite[Prop.~3.5]{ps:gal-cor-hopf-bigal}, which shows that
    \(I\in\qid(H)\) is left (right) admissible if
    \(I^{{\op}}\in\qid(H^{{\op}})\) is right (left) admissible, we conclude
    that \(\psi(S)=\bigl(\psi^{{\op}}(S^{{\op}})\bigr)^{{\op}}\) is left
    admissible.

    Let \(I\) be right admissible.  Then \(I^{{\op}}\) is left admissible, and
    so is the subalgebra
    \(\phi^{{\op}}(H^{{\op}}/I^{{\op}})=\phi(H/I)^{{\op}}\).  Thus
    \(\phi(H/I)\) is right admissible, and
    \(H^{{\op}}/I^{{\op}}=\psi^{{\op}}\phi^{{\op}}(H^{{\op}}/I^{{\op}})=\bigl(\psi\phi(H/I)\bigr)^{{\op}}\).
    Hence \(H/I=\psi\phi(H/I)\), by the resulf
    of~\cite[Prop.~3.5]{ps:gal-cor-hopf-bigal}.
\end{proofof}
\begin{remark}\label{rem:admissibility}
    If we consider a faithfully flat \(H\)-Galois extension \(A/\R\), then the
    above result applied to the left \(L(A,H)\)-comodule algebra \(A\) reduces
    to Theorem~\ref{thm:Schauenburg_Galois_closedness} (on
    page~\pageref{thm:Schauenburg_Galois_closedness}).  We already pointed out
    that in this situation the two extra conditions \textit{(ii.b)} and
    \textit{(iii.b)} in our definition of admissibility (Definition
    \ref{defi:admissible_II} on page~\pageref{defi:admissible_II}) naturally
    follow in the Schauenburg context (see the paragraphs after
    Definition~\ref{defi:admissible_II}).  The missing step is that in
    Schauenburg's Theorem~\ref{thm:Schauenburg_Galois_closedness} (right, left)
    admissible generalised quotients of \(L(A,H)\) correspond to (right, left)
    admissible \(H\)-subcomodule algebras rather than just (right, left)
    admissible subalgebras of \(A\).  But right admissible subalgebras in the
    sense of Definition~\ref{defi:admissible_II} for the left
    \(L(A,H)\)-comodule algebra\footnote{Note that though
        Definition~\ref{defi:admissible_II} (on
        page~\pageref{defi:admissible_II}) is written for right comodule
        algebras we are using it here for the left \(L(A,H)\)-comodule algebra
        \(A\). For left comodule algebras their canonical map should be used
        in \textit{(ii.b)} and \textit{(iii.b)}.} \(A\) are \(H\)-subcomodules
    by Remark~\ref{rem:faithfully_flatness} (on
    page~\pageref{rem:faithfully_flatness}) as every subalgebra of
    coinvariants for quotients of \(L(A,H)\) is always an \(H\)-subcomodule,
    since \(A\) is an \(L(A,H)\)-\(H\)-bicomodule. In a similar way left
    admissible subalgebras in the sense of Definition~\ref{defi:admissible_II}
    (on page~\pageref{defi:admissible_II}) are \(H\)-subcomodules. 
\end{remark}

\section{Galois correspondence for Galois coextensions}
In this section we describe the Galois theory for Galois coextensions.  We
begin with some basic definitions.  The main theorem of this section is
Theorem~\ref{thm:coGalois_mono} (on page~\pageref{thm:coGalois_mono}) which is
a dual version of Theorem~\ref{thm:mono}.  We will need this result in the
following section for a new proof of the Takeuchi correspondence, i.e.
a bijection between generalised quotients and generalised Hopf subalgebras of
a finite dimensional Hopf algebra (Theorem~\ref{thm:newTakeuchi} on
page~\pageref{itm:newTakeuchi_finite}).  For the following two definitions we
refer to~\cite{hs:principal-homogeneuos-spaces}.

\index{module coalgebra}
\index{H-module coalgebra|see{module coalgebra}}
\begin{definition}\label{defi:module-coalgebra}
	Let \(C\) be a coalgebra and \(H\) a Hopf algebra, both over a ring
	\(\R\). We call \(C\) a left \bold{\(H\)-module coalgebra} if it is a left
	\(H\)-module such that the \(H\)-action \mbox{\(H\otimes C\sir C\)} is
	a coalgebra map:
	\begin{equation*}
	  \Delta_C(h\cdot c)=\Delta_H(h)\Delta_C(c),\quad\epsilon_C(h\cdot c)=\epsilon_H(h)\epsilon_C(c).
	\end{equation*}
	Let \(C^H\coloneq\nicefrac{C}{H^+C}\) be the \bold{invariant
	    coalgebra}.  Furthermore, we call \(C\sir C^H\) an
	\(H\)\bold{-coextension}.
\end{definition}
\index{module coalgebra!Galois coextension}
\index{Galois coextension|see{module coalgebra!Galois coextension}}
\begin{definition}\label{defi:coGalois}
    \begin{enumerate}
	\item An \(H\)-module coalgebra \(C\) is called an \bold{\(H\)-Galois
	    coextension} if the canonical map
	\[\can_H:H\otimes C\sir C\cotensor_{C^H}C,\quad h\otimes c\elmap{}hc_{(1)}\otimes c_{(2)}\]	
	is a bijection, where $C$ is considered as a left and right
	    $C^H$-comodule in a standard way. 
	\item More generally, if \(K\in\coid{l}(H)\) is a left coideal
	    then \(K^+C\) is a coideal (at least, when the base ring \(\R\) is
	    a field) and an \(H\)-submodule of \(C\).  The coextension \(C\sir
	    C^K=\nicefrac{C}{K^+C}\) is called Galois if the canonical map
	\begin{equation}\label{eq:can_K}
		\can_K:K\otimes C\sir C\cotensor_{C^K}C,\quad k\otimes c\elmap{}kc_{(1)}\otimes c_{(2)}
	\end{equation}
	is a bijection.
    \end{enumerate}
\end{definition}
A basic example of an \(H\)-module coalgebra is \(H\) itself.  Then \(H^H=\R\)
and \(H\cotensor_{H^H}H=H\otimes H\). The inverse of the canonical map is
given by \(\can_H^{-1}(k\otimes h)=kS\left(h_{(1)}\right)\otimes h_{(2)}\).
\begin{definition}
	Let \(C\) be an \(H\)-module coalgebra.  We let \(\Quot(C)=\{C/I:\
		I\text{- a coideal of }C\}\) with order relation
	\(C/I_1\succcurlyeq C/I_2\Leftrightarrow I_1\subseteq I_2\).  It is
	a complete lattice.  We let \(\Quot(\nicefrac{C}{C^H})\) denote the
	subset \(\{Q\in\Quot(C): C\succcurlyeq Q\succcurlyeq C^H\}\) in
	\(\Quot(C)\).
\end{definition}
\index{module coalgebra!Galois connection}
\begin{proposition}\label{prop:cogalois-extensions}
    Let \(C\) be an \(H\)-module coalgebra over a field~\(\k\). Then there
    exists a \textsf{Galois connection}:
    \begin{align}\label{eq:galois-for-module-coalgebras}
	\Quot(\nicefrac{C}{C^{H}})\lgalois{}{C/(I+\k1_H)^+C\ellmap{}I}\coid{l}(H)
    \end{align}
\end{proposition}
\begin{proof}
    The supremum in \(\coid{l}(H)\) is given by the sum of \(\k\)-subspaces. Thus
    the lattice of left coideals is complete. Furthermore, if \(I\) is a left
    coideal then \(I+\k1_H\) is also a left coideal. Thus \((I+\k1_H)^+\) is
    a coideal of \(H\). As a consequence \((I+\k1_H)^+C\) is a coideal of
    \(C\).  
    It is enough to show that the map
    \(\coid{l}(H)\ni I\mapsto I^+C\in\coId(C)\) preserves all suprema when we
    restrict to left coideals which contain~\(1_H\). Let
    \(I_\alpha\in\coId(H)\) (\(\alpha\in\Lambda\)) be such that \(1_H\in
	I_\alpha\) for all \(\alpha\in\Lambda\). Then \((\sum_\alpha
	I_\alpha)^+=\sum_\alpha(I_\alpha^+)\).  The non trivial inclusion is
    \((\sum_\alpha I_\alpha)^+\subseteq\sum_\alpha(I_\alpha^+)\).  Let
    \(k=\sum_\alpha k_\alpha\in(\sum_\alpha I_\alpha)^+\), i.e.  \(k_\alpha\in
	I_\alpha\) for all \(\alpha\in\Lambda\) and \(\sum_\alpha
	k_\alpha\in\ker\epsilon\).  Then \(\sum_\alpha
	k_\alpha=\sum_\alpha(k_\alpha-\epsilon(k_\alpha)1_H)+\sum_\alpha(\epsilon(k_\alpha)1_H)=\sum_\alpha(k_\alpha-\epsilon(k_\alpha)1_H)\).
    Each \(k_\alpha-\epsilon(k_\alpha)1\in I_\alpha^+\) and hence
    \(k\in\sum_\alpha I_\alpha^+\). Now the proposition follows easily.
\end{proof}
\begin{theorem}\label{thm:coGalois_mono}
    Let \(C\) be an \(H\)-module coalgebra over a field \(\k\) such that the
    canonical map \(\can_H\) is injective.  Let \(K_1,K_2\) be two left
    coideals of \(H\) such that both \(\can_{K_1}\) and \(\can_{K_2}\) are
    bijections.  Then \(K_1=K_2\) whenever \(C^{K_1}=C^{K_2}\). 
\end{theorem}
\begin{proof}
    We have the following commutative diagram:
    \begin{center}
	\begin{tikzpicture}[>=angle 60,thick]
	\matrix[matrix,column sep=7mm,row sep=7mm]{
	    &                         & \node(AQ1) {\(K_1\otimes C\)};\\
	    \node(A) {\(C\cotensor_{C^{K_i}} C\)}; & \node(B) {\(C\cotensor_{C^H}C\)}; & \node(C)   {\(H\otimes C\)};\\
	    &                         & \node(AQ2) {\(K_2\otimes C\)};\\
	};
	\begin{scope}
	    \draw[<-] (A) -- node[above]{$\can_{K_1}$} (AQ1);
	    \draw[<-] (A) -- node[below]{$\can_{K_2}$} (AQ2);
	    \draw[>->] (A) -- (B);
	    \draw[<-<] (B) -- node[above,pos=.4]{$\can$}(C);
	    \draw[<-] (C) -- node[right]{$i_1\otimes\id$}(AQ1);
	    \draw[<-] (C) -- node[right]{$i_2\otimes\id $}(AQ2);
	\end{scope}
	\end{tikzpicture} 
    \end{center}
    It follows that \(i_2\otimes\id\circ
	\left(\can_{K_2}\circ\can_{K_1}^{-1}\right)=i_1\otimes\id\). Thus
    \(K_1\subseteq K_2\), since we are over a field; similarly \(K_2\subseteq
	K_1\).
\end{proof}
\index{module coalgebra!Galois connection!closed left coideal}
\begin{corollary}\label{cor:coGalois_closed_elements}
    Let \(C\) be an \(H\)-coextension such that the canonical map \(\can_H\)
    is injective. Then a left coideal~\(K\) of~\(H\) with \(1_H\in K\) is
    a closed element of the Galois
    connection~\eqref{eq:galois-for-module-coalgebras} if \(C\sir C^K\) is
    \(K\)-Galois.
\end{corollary}
\begin{proof}
    Let \(C\) be a \(K\)-Galois coextension, for some left coideal \(K\)of
    \(H\) such that \(1_H\in K\) and let \(\widetilde K\) be the smallest
    closed left coideal such that \(K\subseteq\widetilde K\). Then we have
    the commutative diagram:
    \begin{center}
	\begin{tikzpicture}
	    \node (A) at (0,0) {$H\otimes C$};
	    \node (B) at (3cm,0) {$C\cotensor_{C^H} C$};
	    \node (C) at (0cm,-1.7cm) {$\widetilde K\otimes C$};
	    \node (D) at (3cm,-1.7cm) {$C\cotensor_{C^{\widetilde K}}C$};
	    \node (E) at (0cm,-3.4cm) {$K\otimes C$};
	    \node (F) at (3cm,-3.4cm) {$C\cotensor_{C^{K}}C$};
	    \begin{scope}[>=angle 60,thick]
		\draw[>->] (A) -- node[above]{$\can_H$} (B);
		\draw[>->] (C) -- (A);
		\draw[>->] (D) -- (B);
		\draw[->]  (C) -- node[above]{$\can_{\widetilde K}$} (D);
		\draw[->]  (E) -- node[below]{$\simeq$} node [above]{$\can_K$}  (F);
		\draw[>->] (E) --(C);
		\draw[->] (F) --node[below,rotate=90]{$=$} (D);
	    \end{scope}
	\end{tikzpicture}
    \end{center}
    From the lower commutative square it follows that \(\can_{\widetilde K}\) is
    onto, and from the upper one that it is a monomorphism. The result follows
    now from the previous theorem.
\end{proof}

\section{Takeuchi correspondence}\label{sec:hopf_alg_case}
We show a new simple proof of the Takeuchi correspondence between left coideal
subalgebras and right \(H\)-module coalgebra quotients of a finite dimensional
Hopf algebra. We also show that for an arbitrary Hopf algebra \(H\),
a generalised quotient \(Q\) is closed if and only if \(H^{co Q}\subseteq H\)
is \(Q\)-Galois.  Similarly, for a left coideal subalgebras it is closed if
and only if \(H\sir H^K\) is a \(K\)-Galois coextension. 
\index{Takeuchi correspondence}
\index{Takeuchi correspondence!Skryabin's result}
\begin{theorem}\label{thm:newTakeuchi}
    Let \(H\) be a bialgebra which is flat as an \(R\)-module and let us
    assume that the antipode of \(H\) is bijective.  Then:
	\begin{equation}
	\begin{array}{ccc}
	    \bigg\{K\subseteq H:\,K\,\hyphen\begin{array}{l}\text{\small
			left coideal}\\\text{\small
			subalgebra}\end{array}\bigg\}&\hspace{-.3cm}\lgalois{\psi}{\phi}&\hspace{-.3cm}\bigg\{H/I:\,I\,\hyphen\begin{array}{l}\text{\small
			right ideal}\\\text{\small coideal}\end{array}\bigg\}\\
	    =:\qsub(H)&&=:\qquot(H)
	\end{array}
    \end{equation}
    where $\phi(Q)=H^{co\,Q},\psi(K)=H/K^+H$ is a \textsf{Galois connection}
    which is a restriction of the Galois
    connection~\eqref{eq:galois-connection}.  Moreover, if \(H\) is a Hopf
    algebra it restricts to normal Hopf subalgebras and conormal Hopf
    quotients, and the following holds:
    \begin{enumerate}[topsep=0pt,noitemsep]
        \item[(i)] \(K\in\qsub(H)\) such that \(H\) is (left, right)
             faithfully flat over \(K\), is a closed element of the above
             Galois connection,
        \item[(ii)] \(Q\in\qquot(H)\) such that \(H\) is (left, right)
             faithfully coflat over \(Q\) and \(Q\) is flat as an \(R\)-module
             is a closed element of the Galois
             connection~\eqref{eq:galois-for-hopf-alg},
        \item[(iii)]\label{itm:newTakeuchi_finite} if \(R\) is a field and
             \(H\) is finite dimensional, then \(\phi\) and \(\psi\) are
             \textsf{inverse bijections}.
    \end{enumerate}
    The Galois correspondence restricts to a bijection between elements
    satisfying~(i) and~(ii).
\end{theorem}
The points (i) and (ii) follow from Theorem~\ref{thm:Galois_closeness} (due to
Schauenburg, see also~\cite[Thm~3.10]{ps:gal-cor-hopf-bigal}), while point
\textit{(iii)} follows
from~\cite[Thm~6.1]{ss:projectivity-over-comodule-algebras}, where it is shown
that if \(H\) is finite dimensional then it is free over each of its right (or
left) coideal subalgebras
(see~\cite[Thm~6.6]{ss:projectivity-over-comodule-algebras} and
also~\cite[Cor.~3.3]{ps-hs:gen-hopf-galois}).  This theorem has a long history.
The study of this correspondence, with Hopf algebraic method, goes back
to~\cite{mt:correspondence,mt:rel-hopf-mod}.  Then Masuoka proved~\textit{(i)}
and~\textit{(ii)} for Hopf algebras over a field (with bijective antipode).
When the base ring is a field, \citet[Thm~1.4]{hs:exact-seq-qg} proved that
this bijection restricts to normal Hopf subalgebras and normal Hopf algebra
quotients.  For Hopf algebras over more general rings it was shown
by~\citet[Thm~3.10]{ps:gal-cor-hopf-bigal}.  We can present a new simple proof
of~\ref{thm:newTakeuchi}\textit{(iii)}, which avoids the Skryabin result. 
\begin{proofof}{Theorem~\ref{thm:newTakeuchi}\textit{(iii)}}
    Whenever \(H\) is finite dimensional, for every \(Q\) the extension
    \(H^{co Q}\subseteq H\) is \(Q\)-Galois
    by~\cite[Cor.~3.3]{ps-hs:gen-hopf-galois}. Using Theorem~\ref{thm:mono} we
    get that the map \(\phi\) is a monomorphism. To show that it is an
    isomorphism it is enough to prove that \(\psi\) is a monomorphism.  We now
    want to consider \(H^*\).  To distinguish \(\phi\) and \(\psi\) for \(H\)
    and \(H^*\) we will write \(\phi_H\) and \(\psi_H\)
    considering~\eqref{eq:galois-for-hopf-alg} for \(H\) and \(\phi_{H^*}\)
    and \(\psi_{H^*}\) considering \(H^*\). It turns out that
    \((\psi_H(K))^*=\phi_{H^*}(K^*)\). Now we show that
    \(\can_{K^*}=(\can_K)^*:H^*\otimes_{^{co\,K^*}{H^*}}H^*\sir K^*\otimes
	H^*\) under some natural identifications (note that we consider
    \(H^*\) as a left \(K^*\)-comodule algebra). First let us observe that
    \(^{co\,K^*}H^*=(H/K^+H)^*\):
    \begin{align*}
	^{co\,K^*}H^* & = \{f\in H^*:\,f_{(1)}|_K\otimes f_{(2)}=\epsilon|_K\otimes f\}\\
	              & = \{f\in H^*:\,\forall_{h\in H,k\in K}f_{(1)}(k)f_{(2)}(h)=\epsilon(k)f(h)\}\\
		      & = \{f\in H^*:\,f|_{K^+H}=0\}\\
		      & = \bigl(H/K^+H\bigr)^*
    \end{align*}
    Now, since \(H\cotensor_{H^K}H\) is defined by the kernel diagram of
    finite dimensional spaces:
    \[H\cotensor_{H^K}H\ir H\otimes H\mprr{}{}H\otimes(H/K^+H)\otimes H\]
    we get the cokernel diagram:
    \[H^*\otimes {^{co\,K^*}H^*}\otimes H^*\mprr{}{}H^*\otimes H^*\ir (H\cotensor_{H^K}H)^*\]
    But the above exact sequence defines \(H^*\otimes_{^{co\,K^*}H^*}H^*\).
    Now, we have a commutative diagram:
    \begin{center}
	\begin{tikzpicture}
	    \matrix[column sep=1.3cm,row sep=1cm]{
		\node (A1) {\(\bigl(H\cotensor_{H^K}H\bigr)^*\)}; & \node (A2) {\(K^*\otimes H^*\)};\\
		\node (B1) {\(H^*\otimes_{^{K^*}H^*}H^*\)}; & \\ 
	    };
	    \draw[->] (A1) --node[above]{\((\can_K)^*\)} (A2);
	    \draw[->] (B1) --node[above,rotate=90]{\(\cong\)} (A1);
	    \draw[->] (B1) --node[below right]{\(\can_{K^*}\)} (A2); 
	\end{tikzpicture}
    \end{center}
    since \((\can_K)^*(f\otimes g)(k\otimes h)=f\otimes g\circ \can_K(k\otimes
	h)=f(kh_{(1)})g(h_{(2)})=f_{(1)}(k)f_{(2)}(h_{(1)})g(h_{(2)})=(f_{(1)}\otimes
	f_{(2)}\ast g)(k\otimes h)=\can_{K^*}(f\otimes g)(k\otimes h)\). 
    
    The canonical map \(\can_{K^*}\)  is an isomorphism, since \(H^*\) is
    finite dimensional, and hence \(\can_K\) is a bijection for every left
    coideal subalgebra \(K\) of \(H\). Now the result follows from
    Theorem~\ref{thm:coGalois_mono} and
    Proposition~\ref{prop:properties-of-adjunction}(iv).
\end{proofof}
\index{Takeuchi correspondence!closed elements}
\begin{theorem}\label{thm:closed-of-qquot}
    Let \(H\) be a Hopf algebra such that \(H\) is a flat \(\R\)-module.  Then: 
    \begin{enumerate}
        \item \(Q\in\qquot(H)\) is a \textsf{closed element} of the Galois
              connection~\eqref{eq:galois-for-hopf-alg} if and only if
              \(H/H^{co\,Q}\) is a \(Q\)-Galois extension,
	\item \(K\in\qsub(H)\) is a \textsf{closed element} of the
	    Galois connection~\eqref{eq:galois-for-hopf-alg} if and only if
	    \(H\sir H^K\) is a \(K\)-Galois coextension
	    (Definition~\ref{defi:coGalois}(ii)), i.e. the map:
	    \[K\otimes H\sir H\cotensor_{H/K^+H}H,\ k\otimes h\selmap{}kh_{(1)}\otimes h_{(2)}\]
	    is an isomorphism.
    \end{enumerate}
\end{theorem}
Note that we do not assume that the antipode of~\(H\) is bijective as it is
done in Theorem~\ref{thm:newTakeuchi}.  The flatness of \(H\) is only
needed to show that if \(K\) is closed then \(H\sir H^K\) is a \(K\)-Galois
coextension.
\begin{proof}
    For the first part, it is enough to show that if $Q$ is closed then $H^{co
	Q}\subseteq H$ is $Q$-Galois (see
    Corollary~\ref{cor:Q-Galois_closed}).  If $Q$ is closed then $Q=H/(H^{co
	Q})^+H$.  One can show that for any $K\in\qsub(H)$ the following map
    is an isomorphism:
    \begin{equation}\label{eq:canK}
	H\otimes_{K}H\mpr{}H\otimes H/K^+H,\quad h\otimes_K h'\elmap{}hh'_{\mathit{(1)}}\otimes\overline{h'_{\mathit{(2)}}}
    \end{equation}
    Its inverse is given by $H\otimes H/K^+H\ni h\otimes
    \overline{h'}\,\elmap{}\,hS(h'_{\mathit{(1)}})\otimes_K
    h'_{\mathit{(2)}}\in H\otimes_K H$ which is well defined since \(K\) is
    a left coideal.  Plugging $K=H^{co\,Q}$ into equation~\eqref{eq:canK} we
    observe that this map is the canonical map~\eqref{eq:canonical-map}
    associated to \(Q\).

    Now, if \(H\sir H/K^+H\) is a \(K\)-Galois coextension then it follows
    from Theorem~\ref{thm:coGalois_mono}, using the same argument as in
    Corollary~\ref{cor:coGalois_closed_elements}, that \(K\) is a closed
    element.  Now, let us assume that \(K\) is closed.  Then \(K=H^{co\,Q}\) for
    \(Q=H/K^+H\). By~\cite[Thm~1.4(1)]{hs:normal-bases} we have an isomorphism:
    \begin{equation}\label{eq:cocanK}
	H^{co\,Q}\otimes H\ir H\cotensor_QH\quad k\otimes h\elmap{}kh_{(1)}\otimes h_{(2)}
    \end{equation}
    with inverse \(H\cotensor_QH\ni k\otimes h\selmap{}kS(h_{(1)})\otimes
	h_{(2)}\in H^{co\,Q}\otimes H\).  The above map is the canonical
    map~\eqref{eq:can_K} since \(K=H^{co\,Q}\) and \(Q=H/K^+H\) .
\end{proof}
Let us note that the statement (i) in the above result
generalises~\cite[Corollary~3.3(6)]{ps-hs:gen-hopf-galois}.  Now we get an
answer to the question when the bijection~\eqref{eq:galois-for-hopf-alg} holds
without extra assumptions.
\index{Takeuchi correspondence!without flatness/coflatness}
\begin{corollary}
    The bijective correspondence~\eqref{eq:galois-for-hopf-alg} holds without
    flatness/coflatness assumptions if and only if for every \(Q\in\qquot(H)\)
    \(H/H^{co\,Q}\) is a \(Q\)-Galois extension and for every \(K\in\qsub(H)\)
    \(H/H^K\) is a \(K\)-Galois coextension.
\end{corollary}
For Hopf algebras which are commutative (as algebras) or whose coradical is
cocommutative, the correspondence~\eqref{eq:galois-for-hopf-alg} restricts to
a bijection between normal Hopf subalgebras and normal quotients
(\cite[Thm~3.4.6]{sm:hopf-alg}).
\begin{corollary} 
    Let \(H\) be a finite dimensional Hopf algebra and \(K\) be its left
    coideal subalgebra.  Then \(H/H^K\) is \(K\)-Galois coextension.  
\end{corollary}
\begin{proof}
    Combine Theorem~\ref{thm:newTakeuchi}\textit{(iii)} and
    Theorem~\ref{thm:closed-of-qquot}\textit{(ii)}.
\end{proof}
\index{Takeuchi correspondence!universal enveloping algebra}
\index{universal enveloping algebra!Takeuchi correspondence}
\begin{example}\label{ex:Takeuchi_for_U(g)}
    Let \(\mathfrak{g}\) be a finite dimensional Lie algebra and let us
    consider the Hopf algebra \(\mathcal{U}(\mathfrak{g})\) then the Takeuchi
    correspondence~\eqref{eq:galois-for-hopf-alg} is a bijection.
\end{example}
\begin{proof}
    Takeuchi showed that~\eqref{eq:galois-for-hopf-alg} gives a bijective
    correspondence between coideal subalgebras of a commutative Hopf algebra
    \(H\) and generalised quotients over which \(H\) is faithfully coflat
    (see~\cite[Thm.~4]{mt:rel-hopf-mod}).  By
    Proposition~\ref{prop:enveloping-algebra-quotients_} and
    Theorem~\ref{thm:enveloping-algebra-quotients} we get that every
    generalised quotient of \(\mathcal{U}(\mathfrak{g})\) is of the form
    \(\mathcal{U}(\mathfrak{g})/K^+\mathcal{U}(\mathfrak{g})\) for some
    coideal subalgebra \(K\).
\end{proof}
Moreover, by Proposition~\ref{prop:enveloping-algebra-quotients_} and
Theorem~\ref{thm:enveloping-algebra-quotients} we get that every coideal
subalgebra of \(\mathcal{U}(\mathfrak{g})\) is of the form
\(\mathcal{U}(\mathfrak{h})\) where \(\mathfrak{h}\) is a Lie subalgebra of
\(\mathfrak{g}\), and every generalised quotient is of the form
\(\mathcal{U}(\mathfrak{g})/\mathcal{U}(\mathfrak{h})^+\mathcal{U}(\mathfrak{g})\).

Theorem~\ref{thm:closed-of-qquot} and Theorem~\ref{thm:connection_over_field}
give the following
\index{comodule algebra!Galois correspondence!over a field}
\begin{corollary}
    Let \(A/B\) be an \(H\)-extension over a field \(k\), where \(H\) is
    a Hopf algebra.  Then every closed
    generalised quotient \(Q\in\qquot(H)\) is of the form \(H/K^+H\) by
    Theorem~\ref{thm:connection_over_field}.  Thus \(Q\) cannot be closed
    in~\eqref{eq:galois-connection} if \(Q\) is not closed in the Galois
    connection~\eqref{eq:galois-for-hopf-alg}, i.e. when
    \(can_Q:H\otimes_{H^{co\,Q}}H\sir H\otimes Q\) is not bijective.
\end{corollary}
We can show what is actually lacking to get the converse of
Corollary~\ref{cor:Q-Galois_closed}.
\index{comodule algebra!Galois correspondence!closed generalised quotients}
\begin{theorem}\label{thm:Q-Galois_closed_over_field}
    Let \(A/B\) be an \(H\)-Galois extension over a field \(\k\).  Then the
    following two conditions are equivalent:
    \begin{enumerate}
	    \item the canonical map \(\can_Q:A\otimes_{A^{\co\,Q}}A\sir A\otimes
			Q\) is a bijection, and
	    \item
	    \begin{enumerate}
                \item[(a)] \(Q\in\qquot(H)\) is closed in the Galois
                     connection~\eqref{eq:galois-connection}, and
                \item[(b)] the map
                     \(\delta_A\otimes\delta_A:A\otimes_{A^{\co\,Q}}A\sir(A\otimes
                         H)\otimes_{A\otimes H^{\co\,Q}}(A\otimes H)\) is an
                     injection.  
	    \end{enumerate}
    \end{enumerate}
\end{theorem}
The map \(\delta_A\otimes\delta_A\) is well defined since \(\delta_A\) is
a morphism of \(Q\)-comodules.  However the splitting of \(\delta_A\) given by
the counit \(\epsilon_H\), does not define a splitting of
\(\delta_A\otimes\delta_A\) in general.
\begin{proof}
    First let us assume that \(Q\) is a closed element.  Then by
    theorem~\ref{thm:connection_over_field} \(Q=H/K^+H\) for some
    \(K\in\qsub(H)\) and thus \(Q\) is closed
    in~\eqref{eq:galois-for-hopf-alg}.  Hence, by
    Theorem~\ref{thm:closed-of-qquot}(ii) the map
    \(\can_Q:H\otimes_{H^{\co\,Q}H}\sir H\otimes Q\) is a bijection.  Now let
    us consider the following commutative diagram:
    \begin{center}
	\begin{tikzpicture}
	    \matrix[matrix of nodes, column sep=1cm, row sep=1cm]{
		|(A1)| \(A\otimes_{A^{\co\,Q}}A\) & |(A2)| \((A\otimes H)\otimes_{A\otimes H^{\co\,Q}}(A\otimes H)\) & |(A3)| \(A\otimes(H\otimes_{H^{\co\,Q}}H)\) \\
		|(B1)| \(A\otimes Q\)                   &                                                & |(B3)| \(A\otimes H\otimes Q\) \\
	    };
	    \draw[->] (A1) --node[left]{\(\can_Q\)} (B1);
	    \draw[->] (A1) --node[above]{\(\delta_A\otimes\delta_A\)} (A2);
	    \draw[->] (A2) --node[above]{\(\cong\)} (A3);
	    \draw[->] (A3) --node[right]{\(\id_A\otimes\can_Q\)} (B3);
	    \draw[->] (B1) --node[below]{\(\delta_A\otimes\id_Q\)} (B3);
	\end{tikzpicture}
    \end{center}
    where \(\delta_A\) is the \(H\)-comodule structure map.  Now, since
    \(\delta_A\otimes\delta_A\) is an injection and the right vertical map is
    a bijection we conclude that the left vertical homomorphism is
    a monomorphism, and it is onto since \(A/A^{\co\,H}\) is \(H\)-Galois.

    Conversely, if \(\can_Q:A\otimes_{A^{\co\,Q}}A\sir A\otimes Q\) is
    a bijection, then by Corollary~\ref{cor:Q-Galois_closed} it is a closed
    element.  By Theorem~\ref{thm:connection_over_field} we see that
    \(Q=H/K^+H\) for some \(K\in\qsub(H)\). Thus \(Q\) is a closed element in
    the Galois connection~\ref{eq:galois-for-hopf-alg}.  By
    Theorem~\ref{thm:closed-of-qquot}(i) the map
    \(\can_Q:H\otimes_{H^{\co\,Q}}H\sir H\otimes Q\) is an isomorphism.  By
    the commutativity of the above diagram it follows that
    \(\delta_A\otimes\delta\) is an injective map. 
\end{proof}

\section{Galois theory of crossed products}
Let us recall the crossed product construction denoted by \(B\#_\sigma H\) and
introduced in Example~\ref{ex:comodule_algebras}(iii).  We describe closed
elements of the Galois connection~\eqref{eq:galois-connection} when \(A\) is
a crossed product. 

Let us note that the results presented below apply to finite Hopf--Galois
extensions of division rings, since these are always crossed products
by~\cite[Thm~8.3.7]{sm:hopf-alg}.
\index{crossed product!Galois correspondence}
\index{comodule algebra!Galois correspondence!crossed product}
\begin{theorem}\label{thm:cleft-case}
    Let \(A=B\#_\sigma H\) be an \(H\)-crossed product over a ring~\(\R\) such
    that \(B\) is a flat Mittag-Leffler \(R\)-module and \(H\) is flat as an
    \(R\)-module. Then the Galois correspondence~\eqref{eq:galois-connection}
    exists. Moreover, let us assume that the unit morphism \(1_A:R\sir
        A,\;r\selr{}r1_A\) is a pure homomorphism. Then an element
    \(Q\in\qquot(H)\) is \textsf{closed} if and only if the extension
    \(A/A^{co Q}\) is \textsf{\(Q\)-Galois}.
\end{theorem}
\begin{proof}
    First of all, the Galois connection~\eqref{eq:galois-connection} exists,
    since we have a diagram:
    \begin{center}
	\begin{tikzpicture}
	    \matrix[column sep=1cm, row sep=1cm]{
		\node (C) {\(\Sub_\Alg(B\#_\sigma H/B)\)}; & \node (A) {\(\qsub(H)\)};  & \node (B) {\(\qquot(H)\)};\\
	    };
	    \draw[->] ($(A)+(10mm,1.5mm)$) --node[above]{\(\psi\)} ($(B) + (-11mm,1.5mm)$);
	    \draw[<-] ($(A)+(10mm,0mm)$) --node[below]{\(\phi\)} ($(B) + (-11mm,0mm)$);
	    \draw[->] ($(A)+(-10mm,0mm)$) --node[below]{\(\zeta\)} ($(C) + (15.5mm,0mm)$);
	    \draw[<-,dashed] ($(A)+(-10mm,1.5mm)$) --node[above]{\(\omega\)} ($(C) + (15.5mm,1.5mm)$);
	\end{tikzpicture}
    \end{center}
    where \((\phi,\psi)\) is the Galois connection~\eqref{eq:galois-for-hopf-alg}
    and \(\zeta(S)=B\otimes S\). The poset \(\qsub(H)\) is complete, since it
    is closed under arbitrary joins: for a family \(K_i\in\qsub(H)\) (\(i\in
	I\)) the join \(\bigvee_{i\in I}K_i\in\qsub(H)\) is equal to the
    subalgebra generated by \(\sum_{i\in I}K_i\). The map \(\zeta\) preserves
    all intersections. Thus it has a left adjoint \(\omega\), i.e. an order
    preserving map \(\omega:\Sub_\Alg(B\#_\sigma H/B)\sir\qsub(H)\) such that: 
    \[\omega(S)\subseteq K\iff S\subseteq\zeta(K)\]
    for \(S\in\Sub_\Alg(B\#_\sigma H)\) and \(K\in\qsub(H)\).  Then the Galois
    connection~\eqref{eq:galois-connection} has the form
    \((\zeta\circ\phi,\psi\circ \omega)\), since by flatness of~\(B\) we have
    \(B\otimes H^{co\,Q}=(B\otimes H)^{co\,Q}\). 

    Let \(Q\in\qquot(H)\) be a closed element of this Galois connection, i.e.
    \(Q=\psi\omega\zeta\phi(Q)\). Then it is closed element
    of~\eqref{eq:galois-for-hopf-alg}, since it belongs to the image of
    \(\psi\). Thus, by Theorem~\ref{thm:closed-of-qquot}(i), the map:
    \(\can'_{Q}:H\otimes_{H^{co\,Q}}H\sir H\otimes Q,\,h\otimes
	k\selr{}hk_{(1)}\otimes k_{(1)}\) is an isomorphism.  We have
    a commutative diagram:
    \begin{center}
	{\hfill\begin{tikzpicture}
	    \matrix[column sep=1.4cm,row sep=7mm]{
		\node (A0) {\(A\otimes_{A^{co\,Q}}A\)}; & \node (A1) {\(A\otimes Q=B\#_\sigma H\otimes Q\)}; \\
		\node (B0) {\(B\otimes(H\otimes_{H^{co\, Q}}H)\)}; &  \\
		\node (C0) {\(B\otimes(H\otimes_{H^{co\, Q}}H)\)}; & \\
	    };
	    \draw[->] (A0) --node[above]{\(\can_Q\)} (A1);
	    \draw[->] (C0) --node[below right]{\(\id_B\otimes\can'_Q\)} (A1);
	    \draw[->] (A0) --node[left]{\(\beta\)} (B0);
	    \draw[->] (B0) --node[left]{\(\gamma\)} (C0);
	\end{tikzpicture}
	\hfill\refstepcounter{equation}\raisebox{12mm}{(\theequation)}\label{diag:smash_product}\\
	}
    \end{center}
    where \(\beta:(B\#_\sigma H)\otimes (B\#_\sigma H)\sir B\otimes
	(H\otimes_{H^{co\,Q}} H)\) is defined by \(\beta(b\#_\sigma h\otimes
	b'\#_\sigma h')=(b\#_\sigma h\cdot b'\#_\sigma 1_H)\otimes h'\) is an
    isomorphism with the inverse \(\beta^{-1}(b\otimes h\otimes k)=(b\#_\sigma
	h)\otimes(1_B\#_\sigma k)\); and
    \(\gamma:B\otimes(H\otimes_{H^{co\,Q}}H)\sir
	B\otimes(H\otimes_{H^{co\,Q}}H)\) is defined by \(\gamma(b\otimes
	(h\otimes_{H^{co\,Q}}k))=b\sigma(h_{(1)},k_{(1)})\otimes(h_{(2)}\otimes_{H^{co\,Q}}k_{(2)})\).
    Note that \(\gamma\) is well defined since it equal to the composition
    \(\id_B\otimes(\can_Q')^{-1}\circ\can_Q\circ\beta^{-1}\). Observe that
    \(\gamma\) is an isomorphism with inverse \(\gamma^{-1}(b\otimes
	(h\otimes_{H^{co\,Q}}k))=b\sigma^{-1}(h_{(1)},k_{(1)})\otimes(h_{(2)}\otimes_{H^{co\,Q}}k_{(2)})\),
    which is well defined since \(A\#_{\sigma^{-1}}H\) is a crossed product,
    though it might be nonassociative and nonunital, but most importantly it
    is a comodule algebra. 
    It follows that \(\can_Q\) is an isomorphism. The converse follows from
    Corollary~\ref{cor:Q-Galois_closed}, since \(\can:A\otimes_BA\sir A\otimes
	H\) is a bijection.
\end{proof}
Now we formulate a criterion for closedness of subextensions of \(A/B\)
which generalises~\ref{thm:closed-of-qquot}(ii).
\index{crossed product!Galois correspondence}
\index{comodule algebra!Galois correspondence!crossed product}
\begin{theorem}
    Let \(A/B\) be an \(H\)-crossed product extension over a ring \(\R\)
    with~\(B\) a faithfully flat Mittag--Leffler module and~\(H\) a flat (thus
    faithfully flat) \(\R\)-module. Then a subalgebra \(S\in\Sub_\Alg(A/B)\)
    is a \textsf{closed element} of the Galois
    connection~\eqref{eq:galois-connection} if and only if the canonical map:
    \begin{equation}
	\can_S: S\otimes_B A\ir A\cotensor_{\psi(S)} H,\quad\can_S(a\otimes_B b)=ab_{(1)}\otimes b_{(2)}
    \end{equation}
    is an \textsf{isomorphism}.
\end{theorem}
\begin{proof}
    First let us note that the map \(\can_S\) is well defined since it is
    a composition of \(S\otimes_B A\sir A^{co\,\psi(S)}\otimes_B A\), induced by
    the inclusion \(S\subseteq A^{co\psi(S)}\), with \(A^{co\,\psi(S)}\otimes_B
	A\sir A\cotensor_{\psi(S)}H\) of \cite[Thm~1.4]{hs:normal-bases}.

    Let us assume that \(\can_S\) is an isomorphism.  We let
    \(K=H^{co\,\psi(S)}\). Since \(B\) is a flat \(\R\)-module we have
    \(\left(B\#_\sigma H\right)^{co\,\psi(S)}=B\#_\sigma K\). The following
    diagram commutes:
    \begin{center}
	\hfill\begin{tikzpicture}
	    \matrix[matrix, column sep=1.5cm, row sep=.75cm]{
		\node[anchor=east] (A) {\(S\otimes_B A\)};
		& \node[anchor=west] (B) {\(A\cotensor_{\psi(S)} H\)};\\
		\node[anchor=east] (C) {\(B\#_\sigma K\otimes H\)};
		& \node[anchor=west] (D)
		{\(B\otimes\left(H\cotensor_{\psi(S)}H\right)\)};\\ };
	    \draw[->] (A) --node[above]{\(\can_S\)} (B); \draw[->] (C)
	    --node[below]{\(\beta\)} (D); \draw[->] (A)
	    --node[left]{\(\alpha\)} ($(A)+(0,-1)$); \draw[->] (B)
	    --node[above,rotate=-90]{\(\simeq\)} ($(B)+(0,-1)$); 
	\end{tikzpicture}
	\hfill\refstepcounter{equation}\raisebox{12mm}{(\theequation)}\label{diag:1}\\
    \end{center}
    where \(\alpha:S\otimes_BA\sir\left(B\#_\sigma K\right)\otimes_B\left(B\#_\sigma
	    H\right)\sir B\#_\sigma K\otimes H\) is given by
    \[\alpha\left(a\otimes_Bb\#_\sigma h\right)=\left(a\cdot b\#_\sigma 1_H\right)\otimes h\quad\text{for }a\in S\subseteq B\#_\sigma K,\ b\#_\sigma h\in A\] 
    It is well defined since \(K\) is a left comodule subalgebra of \(H\).
    The second vertical map \(\left(B\#_\sigma H\right)\cotensor_{\psi(S)}
	H\sir B\otimes\left(H\cotensor_{\psi(S)}H\right)\) is the natural
    isomorphism, since \(B\) is flat over \(\R\). The map \(\beta:B\#_\sigma
	K\otimes H\sir B\otimes\left(H\cotensor_{\psi(S)}H\right)\) is defined
    by \(\beta(a\#_\sigma k\otimes h)=\left(a\#_\sigma k\cdot 1_B\#_\sigma
	    h_{(1)}\right)\otimes h_{(2)}\). Note that
    \(\beta=\id_B\otimes\can_K\circ\gamma\) where \(\gamma:\left(B\#_\sigma
	    K\right)\otimes H\sir\left(B\#_\sigma K\right)\otimes H\),
    \(\gamma\left(a\#_\sigma k\otimes h\right)\coloneq a\sigma(k_{(1)},
	h_{(1)})\otimes k_{(2)}\otimes h_{(2)}\) is an isomorphism, since
    \(\sigma\) is convolution invertible, while \(\can_K\) denotes the canonical
    map~\eqref{eq:cocanK}. Furthermore, the map \(\can_K\) is an isomorphism
    by Theorem~\ref{thm:closed-of-qquot}(ii), since \(K:=H^{co\,\psi(S)}\).
    By commutativity of the above diagram it follows that \(\alpha\) is an
    isomorphism. Now, let us consider the commutative diagram:
    \begin{center}
	\begin{tikzpicture}
	    \matrix[matrix, column sep=1.5cm, row sep=.75cm]{
		\node (A) {\(S\otimes_B\left(B\#_\sigma H\right)\)}; & \node (B) {\(B\#_\sigma K\otimes H\)};\\
		\node (C) {\(\left(B\#_\sigma K\right)\otimes_B \left(B\#_\sigma H\right)\)}; & \\
	    };
	    \draw[->] (A) --node[above]{\(\alpha\)} (B);
	    \draw[->] (A) -- (C);
	    \draw[->] (C) --node[below right]{\(\simeq\)} (B);
	\end{tikzpicture}
    \end{center}
    Thus \(S=B\#_\sigma K\) is indeed closed, since \(\alpha\) is an
    isomorphism and \(B\#_\sigma H\) is a faithfully flat \(B\)-module under
    the assumptions made.

    If \(S\) is closed then \(S=A^{co\,\psi(S)}=B\#_\sigma
	(H^{co\,\psi(S)})\), since \(B\) is a flat as an \(\R\)-module. One
    easily checks that \(\alpha\) is an isomorphism with inverse:
    \[\alpha^{-1}(a\#_\sigma k\otimes h)=\left(a\#_\sigma k\right)\otimes_B\left(1_B\#_\sigma h\right)\]
    The left coideal subalgebra \(K=H^{co\,\psi(S)}\) is a closed element
    of~\eqref{eq:galois-for-hopf-alg} hence by
    Theorem~\ref{thm:closed-of-qquot}(ii) \(\can_K\) is an isomorphism and
    thus \(\beta\) is an isomorphism. It follows from~\eqref{diag:1} that
    \(\can_S\) is an isomorphism as well.
\end{proof}

\chapter{Coring approach}\label{chap:coring_approach}
In this chapter we show, how the concept we have presented of Galois Theory
includes classical Galois theory in Field Theory, and which part of classical
theory can be covered by the theory developed.  In the first section we recall
T.Maszczyk approach to Galois extensions using corings.  He showed that
a field extension \(\bE/\bF\) with Galois group \(G\) is a Galois extension if
and only if there is a (concrete) isomorphism of corings
\(\bE\otimes_\bF\bE\cong\Map(G,\bE)\).  Hence the latter is a Galois coring in
the sense of \cite{tb:the_structure_of_corings} (see
also~\cite{rw:from_galois_ext_to_galois_comodules}).  In the next sections we
are going to generalise his approach to a noncommutative setting.  We will
first develop a Galois correspondence between subalgebras of an \(H\)-module
algebra and \(H\)-module subalgebras of \(H\) in
Section~\ref{sec:Galois_connection_for_H-mod_alg} (on
page~\pageref{sec:Galois_connection_for_H-mod_alg}).  Then we will replace the
Galois group acting on the field \(\bE\) with a Hopf algebra action on
a domain \(A\) and the coring \(\Map(G,A)\) with \(\Hom_\k(H,A)\).  For a left
\(H\)-module domain \(A\), for which the action satisfies the requirements of
Definition~\ref{defi:mono_action} (on page~\pageref{defi:mono_action}), we
show that \(\Hom_\k(H,A)\) has a coring structure (see
Proposition~\ref{prop:hom_coring} on page~\pageref{prop:hom_coring}).  In
Definition~\ref{defi:mono_action} we require that \(H\) has a basis of
elements which act on \(A\) as monomorphisms.  There is also a canonical
coring homomorphism from the Sweedler coring:
\begin{equation}\label{eq:coring_can}
    \can:A\otimes_{A^H}A\ir\Hom_\k(H,A),\;\can(a\otimes_{A^H} b)=\bigl(H\ni h\selmap{}a(h\cdot b)\in A\bigr)
\end{equation}
where \(A^H\coloneq\{a\in H:\forall_{h\in H}h\cdot a =\epsilon(h)a\}\).  Then
we construct a Galois connection between subalgebras of \(A\) and quotient
corings of \(\Hom_\k(H,A)\),  which extends the Galois connection between
subalgebras of \(A\) and generalised subalgebras of \(H\) (cf.
Proposition~\ref{prop:coring_connection} on
page~\pageref{prop:coring_connection}) via: \(\qsub(H)\ni
    K\selmap{}\Hom_\k(K,A)\in\Quot(\Hom_\k(H,A))\), where
\(\Quot(\Hom_\k(H,A))\) is a complete lattice of quotient corings.
Interestingly, this extension gives rise to a bijection between closed
subalgebras of \(H\) and closed quotient corings of \(\Hom_\k(H,A)\) (see
Corollary~\ref{cor:coring_closed} on page~\pageref{cor:coring_closed}).
Furthermore, if we assume that the above canonical map~\eqref{eq:coring_can}
is an epimorphism then a coring \(\Hom_\k(K,A)\in\Quot(\Hom_\k(H,A))\) is
a closed element if the canonical map:
\begin{equation*}\label{eq:coring_can_K}
    \can:A\otimes_{A^K}A\ir\Hom_\k(K,A),\;\can(a\otimes_{A^K} b)=\bigl(K\ni k\selmap{}a(k\cdot b)\in A\bigr)
\end{equation*}
is an isomorphism, where \(A^K\coloneq\{a\in A:\forall_{k\in K} k\cdot
	a=\epsilon(k)a\}\).  Moreover, if \(H\) is a finite dimensional Hopf
algebra and \(A\) is \(\Hom_\k(H,A)\)-Galois, i.e.  the homomorphism
\(\can\)~\eqref{eq:coring_can} is an isomorphism, then a subalgebra \(S\) is
closed if the map:
\[\can_S:A\otimes_SA\sir A\otimes_{A^{\Psi(S)}}A\sir\Hom(\Psi(S),A)\] 
is an isomorphism.

\section{Classical Galois theory and corings}
Here we quote the Maszczyk approach to Galois theory using corings.

\index{Maszczyka approach}
\begin{theorem}[{\cite[Corollary~2.3]{tm:galois-struct}}]\label{thm:coring}
    Let $\bF\subseteq \bE$ be a finite field extension with the Galois group
    $G=\Gal(\bE/\bF)$.  Then $\bF\subseteq \bE$ is a Galois extension if and
    only if the map:
    \begin{center}
	\begin{tikzpicture}
	    \matrix[matrix of nodes,column sep=1cm]{
	    |(A)| $\can:\bE\otimes_{\bF}\bE$	& |(B)| $\Map(G,\bE)$\\
	    |(C)| $\phantom{\can:}e_1\otimes_\bF e_2$		& |(D)| $\hspace{2mm}(g\mapsto e_1g(e_2))$\\
	    };
	    \begin{scope}
	    \draw[->,>=angle 60,thick] (A) -- (B);
	    \draw[|->,rotate=-90,>=angle 60,thick] (C) -- (D);
	    \end{scope}
	\end{tikzpicture}
    \end{center}
    is a bijection.
\end{theorem}
We will omit the proof as it is very similar to the proof of
Proposition~\ref{prop:can_H_for_field_ext} (on
page~\pageref{prop:can_H_for_field_ext}).
Furthermore, in~\cite[Proposition 2.5]{tm:galois-struct} it is shown that
$\can$ is an isomorphism of corings.
\index{coring}
\begin{definition}
    Let \(A\) be a ring (not necessarily commutative).  Then
    an \(A\)-\bold{coring} is a triple \((\mathcal{K},\Delta,\epsilon)\), where
    \(\mathcal{K}\) is an \(A\)-bimodule together with two \(A\)-bilinear
    maps: \(\Delta:\mathcal{K}\sir\mathcal{K}_A\otimes_A\mathcal{K}\) and
    \(\mathcal{K}\sir A\) such that the  following diagrams commute: 
    \begin{center}
	\begin{tikzpicture}
	    \matrix[column sep=1.5cm, row sep=1cm]{
		\node (A1) {\(\mathcal{K}\)};    & \node (A2) {\(\mathcal{K}{_A\otimes_A}\mathcal{K}\)};\\
		\node (B1) {\(\mathcal{K}{_A\otimes_A}\mathcal{K}\)}; & \node (B2) {\(\mathcal{K}{_A\otimes_A}\mathcal{K}{_A\otimes_A}\mathcal{K}\)};\\
	    };
	    \draw[->] (A1) --node[above]{\(\Delta\)} (A2);
	    \draw[->] (A1) --node[left]{\(\Delta\)} (B1);
	    \draw[->] (A2) --node[right]{\(\Delta{_A\otimes_A}\id_\mathcal{K}\)} (B2);
	    \draw[->] (B1) --node[below]{\(\id_\mathcal{K}{_A\otimes_A}\Delta\)} (B2);
	\end{tikzpicture}
    \end{center}
    \begin{center}
	\begin{tikzpicture}
	    \matrix[column sep=1cm,row sep=1cm]{
		\node (A1) {\(\mathcal{K}\)}; & \node (A2) {\(\mathcal{K}{_A\otimes_A}\mathcal{K}\)}; \\
		                    & \node (B) {\(A{_A\otimes_A}\mathcal{K}\)};\\
	    };
	    \draw[->] (A1) --node[above]{\(\Delta\)} (A2);
	    \draw[->] (A1) --node[below left]{\(\cong\)} (B);
	    \draw[->] (A2) --node[right]{\(\epsilon{_A\otimes_A}\id_\mathcal{K}\)} (B);
	\end{tikzpicture}
	\begin{tikzpicture}
	    \matrix[column sep=1cm,row sep=1cm]{
		\node (A1) {\(\mathcal{K}\)}; & \node (A2) {\(\mathcal{K}{_A\otimes_A}\mathcal{K}\)}; \\
		                              & \node (B) {\(\mathcal{K}{_A\otimes_A}A\)};\\
	    };
	    \draw[->] (A1) --node[above]{\(\Delta\)} (A2);
	    \draw[->] (A1) --node[below left]{\(\cong\)} (B);
	    \draw[->] (A2) --node[right]{\(\id_\mathcal{K}{_A\otimes_A}\epsilon\)} (B);
	\end{tikzpicture}
    \end{center}
    In other words an \(A\)-coring is a comonoid in the monoidal category of
    \(A\)-bimodules.  A morphism of corings
    \(f:(\mathcal{K},\Delta_\mathcal{K},\epsilon_\mathcal{K})\ir(\mathcal{K}',\Delta_{\mathcal{K}'},\epsilon_{\mathcal{K}'})\)
    is an \(A\)-bilinear map \(f:\mathcal{K}\sir\mathcal{K}'\) such that:
    \(\Delta_{\mathcal{K'}}\circ f = f\otimes f\circ \Delta_\mathcal{K}\) and
    \(\epsilon_{\mathcal{K}'}\circ f=\epsilon_\mathcal{K}\).
\end{definition}
\index{coring!Sweedler coring}
\begin{example}\label{ex:corings}
    Here we give two important examples of corings.
    \begin{enumerate}
	\item Let $\bE/\bF$ a be field extension then $\bE\otimes_\bF\bE$ is
	      a coring, called the \emph{Sweedler coring}, with the
	      comultiplication:
	      \[\Delta:\bE\otimes_\bF\bE\ni e'\otimes_\bF e\lelmap{} (e\otimes_\bF 1_\bE)\otimes_\bE(1_\bE\otimes_\bF e')\in (\bE\otimes_\bF\bE)\otimes_\bE(\bE\otimes_\bF\bE)\] 
	      and counit map $m:\bE\otimes_\bF\bE\mpr{}\bE$, the multiplication
	      map of $\bE$.  
	\index{field extension!coring|see{coring!field extension}}
	\index{coring!field extension}
	\item Let \(\bE/\bF\) be a field extension and let $G$ be a finite
	      subgroup of $\Gal(\bE/\bF)$.  Then $\Map(G,\bE)$ is a coring
	      with comultiplication induced by the multiplication of $G$,
	      $m:G\times G\mpr{}G$.
	      \[\Delta:\Map(G,\bE)\lmpr{m^*}\Map(G\times G,\bE)\simeq \Map(G,\bE)\otimes_\bE\Map(G,\bE)\]
	      where $\Map(G,\bE)\otimes_\bE\Map(G,\bE)\simeq \Map(G\times
	      G,\bE)$ is given by
	      \[\phi_1\otimes_\bE\phi_2\mapsto \Big((g_1,g_2)\;\mapsto\;\phi_1(g_1)g_2\big(\phi_2(g_2)\big)\Big)\] 
	      and the counit $\epsilon:\Map(G,\bE)\mpr{}\bE$ given by
	      $\epsilon(\phi)=\phi(e_G)$, where $e_G$ denotes the identity of
	      the group $G$.
    \end{enumerate}
\end{example}
The coring $\Map(G,\bE)$ was introduced by T.~Maszczyk.  He observed that the
canonical map $\can$ of Theorem~\ref{thm:coring} is a morphism of corings
(see~\cite[Proposition~2.5]{tm:galois-struct}).
\begin{proposition}
Let $\bF\subseteq \bE$ be a finite field extension and $G\leq Gal(\bE/\bF)$ then the map 
\begin{center}
    \begin{tikzpicture}
	\matrix[matrix of nodes,column sep=1cm]{
	    |(A)| $\can:\bE\otimes_{\bF}\bE$          & |(B)| $\Map(G,\bE)$\\
	    |(C)| $\phantom{\can:}e_1\otimes_\bF e_2$ & |(D)| $\hspace{2mm}(g\mapsto e_1g(e_2))$\\
	};
	\begin{scope}
	    \draw[->,>=angle 60,thick] (A) -- (B);
	    \draw[|->,rotate=-90,>=angle 60,thick] (C) -- (D);
	\end{scope}
    \end{tikzpicture}
\end{center}
is a morphism of corings.
\end{proposition}

Under the assumption that \(\bF\subseteq\bE\) is a finite Galois extension
with the Galois group \(Gal(\bE/\bF)=G\) the coring \(\Map(G,\bE)\) can be
realised by many Hopf algebras, i.e. there are many Hopf algebras \(H\) such
that \(\Map(G,\bE)\simeq \bE\otimes H\) as \(\bE\)-corings.  Let \(\bE/\bF\)
be a finite separable field extension. Then Hopf algebras \(H\) for which
\(\bE/\bF\) is \(H\)-Galois are classified as forms of \(\widetilde\bE[G]\),
where \(\widetilde\bE\) is the normal closure of \(\bE\) and
\(G\coloneq\Gal(\widetilde\bE/\bF)/\Gal(\widetilde\bE/\bE)\)
(see~\cite{cg-bp:separable_field_extensions}). 

The canonical map of corings \mbox{$\bE\otimes_\bF \bE\mpr{}\Map(G,\bE)$} is
closely related to the canonical map of the $\bK[G]^*$-Hopf--Galois extension
$\bF\subseteq \bE$ (where \(\bK\subseteq\bF\) is a finite field extension).
The \(\bK[G]^*\)-comodule structure map
\(\delta:\bE\mpr{}\bE\otimes_\bK\bK[G]^*\) is given by
\begin{center}
    \begin{tikzpicture}
	\matrix[matrix of math nodes,column sep=1cm]{
	    |(A)| \phantom{\otimes_\bF}\bE & |(B)| \bE\otimes_{\bF}\bE & |(C)| \Map(G,\bE) & |(D)| \bE\otimes_\bK \bK[G]^*\\
	    |(E)| \phantom{\otimes_\bF}e   & |(F)| 1\otimes_\bF e      & |(G)| (g\mapsto g(e))   & |(H)| e_{(0)}\otimes e_{(1)}\\
	};
	\begin{scope}
	    \draw[->,>=angle 60,thick] (A) -- (B); 
	    \draw[->,>=angle 60,thick] (B) -- node[above,pos=.4]{$\can$} (C);
	    \draw[->,>=angle 60,thick] (C) -- node[above]{$\simeq$} (D);
	    \draw[|->,>=angle 60,thick] ($(F)-(1.7cm,0)$) -- (F); 
	    \draw[|->,>=angle 60,thick] (F) -- +(1.8cm,0);
	    \draw[|->,>=angle 60,thick] (G) -- +(2.2cm,0);
	\end{scope}
    \end{tikzpicture}
\end{center}
where \(G=\Gal(\bE/\bF)\) and the last isomorphism is given by
\[\bE\otimes_\bK \bK[G]^*\sir{}\Map(G,\bE),\quad e\otimes f\selmap{}\bigl(G\ni g\selmap{}f(g)e\in\bE\bigr)\]
The coaction \(\delta\) is a \(\bK\)-linear algebra homomorphism.  The
composition:
\[\bE\otimes_{\bF}\bE\mpr{\can}\Map(G,\bE)\cong\bE\otimes_\bF\bK[G]^*\]
is the canonical map of \(\bE\) as a \(\bK[G]^*\)-comodule algebra.  Thus we
can realise the isomorphism of corings \(\can\) as a canonical map of
a Hopf--Galois extension. Note that there is no canonical Hopf algebra:
\(\bK\) was any field such that \(\bF\) is its finite extension.  For an
infinite group \(G\) we cannot use the canonical coring \(\Map(G,\bE)\) but we
can still use Hopf algebras. 

\section{Lattices of subcorings and quotient corings}\label{sec:corings_lattices}
First we focus on the lattices of subcorings and quotient corings.  We show
that both of them are complete. Then  we prove that there is an epimorphism
from the lattice of subcorings of the coring \(\Map(G,\bE)\), considered in
the previous section, to the lattice of congruences of \(G\) when considered
as a semigroup.  We also show that there is an epimorphism from the lattice of
coideals of \(\Map(G,\bE)\) to the lattice of submonoids of \(G\).

\index{subcoring}
\index{coring!subcoring}
\begin{definition}
    A \textbf{subcoring} $\mathcal{K}'$ of an $A$-coring $\mathcal{K}$ is an
    $A$-coring $\mathcal{K}'$ such that $\mathcal{K}'$ is a subbimodule of
    $\mathcal{K}$ and the inclusion $\mathcal{K}'\subseteq \mathcal{K}$ is
    a morphism of $A$-corings.
\end{definition}
\noindent If $\mathcal{K}$ is pure as a left and right $A$-module then a subcoring is
a subbimodule such that $\Delta|_{\mathcal{K}'}$ takes values in $\mathcal{K}'{_A\otimes_A} \mathcal{K}'\subseteq
\mathcal{K}{_A\otimes_A} \mathcal{K}$.
\index{coring!lattice of subcorings}
\begin{proposition}
    Subcorings of an $A$-coring $\mathcal{K}$ form a complete lattice ordered by inclusion.
\end{proposition}
\begin{proof}
    Let $\mathcal{K}'$ and $\mathcal{K}''$ be subcorings with structure maps
    $\Delta',\epsilon'$ and $\Delta'',\epsilon''$, respectively.  Then
    $\mathcal{K}'+\mathcal{K}''$ is a subcoring of the coring $\mathcal{K}$:  
    \begin{center}
	\begin{tikzpicture}
	    \matrix[matrix of math nodes,column sep=1cm,row sep=5mm]{
		|(A)| (\mathcal{K}'\otimes \mathcal{K}')\oplus (\mathcal{K}''\otimes \mathcal{K}'') & |(B)| (\mathcal{K}'\oplus \mathcal{K}'')\otimes (\mathcal{K}'\oplus \mathcal{K}'') \\
		|(C)| \mathcal{K}'\oplus \mathcal{K}''                                  & \\
		|(E)| \mathcal{K}'+\mathcal{K}''                                   & |(D)| (\mathcal{K}'+\mathcal{K}'')\otimes (\mathcal{K}'+\mathcal{K}'') \\
		|(F)| \mathcal{K}                                                  & |(G)| \mathcal{K}\otimes \mathcal{K} \\
	    };	
	    \begin{scope}
		\draw[->] (A) -- (B);
		\draw[->] (C) --node[left] {$\Delta'\oplus\Delta''$} (A);
		\draw[->] (B) -- (D);
		\draw[->] (D) --node[right] {$i\otimes i$} (G);
		\draw[<-] (E) -- (C);
		\draw[->,dashed] (E) --node[above]{$\exists$} (D);
		\draw[->] (F) --node[below]{$\Delta$} (G);
		\draw[->] (E) --node[left]{$i$} (F);
	    \end{scope}
	\end{tikzpicture}
    \end{center}
    The kernel of the composition $f:\mathcal{K}'\oplus
    \mathcal{K}''\mpr{}(\mathcal{K}'+\mathcal{K}'')\otimes(\mathcal{K}'+
    \mathcal{K}'')$ contains the kernel of the map $\mathcal{K}'\oplus
    \mathcal{K}''\mpr{}\mathcal{K}'+\mathcal{K}''$.  Thus the map
    $\mathcal{K}'+\mathcal{K}''\mpr{}(\mathcal{K}'+\mathcal{K}'')\otimes(\mathcal{K}'+
    \mathcal{K}'')$ exists.  Let $(l,-l)\in\ker(\mathcal{K}'\oplus
    \mathcal{K}''\mpr{}\mathcal{K}'+\mathcal{K}'')$ (where $l\in
    \mathcal{K}'\cap \mathcal{K}''$).  Then
    $f(l,-l)=(l_{\mathit{(1)'}}-l_{\mathit{(1)''}})\otimes(l_{\mathit{(2)'}}-l_{\mathit{(2)''}})$,
    where we write $\Delta'(l)=l_{\mathit{(1)'}}\otimes l_{\mathit{(2)'}}$ and
    $\Delta''(l)=l_{\mathit{(1)''}}\otimes l_{\mathit{(2)''}}$.  Then
    $i\otimes i\circ f(l,-l)=0$.  It remains to observe that the map $i\otimes
    i:(\mathcal{K}'+\mathcal{K}'')\otimes(\mathcal{K}'+\mathcal{K}'')\mpr{}\mathcal{K}\otimes
    \mathcal{K}$ is a monomorphism, which follows from the commutativity of
    the diagram: 
    \begin{center}
	\begin{tikzpicture}
	    \matrix[matrix of math nodes, column sep=1.5cm,row sep=5mm]{
		|(A)| \mathcal{K}'+\mathcal{K}'' & |(B)| (\mathcal{K}'+\mathcal{K}'')\otimes(\mathcal{K}'+\mathcal{K}'') \\
		|(C)| \mathcal{K}                & |(D)| \mathcal{K}\otimes \mathcal{K} \\
	    };
	    \begin{scope}
		\draw[->] (A) --node[above]{$\Delta'+\Delta''$} (B);	
		\draw[->] (A) --node[left]{$i$} (C);  
		\draw[->] (B) --node[right]{$i\otimes i$} (D);  
		\draw[->] (C) --node[below]{$\Delta$} (D);	
	    \end{scope}
	\end{tikzpicture}
    \end{center}
\end{proof}
\index{coring!coideal}
\begin{definition}
    Let \(\mathcal{K}\) be an \(A\)-coring.  Then a coideal of \(\mathcal{K}\)
    is a kernel of a morphism of \(A\)-corings.
\end{definition}
If \(I\) is pure as a left and right \(A\)-module then a coideal is an
\(A\)-subbimodule \(I\) such that \(\Delta(I)\subseteq I{_A\otimes_A}
    \mathcal{K}+\mathcal{K}{_A\otimes_A}I\) and \(I\subseteq\ker\epsilon\).
Note that by~\cite[17.17]{tb-rw:corings-and-comodules} the above definition is
equivalent to~\cite[17.14]{tb-rw:corings-and-comodules}, where coideals are
defined as kernels of morphisms of corings which are onto.
\index{coring!lattice of quotients}
\begin{proposition}
    Let \(\mathcal{K}\) be an \(A\)-coring.  Then the poset
    \(\Quot(\mathcal{K})\) of quotient \(A\)-corings is complete. 
\end{proposition}
\begin{proof}
    It is enough to show that the dual poset of coideals of
    \(\mathcal{K}\) is complete.  In order to show this we prove that it is
    closed under arbitrary suprema.  Let \((J_i)_{i\in I}\) be a family of coideals of the
    coring \(\mathcal{K}\).  We let \(J\coloneq\sum_{i\in I}J_i\),
    \(\pi_i:\mathcal{K}\sir\mathcal{K}/J_i\) and
    \(\pi:\mathcal{K}\sir\mathcal{K}/J\) be the natural projections.
    According to~\cite[17.14]{tb-rw:corings-and-comodules} it is enough to
    show that there exists a coring structure on \(\mathcal{K}/J\).
    Clearly, we have a diagram:
    \begin{center}
	\begin{tikzpicture}
	    \matrix[column sep=1.5cm, row sep=9mm]{
		\node (A) {\(\mathcal{K}\)};   & \node (B) {\(\mathcal{K}{_A\otimes_A}\mathcal{K}\)};\\
		\node (C) {\(\mathcal{K}/J\)}; & \node (D) {\(\mathcal{K}/J{_A\otimes_A}\mathcal{K}/J\)};\\
	    };
	    \draw[->] (A)   --node[above]{\(\Delta\)} (B);
	    \draw[->>] (A)  --node[left]{\(\pi\)} (C);
	    \draw[->>] (B)  --node[right]{\(\pi\otimes\pi\)} (D);
	    \draw[->,dashed] (C) --node[below]{\(\overline\Delta\)} (D);
	\end{tikzpicture}
    \end{center}
    The map \(\overline{\Delta}\) exists since for any \(\sum_{i}x_i\in
	J\), where \(x_i\in J_i\,\forall i\in I\), we have
    \(\pi_i\otimes\pi_i\circ\Delta(x_i)=\Delta_i\circ\pi_i(x_i)=0\).  We
    get \(\pi\otimes\pi\circ\Delta(\sum_{i}x_i)=0\) and hence
    \(\ker\pi\subseteq\ker(\pi\otimes\pi\circ\Delta)\).  Furthermore,
    \(\overline{\Delta}\) is coassociative.  For this we consider the
    following commutative diagram:
    \begin{center}
	\begin{tikzpicture}
	    \matrix[column sep=1.5cm, row sep=1.2cm]{
		\node (A1) {\(\mathcal{K}\)};   & \node (A2) {\(\mathcal{K}^{\otimes_A2}\)};   & \node (A3) {\(\mathcal{K}^{\otimes_A3}\)};\\
		\node (B1) {\(\mathcal{K}/J\)}; & \node (B2) {\(\mathcal{K}/J^{\otimes_A2}\)}; & \node (B3) {\(\mathcal{K}/J^{\otimes_A3}\)};\\
	    };
	    \draw[->] (A1)   --node[above]{\(\Delta\)} (A2);
	    \draw[->] ($(A2)+(5mm,1mm)$) --node[above]{\(\Delta\otimes\id\)} ($(A3)+(-5mm,1mm)$);
	    \draw[->] ($(A2)+(5mm,-0.7mm)$) --node[below]{\(\id\otimes\Delta\)} ($(A3)+(-5mm,-0.7mm)$);

	    \draw[->>] (A1)  --node[left]{\(\pi\)} (B1);
	    \draw[->>] (A2)  --node[left]{\(\pi^{\otimes2}\)} (B2);
	    \draw[->>] (A3)  --node[right]{\(\pi^{\otimes3}\)} (B3);

	    \draw[->] (B1) --node[below]{\(\overline\Delta\)} (B2);
	    \draw[->] ($(B2)+(7mm,1mm)$) --node[above]{\(\overline{\Delta}\otimes\id\)} ($(B3)+(-7mm,1mm)$);
	    \draw[->] ($(B2)+(7mm,-0.7mm)$) --node[below]{\(\id\otimes\overline{\Delta}\)} ($(B3)+(-7mm,-0.7mm)$);
	\end{tikzpicture}
    \end{center}
    Since \(\pi\) is an epimorphism and both
    \(\overline{\Delta}\otimes\id\circ\overline{\Delta}\) and
    \(\id\otimes_A\overline{\Delta}\circ\overline{\Delta}\) makes the outer
    diagrams commute, they must be equal.  The counit is constructed in
    a similar way:
    \begin{center}
	\begin{tikzpicture}
	    \matrix[column sep=1cm,row sep=5mm]{
		\node (A) {\(\mathcal{K}\)};   & \\
		& \node (B) {\(A\)};\\
		\node (C) {\(\mathcal{K}/J\)}; & \\
	    };
	    \draw[->] (A) --node[above right]{\(\epsilon\)} (B);
	    \draw[->,dashed] (C) --node[below right]{\(\overline{\epsilon}\)} (B);
	    \draw[->>] (A) --node[left]{\(\pi\)} (C);
	\end{tikzpicture}
    \end{center}
    The counit \(\overline{\epsilon}\) exists and is unique such that the
    above diagram commutes since \(\ker\pi\subseteq\ker\epsilon\).  The
    counit axiom can be proved by considering the following diagram:
    \begin{center}
	\begin{tikzpicture}
	    \begin{scope}[style={>=angle 60,thin}]
		\node (A1) at (0cm,0cm) {\(\mathcal{K}/J^{\otimes_A2}\)};
		\node (A2) at (2cm,1.75cm) {\(\mathcal{K}/J\)};
		\node (A3) at (4cm,0cm) {\(\mathcal{K}/J\)};

		\draw[->] ($(A1)+(7mm,1mm)$) --node[above]{\(\overline{\epsilon}\otimes\id\)} ($(A3)+(-3.9mm,1mm)$);
		\draw[->] ($(A1)+(7mm,-0.5mm)$) --node[below]{\(\id\otimes\overline{\epsilon}\)} ($(A3)+(-3.9mm,-0.5mm)$);
		\draw[->] (A2) --node[above left]{\(\overline{\Delta}\)} (A1);
		\draw[->] (A2) --node[above right]{\(=\)} (A3);

		\node (B1) at (0cm,3.5cm) {\(\mathcal{K}^{\otimes_A2}\)};
		\node (B2) at (2cm,5.25cm) {\(\mathcal{K}\)};
		\node (B3) at (4cm,3.5cm) {\(\mathcal{K}\)};

		\draw[->] ($(B1)+(7mm,1mm)$) --node[above]{\(\epsilon\otimes\id\)} ($(B3)+(-3.9mm,1mm)$);
		\draw[->] ($(B1)+(7mm,-0.5mm)$) --node[below]{\(\id\otimes\epsilon\)} ($(B3)+(-3.9mm,-0.5mm)$);
		\draw[->] (B2) --node[above left]{\(\Delta\)} (B1);
		\draw[->] (B2) --node[above right]{\(=\)} (B3);

		\draw[->>] (B1) -- (A1);
		\draw[->>] (B2) -- (A2);
		\draw[->>] (B3) -- (A3);
	    \end{scope}
	\end{tikzpicture}
    \end{center}
    where all the vertical arrows are the respective projection maps \(\pi\).
    From the above diagram we deduce that
    \begin{align*}
	\overline{\epsilon}\otimes\id\circ\overline{\Delta}\circ\pi & = \pi\\
	\id\otimes\overline{\epsilon}\circ\overline{\Delta}\circ\pi & = \pi\\
    \end{align*}
    The map \(\pi\) is an epimorphism, thus \(\overline{\epsilon}\) is
    indeed the counit.  Clearly, \(J\) is the join of \((J_i)_{i\in I}\) in
    \(\coId(\mathcal{K})\).  Hence the poset \(\coId(\mathcal{K})\) is
    complete with respect to the join and thus it is a complete lattice.
\end{proof}
\begin{proposition}\label{prop:ge1}
    There is a Galois epimorphism from the lattice of subcorings of
    $\Map(G,\bE)$ to the lattice of congruences of $G$, denoted as $\Con(G)$,
    where $G$ is considered as a semigroup.
\end{proposition}

As a reminder, a congruence is an equivalence relation which is compatible with
all the algebraic operations, in this case, only with multiplication, i.e. if
\(g_1\) and \(h_1\) are congruent and \(g_2\) and \(h_2\) are then so is the
pair \(g_1g_2\) and \(h_1h_2\), where \(g_i,h_i\in G\), \(i=1,2\).  The lattice
of congruences of a semigroup is dually isomorphic to the lattice of its
quotient semigroups.

\begin{proof}
    Let $C$ be a subcoring of $\Map(G,\bE)$.  Then we define the corresponding
    congruence $\theta_C$ of $G$ as follows:
    \[x\theta_Cy\ \Leftrightarrow\ \mathop\forall\limits_{C\ni f:\,G\tikz[baseline=-.03cm]\draw[->,>=angle 60,thin] (0,2pt) -- (4mm,2pt); \bE}\ f(x)=f(y).\]
    For any $\theta\in \Con(G)$ we can associate a subcoring
    $\Map(G/\theta,\bE)$.  We identify $\Map(G/\theta,\bE)$ with its image under
    the map
    \[\Map(G/\theta,\bE)\lmpr{}\Map(G,\bE)\] 
    induced by the quotient map $G\mpr{}G/\theta$.  We use the following notation:
    \begin{align}
    \Sub_{\textit{coring}}(\Map(G,\bE)) & \lmpr{}\Con(G)\notag\\[-.2cm]    C                             & \lelmap{}\ \theta_C\notag\\[.1cm]
    \Con(G)                             & \lmpr{}\Sub_{\textit{coring}}(\Map(G,\bE))\notag\\[-.2cm]    \theta & \lelmap{}\ C_\theta\notag
    \end{align}
    It is straightforward to see that we have just defined antimonotonic
    morphisms of posets.  We check now the Galois epi property: $C\subseteq
    C_{\theta_C}$ and $\theta=\theta_{C_{\theta}}$.  Let us prove the first
    property: let $x\in C$ then it is easy to see that one can factorise the
    map $x$ through $G/\theta_C$ so $x$ belongs to $\Map(G/\theta_C,\bE)$ and
    thus $C\subseteq \Map(G/\theta_C,\bE)$.  Now we will show that
    $\theta=\theta_{C_\theta}$.  Let us suppose that $x\theta y$.  Then for
    all $f\in C_\theta=\Map(G/\theta,\bE)$, $f(x)=f(y)$ thus by the definition
    of $\theta_{C_\theta}$ we have $x\theta_{C_\theta}y$.  Now if
    $x\theta_{C_\theta}y$ and $x\,\neg\theta\, y$ then one can construct
    a map, as any field $\bE$ has at least two elements, from $G$ to $\bE$
    which factorises through $G/\theta$ and such that $f(x)\neq f(y)$.  But
    this contradicts our assumption $x\theta_{C_\theta}y$ so $x$ and $y$ must
    belong to the same congruence class of $\theta$.
\end{proof}
\begin{proposition}\label{prop:ge2}
    Let \(G\) be a finite group.  Then there exists a Galois epimorphism from
    the lattice of coideals of the coring \(\Map(G,\bE)\) to the lattice of
    submonoids of \(G\).
\end{proposition}
\begin{proof}
    Let $I$ be a coideal of $\Map(G,\bE)$ then $\bigcap_{f\in I} \ker f$,
    where $\ker f$ denotes the set of elements of $G$ which are mapped to $0$,
    is a submonoid of $G$.  It contains the identity of the group $G$, because
    $I\subseteq \ker\epsilon$, where $\epsilon$ is the evaluation at the
    identity of $G$.  Now let us observe that $\bigcap_{f\in I} \ker f$ is
    closed under multiplication: if $g_1$ and $g_2$  belong to all of the
    kernels of elements of $I$ then for any $f\in I$
    $f(g_1g_2)=\Delta(f)(g_1,g_2)=0$, because $\Delta(f)\in I\otimes_\bE
    \Map(G,\bE)\;+\;\Map(G,\bE)\otimes_\bE I$.  The map
    $\theta:I\mapsto\bigcap_{f\in I}\ker f$ reverses the order.  The second
    map of the Galois connection is given as follows.  For any submonoid $G_0$
    of $G$, $\Map(G_0,\bE)$ forms a coring (the identity element of $G$ is
    needed to set the counit as evaluation on it).  $\Map(G_0,\bE)$ is
    a quotient coring of $\Map(G,\bE)$ via the restriction map $f\mapsto
    f|_{G_0}$.  Thus we have a map $\xi$ from submonoids of $G$ to coideals of
    the coring $\Map(G,\bE)$ defined as $\xi(G_0)=\ker\bigl
    (\Map(G,\bE)\mpr{}\Map(G_0,\bE)\bigr)$ which reverses the order.
    Furthermore, one can easily verify that 
    \[\theta\xi=\id_{\Sub_{\textit{mono}}(G)}\quad and\quad \xi\theta\geq \id_{\coId(\Map(G,\bE))}\]
    hence $\theta$ is an epimorphism which is a part of the Galois connection
    $(\theta,\xi)$.  The second inequality follows since \(f\in\xi\theta(I)\)
    if and only if \({\bigcap_{g\in I}\ker g}\subseteq\ker f\).
\end{proof}
In an analogous result for the Hopf algebra \(\k[G]^*\) we were able to show
that every generalised quotient comes from a subgroup of \(G\)
(Proposition~\ref{prop:dual_gruop_algebra_quotients} on
page~\pageref{prop:dual_gruop_algebra_quotients}).  The reason behind this is
that right ideals of \(\k[G]^*\) are always spanned by some of the
\(\delta_g\in\Map(G,\bE)\) (\(\delta_g(h)=1\) if and only if \(h=g\),
otherwise it is \(0\)).

\section{Galois connection for $H$-module algebras}\label{sec:Galois_connection_for_H-mod_alg}
We let \(A\#H\) denote the smash product of \(A\) and \(H\), where \(A\) is an
\(H\)-module algebra (see
Example~\ref{ex:comodule_algebras}\ref{itm:crossed_product}).  The underlying
vector space of \(A\#H\) is \(A\otimes H\).  We will denote by \(a\#h\) (for
\(a\in A\) and \(h\in H\) the simple tensor \(a\otimes h\) as an element of
the smash product \(A\#H\)).  The multiplication is defined by the rule:
\(a\#h\cdot b\#k=a(h_{(1)}\cdot b)\#h_{(2)}k\) for \(a,b\in A\) and \(h,k\in
    H\).  We will denote by \(\Sub_{\Alg^H}(H)\) the poset of right coideal
subalgebras of~\(H\).  
\begin{lemma}\label{lem:invariants}
    Let \(A\) be an \(H\)-module algebra.  Then for any
    \(K\in\Sub_{\Alg^H}(H)\) we have:
    \[\bigl\{a\in A:\forall_{k\in K}\ k\cdot a=\epsilon(k)a\bigr\}=A\cap\Cent_{A\#H}(1_A\#K)\]
    where \(\Cent_{A\#H}(1_A\#K)=\bigl\{x\in A\#H:\forall_{k\in K}\  x\cdot
	    1\#k=1\#k\cdot x\bigr\}\) is the centraliser of \(1_A\#K\) in
    \(A\#H\).
\end{lemma}
\begin{proof}
    The condition \(a\#1\cdot 1\#k=1\#k\cdot a\#1\) translates to
    \(a\#k=(k_{(1)}\cdot a)\# k_{(2)}\).  When we compute
    \(\id_A\otimes\epsilon\) of this equality we get \(k\cdot
	a=\epsilon(k)a\).  The other inclusion is obvious.
\end{proof}
\index{module algebra!Galois connection}
\begin{proposition}\label{prop:H-module_galois_connection}
    Let \(A\) be a left \(H\)-module algebra over a field \(\k\) and let
    \(B=A^H\).  We define two order-reversing morphisms:
    \begin{align*}
	\Phi:\Sub_{\Alg^H}(H)\sir\Sub_\Alg(A/B), & \ \Phi(K)\coloneq A^K=\bigl\{a\in A:\forall_{k\in K}\ k\cdot a=\epsilon(k)a\bigr\}\\
	\intertext{for \(K\in\Sub_{\Alg^H}(H)\) a right coideal subalgebra and} 
    \Psi:\Sub_\Alg(A/B)\sir\Sub_{\Alg^H}(H), & \ \Psi(S)\coloneq H\cap\Cent_{A\#H}(S\#1_H)
    \end{align*} 
    for \(S\in\Sub_\Alg(A/B)\).  Then \((\Phi,\Psi)\) is a Galois connection. 
\end{proposition}
Let us note that \(\Phi(K)\) is the largest subalgebra of \(A\) such that
the \(K\)-action is left \(\Phi(K)\)-linear.
\begin{proof}
    Let \(k\in H\) and \(a\in A\).  We have 
    \[\Phi(K)=\left\{a\in A:\forall_{k\in K}\ a\#1\cdot1\#k=1\#k\cdot a\#1\right\}\]
    by Lemma~\ref{lem:invariants}.  It is clear that \(\Phi(K)\) is
    a subalgebra of \(A\).  Let us show that for \(S\in\Sub_\Alg(A/B)\),
    \(\Psi(S)\) is a right coideal subalgebra of \(H\).  From the definition
    of \(\Psi\) it follows that \(\Psi(S)\) is a subalgebra of \(H\). 
    Now, for \(k\in\Psi(S)\) and for all \(a\in S\) we have \((k_{(1)}\cdot
	a)\#k_{(2)}=a\#k\).  Applying \(\id_A\otimes\Delta_H\) we get
    \(\bigl({k_{(1)}}_{(1)}\cdot a\bigr)\#{k_{(1)}}_{(2)}\otimes
	k_{(2)}=a\#k_{(1)}\otimes k_{(2)}\).  It follows that
    \(\Delta_H(k)\in\Psi(S)\otimes H\).  Thus \(K\) is a right coideal
    subalgebra of \(H\).

    Clearly \(\Phi\) and \(\Psi\) are anti-monotone maps.  The Galois
    properties \(\Psi\Phi\geq\id_{\Sub_{\mathit{Hopf}}(H)}\) and
    \(\Phi\Psi\geq\id_{\Sub_\Alg(A)}\) are easily verified.
\end{proof}
Note that \(\Phi(K)\) is a \(K\)-submodule of \(A\):
\begin{equation*}
    k'\cdot(k\cdot a) = (k'k)\cdot a = \epsilon(k'k)a = \epsilon(k')\epsilon(k)a = \epsilon(k')k\cdot a        
\end{equation*}
for \(k,k'\in K\) and \(a\in\Phi(K)\).  Furthermore, if \(K\) is a normal
Hopf subalgebra then \(\Phi(K)\) is an \(H\)-submodule: 
\[k\cdot(h\cdot a)=(h_{(1)}S(h_{(2)})kh_{(3)})\cdot a=(h_{(1)}\epsilon(S(h_{(2)})kh_{(3)}))\cdot a=\epsilon(k)h\cdot a\]
for any \(k\in K\) and \(h\in H\).  Let us note that if \(K\) is commutative,
\(A^K\) is an \(H\)-submodule, and if moreover \(A\) is a faithful
\(H\)-module then \(K\) is one dimensional.  For every \(h\in H\), \(k\in K\)
and \(a\in A^K\) we have \(k\cdot(h\cdot a)=\epsilon(k)h\cdot a\).  The action
is faithful so we must have \(kh=\epsilon(k)h\).  Then for any non-zero
\(k,k'\in K\) we have \(0\neq kk'=\epsilon(k)k'=\epsilon(k')k\) (it is non
zero since the action is faithful).  

Let \(V\) be a \(\k\)-vector space and let \(W\subseteq V\) then
\(W^\perp\coloneq\{f\in V^*:f|_W=0\}\).  We also will use the same notation
for \(W\subseteq V^*\), \(W^\perp\coloneq\bigcap_{f\in W}\ker f\).  Since we
will use these two maps only if \(V\) is finite dimensional this should not
lead to any confusion (under the identification \(({V^*})^*=V\) these two maps
represent the same morphisms from \(\Sub(V^*)\) to \(\Sub(V)\)).

Let us note that if \(H\) is finite dimensional then \(A\) is a right
\(H\)-comodule algebra if and only if \(A\) is a left \(H^*\)-module
algebra.  Furthermore, for any \(Q\in\qquot(H)\) let \(I_Q\) be the right
ideal coideal of \(H\) such that \(Q=H/I_Q\).  Then we have
\(A^{co\,Q}=A^{I_Q^{\perp}}\) (\(I_Q^{\perp}\) is a right coideal subalgebra
of \(H^*\)).

Let \((\Phi,\Psi)\) be the Galois connection between \(\qsub(H)\) and
\(\Sub_\Alg(A)\), where \(\Phi(K)=A^{K}\).  It exists since
\(\bigvee_{\alpha\in I}K_\alpha\) is equal to the algebra generated by all
the \(K_\alpha\) (\(\alpha\in I\)) and thus \(A^{\bigvee_{\alpha\in I}
	K_{\alpha}}=\bigcap_{\alpha\in I}A^{K_\alpha}\).  Let us note that
\(\phi=\Phi\circ\alpha\), where \(\alpha:\qquot(H)\sir\Sub_{\Alg^H}(H)\),
\(\alpha(Q)\coloneq Q^*\). 
\begin{lemma}
    Let \(A\) be an \(H\)-comodule algebra over a field \(\bK\), with \(H\)
    a finite dimensional Hopf algebra.  Let us also assume that \(\phi\) (and
    thus \(\Phi\)) is injective (this holds for example if \(A\) is
    \(H\)-Galois).  Let \(S\in\Sub_\Alg(A)\).  Then
    \(I_{\psi(S)}=\Psi(S)^{\perp}\). 
\end{lemma}
\begin{proof}
    We have the following diagram:
    \begin{center}
	\begin{tikzpicture}
	    \matrix[column sep=1cm,row sep=1cm]{
		\node (A) {\(\Sub_\Alg(A)\)}; & \node (B) {\(\qquot(H)\)};\\
		& \node (C) {\(\Sub_{\Alg^H}(H^*)\)};\\
	    };
	    \draw[->] ($(A)+(9mm,0.7mm)$) --node[above]{\(\psi\)} ($(B)+(-11mm,0.7mm)$);
	    \draw[->] ($(B)+(-11mm,-0.7mm)$) --node[below]{\(\phi\)} ($(A)+(9mm,-0.7mm)$);
	    \draw[->] ($(A)+(-3mm,-2mm)$) --node[below left]{\(\Psi\)} ($(C)+(-11mm,2.5mm)$);
	    \draw[<-] ($(A)+(0mm,-2mm)$) --node[right]{\(\ \ \Phi\)} ($(C)+(-8mm,2.5mm)$);
	    \draw[<-] (C) --node[left]{\(\cong\)}node[right]{\(\alpha\)} (B);
	\end{tikzpicture}
    \end{center}
    where \(\alpha(Q)=I_Q^\perp\) is a poset isomorphism.  Since
    \(\phi=\Phi\circ\alpha\) we also have \(\Psi=\alpha\circ\psi\) and
    thus \(I_{\psi(S)}=\Psi(S)^\perp\).
\end{proof}

\section{Maszczyk's approach in a noncommutative way}
In this section all algebras are over a fixed base field \(\k\).  Furthermore,
throughout this section we assume that both \(H\) and \(A\) are finite
dimensional over \(\k\).
\index{action through monomorphisms}
\index{module algebra!action through monomorphisms}
\begin{definition}\label{defi:mono_action}
    Let \(A\) be an \(H\)-module algebra with \(H\) finite dimensional.  We
    say that \(H\) \bold{acts through monomorphisms} if there exists a basis
    of \(H\), \(\{h_i\in H:i\in I\}\) for which the following implication
    holds: \(h_i\cdot a=0\Rightarrow a=0\) for all \(i\in I\) and \(a\in A\).
\end{definition}
Note that the definition requires that \(\epsilon(h_i)\neq0\) by letting
\(a=1\).  For example if \(G\) is a group acting by automorphisms on an
algebra \(A\), then it induces a \(\k[G]\)-action through monomorphisms on
\(A\).  Let us note the following theorem:
\index{module algebra!division rings}
\begin{theorem}[{\cite[Thm~8.3.7]{sm:hopf-alg}}]
    Let \(A\) be a left \(H\)-module algebra, where \(A\) is a division ring
    and \(H\) is finite dimensional.  Then the following conditions are
    equivalent:
    \begin{enumerate}
	\item \(A/A^H\) is \(H^*\)-Galois,
	\item \([A:A^H]_r=\dim H\) or \([A:A^H]_l=\dim H\),
	\item \(A^H\subseteq A\) has the normal basis property,
	\item \(A\cong A^H\#_\sigma H^*\) is a crossed product (see
	    Example~\ref{ex:comodule_algebras}(iii)),
	\item \(A\#H\) is simple.
    \end{enumerate}
\end{theorem}
An \(H\)-action on an algebra \(A\), which has a normal basis, satisfies the
requirements of Definition~\ref{defi:mono_action} if there exists a basis
\(\{h_i\in H:i\in I\}\) of \(H\) such that \(Hh_i=H\) for all \(i\in I\),
since the action is induced from the right coaction of \(H^*\) on \(H\).  
\begin{lemma}\label{lem:domain}
    Let \(A\) be a domain and an \(H\)-module algebra, such that the
    \(H\)-action is through monomorphisms (Definition~\ref{defi:mono_action}).
    Let \(\phi_1,\phi_2:H\rightarrow A\) be linear maps.  If for all \(h,k\in
	H\), \(\phi_1(h_{(1)})h_{(2)}\cdot\phi_2(k)=0\) then \(\phi_1=0\) or
    \(\phi_2=0\). 
\end{lemma}
\begin{proof}
    Assume that there exists \(k\in H\) such that \(\phi_2(k)\neq0\).  Then for
    every \(h\in H\), \(\phi_1(h_{(1)})\otimes h_{(2)}=0\), since otherwise
    \(\phi_1(h_{(1)})h_{(2)}\cdot\phi_2(k)\neq0\) (note that we can always
    write \(h_{(1)}\otimes h_{(2)}=\sum_{i\in I} k_i\otimes h_i\) where
    \(h_i,k_i\in H\) and \(\{h_i:i\in I\}\) is a basis of elements guaranteed
    by Definition~\ref{defi:mono_action}).  It follows that for every \(h\in
	H\), \(\phi_1(h)=\phi_1(h_{(1)})\epsilon(h_{(2)})=0\).
\end{proof}
\index{module algebra!Galois connection!coring construction}
\begin{proposition}\label{prop:hom_coring}
    Let \(A\) be a domain and an \(H\)-module algebra such that \(H\)-acts
    through monomorphisms.  Then \(\Hom_\k(H,A)\) is an \(A\)-coring.  The
    \(A\)-bimodule structure is given by:
    \begin{align*}
	(a\cdot\phi)(h)  & =a\phi(h)\\
	(\phi\cdot a)(h) & =\phi(h_{(1)})(h_{(2)}\cdot a)
    \end{align*}
    The comultiplication is given by the commutative diagram:
    \begin{center}
	\begin{tikzpicture}
	    \matrix[column sep=0.7mm,row sep=1.4cm]{
		\node (A) {\(\Hom_\bK(H,A)\)}; &                                   & \node (B) {\(\Hom_\bK(H,A)\otimes_A\Hom_\bK(H,A)\)}; \\
		                               & \node (C) {\(\Hom_\bK(H\otimes H,A)\)}; & \\
	    };
	    \draw[dashed,->] (A) --node[above]{\(\Delta\)} (B);
	    \draw[->] (A) --node[below left]{\(\Hom_\bK(m,A)\)} (C);
	    \draw[->] (B) --node[above left]{\(\simeq\)}node[below right]{\(\alpha\)} (C);
	\end{tikzpicture}
    \end{center}
    where the isomorphism
    \(\alpha:\Hom_\bK(H,A)\otimes_A\Hom_\bK(H,A)\sir\Hom_\bK(H\otimes H,A)\) is
    given by: \(\alpha(\phi_1\otimes_A\phi_2)=\left(h\otimes
	    k\selmap{}\phi_1(h_{(1)})h_{(2)}\cdot\phi_2(k)\right)\). 
    The counit is given by \(\epsilon(\phi)=\phi(1_H)\).
\end{proposition}
\begin{proof}
    The right \(A\) module structure is associative since \(\Hom_\k(H,A)\) is
    an \(H\)-module algebra:
    \begin{align*}
            \left((\phi\cdot a)\cdot b\right)(h) & =\left(\phi\cdot a\right)(h_{(1)})h_{(2)}\cdot b \\
                           & =\phi(h_{(1)})(h_{(2)}\cdot a)(h_{(3)}\cdot b) \\
                           & =\phi(h_{(1)})h_{(2)}\cdot (ab) \\
                           & =\left(\phi\cdot(ab)\right)(h)
    \end{align*}
    for any \(a,b\in A\) and \(h\in H\).  Clearly \(\Hom_\k(H,A)\) is a left
    and right \(A\)-module.  It is an \(A\)-bimodule, since both \(A\) module
    structures commute.

    Now, let us note that, by Lemma~\ref{lem:domain}, \(\alpha\) is
    a monomorphism.  It is an isomorphism since the dimension of both domain
    and codomain is \((\dim_\bK H)^2\cdot\dim_\bK A<\infty\).  We show that
    \(\Delta\) is coassociative.  For this we consider the following diagram:
    \begin{center}
	\hspace*{-1.5cm}\begin{tikzpicture}
	    \matrix[column sep=1cm, row sep=1cm]{
				    &                                            & \node (B3) {\(\Hom_\k(H^{\otimes2},A)\otimes_A\Hom_\k(H,A)\)}; & \\
      \node (A1) {\(\Hom_\k(H,A)\)}; & \node (A2) {\(\Hom_\k(H,A)^{\otimes_A2}\)}; & \node (A3) {\(\Hom_\k(H,A)^{\otimes_A3}\)};                   & \node (A4) {\(\Hom_\k(H^{\otimes3},A)\)}; \\
				    &                                            & \node (C3) {\(\Hom_\k(H,A)\otimes_A\Hom_\k(H^{\otimes2},A)\)}; & \\
	    };
	    \draw[->] (A1) --node[above]{\(\Delta\)} (A2);
	    \draw[->] ($(A2)+(15mm,1.1mm)$)  --node[above]{\(\Delta\otimes_A\id\)} ($(A3)+(-15mm,1.1mm)$);
	    \draw[->] ($(A2)+(15mm,-0.5mm)$) --node[below]{\(\id\otimes_A\Delta\)} ($(A3)-(15mm,0.5mm)$);
	    \draw[->] (A2) --node[above left]{\(\Hom_\k(m,A)\otimes_A\id\)}  (B3);
	    \draw[->] (A2) --node[below left]{\(\id\otimes_A\Hom_\k(m,A)\)} (C3);
	    \draw[->] (A3) --node[right]{\(\alpha\otimes_A\id \)} (B3);
	    \draw[->] (A3) --node[right]{\(\id\otimes_A\alpha\)} (C3);
	    \draw[->] (B3) --node[above right]{\(\beta_1\)} (A4);
	    \draw[->] (C3) --node[below right]{\(\beta_2\)} (A4);
	\end{tikzpicture}
    \end{center}
    where \(\beta_1(\Phi\otimes_A\psi)(h\otimes k\otimes
	l)=\Phi(h_{(1)}\otimes k_{(1)})\left((h_{(2)}k_{(2)})\cdot
	    \psi(l)\right)\), and \(\beta_2(\phi\otimes_A\Psi)(h\otimes
	k\otimes l)=\phi(h_{(1)})\left(h_{(2)}\cdot\Psi(k\otimes
	    l)\right)\). 
    Both maps \(\beta_1\) and \(\beta_2\) are monomorphism, which follows
    using the same argument as in the proof of Lemma~\ref{lem:domain}.  Thus
    both \(\beta_1\circ\alpha\otimes_A\id\) and
    \(\beta_2\circ\id\otimes_A\alpha\) are monomorphisms.  Now, \(\Delta\) is
    coassociative if the following equality holds:
    \[\beta_1\circ\Hom_\k(m,A)\otimes_A\id\circ\Delta=\beta_2\circ\id\otimes_A\Hom_\k(m,A)\]
    By definition of \(\Delta\) we have: 
    \begin{equation}\label{eq:counit} 
	\phi_{(1)}(h_{(1)})\left(h_{(2)}\cdot\phi_{(2)}(k)\right)=\phi(hk)
    \end{equation}
    We claim that 
    \[\left(\beta_1\circ\Hom_\k(m,A)\otimes_A\id\circ\Delta\right)(\phi)(h\otimes k\otimes l)=\phi(hkl)\]
    This follows from commutativity of the following diagram:
    \begin{center}
	\hspace*{-1cm}\begin{tikzpicture}
	    \matrix[column sep=2.5cm, row sep=1.2cm]{
		&                                           & \node (A3){\(\Hom_\k(H^{\otimes3},A)\)};\\
		& \node (B2){\(\Hom_\k(H^{\otimes2},A)\)}; & \node (B3){\(\Hom_\k(H^{\otimes2},A)^{\otimes_A2}\)};\\
		\node (C1){\(\Hom_\k(H,A)\)}; & \node (C2){\(\Hom_\k(H,A)^{\otimes_A2}\)};     & \node[color=gray] (C3){\(\Hom_\k(H,A)^{\otimes_A3}\)}; \\
	    };
	    \draw[->] (C1) --node[below]{\(\Delta\)} (C2);
	    \draw[->,color=gray] (C2) --node[below]{\(\Delta\otimes_A\id \)} (C3);
	    \draw[->] (C1) --node[above left]{\(\Hom_\k(m,A)\)} (B2);
	    \draw[->] (B2) --node[above left]{\(\Hom_\k(m\otimes\id,A)\)} (A3);
	    \draw[->] (C2) --node[fill=white]{\(\alpha\)} (B2);
	    \draw[->] (C2) --node[fill=white]{\(\Hom_\k(m,A)\otimes\id\)} (B3);
	    \draw[->,color=gray] (C3) --node[right]{\(\alpha\otimes\id \)} (B3);
	    \draw[->] (B3) --node[right]{\(\beta_1\)} (A3);
	\end{tikzpicture}
    \end{center}
    The parallelogram diagram commutes since:
    \begin{center}
	\hspace*{-1cm}\begin{tikzpicture}
	    \matrix[column sep=2.5cm, row sep=1.5cm]{
		\node (C){\(\left(h\otimes k\mapsto\phi(h_{(1)})\,h_{(2)}\cdot\phi(k)\right)\)}; & \node (D){\(\Bigl(h\otimes k\otimes l\mapsto \phi(h_{(1)}k_{(1)})\,(h_{(2)}k_{(2)})\cdot \psi(l)\Bigr)\)};\\
		\node (A){\(\phi\otimes_A\psi\)};                           & \node (B){\(\phi\circ m\otimes_A\psi\)}; \\
	    };
	    \draw[|->] (A) --node[fill=white]{\(\alpha\)} (C);
	    \draw[|->] (A) --node[below]{\(\Hom_\k(m,A)\otimes_A\id\)} (B);
	    \draw[|->] (C) --node[above]{\(\Hom_\k(m\otimes\id,A)\)} (D);
	    \draw[|->] (B) --node[fill=white]{\(\beta_1\)} (D);
	\end{tikzpicture}
    \end{center}
    In a similar way we compute
    \(\left(\beta_2\circ\id\otimes_A\Hom_\k(m,A)\circ\Delta\right)(\phi)(h\otimes
	k\otimes l)=\phi(hkl)\).  Thus, indeed,
    \(\beta_1\circ\Hom_\k(m,A)\otimes_A\id\circ\Delta=\beta_2\circ\id\otimes_A\Hom_\k(m,A)\circ\Delta\),
    and hence \(\Delta\) is coassociative.  The counit axiom follows from
    Equation~\eqref{eq:counit}.
\end{proof}
\index{module algebra!Galois extension}
\begin{definition}
    Let \(A\) be an \(H\)-module algebra, which is a domain and such that
    \(H\) acts through monomorphisms (Definition~\ref{defi:mono_action}).  If
    the canonical map:
    \index{module algebra!cannonical map}
    \begin{equation}\label{eq:canonical_map}
	\can:A\otimes_{A^H}A\ir\Hom_\k(H,A),\quad \can(a\otimes a')=\left(h\selmap{}a(h\cdot a')\right)
    \end{equation}
    is a bijection then the extension \(A/A^H\,\) is called
    \(\Hom_\k(H,A)\)-Galois.
\end{definition}

\index{module algebra!cannonical map!coring morphism}
\begin{proposition}
    Let \(A\) be an \(H\)-module algebra, which is a domain and such that
    \(H\) acts through monomorphisms.  Then the canonical
    map~\eqref{eq:canonical_map} is a morphism of corings.
\end{proposition}
\begin{proof}
    First we show that \(\can\) is compatible with comultiplication, i.e.
    that the following diagram commutes:
    \begin{center}
	\begin{tikzpicture}
	    \matrix[column sep=2cm, row sep=7mm]{
		\node (A) {\(A\otimes_{A^H}A\)};                 & \node (B) {\(\Hom_\k(H,A)\)};\\
		& \node (E) {\(\Hom_\k(H^{\otimes2},A)\)};\\
		\node (C) {\((A\otimes_{A^H}A)\otimes_A(A\otimes_{A^H}A)\)}; & \node (D) {\(\Hom_\k(H,A)\otimes_A\Hom_\k(H,A)\)};\\
	    };
	    \draw[->] (A) --node[above]{\(\can\)} (B);
	    \draw[->] (C) --node[below]{\(\can\otimes_A\can\)} (D);

	    \draw[->] (A) -- (C);
	    \draw[->] (B) --node[right]{\(\Hom_\k(m,A)\)} (E);
	    \draw[->] (D) --node[right]{\(\alpha\)} (E);
	\end{tikzpicture}
    \end{center}
    For this we take \(x\otimes y\in A\otimes_{A^H}A\).  Then
    \[\Hom_\k(m,A)\circ\can(x\otimes y)=\left(h\otimes k\selmap{}x\,(hk)\cdot y\right)\]
    while
    \begin{align*}
	\can\otimes_A\can((x\otimes1)\otimes_A(1\otimes y) & = \left((h\selmap{}x\,h\cdot1)\otimes_A(k\selmap{}k\cdot y)\right)
    \end{align*}
    and thus \(\alpha\circ\can\otimes_A\can\left((x\otimes1)\otimes_A(1\otimes
	    y)\right)=\left(h\otimes
	    k\selmap{}x\,\epsilon(h_{(1)})(h_{(2)}k)\cdot y\right)\).  Thus
    the equality follows.  Moreover, the morphism \(\can\) is compatible with
    the counit: \(\epsilon\circ\can(x\otimes y)=xy=\epsilon(x\otimes y)\).
\end{proof}
If \(K\) is a right coideal (unital) subalgebra of \(H\) then
\(\Hom_\k(K,A)\) is a quotient coring of \(\Hom_\k(H,A)\).  The quotient
map is induced by the inclusion \(i:K\subseteq H\) and the
comultiplication is given by:
\begin{center}
    \begin{tikzpicture}
	\matrix[column sep=1cm,row sep=1cm]{
	    \node (A) {\(\Hom_\k(K,A)\)}; &                                 & \node (B) {\(\Hom_\k(K,A)\otimes_A\Hom_\k(K,A)\)}; \\
	    & \node (C) {\(\Hom_\k(K\otimes K,A)\)}; & \\
	};
	\draw[dashed,->] (A) --node[above]{\(\Delta\)} (B);
	\draw[->] (A) --node[below left]{\(\Hom_\k(m,A)\)} (C);
	\draw[->] (B) --node[above left]{\(\simeq\)}node[below right]{\(\alpha\)} (C);
    \end{tikzpicture}
\end{center}
where \(\alpha\) is defined by \(\alpha(\phi_1\otimes_A\phi_2)\left(h\otimes
	k\right)=\phi_1(h_{(1)})\,h_{(2)}\cdot\phi_2(k)\).  It is well defined
since \(K\) is a right coideal and it is an isomorphism by the same argument
as the one used to prove Lemma~\ref{lem:domain}.  The counit of \(\Hom(K,A)\)
is given by \(\epsilon(\phi)=\phi(1_H)\) (\(K\) is a unital subalgebra hence
\(1_H\in K\)).
\index{module coalgebra!Galois connection}
\begin{proposition}\label{prop:coring_connection}
    Let \(A\) and \(H\) be as above.  Then there exists a Galois connection
    \((\Theta,\Upsilon)\):
    \begin{center}
	\hfill\begin{tikzpicture}
	    \matrix[column sep=1cm, row sep=1cm]{
		\node (A) {\(\Sub_\Alg(A)\)}; & \node (B) {\(\qsub(H)\)};\\
		& \node (C) {\(\Quot(\Hom_\k(H,A))\)};\\
	    };
	    \draw[->] (B) --node[right]{\(\Hom(-,A)\)} (C);
	    \draw[->] ($(A)+(10mm,1mm)$) --node[above=-0.5mm]{\(\Psi\)} ($(B)+(-10mm,1mm)$);
	    \draw[<-] ($(A)+(10mm,-0.3mm)$) --node[below]{\(\Phi\)} ($(B)+(-10mm,-0.3mm)$);
	    \draw[->] ($(A)+(.5mm,.7mm)+(0mm,-4mm)$) --node[above right=-1mm]{\(\Upsilon\)} ($(C)+(.5mm,.7mm)+(-17mm,2.7mm)$);
	    \draw[<-] ($(A)+(-.5mm,-.7mm)+(0mm,-4mm)$) --node[below left=-1mm]{\(\Theta\)} ($(C)+(-.5mm,-.7mm)+(-17mm,2.7mm)$);
	\end{tikzpicture}
	\hfill\refstepcounter{equation}\raisebox{1cm}{(\theequation)}\label{diag:coring_connection}
    \end{center}
    where \(\Upsilon\coloneq\Hom_\k(-,A)\circ\Psi\). 
\end{proposition}
\begin{proof}
    The Galois connection \((\Phi,\Psi)\) was constructed in
    Proposition~\ref{prop:H-module_galois_connection}.  To prove that there
    exist Galois connection \((\Theta,\Upsilon)\) it is enough to show
    that \(\Hom_\k(-,A)\) preserves all infima, since all the posets are
    complete.  First let us note that
    \(\Hom_\k(-,A):K\selmap{}\Hom_\k(K,A)\) preserves the order: an
    inclusion \(K_1\subseteq K_2\) induces an epimorphism of corings:
    \(\Hom_\k(K_2,A)\sir\Hom_\k(K_1,A)\) which makes the following
    diagram commute.
    \begin{center}
	\begin{tikzpicture}
	    \matrix[column sep=3mm, row sep=5mm]{
		& \node (A2) {\(\Hom_\k(H,A)\)}; & \\
		\node (B1) {\(\Hom_\k(K_2,A)\)}; &                                 & \node (B3) {\(\Hom_\k(K_1,A)\)};\\
	    };
	    \draw[->>] (A2) -- (B1);
	    \draw[->>] (A2) -- (B3);
	    \draw[->>] (B1) -- (B3);
	\end{tikzpicture}
    \end{center}
    Let \((K_i)_{i\in I}\) be a family of right ideal subalgebras.  Their meet
    in \(\qsub(H)\) is equal to \(\bigcap_{i\in I}K_i\) (see the proof of
    Proposition~\ref{thm:lattice_of_subcomodules}).  We claim that the
    following equility holds \(\bigwedge_{i\in
	    I}\Hom_\k(K_i,A)=\Hom_\k(\bigcap_{i\in I}K_i,A)\).  
    If we show that 
    \begin{equation}\label{eq:intersection}
	\sum_{i}\left\{f\in\Hom_\k(H,A): f|_{K_i}=0\right\}\,=\,\left\{f\in\Hom_\k(H,A): f|_{\bigcap_{i}K_i}=0\right\}
    \end{equation}
    then we get that the sequence:
    \begin{equation*}
	0\sir\left\{f\in\Hom_\k(H,A): f|_{\bigcap_{i}K_i}=0\right\}\sir\Hom_\k(H,A)\sir\bigwedge_{i}\Hom_\k(K_i,A)\sir0
    \end{equation*}
    is exact and the claim follows.  Then we get that
    \(\Hom_\k(-,A)\circ\Psi\) reflects suprema into infima, and hence there
    exists a Galois connection \((\Theta,\Upsilon)\).  Thus it remains to
    prove~\eqref{eq:intersection}.  For this let us observe that an infinite
    intersection of subspaces \(K_i\ (i\in I)\) of a finite dimensional vector
    space \(\Hom_\k(H,A)\) is equal to a finite intersection of \(K_i\ (i\in
	I_0)\), where \(I_0\subseteq I\), \(|I_0|<\infty\) (any intersection
    can be computed as an intersection of a chain \(K_{i_0}\supsetneq
	K_{i_0}\cap K_{i_1}\supsetneq\cdots\) (\(I=\{i_0,i_1,\dots\}\), in
    which every step is a proper inclusion until it stabilises) and this must
    stabilise after finitely many steps, since the space \(\Hom_\k(H,A)\) is
    finite dimensional).  Assuming that~\eqref{eq:intersection} holds for any
    finite intersections we have: 
    \[\sum_{i\in I}K_i^{\perp}\subseteq\left(\bigcap_{i\in I}K_i\right)^{\perp}=\left(\bigcap_{i\in I_0}K_i\right)^\perp=\sum_{i\in I_0}K_i^{\perp}\subseteq\sum_{i\in I}K_i^{\perp}\]
    and thus \eqref{eq:intersection} follows.  To
    prove~\eqref{eq:intersection} in the finite case it is enough to consider
    \(I=\{1,2\}\).  The claim follows since both \(K_1^\perp+K_2^\perp\) and
    \((K_1\cap K_2)^\perp\) are cokernels of the inclusion \((K_1+K_2)^\perp
	\sir K_1^\perp\oplus K_2^\perp,\ (f\selmap{}f\oplus f)\).  In both
    cases the cokernel map sends \(f\oplus g\) to \(f-g\).  It is clear that
    the sequence
    \[0\sir(K_1+K_2)^\perp\sir K_1^\perp\oplus K_2^\perp\sir K_1^\perp+K_2^\perp\sir0\]
    is exact.  It remains to show that
    \[0\sir(K_1+K_2)^\perp\sir K_1^\perp\oplus K_2^\perp\sir (K_1\cap K_2)^\perp\sir0\]
    is exact.  Here the difficulty lies in showing that the map
    \(K_1^\perp\oplus K_2^\perp\ni f\oplus g\selmap{}f-g\in (K_1\cap
	K_2)^\perp\) is an epimorphism.  We can write \(H\) as a direct sum
    \(A\oplus B_1\otimes B_2\oplus C\) where \(A,B_i,C\) are subspaces such
    that \(A=K_1\cap K_2\), \(B_i\oplus A= K_i\) and \(C\) is the complement
    of \(K_1+K_2\) in \(H\).  Let \(f\in(K_1\cap K_2)^\perp\), and let \(g\) be
    such that \(g|_{B_1}=-f|_{B_1}\) and \(g|_{A\oplus B_2\oplus C}=0\).  Then
    \(f=(f+g)-g\) and \(f+g\in K_1^\perp\) and \(g\in  K_2^\perp\).
\end{proof}
\index{module coalgebra!Galois connection}
\begin{corollary}\label{cor:coring_closed}
    Under the assumptions of the previous Proposition, the morphism
    \(\Hom_\k(-,A)\) restricts to a bijection from the closed elements of
    \(\qsub(H)\) in \((\Phi,\Psi)\) to the closed elements of
    \(\Quot(\Hom_\k(H,A))\) in \((\Theta,\Upsilon)\). 
\end{corollary}
\begin{proof}
    Clearly, \(\Hom_\k(-,A):\qsub(H)\sir\Quot(\Hom_\k(H,A))\) is an injective
    map, which maps closed elements of \(\qsub(H)\) (in \((\Phi,\Psi)\)) to
    closed elements of the poset \(\Quot(\Hom_\k(H,A))\) (in
    \((\Theta,\Upsilon)\)).  If \(\mathcal{K}\in\Quot(\Hom_\k(H,A))\) is
    a closed element then there exists \(B\in\Sub_\Alg(A)\) such that
    \(\mathcal{K}=\Upsilon(B)=\Hom_\k(\Psi(B),A)\).  Thus \(\Hom_\k(-,A)\) is
    indeed a bijection between the sets of closed elements.
\end{proof}
\begin{proposition}\label{prop:mono_coring}
    Let \(A\) be a domain and an \(H\)-module algebra, such that the
    \(H\)-action is through monomorphisms (Definition~\ref{defi:mono_action}).
    Moreover, let us assume that the canonical map~\eqref{eq:canonical_map} is
    an epimorphism.  Let \(K,K'\in\qsub(H)\) be such that \(\can_K\) and
    \(\can_{K'}\) are isomorphisms.  Then \(\Hom_\k(K,A)=\Hom_\k(K',A)\) whenever
    \(A^{K}=A^{K'}\)
\end{proposition}
Let us note that the above proposition holds even if \(A\) or \(H\) are
infinite dimensional over the base field \(\k\).
\begin{proof}
    Let \(i:K\subseteq H\) and \(i':K'\subseteq H\) be the inclusions.  The
    proposition follows from the commutative diagram: 
    \begin{center}
	\begin{tikzpicture}
	    \matrix[column sep=1cm, row sep=1cm]{
		&                             & \node (A3) {\(\Hom_\k(K,A)\)};\\
		\node (B1) {\(\substack{A\otimes_{A^K}A\\=A\otimes_{A^{K'}}A}\)}; & \node (B2) {\(A\otimes_{A^H}A\)}; & \node (B3) {\(\Hom_\k(H,A)\)};\\
		&                             & \node (C3) {\(\Hom_\k(K',A)\)};\\
	    };
	    \draw[->>] (B2) --node[above]{\(\can_H\)} (B3);
	    \draw[->>] (B2) -- (B1);
	    \draw[->] (B3) --node[right]{\(\Hom_\k(i,A)\)} (A3);
	    \draw[->] (B3) --node[right]{\(\Hom_\k(i',A)\)} (C3);
	    \draw[->] (B1) --node[above left]{\(\can_K\)}node[below right]{\(\simeq\)} (A3);
	    \draw[->] (B1) --node[below left]{\(\can_{K'}\)}node[above right]{\(\simeq\)} (C3);
	\end{tikzpicture}
    \end{center}
    The map \(f=\can_K\circ\can_{K'}^{-1}\) is a map of corings what easily
    follows from commutativity of the diagram:
    \begin{center}
	\hspace*{-3mm}\begin{tikzpicture}
	    \begin{scope}
		\node (A1) at (0cm,0cm) {\(\Hom_\k(K',A){_A\otimes_A}\Hom_\k(K',A)\)};
		\node (A2) at (3.5cm,1.75cm) {\(\Hom_\k(H,A){_A\otimes_A}\Hom_\k(H,A)\)};
		\node (A3) at (7cm,0cm) {\(\Hom_\k(K,A){_A\otimes_A}\Hom_\k(K,A)\)};
		\draw[->>] (A1) --node[below]{\(f\otimes f\)} (A3);
		\draw[->>] (A2) --node[pos=0.6,above left]{\(\Hom_\k(i',A)\otimes\Hom_\k(i',A)\)} (A1);
		\draw[->>] (A2) --node[pos=0.6,above right]{\(\Hom_\k(i,A)\otimes\Hom_\k(i,A)\)} (A3);
		\node (B1) at (0cm,4.5cm) {\(\Hom_\k(K',A)\)};
		\node (B2) at (3.5cm,6.25cm) {\(\Hom_\k(H,A)\)};
		\node (B3) at (7cm,4.5cm) {\(\Hom_\k(K,A)\)};
		\draw[->>] (B1) --node[pos=0.6,above]{\(f\)} (B3);
		\draw[->>] (B2) --node[above left]{\(\Hom_\k(i',A)\)} (B1);
		\draw[->>] (B2) --node[above right]{\(\Hom_\k(i,A)\)} (B3);
		\draw[->] (B1) --node[left]{\(\Delta_{\Hom_\k(K',A)}\)} (A1);
		\draw[->] (B2) --node[pos=0.6,right]{\(\Delta_{\Hom_\k(H,A)}\)} (A2);
		\draw[->] (B3) --node[right]{\(\Delta_{\Hom_\k(K,A)}\)} (A3);
	    \end{scope}
	\end{tikzpicture}
    \end{center}
    In this way we have proved that \(\Hom_\k(K',A)\geq\Hom_\k(K,A)\) in the
    poset \(\Quot(\Hom_\k(H,A))\).  In a similar way we get
    \(\Hom_\k(K,A)\geq\Hom_\k(K',A)\) and thus they are equal.
\end{proof}
\index{module coalgebra!Galois connection!closed elements}
\begin{corollary}\label{cor:coring_closed_el}
    Let \(H\) and \(A\) be as in Proposition~\ref{prop:mono_coring}.  Let
    \(K\in\qsub(H)\) be such that \(\can_K\) is an isomorphism.  Then
    \(\Hom_\k(K,A)\) is a closed element of \(\Quot(\Hom_\k(H,A))\) in
    \((\Theta,\Psi)\) and thus, by Corollary~\ref{cor:coring_closed}, \(K\) is
    a closed element of the lattice \(\qsub(H)\) in \((\Phi,\Psi)\). 
\end{corollary}
\begin{proof}
    Let \(\widetilde{K}=\Psi\circ\Theta(\Hom_\k(K,A))\). 
    Then, by Corollary~\ref{cor:coring_closed}, \(\widetilde{K}\) is the
    smallest closed element covering \(K\) in \(\qsub(K)\) and thus
    \(A^{\widetilde{K}}=A^K\).  Let \(i:K\subseteq\widetilde{K}\) be the
    inclusion.  We have the following commutative diagram:
    \begin{center}
	\begin{tikzpicture}
	    \matrix[column sep=1cm,row sep=1cm]{
		\node (A1) {\(A\otimes_{A^H}A\)};              & \node (A2) {\(\Hom_\k(H,A)\)};\\
		\node (B1) {\(A\otimes_{A^{\widetilde{K}}}A\)}; & \node (B2) {\(\Hom_\k(\widetilde{K},A)\)};\\
		\node (C1) {\(A\otimes_{A^K}A\)};              & \node (C2) {\(\Hom_\k(K,A)\)};\\
	    };
	    \draw[->>] (A1) --node[above]{\(\can_H\)} (A2);
	    \draw[->]  (B1) --node[above]{\(\can_{\widetilde{K}}\)} (B2);
	    \draw[->]  (C1) --node[below]{\(\can_{K}\)}node[above]{\(\simeq\)} (C2);
	    \draw[->>] (A1) -- (B1);
	    \draw[->>] (A2) -- (B2);
	    \draw[->>] (B2) --node[right]{\(\Hom_\k(i,A)\)} (C2);
	    \draw[->]  (B1) --node[left]{\(=\)} (C1);
	\end{tikzpicture}
    \end{center}
    From the upper commutative square it follows that \(\can_{\widetilde{K}}\)
    is an epimorphism.  From the lower commutative square we get that it is
    also a monomorphism, thus it is an isomorphism.  We get that
    \(\Hom_\k(i,A)\) is an isomorphism.  Thus
    \(\Hom_\k(K,A)=\Hom_\k({\widetilde{K}},A)\) is closed. 
\end{proof}

\index{module coalgebra!Galois connection!closed elements}
\begin{theorem}
    Let \(H\) be a finite dimensional Hopf algebra, \(A\) -- a domain and an
    \(H\)-module algebra, such that \(H\) acts through monomorphisms.
    Furthermore, let us assume that \(A\) is \(\Hom_\k(H,A)\)-Galois.  Then
    \(S\in\Sub_\Alg(A)\) is a closed element of the Galois connection
    \((\Phi,\Psi)\) if the map
    \[\can_S:A\otimes_SA\sir A\otimes_{A^{\Psi(S)}}A\sir\Hom(\Psi(S),A)\] 
    is an isomorphism and \(A\) is
    faithfully flat as a right or left \(S\)-module.  Conversely, if \(S\) is
    closed then \(\can_S\) is an isomorphism.
\end{theorem}
\begin{proof}
    Since \(H\) is finite dimensional, \(A\) is a right \(H^*\)-comodule
    algebra, and \(A^{co\,H^*}=A^H\).  For \(K\in\qsub(H)\)
    the dual \(K^*\in\qquot(H^*)\) and we have a commutative diagram:
    \begin{center}
	\begin{tikzpicture}
	    \matrix[column sep=1cm, row sep=1cm]{
		\node (A1) {\(A\otimes_{A^K}A\)}; & \node (A2) {\(\Hom(K,A)\)}; \\
		\node (B1) {\(A\otimes_{A^{co\,K^*}}A\)}; & \node (B2) {\(A\otimes K^*\)}; \\ 
	    };
	    \draw[->] (A1) --node[above,rotate=90]{\(=\)} (B1);
	    \draw[->] (A1) --node[above]{\(\can_K\)} (A2);
	    \draw[->] (A2) --node[above,rotate=-90]{\(\cong\)} (B2);
	    \draw[->] (B1) --node[below]{\(\can_{K^*}\)} (B2); 
	\end{tikzpicture}
    \end{center}
    Since \(\can_H\) is an isomorphism, \(\can_{H^*}\) is an isomorphism as
    well, hence \(\can_{K^*}\) is  an epimorphism and
    by~\cite[Cor.~3.3]{ps-hs:gen-hopf-galois} it is an isomorphism.  By the
    above commutative diagram \(\can_K\) is an isomorphism.  Hence, if
    \(S=A^{\Psi(S)}\) then \(\can_{\Psi(S)}\) is an isomorphism.  Now, if
    \(\can_S\) is an isomorphism, then since \(\can_{\Psi(S)}\) is an
    isomorphism we must have \(A\otimes_SA=A\otimes_{A^{\Psi(S)}}A\), since
    \(A\) is a domain and by Remark~\ref{rem:faithfully_flatness}(ii) we have
    \(S=A^{\Psi(S)}\). 
\end{proof}
\label{chap:coring_approach_end}

\bibliographystyle{plainnat}
\bibliography{Mat}

\begin{thebibliography}{46}
\providecommand{\natexlab}[1]{#1}
\providecommand{\url}[1]{\texttt{#1}}
\expandafter\ifx\csname urlstyle\endcsname\relax
  \providecommand{\doi}[1]{doi: #1}\else
  \providecommand{\doi}{doi: \begingroup \urlstyle{rm}\Url}\fi

\bibitem[Abe(1980)]{ea:hopf_algebras}
Eiichi Abe.
\newblock \emph{Hopf algebras}, volume~74 of \emph{Cambridge Tracts in
  Mathematics}.
\newblock Cambridge University Press, Cambridge, 1980.
\newblock ISBN 0-521-22240-0.
\newblock Translated from the Japanese by Hisae Kinoshita and Hiroko Tanaka.

\bibitem[Birkhoff and Frink(1948)]{gb-of:represantations_of_lattices}
Garrett Birkhoff and Orrin Frink, Jr.
\newblock Representations of lattices by sets.
\newblock \emph{Trans. Amer. Math. Soc.}, 64:\penalty0 299--316, 1948.
\newblock ISSN 0002-9947.

\bibitem[Blattner and Montgomery(1989)]{rb-sm:crossed-products}
Robert~J. Blattner and Susan Montgomery.
\newblock Crossed products and {G}alois extensions of {H}opf algebras.
\newblock \emph{Pacific J. Math.}, 137\penalty0 (1):\penalty0 37--54, 1989.
\newblock ISSN 0030-8730.
\newblock URL
  \url{http://projecteuclid.org/getRecord?id=euclid.pjm/1102650535}.

\bibitem[Blattner et~al.(1986)Blattner, Cohen, and
  Montgomery]{rb-mc-sm:crossed_products}
Robert~J. Blattner, Miriam Cohen, and Susan Montgomery.
\newblock Crossed products and inner actions of {H}opf algebras.
\newblock \emph{Trans. Amer. Math. Soc.}, 298\penalty0 (2):\penalty0 671--711,
  1986.
\newblock ISSN 0002-9947.

\bibitem[Borceux and Janelidze(2001)]{fb-gj:galois-theories}
Francis Borceux and George Janelidze.
\newblock \emph{Galois {T}heories}, volume~72 of \emph{Cambridge Studies in
  Advanced Mathematics}.
\newblock Cambridge University Press, Cambridge, 2001.
\newblock ISBN 0-521-80309-8.

\bibitem[Brzezi{\'n}ski(2002)]{tb:the_structure_of_corings}
Tomasz Brzezi{\'n}ski.
\newblock The structure of corings: induction functors, {M}aschke-type theorem,
  and {F}robenius and {G}alois-type properties.
\newblock \emph{Algebr. Represent. Theory}, 5\penalty0 (4):\penalty0 389--410,
  2002.
\newblock ISSN 1386-923X.

\bibitem[Brzezi{\'n}ski and Wisbauer(2003)]{tb-rw:corings-and-comodules}
Tomasz Brzezi{\'n}ski and Robert Wisbauer.
\newblock \emph{Corings and Comodules}.
\newblock London Mathematical Society Lecture Notes Series 309. Cambridge
  University Press, New York, 2003.

\bibitem[Burris and Sankappanavar(1981)]{sb-hs:universal_algebra}
Stanley Burris and H.~P. Sankappanavar.
\newblock \emph{A course in universal algebra}, volume~78 of \emph{Graduate
  Texts in Mathematics}.
\newblock Springer-Verlag, New York, 1981.
\newblock ISBN 0-387-90578-2.

\bibitem[Chase and Sweedler(1969)]{sc-ms:hopf-algebras-and-galois-theory}
Stephen~U. Chase and Moss~E. Sweedler.
\newblock \emph{{H}opf algebras and {G}alois theory}.
\newblock Lecture Notes in Mathematics, Vol. 97. Springer-Verlag, Berlin, 1969.

\bibitem[Chase et~al.(1965)Chase, Harrison, and
  Rosenberg]{sc-dh-ar:galois-theory}
Stephen~U. Chase, D.~K. Harrison, and Alex Rosenberg.
\newblock Galois theory and {G}alois cohomology of commutative rings.
\newblock \emph{Mem. Amer. Math. Soc. No.}, 52:\penalty0 15--33, 1965.
\newblock ISSN 0065-9266.

\bibitem[D{\u{a}}sc{\u{a}}lescu et~al.(2001)D{\u{a}}sc{\u{a}}lescu,
  N{\u{a}}st{\u{a}}sescu, and Raianu]{sd-cn-sr:hopf-alg}
Sorin D{\u{a}}sc{\u{a}}lescu, Constantin N{\u{a}}st{\u{a}}sescu, and
  {\c{S}}erban Raianu.
\newblock \emph{Hopf algebras}, volume 235 of \emph{Monographs and Textbooks in
  Pure and Applied Mathematics}.
\newblock Marcel Dekker Inc., New York, 2001.
\newblock ISBN 0-8247-0481-9.
\newblock An introduction.

\bibitem[Davey and Priestley(2002)]{bd-hp:introduction-to-lattices}
B.~A. Davey and H.~A. Priestley.
\newblock \emph{Introduction to lattices and order}.
\newblock Cambridge University Press, New York, second edition, 2002.
\newblock ISBN 0-521-78451-4.

\bibitem[Demazure and Gabriel(1970)]{md-pg:groupes-algebriques}
Michel Demazure and Pierre Gabriel.
\newblock \emph{Groupes alg\'ebriques. {T}ome {I}: {G}\'eom\'etrie
  alg\'ebrique, g\'en\'eralit\'es, groupes commutatifs}.
\newblock Masson \& Cie, \'Editeur, Paris, 1970.
\newblock Avec un appendice {{\i}t Corps de classes local} par Michiel
  Hazewinkel.

\bibitem[Doi and Takeuchi(1986)]{yd-mt:cleft-comodule-algebras}
Yukio Doi and Mitsuhiro Takeuchi.
\newblock Cleft comodule algebras for a bialgebra.
\newblock \emph{Comm. Algebra}, 14\penalty0 (5):\penalty0 801--817, 1986.
\newblock ISSN 0092-7872.

\bibitem[Gr{\"a}tzer(1998)]{gg:lattice-theory}
George Gr{\"a}tzer.
\newblock \emph{General lattice theory}.
\newblock Birkh\"auser Verlag, Basel, second edition, 1998.
\newblock ISBN 3-7643-5239-6; 3-7643-6996-5.
\newblock New appendices by the author with B. A. Davey, R. Freese, B. Ganter,
  M. Greferath, P. Jipsen, H. A. Priestley, H. Rose, E. T. Schmidt, S. E.
  Schmidt, F. Wehrung and R. Wille.

\bibitem[Greither and Pareigis(1987)]{cg-bp:separable_field_extensions}
Cornelius Greither and Bodo Pareigis.
\newblock Hopf {G}alois theory for separable field extensions.
\newblock \emph{J. Algebra}, 106\penalty0 (1):\penalty0 239--258, 1987.
\newblock ISSN 0021-8693.

\bibitem[Heckenberger and
  Kolb(2012)]{ih-ck:homogeneous_right_coideal_subalgebras}
Istv\'{a}n Heckenberger and Stefan Kolb.
\newblock {H}omogeneous right coideal subalgebras of quantized enveloping
  algebras.
\newblock \emph{Bull. London. Math. Soc.}, 44:\penalty0 837--848, 2012.

\bibitem[Heckenberger and
  Schneider(2009)]{ih-hs:right_coideal_subalgebras_of_Nichols_algebras}
Istv\'{a}n Heckenberger and Hans-J{\"u}rgen Schneider.
\newblock Right coideal subalgebras of {N}ichols algebras and the {D}uflo order
  on the {W}eyl groupoid.
\newblock ar{X}iv:0909.0293 [math.QA] 1 Sep 2009, 2009.

\bibitem[Herbera and Trlifaj(2009)]{dh-jt:mittag-leffler}
Dolors Herbera and Jan Trlifaj.
\newblock Almost free modules and {M}ittag-{L}effler conditions.
\newblock ar{X}iv:math.RA/0910.4277v, 2009.

\bibitem[Jacobson(1985)]{nj:basic_algebra_I}
Nathan Jacobson.
\newblock \emph{Basic algebra. {I}}.
\newblock W. H. Freeman and Company, New York, second edition, 1985.
\newblock ISBN 0-7167-1480-9.

\bibitem[Kharchenko(2011)]{vk:right_coideal_subalgebras_of_Uq(so)}
V.~K. Kharchenko.
\newblock Right coideal subalgebras of
  {$\mathcal{U}_q^+({\mathrm{so}}_{2n+1})$}.
\newblock \emph{J. Eur. Math. Soc. (JEMS)}, 13\penalty0 (6):\penalty0
  1677--1735, 2011.
\newblock ISSN 1435-9855.

\bibitem[Kharchenko and
  Sagahon(2008)]{vk-as:right_coideal_subalgebras_of_Uq(sl)}
V.~K. Kharchenko and A.~V.~Lara Sagahon.
\newblock Right coideal subalgebras in {$\mathcal{U}_q(\mathrm{sl}_{n+1})$}.
\newblock \emph{J. Algebra}, 319\penalty0 (6):\penalty0 2571--2625, 2008.
\newblock ISSN 0021-8693.

\bibitem[Kreimer and Takeuchi(1981)]{hk-mt:hopf-algebras-and-galois-extensions}
H.~F. Kreimer and M.~Takeuchi.
\newblock Hopf algebras and {G}alois extensions of an algebra.
\newblock \emph{Indiana Univ. Math. J.}, 30\penalty0 (5):\penalty0 675--692,
  1981.
\newblock ISSN 0022-2518.

\bibitem[Lam(1999)]{tl-modules}
T.~Y. Lam.
\newblock \emph{Lectures on modules and rings}, volume 189 of \emph{Graduate
  Texts in Mathematics}.
\newblock Springer-Verlag, New York, 1999.
\newblock ISBN 0-387-98428-3.

\bibitem[Letzter(2002)]{gl:coideal_subalgebras_and_quantum_symmetric_pairs}
Gail Letzter.
\newblock Coideal subalgebras and quantum symmetric pairs.
\newblock In \emph{New directions in {H}opf algebras}, volume~43 of \emph{Math.
  Sci. Res. Inst. Publ.}, pages 117--165. Cambridge Univ. Press, Cambridge,
  2002.

\bibitem[Mac~Lane(1998)]{smc:categories-for-the-working-mathematician}
Saunders Mac~Lane.
\newblock \emph{Categories for the working mathematician}, volume~5 of
  \emph{Graduate Texts in Mathematics}.
\newblock Springer-Verlag, New York, second edition, 1998.
\newblock ISBN 0-387-98403-8.

\bibitem[Maszczyk(2007)]{tm:galois-struct}
Tomasz Maszczyk.
\newblock Galois {S}tructures.
\newblock Lecture notes available at \url{www.toknotes.mimuw.edu.pl}, 2007.

\bibitem[Montgomery(1993)]{sm:hopf-alg}
Susan Montgomery.
\newblock \emph{Hopf algebras and their actions on rings}, volume~82 of
  \emph{CBMS Regional Conference Series in Mathematics}.
\newblock Published for the Conference Board of the Mathematical Sciences,
  Washington, DC, 1993.
\newblock ISBN 0-8218-0738-2.

\bibitem[Montgomery(2009)]{sm:hopf-galois-survey}
Susan Montgomery.
\newblock Hopf {G}alois theory: a survey.
\newblock In \emph{New topological contexts for {G}alois theory and algebraic
  geometry ({BIRS} 2008)}, volume~16 of \emph{Geom. Topol. Monogr.}, pages
  367--400. Geom. Topol. Publ., Coventry, 2009.

\bibitem[Raynaud and Gruson(1971)]{mr-lg:platitude-et-projectivte}
Michel Raynaud and Laurent Gruson.
\newblock Crit\`eres de platitude et de projectivit\'e. {T}echniques de
  ``platification'' d'un module.
\newblock \emph{Invent. Math.}, 13:\penalty0 1--89, 1971.
\newblock ISSN 0020-9910.

\bibitem[Roman(2008)]{sr:lattices}
Steven Roman.
\newblock \emph{Lattices and ordered sets}.
\newblock Springer, New York, 2008.
\newblock ISBN 978-0-387-78900-2.

\bibitem[Schauenburg(1996)]{ps:hopf-bigalois}
Peter Schauenburg.
\newblock Hopf bi-{G}alois extensions.
\newblock \emph{Comm. Algebra}, 24\penalty0 (12):\penalty0 3797--3825, 1996.
\newblock ISSN 0092-7872.

\bibitem[Schauenburg(1998)]{ps:gal-cor-hopf-bigal}
Peter Schauenburg.
\newblock Galois correspondences for {H}opf bi-{G}alois extensions.
\newblock \emph{J. Algebra}, 201\penalty0 (1):\penalty0 53--70, 1998.
\newblock ISSN 0021-8693.

\bibitem[Schauenburg and Schneider(2005)]{ps-hs:gen-hopf-galois}
Peter Schauenburg and Hans-J{\"u}rgen Schneider.
\newblock On generalized {H}opf {G}alois extensions.
\newblock \emph{J. Pure Appl. Algebra}, 202\penalty0 (1-3):\penalty0 168--194,
  2005.
\newblock ISSN 0022-4049.

\bibitem[Schneider(1990)]{hs:principal-homogeneuos-spaces}
Hans-J{\"u}rgen Schneider.
\newblock Principal homogeneous spaces for arbitrary {H}opf algebras.
\newblock \emph{Israel J. Math.}, 72\penalty0 (1-2):\penalty0 167--195, 1990.
\newblock ISSN 0021-2172.
\newblock Hopf algebras.

\bibitem[Schneider(1992{\natexlab{a}})]{hs:hopf_galois_extensions}
Hans-J{\"u}rgen Schneider.
\newblock Hopf {G}alois {E}xtensions, {C}rossed {P}roducts, and {C}lifford
  {T}heory.
\newblock In \emph{Advances in {H}opf {A}lgebras}, pages 267--298. Marcel
  Deker, Inc., 1992{\natexlab{a}}.

\bibitem[Schneider(1992{\natexlab{b}})]{hs:normal-bases}
Hans-J{\"u}rgen Schneider.
\newblock Normal basis and transitivity of crossed products for {H}opf
  algebras.
\newblock \emph{J. Algebra}, 152\penalty0 (2):\penalty0 289--312,
  1992{\natexlab{b}}.
\newblock ISSN 0021-8693.

\bibitem[Schneider(1993)]{hs:exact-seq-qg}
Hans-J{\"u}rgen Schneider.
\newblock Some remarks on exact sequences of quantum groups.
\newblock \emph{Comm. Algebra}, 21\penalty0 (9):\penalty0 3337--3357, 1993.
\newblock ISSN 0092-7872.

\bibitem[Skryabin(2007)]{ss:projectivity-over-comodule-algebras}
Serge Skryabin.
\newblock Projectivity and freeness over comodule algebras.
\newblock \emph{Trans. Amer. Math. Soc.}, 359\penalty0 (6):\penalty0 2597--2623
  (electronic), 2007.
\newblock ISSN 0002-9947.

\bibitem[Sweedler(1969)]{ms:hopf-alg}
Moss~E. Sweedler.
\newblock \emph{Hopf algebras}.
\newblock Mathematics Lecture Note Series. W. A. Benjamin, Inc., New York,
  1969.

\bibitem[Takeuchi(1972)]{mt:correspondence}
Mitsuhiro Takeuchi.
\newblock A correspondence between {H}opf ideals and sub-{H}opf algebras.
\newblock \emph{Manuscripta Math.}, 7:\penalty0 251--270, 1972.
\newblock ISSN 0025-2611.

\bibitem[Takeuchi(1979)]{mt:rel-hopf-mod}
Mitsuhiro Takeuchi.
\newblock Relative {H}opf modules --- equivalences and freeness criteria.
\newblock \emph{J.~Algebra}, 60\penalty0 (2):\penalty0 452--471, 1979.
\newblock ISSN 0021-8693.

\bibitem[Ulbrich(1982)]{ku:galoiserweiterungen}
K.-H. Ulbrich.
\newblock Galoiserweiterungen von nicht-kommutativen {R}ingen.
\newblock \emph{Comm. Algebra}, 10\penalty0 (6):\penalty0 655--672, 1982.
\newblock ISSN 0092-7872.

\bibitem[van Oystaeyen and Zhang(1994)]{fo-yz:gal-cor-hopf-galois}
F.~van Oystaeyen and Y.~Zhang.
\newblock Galois-type correspondences for {H}opf-{G}alois extensions.
\newblock \emph{$K$-Theory}, 8\penalty0 (3):\penalty0 257--269, 1994.
\newblock ISSN 0920-3036.

\bibitem[Wisbauer(2004)]{rw:coalgebras-and-bialgebras}
Robert Wisbauer.
\newblock Coalgebras and {B}ialgebras.
\newblock Technical report, The {E}gyptian {M}athematical {S}ociety, {T}he
  {M}athematical {S}ciences {R}esearch {C}enter, 2004.
\newblock Lectures given at {C}airo {U}niversity and the {A}merican
  {U}niversity in {C}airo, available at the homepage of R.Wisbauer.

\bibitem[Wisbauer(2005)]{rw:from_galois_ext_to_galois_comodules}
Robert Wisbauer.
\newblock From {G}alois field extensions to {G}alois comodules.
\newblock In \emph{Advances in ring theory}, pages 263--281. World Sci. Publ.,
  Hackensack, NJ, 2005.

\end{thebibliography}
\printindex
\end{document}